\documentclass[a4paper,12pt,reqno]{article}

\usepackage{newtxtext}
\usepackage{anyfontsize}
\usepackage[T1]{fontenc}
\usepackage{lmodern}

\usepackage[leqno]{amsmath}
\usepackage{amssymb}
\usepackage{amsthm}
\usepackage{booktabs}
\usepackage{enumitem}
\setlist[enumerate,1]{label=\textup{(\roman*)}}
\usepackage[pdftex]{graphicx}
\usepackage{ifthen}
\usepackage{mathtools}
\usepackage[numbers,sort&compress]{natbib}
\usepackage[colorinlistoftodos]{todonotes}
\usepackage[normalem]{ulem}
\usepackage[varvw]{newtxmath}
\usepackage[cal=cm,bb=ams]{mathalpha}
\usepackage{upgreek}
\usepackage{url}

\makeatletter
\def\my@tag@font{\normalsize}
\def\maketag@@@#1{\hbox{\m@th\normalfont\my@tag@font#1}}
\let\amsmath@eqref\eqref
\renewcommand\eqref[1]{{\let\my@tag@font\relax\amsmath@eqref{#1}}}
\makeatother

\usepackage[british]{babel}
\datebritish

\usepackage[babel=true]{microtype}

\usepackage[fontsize=12pt]{scrextend}
\usepackage[twoside=semi, DIV=default]{typearea}

\usepackage{tikz}
\usepackage{tikz-cd}
\usetikzlibrary{calc,math}
\usepgflibrary{shapes.geometric}
\usepgflibrary{shapes.misc}
\usetikzlibrary{positioning}
\usetikzlibrary{decorations}
\usetikzlibrary{arrows}

\definecolor{darkblue}{rgb}{0.0,0.0,0.6}
\usepackage[pdflang=en-UK, colorlinks, urlcolor=darkblue, linkcolor=darkblue, citecolor=darkblue]{hyperref}

\hypersetup{
pdftitle={Cluster structures via representation theory: cluster ensembles, tropical duality, cluster characters and quantisation},
pdfauthor={Jan E. Grabowski and Matthew Pressland},
pdfstartview=FitH
}

\newcommand{\sectionbreak}{\vfill\pagebreak}
\renewcommand{\sectionbreak}{}

\allowdisplaybreaks

\theoremstyle{plain}
\newtheorem{theorem}{Theorem}[section]
\newtheorem*{theorem*}{Theorem}
\newtheorem{proposition}[theorem]{Proposition}
\newtheorem*{proposition*}{Proposition}
\newtheorem{lemma}[theorem]{Lemma}
\newtheorem{corollary}[theorem]{Corollary}

\theoremstyle{definition}
\newtheorem{definition}[theorem]{Definition}
\newtheorem*{question*}{Question}
\theoremstyle{remark}
\newtheorem{remark}[theorem]{Remark}
\newtheorem{example}[theorem]{Example}

\newtheorem*{example*}{Example}
\newtheorem*{examplectd*}{Example (continued)}

\numberwithin{equation}{section}

\newcommand{\cA}{\mathcal{A}}
\newcommand{\cB}{\mathcal{B}}
\newcommand{\cC}{\mathcal{C}}
\newcommand{\cD}{\mathcal{D}}
\newcommand{\cE}{\mathcal{E}}
\newcommand{\cP}{\mathcal{P}}

\newcommand{\cT}{\mathcal{T}}
\newcommand{\cU}{\mathcal{U}}
\newcommand{\cV}{\mathcal{V}}
\newcommand{\cTroot}{\mathcal{T}_{0}}
\newcommand{\cX}{\mathcal{X}}

\newcommand{\complex}{\ensuremath \mathbb{C}}
\newcommand{\bK}{\mathbb{K}} 
\newcommand{\nat}{\ensuremath \mathbb{N}}
\newcommand{\integ}{\ensuremath{\mathbb{Z}}}
\newcommand{\bP}{\mathbb{P}}
\newcommand{\bQ}{\mathbb{Q}}
\newcommand{\bbT}{\mathbb{T}}
\newcommand{\bV}{\mathbb{V}}

\newcommand{\bT}{\mathbf{T}}

\newcommand{\fm}{\mathfrak{m}}

\DeclareMathOperator{\fpmod}{mod}
\DeclareMathOperator{\lfd}{lfd}
\DeclareMathOperator{\Mod}{Mod}
\DeclareMathOperator{\pcMod}{pc}
\DeclareMathOperator{\fd}{fd}

\DeclareMathOperator{\proj}{proj}
\DeclareMathOperator{\coker}{Coker}
\DeclareMathOperator{\per}{per}

\newcommand{\add}{\operatorname{add}}
\newcommand{\addcat}[1]{{#1}^{\textup{\text{add}}}}
\newcommand{\adj}[1]{#1^{\dagger}}
\newcommand{\bdcat}[1]{\mathcal{D}^{\textup{b}}(#1)}
\newcommand{\bhcat}[1]{\mathcal{K}^{\textup{b}}(#1)}
\newcommand{\bigdsum}{\bigoplus}

\newcommand{\bigunion}{\bigcup}
\newcommand{\blank}{\mathord{\char123}}
\newcommand{\card}[1]{\lvert #1 \rvert}
\newcommand{\cind}[2]{\smash{\text{\textup{(co)ind}}\cindspace}{}_{#1}^{#2}}
\newcommand{\cindbar}[2]{\overline{\text{\textup{(co)ind}}\cindspace}{}_{#1}^{#2}}
\newcommand{\stabcind}[2]{\underline{\smash{\text{\textup{(co)ind}}}\cindspace}{}_{#1}^{#2}}
\newcommand{\stabcindbar}[2]{\underline{\smash{\overline{\text{\textup{(co)ind}}}}\cindspace}{}_{#1}^{#2}}
\newcommand{\coind}[2]{\text{\textup{coind}\cindspace}{}{}_{#1}^{#2}}
\newcommand{\coindbar}[2]{{\overline{\text{\textup{coind}}}\cindspace}{}{}_{#1}^{#2}}

\newcommand{\stabcoindbar}[2]{{\underline{\overline{\text{\textup{coind}}}}\cindspace}{}_{#1}^{#2}}
\newcommand{\stabCoindbar}[2]{{\underline{\overline{\text{\textup{Coind}}}}\cindspace}{}_{#1}^{#2}}
\DeclareMathOperator{\Coker}{Coker}
\newcommand{\colim}{\varprojlim}
\newcommand{\cpa}[2]{#1\langle\hspace{-0.1em}\langle #2\rangle\hspace{-0.1em}\rangle}
\newcommand{\cross}{\times}
\newcommand{\curly}[1]{\ensuremath{\mathcal{#1}}}
\newcommand{\curlyform}[2]{\ensuremath \{#1,#2\}}

\newcommand{\cvecset}[3]{\mathbf{c}^{#1}_{#2}(#3)}
\newcommand{\cvecplus}[2]{\cvecset{+}{#1}{#2}}
\newcommand{\cvecminus}[2]{\cvecset{-}{#1}{#2}}
\newcommand{\defeq}{\coloneqq}

\let\degsave\deg
\renewcommand{\deg}{\underline{\degsave}}

\newcommand{\divalg}[1]{D_{#1}}
\newcommand{\dimdivalg}[1]{d_{#1}}
\newcommand{\dsum}{\ensuremath{ \oplus}}
\newcommand{\dual}[1]{\ensuremath {#1}^{*}}
\newcommand{\ddual}[1]{\ensuremath {#1}^{**}}
\newcommand{\End}[2]{\ensuremath \operatorname{End}_{#1}(#2)}

\renewcommand{\epsilon}{\varepsilon}
\newcommand{\evform}[3][]{\ensuremath \ip{#2}{#3}^{#1}_{\mathrm{ev}}}
\newcommand{\exchmat}[1]{B_{#1}}
\newcommand{\exchmatentry}[1]{b_{#1}}
\newcommand{\exchmon}[3]{{#2}_{#1}^{#3}}
\newcommand{\Ext}[4]{\ensuremath \operatorname{Ext}^{#1}_{#2}(#3,#4)}

\newcommand{\ext}[5][]{\ensuremath \dim_{\bK}\operatorname{Ext}^{#2}_{#3}#1(#4,#5#1)}
\newcommand{\extbig}[4]{\ensuremath \dim_{\bK}\operatorname{Ext}^{#1}_{#2}\bigl(#3,#4\bigr)}
\newcommand{\Extfun}[1]{\functor{E}^{#1}}
\newcommand{\functor}[1]{\mathrm{#1}}
\newcommand{\Gabmat}[1]{C_{#1}}
\newcommand{\Gabmatentry}[1]{c_{#1}}
\newcommand{\Gabquiv}[1]{Q(#1)}
\newcommand{\gldim}{\operatorname{gldim}}

\newcommand{\gvecset}[3]{\mathbf{g}^{#1}_{#2}(#3)}
\newcommand{\gvecplus}[2]{\gvecset{+}{#1}{#2}}
\newcommand{\gvecminus}[2]{\gvecset{-}{#1}{#2}}
\newcommand{\halfinteg}{\tfrac{1}{2}\integ}
\newcommand{\Hom}[3]{\ensuremath \operatorname{Hom}_{#1}(#2,#3)}
\renewcommand{\hom}[3]{\ensuremath \dim_{\bK}\operatorname{Hom}_{#1}(#2,#3)}
\newcommand{\stabHom}[3]{\ensuremath \operatorname{\underline{Hom}}_{#1}(#2,#3)}
\newcommand{\Ltens}[1]{\mathbin{\stackrel{\mathbf{L}}{\otimes}_{#1}}}
\newcommand{\id}{\ensuremath \textup{id}}
\newcommand{\idcomp}[1]{\ensuremath{#1}^{\kappa}}
\renewcommand{\iff}{\ensuremath \Longleftrightarrow}
\DeclareMathOperator{\image}{im}

\newcommand{\ind}[2]{\text{\textup{ind}\cindspace}{}_{#1}^{#2}}
\newcommand{\indbar}[2]{{\overline{\text{\textup{ind}}}\cindspace}{}_{#1}^{#2}}

\newcommand{\stabind}[2]{{\underline{\text{\textup{ind}}}\cindspace}{}_{#1}^{#2}}
\newcommand{\stabindbar}[2]{{\underline{\overline{\text{\textup{ind}}}}\cindspace}{}_{#1}^{#2}}
\newcommand{\stabIndbar}[2]{{\underline{\overline{\text{\textup{Ind}}}}\cindspace}{}_{#1}^{#2}}
\newcommand{\inj}{\hookrightarrow}
\newcommand{\intersection}{\mathrel{\cap}}
\newcommand{\ip}[2]{\ensuremath \langle\;\!#1,#2\;\!\rangle}
\newcommand{\bigip}[2]{\ensuremath \bigl\langle\;\!#1,#2\;\!\bigr\rangle}
\newcommand{\irr}[3]{\ensuremath\operatorname{irr}_{#1}(#2,#3)}
\newcommand{\iso}{\ensuremath \cong}
\newcommand{\isoto}{\stackrel{\sim}{\to}}
\DeclareMathOperator{\Ker}{Ker}
\newcommand{\Kgp}[1]{\mathsf{K}_{0}(#1)}
\newcommand{\Kgpbig}[1]{\mathsf{K}_{0}\bigl(#1\bigr)}
\newcommand{\Kgpnum}[1]{\mathsf{K}^{\textup{num}}_{0}(#1)}
\newcommand{\Kgpsplit}[1]{\mathsf{K}_{0}^{\textup{split}}(#1)}

\DeclareMathOperator{\lcm}{lcm}
 
\newcommand{\modlift}[3]{{#1}_{#2}^{#3}}

\newcommand{\onto}{\twoheadrightarrow}
\newcommand{\op}[1]{#1^{\textup{op}}}
\newcommand{\projmod}[2]{\Yonfun{#1}{#2}}
\newcommand{\QGra}[2]{\operatorname{Gr}_{#1}({#2})}

\newcommand{\rad}[2][]{\operatorname{rad}^{#1}_{#2}}
\newcommand{\radfun}[2][]{\operatorname{rad}^{#1}_{#2}\functor{H}^{#2}}

\newcommand{\radHom}[4][]{\ensuremath \operatorname{rad}^{#1}_{#2}(#3,#4)}

\newcommand{\rank}{\ensuremath\operatorname{rank}}
\newcommand{\real}{\ensuremath \mathbb{R}}

\newcommand{\RHom}[3]{\ensuremath \operatorname{\mathbf{R}Hom}_{#1}(#2,#3)}

\newcommand{\seedlat}[1]{\mathsf{#1}}

\newcommand{\simpmod}[2]{S^{#1}_{#2}}

\newcommand{\Spec}{\ensuremath\operatorname{Spec}}
\newcommand{\stab}[1]{\underline{#1}}

\newcommand{\supp}[1]{\mathop{\mathrm{Supp}}( #1 )}

\newcommand{\tensor}{\ensuremath \otimes}

\newcommand{\uf}{\textup{uf}}
\newcommand{\union}{\mathrel{\cup}}
\newcommand{\Var}{\operatorname{Var}}
\newcommand{\vsim}{\rotatebox{90}{$\sim$}}
\newcommand{\Yonfun}[1]{\functor{H}^{#1}}
\newcommand{\opYonfun}[1]{\functor{H}_{#1}}

\newcommand{\Poly}[1]{\mathscr{P}(#1)}
\newcommand{\Laurent}[1]{\mathscr{L}(#1)}
\newcommand{\Frac}[1]{\mathscr{F}(#1)}
\newcommand{\Fracbar}[1]{\overline{\mathscr{F}}(#1)}

\newcommand{\ctsubcat}{\mathrel{\subseteq_{\mathrm{ct}\,}}}

\renewcommand{\leq}{\leqslant}
\renewcommand{\geq}{\geqslant}

\usepackage[only,llbracket,rrbracket]{stmaryrd}
\newcommand{\powser}[2]{#1\llbracket#2\rrbracket}

\newcommand{\indec}{\operatorname{indec}}

\newcommand{\exch}{\operatorname{mut}}

\newcommand{\leftapp}[2]{L_{#1}{#2}}
\newcommand{\leftcok}[2]{C_{#1}{#2}}
\newcommand{\rightapp}[2]{R_{#1}{#2}}
\newcommand{\rightker}[2]{K_{#1}{#2}}
\newcommand{\mut}[2]{\mu_{#1}{#2}}

\newcommand{\canform}[3]{\ip{#1}{#2}_{#3}}
\newcommand{\bigcanform}[3]{\bigip{#1}{#2}_{#3}}
\newcommand{\pform}[3]{\ip{#1}{#2}_{#3}^{\textup{p}}}
\newcommand{\sform}[3]{\ip{#1}{#2}_{#3}^{\textup{s}}}

\newcommand{\sinc}[1]{\iota_{#1}^{\textup{s}}}
\newcommand{\sproj}[1]{\pi_{#1}^{\textup{s}}}

\newcommand{\pproj}[1]{\pi_{#1}^{\textup{p}}}
\newcommand{\pdual}[1]{\delta_{#1}^{\textup{p}}}
\newcommand{\sdual}[1]{\delta_{#1}^{\textup{s}}}

\newcommand{\lefterror}[2]{\ell_{#1}^{#2}}
\newcommand{\righterror}[2]{r_{#1}^{#2}}
\newcommand{\lefterrormap}[4]{(\ell_{#1}^{#2})_{#3}^{#4}}
\newcommand{\righterrormap}[4]{(r_{#1}^{#2})_{#3}^{#4}}

\newcommand{\lefterrorold}[4]{\ell_{#1}^{#2}#4\ifthenelse{\equal{#3}{\blank}}{}{\ifthenelse{\equal{#3}{}}{}{(#3)}}}
\newcommand{\righterrorold}[4]{r_{#1}^{#2}#4\ifthenelse{\equal{#3}{\blank}}{}{\ifthenelse{\equal{#3}{}}{}{(#3)}}}

\newcommand{\Aside}{\mathcal{A}}
\newcommand{\Xside}{\mathcal{X}}
\newcommand{\clucha}[2][]{\mathrm{CC}_{#2}^{#1}}
\newcommand{\Fpoly}{\mathcal{F}}

\newcommand{\infl}{\rightarrowtail}
\newcommand{\defl}{\twoheadrightarrow}
\newcommand{\confl}{\dashrightarrow}
\tikzcdset{
infl/.style={rightarrowtail},
defl/.style={twoheadrightarrow},
confl/.style={dashed},
exists/.style={dotted}}

\newlength{\leftstackrelawd}
\newlength{\leftstackrelbwd}
\def\leftstackrel#1#2{\settowidth{\leftstackrelawd}{${{}^{#1}}$}\settowidth{\leftstackrelbwd}{$#2$}\addtolength{\leftstackrelawd}{-\leftstackrelbwd}\leavevmode\ifthenelse{\lengthtest{\leftstackrelawd>0pt}}{\kern-.5\leftstackrelawd}{}\mathrel{\mathop{#2}\limits^{#1}}}

\renewcommand{\subset}{\subseteq}

\let\midsave\mid
\newcommand{\divides}{\midsave}
\renewcommand{\mid}{:}
\newcommand{\cindspace}{} \newcommand{\cTU}{\cT\!,\,\cU} 

\deffootnote[1em]{1em}{1em}{\textsuperscript{\thefootnotemark}\,}

\title{Cluster structures via representation theory:\\ cluster ensembles, tropical duality, cluster characters and quantisation}
\author{Jan E. Grabowski\footnotemark[2] 
\\ \small{\textit{School of Mathematical Sciences, Lancaster University,}}
\\ \small{\textit{Lancaster, LA1 4YF, UK}}
\and Matthew Pressland\footnotemark[3]
\\ \small{\textit{School of Mathematics and Statistics, University of Glasgow,}}
\\ \small{\textit{Glasgow, G12 8QQ, United Kingdom}}
}
\date{\today}

\begin{document}

\maketitle
\vspace{-1em}

\renewcommand{\thefootnote}{\fnsymbol{footnote}}
\footnotetext[2]{Email: \url{j.grabowski@lancaster.ac.uk}.  Website: \url{http://www.maths.lancs.ac.uk/~grabowsj/}}
\footnotetext[3]
{Email: \url{matthew.pressland@unicaen.fr}. Website: \url{https://mdpressland.github.io}\newline
Current address: Laboratoire de Mathématiques Nicolas Oresme, Université de Caen Normandie, Boulevard Maréchal Juin, 14032 Caen Cedex 5, France}
\renewcommand{\thefootnote}{\arabic{footnote}}
\setcounter{footnote}{0}

\begin{abstract}
We develop a general theory of cluster categories, applying to a $2$-Calabi--Yau extriangulated category $\cC$ and cluster-tilting subcategory $\cT$ satisfying only mild finiteness conditions.
We show that the structure theory of $\cC$ and the representation theory of $\cT$ give rise to the rich combinatorial structures of seed data and cluster ensembles, via Grothendieck groups and homological algebra.
We demonstrate that there is a natural dictionary relating cluster-tilting subcategories and their tilting theory to $\cA$-side tropical cluster combinatorics and, dually, relating modules over $\stab{\cT}$ to the $\cX$-side; here $\stab{\cT}$ is the image of $\cT$ in the triangulated stable category of $\cC$.
Moreover, the exchange matrix associated to $\cT$ arises from a natural map $p_{\cT}\colon\Kgp{\fpmod \stab{\cT}}\to\Kgp{\cT}$ closely related to taking projective resolutions.

Via our approach, we categorify many key identities involving mutation, $\mathbf{g}$-vectors and $\mathbf{c}$-vectors, including in infinite rank cases and in the presence of loops and $2$-cycles.
We are also able to define $\cA$- and $\cX$-cluster characters, which yield $\cA$- and $\cX$-cluster variables when there are no loops or $2$-cycles, and which enable representation-theoretic proofs of cluster-theoretical statements.

Continuing with the same categorical philosophy, we give a definition of a quantum cluster category, as a cluster category together with the choice of a map closely related to the adjoint of $p_{\cT}$.
Our framework enables us to show that any Hom-finite exact cluster category admits a canonical quantum structure, generalising results of Geiß--Leclerc--Schröer.


\end{abstract}

\setcounter{tocdepth}{2}
\tableofcontents

\sectionbreak

\section{Introduction}

Cluster categories arose as structures of interest around 2004, when Fomin--Zelevinsky's theory of cluster algebras \cite{FZ-CA1} came to the attention of representation theorists studying tilting theory, and the notions of cluster-tilting objects and their mutations were introduced \cite{BMRRT} (see also \cite{Iyama-HigherAR,IyamaYoshino}).
There were very rapid developments in both finding the most general framework for this theory \cite{Amiot,FuKeller}, and in the identification of important families of examples \cite{GLS-Rigid,GLS-PFV}.

Simultaneously, Fock and Goncharov \cite{FockGoncharov} brought ideas from geometry, and in particular mirror symmetry, in the form of cluster varieties.  Cluster varieties live on one of two `sides', the $\Aside$-side or the $\Xside$-side (also known \cite{ShenWeng} as the $K_2$ side and the Poisson side, reflecting the natural geometric properties of the varieties).
The work on cluster categories referred to above relates most strongly to the $\Aside$-side in this philosophy.

In this work, we have two main goals.
The first of these is provide a treatment of the theory of cluster categories in the spirit of the Fock--Goncharov approach to cluster varieties, in which we see that the relationship between cluster-tilting subcategories and categories of modules over them gives rise to a tropical duality.
In particular, this allows us to use cluster categories to describe cluster-theoretic phenomena on the $\Xside$-side via `dual' or `mirror' results to those describing the $\Aside$-side.
We use these results to categorify a number of formulæ relating to $\cX$-variables, culminating in the construction of an $\cX$-cluster character.

The second goal is to identify the additional datum needed for quantisation, and hence to give a definition and examples of quantum cluster categories.
We are able to show that a large class of exact cluster categories have quantisations, significantly expanding the class of examples previously known, by showing that this is an emergent feature from the cluster categorical properties rather than specific to particular constructions.

There are a number of distinctive themes running through our approach:

\begin{itemize}
\item \emph{Duality} is at its heart, both as a guiding philosophy and in concrete statements, relating objects of interest (certain Grothendieck groups) and maps between them via adjunction.
\item Another theme is the use of a \emph{basis-free approach}, avoiding wherever possible indexing of elements and instead proving properties of maps and subspaces.
This removes a layer of complication from the presentation of many results.
\item The basis-free approach also enables us to prove many of our results without the assumption of finite rank (that is, additively finite cluster-tilting subcategories).
We identify a \emph{minimal amount of finiteness} needed for the desired results and show that this can be significantly weaker than the usual prevailing assumptions, e.g.\ of finite rank or Hom-finiteness.
\item We also treat the triangulated and (Frobenius) exact cases together, by working in the `greatest common generality' of \emph{extriangulated categories}.
We use homological arguments that are valid in this setting, and thus apply in particular to the special cases of triangulated or exact categories.
\item We work with $\bK$-linear categories but do not assume that $\bK$ is algebraically closed. By doing so, we are able to obtain cluster algebras having \emph{skew-symmetrizable} exchange matrices, rather than restricting to the skew-symmetric case.
\item We return to the tilting theory roots of cluster categories, by studying the relationships between \emph{arbitrary} pairs of cluster-tilting subcategories, not only those related by a single mutation.
We show that these relationships are given by index and coindex maps on the respective Grothendieck groups, which arise from a tilting phenomenon.  

We explain how the failure of these maps to be `transitive' (i.e.\ the composition of two (co)index maps is not equal to another (co)index map) is controlled, with error terms lying in the image of the exchange matrix. Transitivity is restored when we instead look at the mutation of a natural bilinear form associated to the exchange matrix and its image under index and coindex.

This also allows us to transcend the constraints of arguments based on iterated (Fomin--Zelevinsky) mutation.
Indeed, only very occasionally do we need our categorical mutation to align with Fomin--Zelevinsky mutation; essentially this is only needed when we wish to decategorify and make a statement about an associated Fomin--Zelevinsky cluster algebra.
In particular, we do not rely on any theorems about cluster algebras or their combinatorics for our main results.

\item Finally, we gain the freedom to be \emph{largely agnostic about the presence of loops or $2$-cycles}, again unless we wish to decategorify to a Fomin--Zelevinsky cluster algebra.
This opens up the applicability of our results, and hence many cluster-categorical theorems, to a much wider class of examples, such as those coming from geometry, where loops are abundant.
There may still be interesting decategorifications in these cases, and indeed we give an example which decategorifies to a generalised cluster algebra in the sense of Chekhov and Shapiro \cite{CheSha}.
\end{itemize}  

Consequently, our work provides a uniform approach to the methods of additive categorification in cluster theory, through which the various methods used to prove cluster algebraic conjectures may be implemented, often in wider generality than that in which they were originally stated.
For example, by applying the methods of \cite{CILF}, we can straightforwardly deduce the linear independence of cluster monomials via our cluster character, with fewer assumptions on the input cluster category and hence on the cluster algebra it decategorifies to.
We also expect to be able to use this framework to obtain new results, notably on quantum and generalised cluster algebras.

Note that we do not address the issue of constructing a suitable additive categorification of a given cluster algebra (or generalisation thereof); rather, our focus is on what can be deduced from such a categorification when it exists.
The additive categorification programme for cluster algebras associated to skew-symmetric exchange matrices is essentially complete, through work of many authors, originating in \cite{Amiot,BMRRT,Plamondon-ClustCat} for the case of no frozen variables, continuing with \cite{Pressland} for cluster algebras with `enough' frozen variables (to admit categorification via an exact category) and culminating in the most general results to date in \cite{Wu,KellerWu}, where arbitrary collections of frozen variables are handled via Nakaoka--Palu's extriangulated categories \cite{NakaokaPalu}.
Important families of examples, especially those arising in Lie theory, also admit more explicit additive categorifications that can be obtained independently of these general constructions \cite{GLS-PFV,DemonetIyama,JKS,PresslandPostnikov}.
We also note that the categorification problem is addressed via other methods in the monoidal setting, which we do not discuss at all; here, the most general result is to be found in \cite{KKOP}.

We now expand on the above summary, beginning with an explanation of how our cluster categories give rise to (most of) the data of a cluster ensemble, which is the starting point for the study of cluster varieties as instigated in \cite{FockGoncharov} and continued in \cite{GHK,GHKK}.

\subsection{Categorifying cluster ensembles}\label{ss:intro-cl-ensembles}

The notion of a cluster ensemble originated with Fock and Goncharov \cite{FockGoncharov} and is key to the connection between cluster algebras and geometry, in particular mirror symmetry.
Of course, cluster algebras and cluster categories are also intimately related, via (de)categorification.
In this paper, we will make direct connections between categorical and geometric information, helping to illuminate both sides and enabling generalisation of existing constructions.

In this section, we explain the dictionary between cluster ensembles and cluster categories, beginning by recalling the definition of the former \cite{FockGoncharov,GHK}.
To align with existing definitions in the literature and to simplify the exposition here, we will assume that we are in the \emph{finite rank} case; that is, the indexing sets of various data below will be finite.
We emphasise, though, that in the main body of this paper, this assumption is not made, unless stated explicitly, and consequently our results extend much of what follows to the infinite rank case.

A \emph{seed datum} consists of the following pieces of data:
\begin{enumerate}
\item\label{d:seed-datum-lattices} a lattice (that is, a free abelian group) $\seedlat{N}$ of finite rank, a distinguished saturated sublattice $\seedlat{N}_{\uf}$ and a sublattice $\seedlat{N}^{\circ}$ such that $\seedlat{N}/\seedlat{N}^{\circ}$ is torsion;
\item\label{d:seed-datum-forms} a skew-symmetric bilinear form $\curlyform{\blank}{\blank}\colon \seedlat{N}\cross \seedlat{N}\to \bQ$ such that $\curlyform{n_{1}}{n_{2}}=-\curlyform{n_{2}}{n_{1}}$ if $n_{1},n_{2}\in \seedlat{N}_{\uf}$, $\curlyform{\seedlat{N}^\circ}{\seedlat{N}_{\uf}}\subseteq \integ$ and $\curlyform{\seedlat{N}}{\seedlat{N}_{\uf}^{\circ}}\subseteq \integ$ where $\seedlat{N}_{\uf}^{\circ}=\seedlat{N}_{\uf}\intersection \seedlat{N}^{\circ}$; and
\item\label{d:seed-datum-bases} a basis $\{ e_{i} \mid i\in I \}$ of $\seedlat{N}$ such that $\{ e_{i} \mid i\in I_{\uf} \}$ is a basis of $\seedlat{N}_{\uf}$ for some $I_{\uf}\subseteq I$ and such that there exist $d_{i}\in \integ_{>0}$ such that $\{ d_{i}e_{i} \mid i\in I \}$ is a basis of $\seedlat{N}^{\circ}$.
\end{enumerate}

Here saturation of $\seedlat{N}_{\uf}$ inside $\seedlat{N}$ is a technical condition, meaning that $\seedlat{N}/\seedlat{N}_{\uf}$ is again free abelian.
Since $\seedlat{N}$ has finite rank, $\seedlat{N}/\seedlat{N}^{\circ}$ being torsion is equivalent to $\seedlat{N}^{\circ}$ being a finite index sublattice (but the former condition is weaker when $\seedlat{N}$ has infinite rank).

From a seed datum, one also obtains
\begin{enumerate}
\setcounter{enumi}{3}
\item\label{d:seed-datum-dual-lattices} $\seedlat{M}=\dual{\seedlat{N}}=\Hom{}{\seedlat{N}}{\integ}\subset \Hom{}{\seedlat{N}^{\circ}}{\integ}=\seedlat{M}^{\circ}$ and $\seedlat{M}_{\uf}=\seedlat{M}^{\circ}/\seedlat{N}_{\uf}^{\perp}$;
\item bases $\{ \dual{e_{i}} \mid i\in I \}$ of $\seedlat{M}$ and $\{f_{i}=d_{i}^{-1}\dual{e_{i}} \mid i\in I\}$ of $\seedlat{M}^{\circ}$;
\item\label{d:seed-datum-rescaled-form} a $\integ$-bilinear form $\ip{e_{i}}{e_{j}}=\curlyform{e_{i}}{e_{j}}d_{j}$ (also denoted $\epsilon_{ij}$); and
\item\label{d:seed-datum-maps} maps $\dual{p}_{1}\colon \seedlat{N}_{\uf}^{\circ} \to \seedlat{M},\ n\mapsto \curlyform{n}{\blank}$ and $\dual{p}_{2}\colon \seedlat{N}^{\circ} \to \seedlat{M}_{\uf},\ n\mapsto \curlyform{n}{\blank}|_{\seedlat{N}^{\circ}}$.
\end{enumerate}

We remark briefly that there is a small deviation here from \cite{GHK,GHKK}.
There, the maps are $\dual{p}_{1}\colon \seedlat{N}_{\uf}\to \seedlat{M}^{\circ}$ and $\dual{p}_{2}\colon \seedlat{N}\to \seedlat{M}_{\uf}^{\circ}$.
The shift of the ``$\blank^{\circ}$'' makes essentially no difference to the seed datum; it can be obtained by simultaneous rescaling using the lowest common multiple of the $d_{i}$.
However, we will see below that, in the categorical setting, it is more natural for us to make the other choice as in \ref{d:seed-datum-maps}.

A \emph{cluster ensemble} consists of the data \ref{d:seed-datum-lattices}--\ref{d:seed-datum-maps} above together with the choice of a map $\dual{p}\colon \seedlat{N}^{\circ}\to \seedlat{M}$ that yields $\dual{p}_{1}$ and $\dual{p}_{2}$ under composition with the natural inclusion and projection maps, respectively.

Data of this sort is familiar in toric geometry, since for any lattice $\seedlat{L}$, one has the algebraic torus $\bbT^{\seedlat{L}}\defeq \operatorname{Spec} \complex[\seedlat{L}]$, whose character lattice identifies naturally with $\seedlat{L}$.  In particular, given a cluster ensemble, one has a map of tori $p\colon \bbT^{\seedlat{M}}\to \bbT^{\seedlat{N}^{\circ}}$.

A key insight was that by defining mutation of seed data appropriately, one obtains birational maps between these associated tori, along which one may glue to obtain the \emph{cluster varieties}.
If $s$ is a seed datum, we write a subscript $s$ on $\seedlat{N}$, $\seedlat{M}$ etc., to indicate the lattices associated to that seed.
Then letting $\mathbf{s}$ denote the set of seeds obtained by iterated mutation from some chosen initial seed, we have $\cA=\bigunion_{s\in \mathbf{s}} \bbT^{\seedlat{M}_{s}}$ and $\cX=\bigunion_{s\in \mathbf{s}} \bbT^{\seedlat{N}^{\circ}_{s}}$.
Moreover, these varieties have well-defined \emph{positive parts} $\cA^{>0}$ and $\cX^{>0}$, and the maps $p_{s}$ glue to give $p\colon \cA\to \cX$, restricting to $p^{>0}\colon\cA^{>0}\to\cX^{>0}$.

The motivating class of examples in \cite{FockGoncharov,FockGoncharov-ClusterPoisson} are from (higher) Teichmüller theory: for suitable input data, $\cA^{>0}$ is a decorated Teichmüller space, $\cX^{>0}$ its undecorated analogue and $p^{>0}$ is the map that forgets the decoration.

\begin{remark}
The cluster variety $\cA$ defined in this way is a subset of $\Spec{\mathscr{A}}$ for Fomin--Zelevinsky's cluster algebra $\mathscr{A}$. The complement of $\cA$ in $\Spec{\mathscr{A}}$ is called the \emph{deep locus}, and is studied in \cite{BeyMul,CGSS}; while it has codimension at least $2$, it need not be empty.
It is sometimes $\Spec{\mathscr{A}}$, rather than $\cA$, which is referred to as the $\cA$-cluster variety.
\end{remark}

From seed data, one can extract cluster algebra seeds in the sense of Fomin--Zelevinsky; we will not explain this in detail now, as it will become evident from what follows.
However, we emphasise at this point that the connection back from geometry to cluster algebras is that the coordinate ring of $\cA$ is Fomin--Zelevinsky's upper cluster algebra, which contains the cluster algebra itself (identified as the subalgebra generated by \emph{global monomials}).

We now explain how cluster categories naturally give rise to cluster ensembles.  We first set some notation.  Readers unfamiliar with the level of generality below are encouraged to think of their favourite cluster categories, such as ``classical'' triangulated cluster categories \cite{BMRRT,Amiot}, the exact cluster categories associated to partial flag varieties \cite{BIRS1,GLS-KacMoody}, Grassmannians \cite{JKS} and positroids \cite{PresslandPostnikov}, or Higgs categories \cite{Wu,KellerWu}.

Let $\cC$ denote an algebraic Frobenius extriangulated category over a perfect field $\bK$ (Definitions~\ref{d:Frob-extri} and \ref{d:algebraic}) such that
\begin{itemize}
\item $\cC$ is Krull--Schmidt, enriched in pseudocompact vector spaces (e.g.\ Hom-finite) and $\dimdivalg{X}=\dim \op{\End{\cC}{X}}/\rad{}{\op{\End{\cC}{X}}}<\infty$ for all $X\in \cC$ (Definition~\ref{d:comp-clustcat}, \S\ref{ss:pseudocompact});
\item the stable category $\stab{\cC}$ is a (Hom-finite) 2-Calabi--Yau (triangulated) category (Definition~\ref{d:CY}); and
\item $\cC$ has a cluster-tilting subcategory $\cT$ and all such are radically pseudocompact (Definition~\ref{d:rad-pc}).
\end{itemize}
The final condition on radical pseudocompactness is not overly strong: for example, it holds if $\cT$ is additively finite (i.e.\ $\cT=\add T$ for some object $T$) and $\op{\End{\cC}{X}}$ has a finite Gabriel quiver.  We write $\cT \ctsubcat \cC$ to indicate that $\cT$ is a cluster-tilting subcategory of $\cC$ and denote by $\indec \cT$ the indecomposable objects of $\cT$.

In what follows, we call such categories \emph{compact cluster categories} (Definition~\ref{d:comp-clustcat}).
Some results also hold with fewer or different assumptions---see \S\ref{ss:clust-cats}---but for this exposition, we will assume our cluster categories to be compact and furthermore, to align with the above, we will also assume that the cluster-tilting subcategories of $\cC$ are additively finite (corresponding to the finite rank case).
In particular, $\cC$ has a weak cluster structure in the sense of Definition~\ref{d:weak-cluster-structure}, meaning that its cluster-tilting subcategories may be mutated at any indecomposable non-projective object (i.e.\ they are maximally mutable in the sense of Definition~\ref{d:max-mut}); this follows from Corollary~\ref{c:weak-clust-struct}.
If $\cC$ is a cluster category in this sense, so is its stable category $\stab{\cC}$.

The first key identifications of the data of a cluster ensemble are the lattices in \ref{d:seed-datum-lattices} and \ref{d:seed-datum-dual-lattices}.
For a cluster category $\cC$ as above and $\cT \ctsubcat \cC$, the Grothendieck group $\Kgp{\cT}$ of $\cT$ (as an additive category) is a lattice with basis $\{ [T] \mid T\in \indec \cT \}$.

Dually, one can consider the category $\fd \cT$ of finite-dimensional $\cT$-modules, that is, functors $M\colon \cT \to \fd \bK$ taking values in finite-dimensional vector spaces (\S\ref{s:modules}).  Then $\Kgp{\fd \cT}$ is also a lattice with basis $\{ [\simpmod{\cT}{T}] \mid T\in \indec \cT \}$, where $\simpmod{\cT}{T}$ is the simple functor supported at $T$ (Proposition~\ref{p:KS-simples}).  We also have a natural inclusion $\indec\stab{\cT} \subseteq \indec\cT$ and hence injection $\Kgp{\fd \stab{\cT}}\inj \Kgp{\fd \cT}$.

Furthermore, given a finite-dimensional $\cT$-module $M$ and an object $T$ of $\cT$, we may compute $\canform{M}{T}{\cT}\defeq \dim M(T)$.  Indeed, this yields (Proposition~\ref{p:K-duality}) a non-degenerate $\integ$-bilinear form 
\[ \canform{\blank}{\blank}{\cT}\colon \Kgp{\fd \cT}\cross \Kgp{\cT}\to \integ, \ \canform{[M]}{[T]}{\cT}=\dim M(T),\]
and hence injective maps
\begin{align*}
\pdual{\cT} & \colon \Kgp{{\cT}}\to\dual{\Kgp{\fd{\cT}}},\ \pdual{\cT}[T]=\canform{\blank}{[T]}{\cT}, \\
\sdual{\cT} & \colon \Kgp{\fd{\cT}}\to\dual{\Kgp{{\cT}}},\ \sdual{\cT}[M]=\canform{[M]}{\blank}{\cT},
\end{align*}
allowing us to begin to build our dictionary between cluster ensembles and cluster categories as follows:
\begin{center}
\begin{tabular}{lll} \toprule
\multicolumn{2}{l}{Cluster ensemble} & Cluster category \\ \midrule
$\seedlat{N}$ & lattice & $\dual{\Kgp{\cT}}$ \\
$\seedlat{N}_{\uf}$ & saturated sublattice & $\dual{\Kgp{\stab{\cT}}}$ \\
$\seedlat{N}^{\circ}$ & finite-index sublattice & $\Kgp{\fd \cT}$ \\
$\seedlat{N}_{\uf}^{\circ}$ & $\seedlat{N}_{\uf}\intersection \seedlat{N}^{\circ}$ & $\Kgp{\fd \stab{\cT}}$ \\
$e_{i}$ & basis of $\seedlat{N}$ & $\dual{[T]}$ \\
$I$ & indexing set for basis & $\indec \cT$ \\
$I_{\uf}$ & subset of $I$ & $\indec \stab{\cT}$ \\
$d_{i}$ & multiplier & $d_{T}$ \\
$d_{i}e_{i}$ & basis of $\seedlat{N}^{\circ}$ & $[\simpmod{\cT}{T}]$ \\
$\seedlat{M}$ & dual lattice to $\seedlat{N}$ & $\Kgp{\cT}$ \\
$\seedlat{M}^{\circ}$ & contains $\seedlat{M}$ as finite-index sublattice & $\dual{\Kgp{\fd \cT}}$ \\
$\dual{e_{i}}$ & basis of $\seedlat{M}$ dual to $\{e_{i}\}$ & $[T]$ \\
$f_{i}$ & basis of $\seedlat{M}^{\circ}$ such that $f_{i}=d_{i}^{-1}\dual{e}_{i}$ & $\dual{[\simpmod{\cT}{T}]}$ \\ \bottomrule		
\end{tabular}
\end{center}
In particular, for $T,U\in \indec \cT$, the fact that $\canform{[\simpmod{\cT}{T}]}{[U]}{\cT}=\dim \simpmod{\cT}{T}(U)=\dimdivalg{T}\delta_{TU}$ implies that  $\sdual{\cT}[\simpmod{\cT}{T}]=\dimdivalg{T}\dual{[T]}$ and $\pdual{\cT}[T]=\dimdivalg{T}\dual{[\simpmod{\cT}{T}]}$, so that $\dual{[\simpmod{\cT}{T}]}=\dimdivalg{T}^{-1}\pdual{\cT}[T]$ as required\footnote{In the cluster ensemble setup, there is no obvious reason to explicitly name the two lattice inclusions whereas the natural bases of the Grothendieck groups we consider do make it appropriate to do so.  We see this when comparing, for example, $d_{i}e_{i}$ with $\dimdivalg{T}\dual{[T]}=\sdual{\cT}[\simpmod{\cT}{T}]$. However, we do suppress the inclusions $\Kgp{\stab{\cT}}\subseteq \Kgp{\cT}$ and $\Kgp{\fd \stab{\cT}}\subseteq \Kgp{\fd \cT}$ corresponding to $\blank_{\uf}$, the unfrozen sublattices.}.

Extending results of Palu \cite{Palu} and Fu--Keller \cite{FuKeller}, we show in \S\ref{s:cat-ex-mat} that there is a $\integ$-bilinear form
\[ \sform{\blank}{\blank}{\cT}\colon \Kgp{\fd \cT} \cross \Kgp{\fd \stab{\cT}} \to \integ \]
defined by
\begin{multline}
\tag{\ref{eq:s-form-rewrite2}}
\sform{[M]}{[N]}{\cT}=-\hom{\cT}{M}{N}+\ext{1}{\cT}{M}{N}\\-\ext{1}{\cT}{N}{M}+\hom{\cT}{N}{M}
\end{multline}
and, moreover, the restriction of this form to $\Kgp{\fd \stab{\cT}}\cross \Kgp{\fd \stab{\cT}}$ is skew-symmetric (Lemma~\ref{l:s-form-skew-sym}).  On the basis of simple functors, we have  \begin{equation} \tag{\ref{eq:s-gram-matrix}} \sform{[\simpmod{\cT}{T}]}{[\simpmod{\cT}{U}]}{\cT}=\dimdivalg{T}\exchmatentry{T,U}
\end{equation} where $\exchmatentry{T,U}\in \integ$ is obtained from the Cartan matrix of $\cT$ (Definition~\ref{d:exch-mat}) and is the relevant entry of the usual exchange matrix. 

Let
\[ \curlyform{\blank}{\blank}\colon \dual{\Kgp{\cT}} \cross \dual{\Kgp{\cT}} \to \bQ \]
be defined by
\[\curlyform{\dual{[T]}}{\dual{[U]}}=d_{T}^{-1}d_{U}^{-1}\sform{[\simpmod{\cT}{T}]}{[\simpmod{\cT}{U}]}{\cT}.\] 
Strictly, this formula is not valid unless $U\in\indec \stab{\cT}$; however, we may extend to the case that $T\in\indec{\stab{\cT}}$ by requiring that $\{[T]^*,[U]^*\}=-\{[T]^*,[U]^*\}$, and extend arbitrarily to a skew-symmetric form on $\dual{\Kgp{\cT}}$.
Constructions involving seed data typically only use the restrictions of the form to $\dual{\Kgp{\stab{\cT}}}\times\dual{\Kgp{\cT}}$ and $\dual{\Kgp{\cT}}\times\dual{\Kgp{\stab{\cT}}}$, or equivalently (see below) the maps $p_1^*$ and $p_2^*$, and so are insensitive to this final choice of extension.

To see that $\curlyform{\dual{\Kgp{\stab{\cT}}}}{\sdual{\cT}\Kgp{\fd \cT}}\subseteq \integ$ and $\curlyform{\dual{\Kgp{\cT}}}{\sdual{\cT}\Kgp{\fd \stab{\cT}}}\subseteq \integ$, we calculate
\[\curlyform{\dual{[T]}}{\sdual{\cT}[\simpmod{\cT}{U}]}=\curlyform{\dual{[T]}}{\dimdivalg{U}\dual{[U]}}=\dimdivalg{T}^{-1}\sform{[\simpmod{\cT}{T}]}{[\simpmod{\cT}{U}]}{\cT}=\dimdivalg{T}^{-1}(\dimdivalg{T}\exchmatentry{T,U})=\exchmatentry{T,U}\in \integ, \]
valid provided that at least one of $T$ and $U$ belongs to $\indec \stab{\cT}$, hence the two desired containments.
Thus, $\curlyform{\blank}{\blank}$ is the desired form for \ref{d:seed-datum-forms} above.

In particular, the form $\ip{\dual{[T]}}{\dual{[U]}}\defeq \curlyform{\dual{[T]}}{\dual{[U]}}\dimdivalg{U}=\exchmatentry{T,U}$ is that appearing in \ref{d:seed-datum-rescaled-form}, and has the exchange matrix $\exchmat{\cT}$ as its Gram matrix.
Letting $D_{\cT}$ be the diagonal matrix with entries $\dimdivalg{T}$, we have that $D_{\cT}\exchmat{\cT}$ is the Gram matrix of $\sform{\blank}{\blank}{\cT}$, and in particular has skew-symmetric principal part.

The form $\curlyform{\blank}{\blank}\colon\seedlat{N}_{\uf}\times\seedlat{N}\to\mathbb{Q}$ restricted from that in \ref{d:seed-datum-forms} is equivalent data to the map $p_1^*\colon\seedlat{N}_{\uf}^\circ\to\seedlat{M}$, since $\seedlat{N}_{\uf}/\seedlat{N}_{\uf}^\circ$ is torsion.
From the categorical viewpoint, this map arises more naturally than the form, via a process related to taking projective resolutions.
Indeed, we obtain a map $p_{\cT}\colon\Kgp{\fd\stab{\cT}}\to\Kgp{\cT}$ (Definition~\ref{d:proj-res-map}) by taking projective resolutions of $\stab{\cT}$-modules, in an appropriately enlarged category in which they have finite projective dimension.
Defining $\beta_{\cT}=-p_{\cT}$ (to align with existing cluster theoretic conventions), we may then define
$\sform{\blank}{\blank}{\cT}=\canform{\blank}{\beta_{\cT}(\blank)}{\cT}$ (Definition~\ref{d:s-form}) and prove (Remark~\ref{r:s-form-rewrite}) that this form may also be expressed as above in terms of Hom and Ext spaces.
Under our dictionary, $\beta_{\cT}$ corresponds to the restriction of $p_1^*$ to $\seedlat{N}_{\uf}^\circ$, which uniquely determines $p_1^*$, using again that $\seedlat{N}_{\uf}/\seedlat{N}_{\uf}^\circ$ is torsion.

The map $\dual{p}_2$ is similarly related to an appropriate \emph{adjoint} $\adj{\beta}_{\cT}\colon \Kgp{\fd \cT}\to \Kgp{\stab{\cT}}$ to $\beta_{\cT}$ (Proposition~\ref{p:adjunction}); equivalently, it can be obtained from $\curlyform{\blank}{\blank}$ above as required. 
In this notation, $B_{\cT}$ having skew-symmetrizable principal part is expressed as $\adj{(\pdual{\stab{\cT}}\circ \beta_{\stab{\cT}})}=-(\pdual{\stab{\cT}}\circ \beta_{\stab{\cT}})$ (Corollary~\ref{c:beta-skew-symmetric}).

The above shows how to associate a seed datum to any cluster-tilting subcategory $\cT \ctsubcat \cC$ of a compact cluster category $\cC$.
To complete this to a cluster ensemble, we require a map
\[ \dual{p}\colon \Kgp{\fd \cT} \to \Kgp{\cT} \]
with the property that $\dual{p}\circ \iota_{\uf}^{\seedlat{N}^{\circ}}=\dual{p}_{1}$ and $\iota_{\uf}^{\seedlat{M}}\circ \dual{p}_{2}=\dual{p}$, as is required, where $\iota_{\uf}^{\blank}$ is the inclusion $\blank_{\uf} \inj \blank$.  
The natural candidate is a map defined using projective resolutions as above, now of arbitrary finite-dimensional $\cT$-modules rather than only $\stab{\cT}$-modules.
Various assumptions on $\cC$ and $\cT$ ensure that this is valid, e.g.\ when $\cC$ is exact, and $\cT$ is locally finite with finite global dimension.
If $\cC$ is obtained via a particular construction, such as that of Higgs categories, extensions $\dual{p}$ can also be made using natural dg enhancements.
Hence, in many cases, we obtain a cluster ensemble.

We remark, however, that throughout this work the map $\dual{p}$ will not be needed: it will suffice to study $\beta_{\cT}$ (corresponding to $\dual{p}_{1}$), with $\adj{\beta}_{\cT}$ (corresponding to $\dual{p}_{2}$) also appearing in some situations.

\subsection{Tropical cluster data}\label{ss:tropical-info}

An early observation \cite{FZ-CA4} in the theory of cluster algebras was that, to a cluster algebra, one can associate a number of collections of integer vectors and polynomials that capture key cluster-theoretic information.
These include $\mathbf{g}$-vectors, $\mathbf{c}$-vectors and $\curly{F}$-polynomials.  

As the understanding of the theory matured, it was realised that $\mathbf{g}$- and $\mathbf{c}$-vectors are \emph{tropical} in nature and this opened up the subject to the use of methods of tropical geometry.
Tropical geometry is closely related to toric geometry, and cluster varieties are built by gluing many tori, so many techniques become available; this observation is at the heart of the scattering diagram technology of \cite{GHKK}, which led to geometric (re-)proofs of many of the main conjectures on cluster algebras posed by Fomin and Zelevinsky. 

A key theorem states that the cluster monomials (that is, monomials in the cluster variables of a particular cluster) are distinguished by their $\mathbf{g}$-vectors.
This is a fundamental step in showing that cluster monomials are linearly independent, and hence understanding how they relate to canonical bases, which was a key motivation for their introduction. 

Results such as these were the inspiration for categorification of cluster algebras: indeed, categorification provides proofs of the above results on $\mathbf{g}$-vectors (see e.g.\ \cite{GLS-Rigid,DehyKeller,FuKeller}) and linear independence \cite{CKLP} for large classes of cluster algebras.
By widening the scope of cluster categorical methods in this paper, we are able to expand the scope of these techniques even further (e.g.\ Theorem~\ref{t:g-vectors-cl-mons}).

In the categorical setting, the fundamental observation is that $\cT$-approximations of objects in a cluster category $\cC$, where $\cT$ is a cluster-tilting subcategory, yield elements of $\Kgp{\cT}$ known as the \emph{index} and \emph{coindex}.
These are sometimes referred to as \emph{homological $\mathbf{g}$-vectors} since one can show (Theorem~\ref{t:c-vec-mut-formula}) that they recover precisely the ordinary $\mathbf{g}$-vectors under decategorification.
(One has both indices and coindices, corresponding to left and right approximations and to a sign choice in the definition of combinatorial $\mathbf{g}$-vectors.)
Then, by showing that rigid objects are determined by their indices (see Proposition~\ref{p:rigid-index}, building on \cite{DehyKeller} and \cite{FuKeller}), one obtains the above claim on cluster monomials.

In this work, we build on the use of indices and coindices, noting that they, and correspondingly $\mathbf{g}$-vectors, lie on the $\cA$-side.
We shift perspective from vectors to linear maps, introducing functions
\[
\ind{\cT}{\cU}\colon \Kgp{\cT} \isoto \Kgp{\cU},\quad
\coind{\cT}{\cU}\colon \Kgp{\cT} \isoto \Kgp{\cU}
\]
associated to pairs of cluster-tilting subcategories $\cTU \ctsubcat \cC$ (Definitions~\ref{d:index} and \ref{d:coindex}).
These functions are even isomorphisms (Proposition~\ref{p:ind-coind-inverse}), from which we see that, under mild assumptions, $\card{\indec \cT}=\card{\indec \cU}$ and $\card{\exch \cT}=\card{\exch \cU}$ (Corollary~\ref{c:clusters-same-size}) where the latter denotes the \emph{mutable} objects in $\cT$ (Definition~\ref{d:mutable}).

Using the non-degenerate form $\canform{\blank}{\blank}{\cT}$ above, we can take adjoints of these maps and obtain
\[
\indbar{\cU}{\cT}=\adj{(\coind{\cT}{\cU})}\colon \Kgp{\fd \cU} \isoto \Kgp{\fd \cT},\quad
\coindbar{\cU}{\cT}=\adj{(\ind{\cT}{\cU})}\colon \Kgp{\fd \cU} \isoto \Kgp{\fd \cT},
\]
which are also isomorphisms (Definition~\ref{d:ind-bar-def}).\footnote{
Recall that in this introductory exposition, we restrict to the finite rank case; outside this setting, more care is needed due to the fact that $\fd \cT$, $\fpmod \cT$ and $\lfd \cT$ (that is, the categories of finite-dimensional, finitely presented and locally finite-dimensional modules, respectively) need not coincide.
Duality statements also need to be refined, e.g.\ in the definition of the form $\canform{\blank}{\blank}{\cT}$, we need to replace $\Kgp{\fd \cT}$ by a different Grothendieck group, $\Kgpnum{\lfd \cT}$ (Definition~\ref{d:canform}).}

We study these maps and their compositions in detail, obtaining results that justify the claim that the maps $\cind{}{}$ and $\cindbar{}{}$ evaluated on the natural bases of indecomposable objects of $\cT$ (playing the role of projective objects via $\cT\simeq \proj \cT$) and simples respectively compute $\mathbf{g}$- and $\mathbf{c}$-vectors.  For example, we deduce (Theorem~\ref{t:c-vec-mut-formula}) the following categorification of a formula for the mutation of $\mathbf{g}$-vectors \cite{NakanishiBook,NakanishiZelevinsky}:
\begin{align}
\ind{\mut{U}{\cU}}{\cT}[\mut{\cU}{U}] & = -\ind{\cU}{\cT}[U]+\Bigl(\sum_{W\in \cU\setminus U} [\exchmatentry{W,U}]_{-}\ind{\cU}{\cT}[W]\Bigr)-\beta_{\cT}\bigl[\stabindbar{\cU}{\cT}[\simpmod{\cU}{U}]\bigr]_- \tag{\ref{eq:g-vec-mut-formula}}, \\
\ind{\mut{U}{\cU}}{\cT}[V] & =\ind{\cU}{\cT}[V], \notag
\end{align}
and the corresponding formula for mutation of $\mathbf{c}$-vectors:
\begin{align}
\indbar{\mut{U}{\cU}}{\cT}[\simpmod{\mut{U}{\cU}}{\mut{\cU}{U}}]& =-\indbar{\cU}{\cT}[\simpmod{\cU}{U}], \notag \\ \tag{\ref{eq:c-vec-mut-formula}}
\indbar{\mut{U}{\cU}}{\cT}[\simpmod{\mut{U}{\cU}}{V}] & =\indbar{\cU}{\cT}[\simpmod{\cU}{V}]+[\exchmatentry{U,V}^{\cU}]_{+}\indbar{\cU}{\cT}[\simpmod{\cU}{U}]+\exchmatentry{U,V}^{\cU}\bigl[\stabindbar{\cU}{\cT}[\simpmod{\cU}{U}]\bigr]_{-}. \end{align} 

Using this analysis, we are able to obtain the following result, from which a significant number of desirable consequences for mutation in cluster categories follow (Corollary~\ref{c:rank-invariant}, Proposition~\ref{p:mut-of-Euler-form}, Theorem~\ref{t:exch-mat-mutation-clust-str}).  
Here, $\beta_{\cT}$ is the map defined above, obtained via projective resolution.

\begin{theorem*}[Theorem~\ref{t:exch-isos}, Corollary~\ref{c:exch-isos}] Let $\cC$ be a compact cluster category.  Then we have commutative diagrams
\[\begin{tikzcd}
\Kgp{\fd \stab{\cT}} \arrow{r}{\beta_{\cT}} \arrow{d}[swap]{\stabindbar{\cT}{\cU}} & \Kgp{\cT} \arrow{d}{\ind{\cT}{\cU}} \\
\Kgp{ \fd \stab{\cU}} \arrow{r}{\beta_{\cU}} & \Kgp{\cU} 
\end{tikzcd}\qquad
\begin{tikzcd}
\Kgp{\fd \stab{\cT}} \arrow{r}{\beta_{\cT}} \arrow{d}[swap]{\stabcoindbar{\cT}{\cU}} & \Kgp{\cT} \arrow{d}{\coind{\cT}{\cU}} \\
\Kgp{ \fd \stab{\cU}} \arrow{r}{\beta_{\cU}} & \Kgp{\cU} 
\end{tikzcd}\]
for any $\cTU\ctsubcat\cC$.
\end{theorem*}

This categorifies similar diagrams first written down by Fock and Goncharov \cite{FockGoncharov} and is often referred to as \emph{tropical duality}.  However, we do not need to assume that $\cT$ and $\cU$ are reachable from each other by a sequence of mutations: the claims hold for any pair of cluster-tilting subcategories.

We obtain proofs of various properties of $\mathbf{g}$- and $\mathbf{c}$-vectors, analogous to earlier work with stronger assumptions on $\cC$.
In particular, the proof due to Dehy and Keller \cite{DehyKeller} that $\mathbf{g}$-vectors are (row) sign-coherent transfers across to our setting with minimal changes (Proposition~\ref{p:g-vector-sign-coherence}).
Since $\cindbar{}{}$ is adjoint to $\cind{}{}$, the (column) sign-coherence of $\mathbf{c}$-vectors is then an immediate corollary (Corollary~\ref{c:sign-coh-c-vectors}). Note also Theorem~\ref{t:g-vectors-cl-mons} and Remark~\ref{r:g-vecs-distinguish-cl-mons}, extending the categorical proofs of conjectures of Fomin and Zelevinsky to our generality.

As we describe in \S\ref{s:sign-coherence}, the above gives us the starting point for the construction of scattering diagrams \emph{\`{a} la} Gross--Hacking--Keel--Kontsevich \cite{GHKK} and we are able to use the categorical structure to deduce some basic properties, e.g.\ the interiors of the cones defined by $\mathbf{g}$-vectors do not intersect; see also Proposition~\ref{p:beta-maps-cones}, reproving a recent result of Melo and Nájera Chávez \cite{MNC}.

\subsection{Cluster characters}\label{ss:cl-cats-intro}

The aforementioned relationships among the functions $\cind{}{}$ and $\cindbar{}{}$ also play an important role in proving the required properties of cluster characters.
Defining these requires one more input, namely $\curly{F}$-polynomials, which give us the non-tropical part.
These should remind one of other classical polynomial invariants associated to moduli problems, such as Poincaré polynomials or their generalisations (Serre polynomials, mixed Hodge polynomials, etc.).
In our setting, the moduli problem is that of counting submodules of certain modules according to dimension vectors, where the existence of infinitely many such submodules is handled by taking the Euler characteristic of the corresponding (quiver) Grassmannian, the latter being defined as follows: for $M\in \fd \stab{\cT}$ and $[L]\in\Kgp{\fd\stab{\cT}}$, the \emph{quiver Grassmannian} $\QGra{[L]}{M}$ is the algebraic variety whose points parametrise submodules $L'\leq M$ with $[L']=[L]\in \Kgp{\fd \stab{\cT}}$.

Let $\cC$ be a cluster category, and fix a cluster-tilting subcategory $\cT\ctsubcat \cC$.
Let $\bK\Kgp{\cT}$ be the group $\bK$-algebra of $\Kgp{\cT}$, which we write as $\bK\Kgp{\cT}=\text{span}_{\bK}\{ a^{t} \mid t\in \Kgp{\cT} \}$.
Similarly, let $\bK\Kgp{\fd \stab{\cT}}$ be the group algebra of $\Kgp{\fd \stab{\cT}}$, with its canonical basis $\{ x^n \mid n\in \Kgp{\fd \stab{\cT}} \}$.
The letters `$a$' and `$x$' for the formal variables are chosen to be compatible with the $\cA$-side and the $\cX$-side. Note that taking the group algebra is an `exponentiation' operation, whose inverse `log' operation is, in this setting, tropicalisation (see Remark~\ref{r:clucha-A}\ref{r:clucha-A-trop} for a more precise statement).

Then for $M\in \fd \stab{\cT}$, define its \emph{$\Fpoly$-polynomial} to be
\[ \Fpoly(M)=\sum_{[L]\in \Kgp{\fd \stab{\cT}}} \chi(\QGra{[L]}{M})x^{[L]} \in \bK\Kgp{\fd \stab{\cT}},\]
so that $\curly{F}$-polynomials naturally live on the $\cX$-side.
However, when we need them on the $\cA$-side, we can use the map $\beta_{\cT}$ to transfer them.
Specifically,
$\beta_{\cT}\colon \Kgp{\fd \stab{\cT}}\to \Kgp{\cT}$ induces a map $(\beta_{\cT})_{*}\colon \bK\Kgp{\fd{\stab{\cT}}}\to \bK\Kgp{\cT}$ of group algebras, with $(\beta_{\cT})_{*}(x^{[L]})=a^{\beta_{\cT}[L]}$.
As always in the introduction, we are restricting to the finite rank case, outside which some additional technicalities are needed.

The $\Aside$-cluster character is a function from $\cC$ to $\bK\Kgp{\cT}$, which has two elements, tropical and non-tropical.  For $X\in \cC$, the tropical part is $a^{\ind{\cC}{\cT}[X]}$, recording the $\cT$-approximation of $X$. The non-tropical part is given by pushing forward the $\curly{F}$-polynomial of the $\stab{\cT}$-module $\Extfun{\cT}X\defeq \Ext{1}{\cC}{\blank}{X}|_{\cT}$.
Combining the two parts, the $\cA$-cluster character is given on objects $X$ of $\cC$ by
\begin{equation} \tag{\ref{eq:clucha-A-alt}} \clucha[\cT]{\Aside}(X)=a^{\ind{\cC}{\cT}[X]}(\beta_{\cT})_{*}\Fpoly(\Extfun{\cT}X). \end{equation}
This function has the property that
\begin{equation} \tag{\ref{eq:clucha-A-split}} \clucha[\cT]{\Aside}(X\oplus Y)=\clucha[\cT]{\Aside}(X)\clucha[\cT]{\Aside}(Y),\end{equation}
justifying the terminology `character'.

A fundamental fact we use throughout the paper is that the functor $\Extfun{\cT}\colon \cC/\cT \to \fpmod \stab{\cT}$ is an equivalence (Proposition~\ref{p:equiv-to-mod}), and as such modules of the form $\Extfun{\cT}X$ appear throughout our work, as in \eqref{eq:clucha-A-alt}.
In particular, when there is no loop at $T$, the mutant $\mut{\cT}{T}$ is characterised by $\Extfun{\cT}(\mut{\cT}{T})=\simpmod{\cT}{T}$, the latter being the simple functor at $T$, whose class we have seen as part of the basis of $\Kgp{\fd{\stab{\cT}}}$, corresponding to $\seedlat{N}^{\circ}$.

To obtain applications to an associated cluster algebra (or generalisation thereof), the cluster character should relate categorical mutations of cluster-tilting subcategories (Definition~\ref{d:mutable}) to mutations of cluster variables.
To this end, we show that if $X,Y\in\cC$ have the property that any two non-split conflations $Y\infl Z\defl X\confl$ have isomorphic middle term, and similarly for non-split conflations $X\infl Z'\defl Y\confl$, then
\begin{equation} \tag{\ref{eq:clucha-A-non-split}} \clucha[\cT]{\Aside}(X)\clucha[\cT]{\Aside}(Y)=\clucha[\cT]{\Aside}(Z)+\clucha[\cT]{\Aside}(Z'), \end{equation}
as in a cluster algebra exchange relation.
This property of conflations holds in the case that $\ext{1}{\cC}{X}{Y}=1$, but also more generally (Lemma~\ref{l:rk1-unique-middle-term}), and the proof is based on work of Palu \cite{Palu2}.
The fundamental theorem for $\cA$-cluster characters now follows.

\begin{theorem*}[Theorem~\ref{t:Aside-bijection}, based on \cite{BMRRT,FuKeller,PresslandPostnikov}]
If $\cC$ has a cluster structure, then $\clucha[\cT]{\Aside}$ is a bijection between indecomposable objects in cluster-tilting subcategories in the mutation class of $\cT$ and cluster (and frozen) $\Aside$-variables of the cluster algebra with initial exchange matrix $B_{\cT}$, with frozen variables corresponding to projective objects.
This induces a bijection between these cluster-tilting subcategories and the seeds of this cluster algebra, commuting with mutations.
\end{theorem*}

The cluster character is also compatible with partial stabilisation: the diagram
\[
\begin{tikzcd}
\cC \arrow{r}{\clucha[\cT]{\Aside}} \arrow{d}[left]{\pi_{\cC}^{\cC/\cP}} & \bK\Kgp{\cT} \arrow{d}{(\pi_{\cT}^{\cT/\cP})_*}\\
\cC/\cP \arrow{r}[below,yshift=-0.2em]{\clucha[\cT/\cP]{\Aside}} & \bK\Kgp{\cT/\cP}
\end{tikzcd}
\]
commutes (Proposition~\ref{p:A-cl-cat-partial-stab}). In particular, this tells us how to handle setting frozen variables equal to $1$ in the categorical setting.

Already at this point, we have extended the theory of cluster characters, due to the greater generality in which we are working and the removal of previously common assumptions (e.g.\ finite rank or no loops), via a single general construction.

However, the more significant progress afforded by the technical analysis carried out here is in defining a cluster character on the $\cX$-side.  Although ingredients of such a function had been categorified previously (with additional assumptions), we are able to obtain a multiplication formula for the $\cX$-cluster character, and a categorical proof of the separation formula, in broad generality, as we now describe.
We need some of our mild additional finiteness assumptions so that the adjoint maps $\cindbar{}{}$ are well-defined.

For each $\cU\ctsubcat\cC$, and $M\in\fd{\stab{\cU}}$, there exist $\modlift{M}{\cU}{\pm}\in\cU$ such that
\begin{equation*}
\beta_{\cU}[M]=[\modlift{M}{\cU}{+}]-[\modlift{M}{\cU}{-}]\in \Kgp{\cU};
\end{equation*}
any two choices differ only by the addition or removal of common direct summands and this ambiguity has no effect on what follows.

Let $\cC$ be a compact cluster category, and let $\cTU\ctsubcat\cC$.
Let $\Frac{\Kgp{\fd{\cT}}}$ be the field of fractions of the group algebra $\bK\Kgp{\fd{\stab{\cT}}}$.
We define the \emph{$\Xside$-cluster character} for $\cU$ with respect to $\cT$ to be the function 
$ \clucha[\cTU]{\Xside}\colon \fd \stab{\cU} \to \Frac{\Kgp{\fd\stab{\cT}}} $ defined by
\begin{equation} \tag{\ref{eq:clucha-X-alt}}
\clucha[\cTU]{\Xside}(M)=x^{\stabindbar{\cU}{\cT}{[M]}}\Fpoly(\Extfun{\cT}\modlift{M}{\cU}{+})\Fpoly(\Extfun{\cT}\modlift{M}{\cU}{-})^{-1}.
\end{equation}
Note that the image of $\clucha[\cT]{\Xside}$ visibly lies in $\Frac{\Kgp{\fd \cT}}$ and not $\bK\Kgp{\fd \cT}$ except possibly in degenerate situations, i.e.\ the values of the $\cX$-cluster character are (unavoidably) rational functions.
Also, $\clucha[\cTU]{\Xside}(M)$ has two natural tropicalisations which correspond to taking the minimal and maximal submodules in the two $\curly{F}$-polynomial factors.
Under the minimal convention, we obtain $x^{\indbar{\cU}{\cT}{[M]}}$, and under the maximal convention we obtain
\[x^{\indbar{\cU}{\cT}{[M]}+[\Extfun{\cT}\modlift{M}{\cU}{+}]-[\Extfun{\cT}\modlift{M}{\cU}{-}]}=x^{\coindbar{\cU}{\cT}{[M]}}.\]

One can show (Proposition~\ref{p:cluchaX-mult}) that if $[M]=[M_1]+[M_2]\in\Kgp{\fd\stab{\cU}}$, then
\[\clucha[\cTU]{\Xside}(M)=\clucha[\cTU]{\Xside}(M_1)\clucha[\cTU]{\Xside}(M_2)\]
and hence we have a well-defined character $ \clucha[\cTU]{\Xside}\colon \Kgp{\fd\stab{\cU}} \to \Frac{\Kgp{\fd \cT}}$.

We continue by showing that if $U\in \exch \cU$ and $M=\Extfun{\cU}{(\mut{\cU}{U})}$, we may choose the objects $\modlift{M}{\cU}{\pm}$ to be the middle terms $U_{\cU}^{\pm}$ of the corresponding exchange conflations.
We immediately obtain (Corollary~\ref{c:X-cc-on-simples}) that
\[\clucha[\cTU]{\Xside}(\Extfun{\cU}{(\mut{\cU}{U})})=x^{\indbar{\cU}{\cT}{[\Extfun{\cU}{(\mut{\cU}{U})}]}}\Fpoly(\Extfun{\cT}\exchmon{\cU}{U}{+})\Fpoly(\Extfun{\cT}\exchmon{\cU}{U}{-})^{-1}.\]
If we also assume there is no loop at $U$, so that $M=\simpmod{\cU}{U}$ is simple, then we may make this expression more explicit, giving us a categorification (Proposition~\ref{p:prod-form-of-X-cc-on-simples}) of the celebrated \emph{separation formula} for $\cX$-cluster variables:
\[ \clucha[\cTU]{\Xside}(\simpmod{\cU}{U})=x^{\indbar{\cU}{\cT}[\simpmod{\cU}{U}]}\prod_{V\in \indec \cU} \Fpoly(\Extfun{\cT}{V})^{\exchmatentry{V,U}^{\cU}}. \]

Our main theorem in relation to $\cX$-cluster characters is, as one would wish, that they satisfy the $\cX$-side mutation rules when there are no loops or $2$-cycles.

\begin{theorem*}[Theorem~\ref{t:X-clucha-mutation}]
Let $\cC$ be a compact or skew-symmetric cluster category. Let $\cTU\ctsubcat\cC$, and assume $\cU$ has no loops or $2$-cycles. Let $U\in\exch{\cU}$, with associated mutation $\mut{U}{\cU}$, also assumed to have no loops.
Then for $V\in\exch\mu_U\cU$, we have
\[\clucha[\cT,\,\mut{U}{\cU}]{\Xside}(\simpmod{\mut{U}{\cU}}{V})=\begin{cases}
\clucha[\cTU]{\Xside}(\simpmod{\cU}{U})^{-1}&\text{if}\ V=\mu_\cU U,\\
\clucha[\cTU]{\Xside}(\simpmod{\cU}{V})\clucha[\cTU]{\Xside}(\simpmod{\cU}{U})^{[\exchmatentry{U,V}^{\cU}]_{+}}(1+\clucha[\cTU]{\Xside}(\simpmod{\cU}{U}))^{-\exchmatentry{U,V}^{\cU}}&\text{otherwise.}
\end{cases}\]
\end{theorem*}

From this, we may deduce the fundamental theorem for $\cX$-cluster characters.

\begin{theorem*}[Theorem~\ref{t:X-clucha}] Assume that $\cC$ has a cluster structure with respect to $\cT \ctsubcat \cC$.
Let $\cU$ be the cluster-tilting subcategory corresponding to a seed $s$ of the cluster algebra with initial exchange matrix $B_{\cT}$ under the bijection of Theorem~\ref{t:Aside-bijection}. Then the functions $\clucha[\cTU]{\Xside}(S)$, as $S$ runs over the simple $\stab{\cU}$-modules, are the $\Xside$-cluster variables of $s$ at mutable vertices.
\end{theorem*}

We also obtain a new relationship between the $\cA$- and $\cX$-cluster characters:

\begin{proposition*}[Proposition~\ref{p:cluchaX-ratio}]
Let $M\in\fd \stab{\cU}$. Then
\[(\beta_\cT)_*\clucha[\cTU]{\Xside}(M)=\frac{\clucha[\cT]{\Aside}(\modlift{M}{\cU}{+})}{\clucha[\cT]{\Aside}(\modlift{M}{\cU}{-})}.\]
\end{proposition*}
This may be interpreted as showing that $(\beta_{\cT})_*$ is closely related to the change of variables from $y$ to $\hat{y}$ appearing in the original work of Fomin--Zelevinsky (\cite[Eq.~3.7]{FZ-CA4}).

\subsection{Quantisation}\label{ss:quantisation-intro}

Following the success of the use of categorification to answer questions in cluster theory, it is natural to ask whether it can be applied to understand \emph{quantum cluster algebras}.  These were introduced by Berenstein and Zelevinsky \cite{BZ-QCA} and give a noncommutative version of classical, commutative cluster algebras.
The motivation for doing so comes from \emph{noncommutative geometry}, which encompasses a wide variety of objects and techniques to study analogues of important classes of varieties such as Grassmannians and partial flag varieties.
Ultimately, the origin for these ideas comes from mathematical physics but very many questions of interest to mathematicians have arisen from it, notably the theory of quantum groups (quantised enveloping algebras and quantised coordinate rings) and their canonical bases; see for example \cite{LusztigBook}.
Indeed, the whole theory of cluster algebras came about---in large part---from a desire to understand better these canonical bases.

Quantum cluster algebras are unusual in the pantheon of noncommutative and quantum algebras, in that typically these are obtained by some deformation process that can lead to wildly different algebraic properties.
In favourable situations, some properties do persist, e.g.\ homological dimensions; there is the notion of a flat deformation, for example, that captures when a noncommutative analogue is not too far from the original algebra. When the noncommutative version is controlled by a deformation parameter $q$, one can also usually recover the unquantised algebra by some sort of classical limit, but this can be technically complex.
Even finding suitable noncommutative or quantum candidates is hard.

However, quantum cluster algebras are \emph{very flat deformations}: we replace the tori used in building the cluster algebra with quantum tori (Laurent polynomial rings in which the generators quasi-commute rather than commute, i.e.\ satisfy relations of the form $x_{i}x_{j}=q^{\lambda_{ij}}x_{j}x_{i}$) and because the Laurent phenomenon persists in the quantum setting, one can show that the cluster combinatorics is identical \cite{BZ-QCA} and, in many cases, simply setting $q=1$ recovers the original cluster algebra \cite{GLS-quantum-specialization}.
Conversely, if you are given a cluster algebra, finding a quantisation of it is a straightforward linear algebra problem; even the associated moduli problem of classifying all quantisations is controlled by a vector space \cite{GellertLampe}.
We also note that there is an intimate link between quantum cluster algebras and Poisson geometry: in fact, quantum cluster data is exactly the data of a log-canonical Poisson structure \cite{GSV-Book}.

One might think that this makes the quantum version of cluster theory uninteresting.
However, the reverse is true, because work of many authors (\cite{GLS-QuantumPFV,GradedQCAs,GoodearlYakimovQCA,GoodearlYakimovDoubleBruhat} and more) have shown that almost all known quantisations of varieties arising in Lie theory have quantum cluster algebra structures---Grassmannians and partial flag varieties, their Schubert cells, double Bruhat cells and more.  Consequently, we have the significant advantage that if we want to study these quantisations, for many of their properties, especially those encoded in the cluster combinatorics, we can reduce to the commutative case.

It is then natural to ask: given a categorification of a (commutative) cluster algebra, when is there a quantum categorification of an associated quantum cluster algebra?  This is another important question that we are able to address in this work, as detailed below, but before explaining this, we provide some context.  

The work of Geiß--Leclerc--Schröer \cite{GLS-QuantumPFV} showed that certain quantum coordinate rings associated to unipotent subgroups of Kac--Moody groups have quantum cluster algebra structures.
They did so by examining their categorification of the cluster structure they had identified on the corresponding commutative coordinate ring and observing two things: firstly, the cluster combinatorics from the category ought to remain the same (as we indicated above) and secondly, there is homological information in the category that encodes the quasi-commutation of the corresponding variables.
Later work by Jensen--King--Su \cite{JKS,JKS-Quantum-Gr} showed that this phenomenon is not limited to the specific categories studied by Geiß--Leclerc--Schröer.
In their work, Jensen--King--Su see the same pattern: their categorification of the Grassmannian cluster category extends to one of the quantum Grassmannian by using the same category and finding the quasi-commutation data there too.

In this work, we introduce a definition of a \emph{quantum cluster category} and explain how it too is tied up with the duality of the $\cA$- and $\cX$-sides as above.
We then use our technical results to establish the basic theory of quantum cluster categories, and we conclude by showing that a very large class of cluster categories admit quantisations, significantly expanding the work of Geiß--Leclerc--Schröer.
We are able to lay the foundations for a theory of quantum cluster characters, but a construction of these is currently out of reach due to profound difficulties in the algebraic geometry (i.e.\ the lack of a suitable quantum Euler characteristic for singular quiver Grassmannians). 

Now, we describe our construction.
The starting point is the following observation: to define a quantum cluster algebra, one chooses a skew-symmetric integer matrix $L$ that is required to be compatible with the exchange matrix $B$ by satisfying the equation $\adj{B}L=J$ where $J$ has diagonal principal part with positive integer entries and is zero otherwise.
Here $\adj{B}$ denotes the transpose of the matrix $B$.
The appearance of both skew-symmetry and the transposition in the compatibility condition is suggestive that we should look for a skew-symmetric form $\pform{\blank}{\blank}{\cT}$, akin to $\sform{\blank}{\blank}{\cT}$, and an associated map $\lambda_{\cT}$ such that there is a relationship between $\lambda_{\cT}$ and $\adj{\beta}_{\cT}$.
Due to the non-uniqueness in choosing quantisations, we do not expect a canonical map $\lambda_{\cT}$ to emerge from the categorical structure in general.

The definition we make is as follows (Definition~\ref{d:quantum-hom}).
We first define a \emph{quantum datum} for $\cT$ to be a map	$\lambda_{\cT}\colon \Kgp{\cT} \to \dual{\Kgp{\cT}}$ such that we have $\adj{\lambda_{\cT}}=-\lambda_{\cT}$ (skew-symmetry) and $\adj{\lambda_{\cT}}\circ\beta_{\cT}=2(\sdual{\cT}\circ\sinc{\cT})$ (compatibility).
Here, $\sinc{\cT}$ is the map induced by the inclusion of categories $\fd \stab{\cT}\subseteq \fd \cT$.
We choose the latter equation over its adjoint because it reveals an important feature: such a $\lambda_{\cT}$ can only exist if $\beta_{\cT}$ is injective, since $\sdual{\cT}$ and $\sinc{\cT}$ are.

The corresponding form is defined (Definition~\ref{d:p-form}) to be $\pform{\blank}{\blank}{\cT}\colon \Kgp{\cT}\cross \Kgp{\cT}\to \integ$,
\[ \pform{[T]}{[U]}{\cT}\defeq \evform{[T]}{\lambda_{\cT}[U]}. \]
The skew-symmetry of $\lambda_{\cT}$ immediately gives skew-symmetry of $\pform{\blank}{\blank}{\cT}$ (Lemma~\ref{l:p-form-skew-symmetric}).
Since $\lambda_{\cT}$ and $\pform{\blank}{\blank}{\cT}$ uniquely determine each other, we will also refer to the form as a quantum datum for $\cT$.

Given a quantum datum $\pform{\blank}{\blank}{\cT}$ for $\cT \ctsubcat \cC$, we may then transport it to a form $\mu_{\cT}^{\cU}\pform{\blank}{\blank}{\cT}\colon\Kgp{\cU}\cross \Kgp{\cU}\to \integ$ for a different cluster-tilting subcategory $\cU$  (Definition~\ref{d:mut-of-p-form}) by defining
\[ \mu_{\cT}^{\cU}\pform{\blank}{\blank}{\cT}=\pform{\ind{\cU}{\cT}(\blank)}{\ind{\cU}{\cT}(\blank)}{\cT}. \]
A choice of quantum datum $\pform{\blank}{\blank}{\cT}$ for every $\cT\ctsubcat\cC$ such that
\[\mu_{\cT}^{\cU}\pform{\blank}{\blank}{\cT}=\pform{\blank}{\blank}{\cU}\]
whenever $\cTU\ctsubcat\cC$ is called (Definition~\ref{d:quant-struct}) a \emph{quantum structure} for $\cC$.
In particular, a quantum structure is uniquely determined by any one of its quantum data.
Our main result is that this determination is `free', in the sense that there are no additional constraints needed on this initial choice for the transported forms $\mu_{\cT}^{\cU}\pform{\blank}{\blank}{\cT}$ to be quantum data, fitting together into a quantum structure.

The strategy for this is as follows.
Translating back from the form to the map, we obtain $\mu_{\cT}^{\cU}(\lambda_{\cT})[U]=\mu_{\cT}^{\cU}\pform{\blank}{[U]}{\cT}$ for each $\cU\ctsubcat \cC$, given one initial choice $\lambda_{\cT}$.
This enables us to give the analogous commutative diagrams (Proposition~\ref{p:lambda-square}) relating $\lambda_{\cT}$ and $\mu_{\cT}^{\cU}(\lambda_{\cT})$ as we had for $\beta_{\cT}$ and $\beta_{\cU}$ above:
\[\begin{tikzcd}
\dual{\Kgp{\cT}} \arrow{d}[swap]{\dual{(\coind{\cU}{\cT})}}& \Kgp{\cT} \arrow{l}[swap]{\lambda_{\cT}} \arrow{d}{\ind{\cT}{\cU}} \\
\dual{\Kgp{\cU}}   & \Kgp{\cU}, \arrow{l}[swap]{\mu_{\cT}^{\cU}(\lambda_{\cT})}
\end{tikzcd}\qquad 
\begin{tikzcd}
\dual{\Kgp{\cT}} \arrow{d}[swap]{\dual{(\ind{\cU}{\cT})}}& \Kgp{\cT} \arrow{l}[swap]{\lambda_{\cT}} \arrow{d}{\coind{\cT}{\cU}} \\
\dual{\Kgp{\cU}}   & \Kgp{\cU}. \arrow{l}[swap]{\mu_{\cT}^{\cU}(\lambda_{\cT})}
\end{tikzcd}\]
From this, we deduce that $\lambda_{\cU}\defeq \mu_{\cT}^{\cU}(\lambda_{\cT})$ is indeed a valid quantum datum for $\cU$ (Proposition~\ref{p:mut-of-lambda-ss-compat}), i.e.\ that $\lambda_{\cU}$ is also skew-symmetric and compatible with $\beta_{\cU}$.
We show that this transport of quantum data is transitive (Proposition~\ref{p:mut-of-p-form-transitive}), which is non-trivial since $\ind{}{}$ is \emph{not} transitive, and we conclude that the family of maps $\lambda_{\cU}$ obtained by transporting any initial quantum datum $\lambda_{\cT}$ in this way do form a quantum structure on $\cC$ as claimed (Corollary~\ref{c:mut-of-lambda-are-q-structure}).

While the above construction of a quantum structure from an initial choice of data did not involve mutation, it is nonetheless true that,  for $\cU\ctsubcat\cC$ and $U\in\exch{\cU}$, the quantum data $\lambda_{\cU}$ and $\lambda_{\mut{U}{\cU}}$ are related by mutation in the sense of Berenstein--Zelevinsky, as we prove in Proposition~\ref{p:BZ-lambda-mut}.

Our final result is the aforementioned extension of the construction of Geiß--Leclerc--Schröer.
Namely, assume $\cE$ is a Hom-finite exact cluster category.
For each $\cT\ctsubcat\cC$ and each $T_1,T_2\in\cT$, define
\[\pform{[T_1]}{[T_2]}{\cT}=\dim\Hom{\cE}{T_1}{T_2}-\dim\Hom{\cE}{T_2}{T_1}.\]
Then the forms $\pform{\blank}{\blank}{\cT}$ defined in this way are a quantum structure on $\cE$ (Theorem~\ref{t:canonical-p-form}).  

Note in particular that the claim is not only that the Hom-difference formula define a quantum datum for each cluster-tilting subcategory, but that making this choice for \emph{all} cluster-tilting subcategories is compatible with mutation (and indeed the more general transport of quantum data above).
Put differently, if we use the Hom-difference formula above to define a single quantum datum $\pform{\blank}{\blank}{\cT}$, for a particular cluster-tilting subcategory $\cT$, then the transported quantum data $\mu_{\cT}^{\cU}\pform{\blank}{\blank}{\cT}$ are given by the same Hom-difference formula, which was not a priori clear, and in particular, the induced quantum structure is independent of the choice of $\cT$.

We also briefly mention an alternative approach to quantum categorification, which is different from the one presented here.
Namely, in \cite{Hernandez-Leclerc-first-mon-cat}, Hernandez and Leclerc showed that by examining the Grothendieck \emph{ring} of certain categories of modules for quantum affine algebras, one can obtain \emph{monoidal} categorifications of cluster algebras, in which the multiplication of cluster variables corresponds to a tensor product operation in the category.
In this paper, we are only concerned with \emph{additive} categorifications: one observation on the difference between the two settings is that we see the tropical story very clearly here, whereas the monoidal setting is able to handle other types of non-tropical questions.
In particular, monoidal categorification has enabled the resolution of a number of important conjectures originating in the work of Lusztig and Kashiwara on canonical bases, in work of Kashiwara--Kim--Oh--Park \cite{KKOP} and subsequent work by a number of authors; see also \cite{Qin-duan-can-basis}.

\subsubsection*{Acknowledgements}

We owe a significant debt to many colleagues and their institutions who have supported, hosted and educated us during the gestation of this paper.  The following list is inevitably not exhaustive, but we would like to mention in particular Xiaofa Chen, Mikhail Gorsky, Sira Gratz, Norihiro Hanihara, Bernhard Keller, Sondre Kvamme, Tim Magee, Lang Mou, Yann Palu, David Pauksztello, Pierre-Guy Plamondon, Konstanze Rietsch, Antoine de Saint Germain and Michael Wemyss.

Over the course of this project, the second author was supported by a fellowship from the Max-Planck-Gesellschaft, the EPSRC Postdoctoral Fellowship EP/T001771/2, and Michael Wemyss' ERC Consolidator Grant 101001227 (MMiMMa).
Parts of this work were done at the \emph{Cluster algebras and representation theory} programme in 2021 at the Isaac Newton Institute for Mathematical Sciences (supported by EPSRC grant no EP/R014604/1).
We also acknowledge financial support from Lancaster University and Universität Stuttgart.

\sectionbreak
\section{Cluster categories}

Throughout the paper, we assume that the ground field $\bK$ is perfect (giving us access to the results of Section~\ref{s:approximations}), but do not assume that this field is algebraically closed unless otherwise stated.

The goal of this first section is to define the class of categories we will study, along with some additional properties one can impose for better behaviour, and to give some of the most immediate and general consequences of these definitions.

\subsection{Extriangulated categories}

In the interests of working in a wide level of generality, covering recent examples constructed by Yilin Wu \cite{Wu}, our cluster categories will be extriangulated in the sense of \cite{NakaokaPalu}.
However, we do not recall the full definition of an extriangulated category here, since we will only use particular examples with a simpler description in terms of exact categories (see Definition~\ref{d:algebraic} below).
What is important for us is that a $\bK$-linear extriangulated category $\cC$ comes with a functor $\Ext{1}{\cC}{\blank}{\blank}\colon\cC\times\op{\cC}\to\Mod{\bK}$ such that each element of $\Ext{1}{\cC}{X}{Y}$ may be realised as a kernel--cokernel pair
\[\begin{tikzcd}
Y\arrow[infl]{r}{i}&E\arrow[defl]{r}{p}&X,
\end{tikzcd}\]
up to isomorphisms of such pairs extending the identity maps on $X$ and $Y$.
A kernel--cokernel pair appearing in this way is called a \emph{conflation}, the map $i$ is called an \emph{inflation} and the map $\pi$ is called a \emph{deflation}.
The notation
\[\begin{tikzcd}
Y\arrow[infl]{r}{i}&E\arrow[defl]{r}{p}&X\arrow[confl]{r}{\delta}&\phantom{}
\end{tikzcd}\]
indicates that a conflation is realised by $\delta\in\Ext{1}{\cC}{X}{Y}$ (cf.~Example~\ref{eg:extri}\ref{eg:extri-tri} below).

\begin{example}\label{eg:extri} {\ }
\begin{enumerate}
\item\label{eg:extri-exact}
If $\cE$ is an exact category, it is naturally extriangulated with $\Ext{1}{\cE}{X}{Y}$ defined in the usual way.
The conflations, inflations and deflations of this extriangulated category are precisely those of the exact category $\cE$ (which are sometimes \cite{Buehler} referred to as \emph{admissible} short exact sequences, monomorphisms and epimorphisms, respectively).

\item\label{eg:extri-tri}If $\cC$ is a triangulated category with suspension functor $\Sigma$, it is extriangulated with $\Ext{1}{\cC}{X}{Y}=\Hom{\cC}{X}{\Sigma Y}$. Each $\delta\in\Ext{1}{\cC}{X}{Y}$ may be completed to a distinguished triangle
\[\begin{tikzcd}
Y\arrow{r}&E\arrow{r}&X\arrow{r}{\delta}&\Sigma Y,
\end{tikzcd}\]
and the distinguished triangles (or, more accurately, their first three terms) are the conflations in $\cC$. Every morphism of $\cC$ is both an inflation and a deflation.

\item\label{eg:extri-ext-closed}
Let $\cC$ be an extriangulated category, and $\cD\subset\cC$ a full subcategory. We say $\cD$ is \emph{extension-closed} if for any conflation $Y\infl E\defl X\confl$ of $\cC$ with $X,Y\in\cD$, the middle term $E$ also lies in $\cD$. In this case $\cD$ becomes extriangulated in its own right by defining $\Ext{1}{\cD}{\blank}{\blank}=\Ext{1}{\cC}{\blank}{\blank}|_{\op{\cD}\times\cD}$, the extension-closure condition ensuring that the realisations of elements of $\Ext{1}{\cD}{X}{Y}$ are kernel--cokernel pairs in $\cD$. If the original category $\cC$ is exact, then so is $\cD$, but if $\cC$ is triangulated then $\cD$ need not be exact or triangulated in general.
\end{enumerate}
\end{example}

As Example~\ref{eg:extri} suggests, many of the techniques for working with extriangulated categories are analogous to those for working with exact or triangulated categories. For example, we will use the following foundational result freely throughout the paper.

\begin{proposition}[{\cite[Cor.~3.12]{NakaokaPalu}}]
\label{p:extri-les}
Let $X\infl Y\defl Z\confl$ be a conflation in an extriangulated category $\cC$, and let $T\in\cC$. Then there are exact sequences
\[
\begin{tikzcd}[column sep=10pt,row sep=0pt,ampersand replacement=\&,font=\small]
\Hom{\cC}{T}{X}\arrow{r}\&\Hom{\cC}{T}{Y}\arrow{r}\&\Hom{\cC}{T}{Z}\arrow{r}\&\Ext{1}{\cC}{T}{X}\arrow{r}\&\Ext{1}{\cC}{T}{Y}\arrow{r}\&\Ext{1}{\cC}{T}{Z},\\
\Hom{\cC}{X}{T}\arrow{r}\&\Hom{\cC}{Y}{T}\arrow{r}\&\Hom{\cC}{Z}{T}\arrow{r}\&\Ext{1}{\cC}{X}{T}\arrow{r}\&\Ext{1}{\cC}{Y}{T}\arrow{r}\&\Ext{1}{\cC}{Z}{T}.
\end{tikzcd}\]
If $\cC$ is exact, then the left-most map in each sequence is injective.
\end{proposition}

\begin{definition}
Let $\cC$ be an extriangulated category. The \emph{Grothendieck group} $\Kgp{\cC}$ of $\cC$ is the free abelian group on generators $[X]$ indexed by objects $X\in\cC$, modulo relations $[X]-[Y]+[Z]$ for each conflation $X\infl Y\defl Z\confl$ in $\cC$.
\end{definition}

\begin{remark}
Since all split exact sequences in an extriangulated category $\cC$ are conflations, we have $[X\oplus Y]=[X]+[Y]\in\Kgp{\cC}$ for any $X,Y\in\cC$.
\end{remark}

\begin{definition}
\label{d:Frob-extri}
Let $\cC$ be an extriangulated category. An object $P\in\cC$ is \emph{projective} if $\Ext{1}{\cC}{P}{\blank}=0$ and \emph{injective} if $\Ext{1}{\cC}{\blank}{P}=0$. We say that $\cC$ has \emph{enough projectives} if for every $X\in\cC$ there exists a deflation $P\defl X$ with $P$ projective, and that $\cC$ has \emph{enough injectives} if for every $X\in\cC$ there exists an inflation $X\infl Q$ with $Q$ injective. We call $\cC$ \emph{Frobenius} if it has enough projectives and injectives, and the projective and injective objects coincide.
\end{definition}

\begin{example}
Definition~\ref{d:Frob-extri} specialises to the usual definition of a Frobenius exact category. Any triangulated category $\cC$ is Frobenius as an extriangulated category, because the only projective or injective object is $0$, and every morphism is both an inflation and a deflation; in particular, $X\infl 0$ is an inflation and $0\defl X$ a deflation for any $X\in\cC$.
\end{example}

\begin{definition}
\label{d:partial-stab}
Let $\cC$ be an extriangulated category, and let $\cP$ be a full and additively closed subcategory of projective-injective objects of $\cC$. Then the \emph{partial stabilisation} of $\cC$ by $\cP$ is the additive quotient category $\cC/\cP$, which is naturally extriangulated with extension groups
\[\Ext{1}{\cC/\cP}{X}{Y}\defeq\Ext{1}{\cC}{X}{Y},\]
by a special case of \cite[Prop.~3.30]{NakaokaPalu} (see also \cite[Thm.~2.8]{FGPPP}).
\end{definition}

\begin{remark}
\label{r:partial-stab-Frob}
The conflations, inflations and deflations of $\cC/\cP$ are the images of those of $\cC$ under the projection functor $\cC\to\cC/\cP$. Since $\cC/\cP$ has the same projective and injective objects as $\cC$, the partial stabilisation $\cC/\cP$ is Frobenius if and only if $\cC$ is.
\end{remark}

When $\cC$ is a Frobenius extriangulated category, one can take $\cP$ to be the full subcategory of all projective-injective objects in $\cC$. In this case we write $\underline{\cC}=\cC/\cP$ and call this the \emph{stable category} of $\cC$. The following result generalises Heller \cite{Heller} (see also Happel \cite[Th.~I.2.6]{Happelbook}) for the case that $\cC$ is exact.

\begin{theorem}[{\cite[Cor.~7.4, Rem.~7.5]{NakaokaPalu}}]
\label{t:extri-stabcat}
If $\cC$ is a Frobenius extriangulated category, then the stable category $\stab{\cC}$ is naturally triangulated. Writing $\Sigma$ for the suspension functor of $\stab{\cC}$, we have
\[\Ext{1}{\cC}{X}{Y}=\stabHom{\cC}{X}{\Sigma Y}.\]
\end{theorem}

\begin{definition}
\label{d:algebraic}
We say that an extriangulated category $\cC$ is \emph{algebraic} if it is equivalent as an extriangulated category to the partial stabilisation $\cE/\cP$ of an exact category $\cE$ by a full and additively closed subcategory $\cP$ of projective-injective objects.
\end{definition}

By Remark~\ref{r:partial-stab-Frob}, an algebraic extriangulated category $\cC$ is Frobenius if and only if it is equivalent to a partial stabilisation $\cE/\cP$ of a Frobenius exact category $\cE$.

\begin{remark}
\label{r:algebraic}
By choosing $\cP=\{0\}$ in Definition~\ref{d:algebraic}, we see that an exact category $\cE=\cE/\{0\}$ is itself an algebraic extriangulated category. For triangulated categories, algebraicity has its usual meaning. Because of the natural isomorphism
$(\cC/\cP)/\cP'=\cC/\add(\cP,\cP')$
for additive subcategories $\cP$ and $\cP'$ of $\cC$, the property of algebraicity is preserved under partial stabilisation as in Definition~\ref{d:partial-stab}. It also follows from this isomorphism that $\stab{\smash{\cC/\cP}}=\stab{\cC}$, so the stable category is invariant under partial stabilisation.

As for triangulated categories, algebraicity has several equivalent formulations \cite[Prop.-Def.~6.20]{ChenThesis}, of which that in terms of exact categories is most useful to us here. On the other hand, one of these reformulations makes it clearer that algebraicity is preserved under passing to extension-closed subcategories, and so in particular Yilin Wu's Higgs categories \cite{Wu} are also algebraic.
\end{remark}

\subsection{Cluster categories and their variations}
\label{ss:clust-cats}

\begin{definition}
\label{d:ct-subcat}
Let $\cC$ be an extriangulated category, and let $\cT\subset\cC$ be a full subcategory. We call $\cT$ \emph{weak cluster-tilting} if
\[\cT =\{ X\in \cC \mid \Ext{1}{\cC}{T}{X}=0\ \text{for all}\ T\in \cT \}=\{ X\in \cC \mid \Ext{1}{\cC}{X}{T}=0\ \text{for all}\ T\in \cT \}.\]
In particular, this means that $\cT$ is additively closed. We say $\cT$ is \emph{cluster-tilting} if it is also \emph{strongly functorially finite}, meaning that any object of $\cC$ admits both a left $\cT$-approximation which is an inflation and a right $\cT$-approximation which is a deflation. We call an object $T\in\cC$ (weak) cluster-tilting if its additive closure $\add{T}$ is a (weak) cluster-tilting subcategory.
\end{definition}

\begin{remark}
\label{r:strong-ff}
In most of the paper, $\cC$ will have enough projective and injective objects, and be (weakly) idempotent complete.
With these assumptions, if $\cT\subseteq\cC$ is a subcategory containing all projective or injective objects (such as a weak cluster-tilting subcategory), then all left $\cT$-approximations are inflations and all right $\cT$-approximations are deflations (cf.~Proposition~\ref{p:WIC} below).
In particular, any functorially finite subcategory containing all projective and injective objects is strongly functorially finite.
\end{remark}

We will often write $\cT\ctsubcat \cC$ to mean that $\cT$ is a cluster-tilting subcategory of $\cC$.
If $\cC$ is Hom-finite, then any subcategory of the form $\add{T}$ for $T\in\cC$ is functorially finite.
In particular, if $\cC$ is also weakly idempotent complete with enough projective and injective objects, then every weak cluster-tilting object in $\cC$ is cluster-tilting.

\begin{definition}\label{d:CY}
A triangulated $\bK$-category $\cC$ with suspension $\Sigma$ is \emph{$d$-Calabi--Yau}, for $d\in\integ$, if it is Hom-finite and $\Sigma^d$ is a Serre functor, meaning that there is a functorial isomorphism
\[\Hom{\cC}{X}{Y}=\dual{\Hom{\cC}{Y}{\Sigma^dX}}\]
for any $X,Y\in\cC$, where $\dual{(\blank)}=\Hom{\bK}{\blank}{\bK}$. We say a Frobenius extriangulated category $\cC$ is \emph{stably $d$-Calabi--Yau} if its triangulated stable category $\stab{\cC}$ is $d$-Calabi--Yau.
\end{definition}

If $\cC$ is a stably $2$-Calabi--Yau Frobenius category, then it follows from the natural identification $\stabHom{\cC}{X}{\Sigma Y}=\Ext{1}{\cC}{X}{Y}$ that
$\Ext{1}{\cC}{X}{Y}=\dual{\Ext{1}{\cC}{Y}{X}}$
for all $X,Y\in\cC$. Thus, the second equality in the definition of weak cluster-tilting is in fact satisfied for any subcategory of $\cC$.

\begin{remark}
For any extriangulated category $\cC$, a strongly functorially finite subcategory $\cT\subset\cC$ is cluster-tilting if and only if $\cT$ contains all projective objects of $\cC$ and
\begin{equation}
\label{eq:half-ct}
\cT =\{ X\in \cC \mid \Ext{1}{\cC}{T}{X}=0\ \text{for all}\ T\in \cT \}.
\end{equation}
We thank Norihiro Hanihara for pointing out the following argument, based on \cite[Prop.~2.2.2]{Iyama-HigherAR}.
The key step is to check that, under these assumptions, if $\Ext{1}{\cC}{X}{T}=0$ for all $T\in\cT$ then $X\in\cT$. Given such an object $X$, let $r\colon R\to X$ be a right $\cT$-approximation, which exists and is a deflation since $\cT$ is strongly functorially finite.
Thus, there is a conflation
\begin{equation}
\label{eq:half-ct-confl}
\begin{tikzcd}
K\arrow[infl]{r}&R\arrow[defl]{r}{r}&X\arrow[confl]{r}&\phantom{}
\end{tikzcd}
\end{equation}
in $\cC$. Since $r$ is a right $\cT$-approximation, $\Hom{\cT}{T}{r}$ is surjective for any $T\in\cT$. By \eqref{eq:half-ct}, we also have $\Ext{1}{\cC}{T}{R}=0$, and so conclude from Proposition~\ref{p:extri-les} that $\Ext{1}{\cC}{T}{K}=0$. By \eqref{eq:half-ct} again, this means that $K\in\cT$. But then $\Ext{1}{\cC}{X}{K}=0$ by the assumption on $X$, so the conflation \eqref{eq:half-ct-confl} splits and $X\in\cT$.
\end{remark}

To reduce the proliferation of adjectives throughout the paper, we make the following definitions by way of abbreviation.

\begin{definition}
\label{d:clustcat}
A \emph{cluster category} is an idempotent complete, algebraic and stably $2$-Calabi--Yau Frobenius extriangulated category $\cC$ with a cluster-tilting subcategory $\cT$.
\end{definition}

\begin{proposition}
\label{p:WIC}
Given morphisms $f\colon X\to Y$ and $g\colon Y\to Z$ in a cluster category,
\begin{enumerate}
\item if $g\circ f$ is an inflation, then $f$ is an inflation, and
\item if $g\circ f$ is a deflation, then $g$ is a deflation.
\end{enumerate}
\end{proposition}
\begin{proof}
Since cluster categories are idempotent complete, they are weakly idempotent complete \cite[Lem.~A.6.2]{TT}, meaning every retraction has a kernel. Weak idempotent completeness is equivalent to the stated properties of morphisms by \cite[Prop.~2.7]{Klapproth}.
\end{proof}

In particular, the observations in Remark~\ref{r:strong-ff} apply to a cluster category $\cC$ to show that functorial finiteness and strong functorial finiteness are equivalent for weak cluster-tilting subcategories.

Since the stable category $\stab{\cC}$ of a cluster category $\cC$ is always Hom-finite, this property being included for us in the definition of a Calabi--Yau category, triangulated cluster categories in our sense are always Hom-finite. More general cluster categories may not be, and a fuller treatment of these cases requires stronger assumptions, making use of the notion of pseudocompactness; see Section~\ref{ss:pseudocompact} for definitions and background.

\begin{definition}
\label{d:comp-clustcat}
A \emph{compact cluster category} is a cluster category $\cC$ such that
\begin{enumerate}
\item $\cC$ is pseudocompact as an additive category,
\item\label{d:comp-clustcat-finite-ds} $\dimdivalg{X}=\dim_{\bK}\divalg{X}<\infty$ for any $X\in\cC$, where $\divalg{X}=\op{\End{\cC}{X}}/\rad{}\op{\End{\cC}{X}}$, and
\item any $\cT\ctsubcat\cC$ is radically pseudocompact.
\end{enumerate}
\end{definition}

We note that all of the properties in the definition of a (compact) cluster category are self-dual, so that if $\cC$ is a (compact) cluster category, so is $\op{\cC}$. Moreover, if $\cT\ctsubcat\cC$ then $\op{\cT}\ctsubcat\op{\cC}$.
The following proposition is an example of how the topological assumptions in Definition~\ref{d:clustcat} let us extend results for Hom-finite categories to a Hom-infinite setting.

\begin{proposition}
\label{p:KS}
A pseudocompact additive category $\cC$ such that $\dimdivalg{X}<\infty$ for all $X\in\cC$ is Krull--Schmidt.
This applies in particular to compact cluster categories.
\end{proposition}
\begin{proof}
Since $\cC$ is pseudocompact, so is $A=\op{\End{\cC}{X}}$ for any object $X\in\cC$. Then (see e.g.\ \cite[Prop.~2.13]{IMacQ-survey}) there is a direct product decomposition $A=\prod_{i\in I}P_i$ into indecomposable $A$-modules $P_i$ (with local endomorphism rings). The projection maps $A\to P_i$ determine linearly independent elements of $\divalg{X}$. In particular, $|I|\leq\dimdivalg{X}<\infty$, so the product is finite and $A=\bigoplus_{i\in I}P_i$. Thus, $A$ is semi-perfect (see \cite[Prop.~4.1]{KrauseKS}), and so $\cC$ is Krull--Schmidt by \cite[Cor.~4.4]{KrauseKS}.
\end{proof}

\begin{remark}
\label{r:rad-pseudocompact}
While Definition~\ref{d:comp-clustcat} is somewhat technical, it simplifies drastically in many situations of interest. For example, if the cluster-tilting subcategories of $\cC$ are locally finite (Definition~\ref{d:loc-finite}) and additively finite, then their radical pseudocompactness is a consequence of the global assumptions on $\cC$, by Proposition~\ref{p:rad-compact-finite} and Proposition~\ref{p:KS}. In particular, a Hom-finite cluster category $\cC$ with additively finite cluster-tilting subcategories is a compact cluster category, and this includes all the examples of \cite{BMRRT,Amiot,Wu}.
\end{remark}

\begin{example}
The Grassmannian cluster categories \cite{JKS} and related categorifications of more general positroid cluster algebras \cite{PresslandPostnikov} admit a topology on their Hom-spaces making them compact cluster categories. (Strictly speaking, \cite[\S3]{JKS} describes two categories, and our assertion applies to that constructed using complete rings.)

Let $X$ and $Y$ be two objects from a Grassmannian cluster category $\cC$; by definition, these are free and finite rank modules over the power series ring $Z=\powser{\complex}{t}$ in one variable. Since $Z$ is a Noetherian principal ideal domain, it follows that $\Hom{\cC}{X}{Y}\subseteq\Hom{Z}{X}{Y}\iso Z^N$ is also a free and finitely generated $Z$-module. Being a complete local ring, $Z$ may be naturally equipped with the $\fm$-adic topology, where $\fm=(t)$ is the unique maximal ideal, and we extend this topology to all free $Z$-modules. This makes the free and finitely generated $Z$-modules pseudocompact, since this is true for $Z$ itself, as exhibited by the system $(\fm^n)_{n\geq0}$.
The cluster categories for connected positroids described in \cite{PresslandPostnikov} may be realised as full subcategories of Grassmannian cluster categories by \cite[Prop.~3.6]{CKP}. They are thus also enriched in free and finitely generated $Z$-modules, and are hence pseudocompact in the $\fm$-adic topology on these modules.

Now if $\cC$ is a positroid cluster category (including the Grassmannian cluster category), then for any $X\in\cC$, there is a surjection $\op{\End{\cC}{X}}/\fm\op{\End{\cC}{X}}\to\divalg{X}$, and the former is finite-dimensional since $\op{\End{\cC}{X}}$ has finite rank over $Z$.
Similarly, if $\cT\ctsubcat\cC$ and $X,Y\in\cT$, then there is a surjection $\Hom{\cT}{X}{Y}/\fm^2\Hom{\cT}{X}{Y}\to\Hom{\cT}{X}{Y}/\radHom[2]{\cT}{X}{Y}$ from a finite-dimensional vector space, and so $\cT$ is locally finite.
It is also additively finite, hence radically pseudocompact as in Remark~\ref{r:rad-pseudocompact}.

The extriangulated Higgs categories described in \cite{KellerWu} (extending \cite{Wu} from the Hom-finite setting) are enriched in pseudocompact vector spaces by construction. In the case of a Higgs category defined from a finite ice quiver with potential, it again follows as in Remark~\ref{r:rad-pseudocompact} that it is a compact cluster category.
\end{example}

\begin{example}
On the other hand, Igusa--Todorov's categories associated to discs with infinitely many marked points on their boundary \cite{IT} are cluster categories in the sense of Definition~\ref{d:clustcat}, but are not compact.
Although they are Hom-finite, they admit cluster-tilting subcategories $\cT$ with $\rad[\infty]{\cT}\ne0$, and which are thus not radically pseudocompact.
Paquette--Yıldırım's related triangulated categories for completed discs are not cluster categories in our sense at all, since they are not $2$-Calabi--Yau.
Modifying their extriangulated structures as in \cite{CKaP} does not resolve this, since the resulting extriangulated categories are not Frobenius.
\end{example}

\begin{remark}
Demonet \cite{Demonet} shows that if $\cC$ is a Hom-finite exact cluster category and $\Gamma$ is a finite group acting on $\cC$, then the skew group (or $\Gamma$-equivariant) category $\cC\#\Gamma$ is again a Hom-finite exact cluster category.
Demonet uses this construction to categorify cluster algebras with strictly skew-symmetrisable exchange matrices, but to do so uses the extra structure on $\cC$ and $\cC\#\Gamma$ coming from the $\Gamma$-action; this structure is used, for example, to write down a cluster character.
As such, while Demonet's categories are cluster categories in our sense, his results are not special cases of ours.
\end{remark}

\begin{definition}\label{d:skew-symmetric}
Let $\cC$ be a cluster category.  We say $\cC$ is \emph{skew-symmetric} if it is Krull--Schmidt and $\dimdivalg{X}=1$ for all $X\in\indec{\cC}$.
\end{definition}

\begin{remark}
In the sequel, we will take the approach of stating minimal assumptions on our input data, at the cost of making some statements slightly more involved.  To help parse the conditions, we give a brief overview of the relationships between the different assumptions we may need to impose.
\begin{description}
\item[Krull--Schmidt:] our cluster categories are assumed to be idempotent complete but not necessarily Krull--Schmidt; this is because, as we will see later, algebraicity implies that our cluster categories have lifts to exact cluster categories, but we cannot ensure these lifts are Krull--Schmidt.
The Krull--Schmidt property is very often an extra assumption for us, but this assumption is very mild---in particular, any Hom-finite cluster category is already Krull--Schmidt.
We will mainly use it in the standard way, i.e.\ to reduce claims to indecomposables and to have known spanning sets for Grothendieck groups.

\item[compact:] this is a stronger condition than Krull--Schmidt, by Proposition~\ref{p:KS}.  It provides a finiteness condition in the specific sense that simple modules are required to be finite-dimensional.
The pseudocompactness conditions are used to give explicit minimal approximations of objects, via constructions explained in Section~\ref{s:approximations}.

As noted above, it is not enough to be Hom-finite to guarantee compactness but many known examples of cluster categories (both Hom-finite and Hom-infinite) are compact.
A Hom-finite cluster category of finite rank (meaning every cluster-tilting subcategory has the form $\add{T}$ for a cluster-tilting object $T$) is always compact.
\item[skew-symmetric:] this is a different flavour of condition; it implies that the cluster-tilting categories have skew-symmetric exchange matrices.

By requiring a perfect duality of simples and projectives, this assumption allows us to make several key constructions easily, without the compactness assumptions needed in general.
\end{description}
\end{remark}

Given $\cT\ctsubcat\cC$, we write $\Yonfun{\cT}\colon\cC\to\Mod{\cT}$ for the restricted Yoneda functor, with $\Yonfun{\cT}{X}=\Hom{\cC}{\blank}{X}|_{\cT}$; see Section~\ref{s:modules}.
The following observation allows us to apply the results of Section~\ref{s:approximations} to deduce many properties of compact cluster categories (over the perfect ground field $\bK$).

\begin{proposition}
\label{p:ct-Yon-fp}
Let $\cC$ be a cluster category and let $\cT\ctsubcat\cC$. Then for any $X\in\cC$, the $\cT$-module $\Yonfun{\cT}{X}$ is finitely presented. In particular, if $\cC$ is a compact cluster category then $\Yonfun{\cT}{X}$ is radically pseudocompact.
\end{proposition}
\begin{proof}
Since $\cT$ is functorially finite by assumption, there is a right $\cT$-approximation $\varphi\colon R\to X$, with $R\in\cT$, and $\varphi$ is a deflation as in Remark~\ref{r:strong-ff}. We may thus complete $\varphi$ to a conflation
\[\begin{tikzcd}
K\arrow[infl]{r}&R\arrow[defl]{r}{\varphi}&X\arrow[confl]{r}&.
\end{tikzcd}\]
Applying $\Yonfun{\cT}$ yields the exact sequence
\[\begin{tikzcd}
\Yonfun{\cT}R\arrow{r}{\Yonfun{\cT}\varphi}&\Yonfun{\cT}X\arrow{r}&\Extfun{\cT}K\arrow{r}&\Extfun{\cT}R=0,
\end{tikzcd}\]
using that $R\in\cT$. Since $\varphi$ is a right $\cT$-approximation, $\Yonfun{\cT}\varphi$ is surjective, so $\Extfun{\cT}K=0$ and $K\in\cT$. Thus
\[\begin{tikzcd}
\Yonfun{\cT}K\arrow{r}&\Yonfun{\cT}R\arrow{r}&\Yonfun{\cT}X\arrow{r}&0
\end{tikzcd}\]
is a projective presentation, and so $\Yonfun{\cT}X\in\fpmod{\cT}$ is finitely presented.

The final assertion then follows from Proposition~\ref{p:fp-radpc}, using the fact that $\cT$ is radically pseudocompact (by definition) when $\cC$ is a compact cluster category.
\end{proof}

We next show that the partial stabilisation of a cluster category is another cluster category with `the same' cluster-tilting subcategories.
This will allow us to describe mutations of cluster-tilting subcategories in general cluster categories by passing to the $2$-Calabi--Yau triangulated stable category, for which there is a well-developed theory.

\begin{lemma}
\label{l:additive-closure}
Let $\cC$ be an additive category, and $\cP$ a full additively closed subcategory.
Then $\pi\colon\cC\to\cC/\cP$ induces a bijection between full additively closed subcategories of $\cC$ containing $\cP$ and full additively closed subcategories of $\cC/\cP$.
\end{lemma}
\begin{proof}
Full and additively closed subcategories are completely determined by their object sets, and $\pi$ acts as the identity on objects, so we need only show that a set of objects is additively closed in $\cC/\cP$ if and only if it is additively closed in $\cC$ and contains the objects of $\cP$.

Any additively closed subcategory of $\cC/\cP$ contains $\cP$, since objects of $\cP$ are isomorphic to $0$.
Now $X'$ is a summand of $X$ in $\cC/\cP$ if and only if there exists $P\in\cP$ such that $X'$ is a summand of $X\oplus P$ in $\cC$, as follows. The reverse implication holds since $\pi$ preserves split monomorphisms and epimorphisms. For the forward implication, our assumption is that there is an object $P\in\cP$ and maps
\[\begin{tikzcd}
P\arrow[shift left]{r}{f}&X'\arrow[shift left]{l}{g}\arrow[shift right,swap]{r}{i}&X\arrow[shift right,swap]{l}{p}
\end{tikzcd}\]
in $\cC$ such that $pi-fg=1_{X'}$. But then $(\begin{smallmatrix}p&-f\end{smallmatrix})(\begin{smallmatrix}i\\g\end{smallmatrix})=1_{X'}$, and so $X'$ is a summand of $X\oplus P$ in $\cC$.
\end{proof}

\begin{lemma}
\label{l:ct-bijection}
Let $\cC$ be an extriangulated category, and let $\cP\subset\cC$ be a full and additively closed subcategory of projective-injective objects.
Then the quotient functor $\pi\colon\cC\to\cC/\cP$ induces a bijection from (weak) cluster-tilting subcategories of $\cC$ to (weak) cluster-tilting subcategories of $\cC/\cP$.
\end{lemma}
\begin{proof}
We show that the bijection of additively closed subcategories from Lemma~\ref{l:additive-closure} restricts to a bijection of (weak) cluster-tilting subcategories, as follows.
Since $\pi$ is exact and essentially surjective, if $\cT\subseteq\cC$ is weak cluster-tilting then so is $\pi\cT\subseteq\cC/\cP$. Conversely, let $\cT$ be a full and additively closed subcategory of $\cC$ containing $\cP$. If $\pi\cT$ is weak cluster-tilting, then the object set
{\small\begin{equation*}
\{X\in\cC:\Ext{1}{\cC}{T}{X}=0\text{ for all }T\in\cT\}=\{X\in\cC/\cP:\Ext{1}{\cC/\cP}{\pi T}{\pi X}=0\text{ for all }T\in\cT\}
\end{equation*}}coincides with that of $\pi\cT$, and hence with that of $\cT$.
Since both are full and additively closed subcategories of $\cC$, we thus have $\cT=\{X\in\cC:\Ext{1}{\cC}{T}{X}=0\text{ for all }T\in\cT\}$ as categories.
The same argument shows that $\cT=\{X\in\cC:\Ext{1}{\cC}{X}{T}=0\text{ for all }T\in\cT\}$ as well, and so $\cT$ is weak cluster-tilting as required.

Since $\pi$ is full, if $\cT\subset\cC$ is functorially finite then so is $\pi\cT\subset\cC/\cP$; we obtain the required approximations in $\cC/\cP$ by projecting those in $\cC$.
Conversely, assume $\pi\cT\subseteq\cC/\cP$ is functorially finite. Then for any $X\in\cC$ there is a morphism $f\colon T\to X$, with $T\in\cT$, such that $\pi(f)$ is a right $\pi\cT$-approximation of $\pi X$. Since $\cC$ has enough projectives, $X$ has a projective cover $p\colon P\to X$.

Now consider the morphism $P\oplus T\stackrel{(\begin{smallmatrix}p&f\end{smallmatrix})}{\longrightarrow}X$.
We claim this is a right $\cT$-approximation of $X$.
Indeed, if $g\colon T'\to X$ is a morphism such that $T'\in\cT$, then $\pi(g)$ factors over $\pi(f)$.
In other words, there is some morphism $q\colon T'\to X$ such that $\pi(q)=0$ and $g-q$ factors over $f$.
But $\pi(q)=0$ implies that $q$ factors over an object of $\cP$, which is projective.
Thus $q$ factors over $p$, and so $g$ factors over $\cramped{(\begin{smallmatrix}p&f\end{smallmatrix})}$ as required. The existence of left $\cT$-approximations is proved dually.
\end{proof}

Since $\cP\subseteq\cT$ when $\cT\ctsubcat\cC$, we have $\pi\cT\simeq\cT/\cP$, and we will sometimes prefer the latter notation to emphasise the chosen subcategory $\cP$.

For completeness, we now discuss the extent to which our various classes of cluster category are closed under taking partial stabilisations.
For general cluster categories, a partial stabilisation may fail to be idempotent complete, but this is the only obstruction.

\begin{proposition}
\label{p:partialstab}
Let $\cC$ be a cluster category, and let $\cP\subset\cC$ be a full additively closed subcategory of projective-injective objects.
If the partial stabilisation $\cC/\cP$ is idempotent complete, then it is itself a cluster category, and the quotient functor $\pi\colon\cC\to\cC/\cP$ induces a bijection from (weak) cluster-tilting subcategories of $\cC$ to (weak) cluster-tilting subcategories of $\cC/\cP$.
\end{proposition}
\begin{proof}
As in Remark~\ref{r:algebraic}, the category $\cC/\cP$ is algebraic and $\stab{\smash{\cC/\cP}}=\stab{\cC}$, so this stable category is $2$-Calabi--Yau.
Moreover, $\cC/\cP$ is idempotent complete by assumption.
The rest of the statement then follows from Lemma~\ref{l:ct-bijection}; in particular, since $\cC$ has a cluster-tilting subcategory by assumption, so does $\cC/\cP$.
\end{proof}

\begin{remark}
For a cluster category $\cC$ and $\cP\subseteq\cC$ a full and additively closed subcategory of projective-injective objects, we may consider the idempotent completion $\idcomp{(\cC/\cP)}$ of $\cC/\cP$ (also known as the Karoubi envelope, motivating our notation).
This carries a natural Frobenius extriangulated structure \cite{WWZZ}, is idempotent complete by construction, and is $2$-Calabi--Yau.
This last fact can be shown by direct calculation using the explicit description of $\idcomp{\stab{\cC/\cP}}$ by Balmer--Schlichting \cite{BalmerSchlichting}, this triangulated category being equivalent to the stable category of $\idcomp{(\cC/\cP)}$ (that is, idempotent completion commutes with stabilisation; see Msapato \cite[\S3]{Msapato} for more details on results of this kind).
However, it is not clear why $\idcomp{(\cC/\cP)}$ should have a cluster-tilting subcategory in general, since it may have rigid objects which are not direct summands of rigid objects in $\cC/\cP$.
\end{remark}

\begin{proposition}
\label{p:partialstab-KS}
In the setting of Proposition~\ref{p:partialstab}, if $\cC$ is a Krull--Schmidt cluster category, then so is $\cC/\cP$.
\end{proposition}
\begin{proof}
It follows from \cite[Cor.~4.4]{KrauseKS} that $\cC/\cP$ is a Krull--Schmidt category.
Since this means in particular that it is idempotent complete, we may deduce the result from Proposition~\ref{p:partialstab}.
\end{proof}

\begin{proposition}
\label{p:radproj}
In the setting of Proposition~\ref{p:partialstab}, if the cluster category $\cC$ is Krull--Schmidt then $\rad{\cC/\cP}=\pi(\rad{\cC})$.
\end{proposition}
\begin{proof}
The inclusion $\pi(\rad{\cC})\subseteq\rad{\cC/\cP}$ follows directly from the definition since $\pi$ is an additive functor.
For the reverse inclusion, first let $X,Y\in\cC$ be indecomposable.
Since $\pi$ is full, any morphism in $\radHom{\cC/\cP}{X}{Y}$ has the form $\pi(f)$ for some $f\in\Hom{\cC}{X}{Y}$.
Since $\pi(f)$ is in the radical, it is not an isomorphism, and so $f$ cannot be either.
Because $\cC$ is Krull--Schmidt, this means that $f\in\radHom{\cC}{X}{Y}$ \cite[Prop.~2.1(b)]{Bautista}, and hence $\radHom{\cC/\cP}{X}{Y}=\pi(\radHom{\cC}{X}{Y})$.

Now $\cC/\cP$ is also Krull--Schmidt by Proposition~\ref{p:partialstab-KS}, and the isoclasses of indecomposable objects in $\cC/\cP$ are a subset of those for $\cC$ (consisting of the isoclasses of indecomposable objects of $\cC$ which are not in $\cP$).
Thus we deduce that $\radHom{\cC/\cP}{X}{Y}=\pi(\radHom{\cC}{X}{Y})$ for arbitrary $X,Y\in\cC$ via direct sum decomposition (cf.~\cite[Lem.~3.4(b)]{ASS-Book}).
\end{proof}

\begin{remark}
\label{r:D_X}
If $\cC$ is Krull--Schmidt and $X\in\cC$ is indecomposable and not in $\cP$, then every morphism $X\to X$ factoring over $\cP$ lies in $\rad{}\op{\End{\cC}{X}}$. Since the projection $\cC\to\cC/\cP$ respects radicals by Proposition~\ref{p:radproj},
it induces a canonical isomorphism
\[\op{\End{\cC}{X}}/\rad{}\op{\End{\cC}{X}}\isoto\op{\End{\cC/\cP}{X}}/\rad{}\op{\End{\cC/\cP}{X}}.\]
Thus, we may write $\divalg{X}$ unambiguously for either algebra, and $\dimdivalg{X}$ for its dimension, independent of whether we view $X$ as an object of $\cC$ or of $\cC/\cP$.
\end{remark}

\begin{proposition}
\label{p:partialstab-compact}
In the setting of Proposition~\ref{p:partialstab}, if the cluster category $\cC$ is compact and $\cP$ is functorially finite, then the cluster category $\cC/\cP$ is also compact.
\end{proposition}
\begin{proof} We verify the three conditions in the definition of compactness.
\begin{enumerate}
\item\label{p:partialstab-compact-pc} Let $X,Y\in\cC$, and write $\cP(X,Y)$ for the subspace of $\Hom{\cC}{X}{Y}$ consisting of morphisms factoring over $\cP$.
Let $p\colon P\to Y$ be a right $\cP$-approximation of $Y$.
Then $\cP(X,Y)=p_*(\Hom{\cC}{X}{P})$ is the image of the continuous function $p_*$.
Thus it is closed, and so $\Hom{\cC/\cP}{X}{P}=\Hom{\cC}{X}{P}/\cP(X,Y)$ is pseudocompact in the quotient topology (cf.~Proposition~\ref{p:pc-abelian}).
Composition in $\cC/\cP$ is continuous in this topology by functoriality of the projection $\cC\to\cC/\cP$.
\item\label{p:partialstab-compact-dX} This follows immediately from the fact that $\op{\End{\cC/\cP}{X}}/\rad{}{\op{\End{\cC/\cP}{X}}}$ is a quotient of $\op{\End{\cC}{X}}/\rad{}{\op{\End{\cC}{X}}}$ (cf.~Remark~\ref{r:D_X}).
\item
By Lemma~\ref{l:ct-bijection}, every cluster-tilting subcategory of $\cC/\cP$ has the form $\pi\cT$ for some $\cT\ctsubcat\cC$.
Note that $\pi\cT\ctsubcat\cC/\cP$ is pseudocompact in the quotient topology induced from the radical topology on $\cT$; this is proved analogously to part \ref{p:partialstab-compact-pc}, using that $\cT$ is radically pseudocompact and $\cP\subset\cT$.
We claim that this quotient topology is coarser than the radical topology, so any system of open sets exhibiting pseudocompactness of $\Hom{\pi\cT}{X}{Y}$ in the quotient topology also exhibits pseudocompactness in the radical topology, demonstrating that $\pi\cT$ is radically pseudocompact as required.

To prove the claim, let $V\subset\Hom{\pi\cT}{X}{Y}$ be closed in the quotient topology, so $\pi^{-1}(V)\subset\Hom{\cT}{X}{Y}$ is closed in the radical topology on $\cT$. Thus, \begin{equation}
\label{eq:closure}
\pi^{-1}(V)=\overline{\pi^{-1}(V)}=\bigcap_{n\in\nat}V+\radHom[n]{\cT}{X}{Y}.
\end{equation}
Since $\cC$ is Krull--Schmidt by Proposition~\ref{p:KS}, and $\cT\subseteq\cC$ and $\pi\cT\subseteq\cC/\cP$ are full, we have $\radHom{\pi\cT}{X}{Y}=\pi(\radHom{\cT}{X}{Y})$ for all $X,Y\in\cT$ by Proposition~\ref{p:radproj}.
Thus applying $\pi$ to \eqref{eq:closure} gives
\[V=\bigcap_{n\in\nat}V+\pi(\radHom[n]{\cT}{X}{Y})=\bigcap_{n\in\nat}V+\radHom[n]{\pi\cT}{X}{Y}=\overline{V},\]
and so $V$ is closed in the radical topology on $\pi\cT$. \qedhere
\end{enumerate}
\end{proof}

A technique we will apply repeatedly is to deduce a statement about general cluster categories via algebraicity: by the following result, it suffices to prove the statement for exact cluster categories and then show that it is preserved under partial stabilisation.

\begin{proposition}\label{p:cl-cat-exact-lift}
For any cluster category $\cC$, there is an exact cluster category $\cE$ and a full and additively closed subcategory $\cP\subseteq\cE$ of projective objects such that $\cE/\cP\simeq\cC$.
\end{proposition}
\begin{proof}
Because $\cC$ is algebraic, there is an exact category $\cE$ and a full and additively closed subcategory $\cP$ of projectives such that $\cC\simeq\cE/\cP$. Let $\idcomp{\cE}$ be the idempotent completion of $\cE$, which is still exact \cite[\S6.1]{Buehler}, and naturally contains the idempotent completion $\idcomp{\cP}$, objects of which are projective in $\idcomp{\cE}$. Using the universal properties of quotient categories and idempotent completion, we see that $\idcomp{\cE}/\idcomp{\cP}$ is the idempotent completion of $\cC$ (as an extriangulated category \cite{WWZZ}). But $\cC$ is idempotent complete, being a cluster category, and so $\cC\simeq\idcomp{\cE}/\idcomp{\cP}$. We may thus assume, without loss of generality, that $\cE$ is idempotent complete.

Now since $\cC$ is Frobenius, so is $\cE$, as pointed out in Definition~\ref{d:partial-stab}. Since $\cE$ is exact, it is automatically algebraic, and it is stably $2$-Calabi--Yau as in Remark~\ref{r:algebraic}. By Lemma~\ref{l:ct-bijection}, it has a cluster-tilting subcategory, and is hence a cluster category.
\end{proof}

\begin{remark}
We do not currently know if the analogous statement for compact cluster categories---namely, that every compact cluster category has the form $\cE/\cP$ for a compact exact cluster category $\cE$ and a full and additively closed subcategory $\cP$ of projectives---is also true.
We also do not have analogous statements for Krull--Schmidt cluster categories, or skew-symmetric cluster categories (except when $\bK$ is algebraically closed).
\end{remark}

\subsection{Cluster structures}
\label{s:cluster-structures}
To begin to relate our categories to Fomin--Zelevinsky's theory of cluster algebras, we will require that our cluster-tilting subcategories have a well-defined concept of mutation.
For the closest possible relationship, under which we may decategorify to such a cluster algebra, we will also need categorical mutation to be compatible with Fomin--Zelevinsky mutation of skew-symmetrisable matrices.

Let $\cA$ be a Krull--Schmidt category. For $X,Y\in\indec{\cA}$, we may define
\[\irr{\cA}{X}{Y}=\radHom{\cA}{X}{Y}/\radHom[2]{\cA}{X}{Y}.\]
Elements of this space are sometimes referred to as irreducible maps from $X$ to $Y$, although strictly they are equivalence classes of maps. If these spaces are finite-dimensional (for example, if $\cA$ is locally finite as in Definition~\ref{d:loc-finite}), then we may further define
\[\Gabmatentry{X,Y}=\rank_{\divalg{X}}\irr{\cA}{X}{Y}<\infty\]
and thus obtain a (possibly infinite) integer-valued $\indec{\cA}\times\indec{\cA}$ matrix $\Gabmat{\cA}=(\Gabmatentry{X,Y})$, the \emph{Cartan matrix} of $\cA$.

\begin{proposition}
\label{p:loc-finite-quiver}
A Krull--Schmidt category $\cA$ is locally finite at $X\in\indec\cA$ if and only if
\begin{enumerate}
\item\label{p:loc-finite-quiver-dX-finite} $\dimdivalg{X}<\infty$,
\item\label{p:loc-finite-quiver-cXY-finite} $\dimdivalg{Y}\Gabmatentry{X,Y},\dimdivalg{Y}\Gabmatentry{Y,X}<\infty$ for all $Y\in\indec{\cA}$, and
\item\label{p:loc-finite-quiver-cXY-zero} $\Gabmatentry{X,Y}=0=\Gabmatentry{Y,X}$ for all but finitely many $Y\in\indec{\cA}$.
\end{enumerate}
\end{proposition}
\begin{proof}
As in Section~\ref{s:modules}, we write $\Yonfun{\cA}{X}=\Hom{\cA}{\blank}{X}$, and $\opYonfun{\cA}{X}=\Hom{\cA}{X}{\blank}$ for each $X\in\cA$.
We have
\[\dim_{\bK}(\Yonfun{\cA}{X}/\rad[2]{\cA}{\Yonfun{\cA}{X}})(Y)= \dimdivalg{Y}(\Gabmatentry{Y,X}+\delta_{Y,X}),\quad
\text{for }\delta_{Y,X}=\begin{cases}1&\text{if}\ X\iso Y,\\0&\text{otherwise.}\end{cases}\]
Indeed, $\irr{\cA}{Y}{X}=\radHom{\cA}{Y}{X}/\radHom[2]{\cA}{Y}{X}$ is free of rank $\Gabmatentry{Y,X}$ over the $\dimdivalg{Y}$-dimensional division algebra $\divalg{Y}$, and
$\dim_{\bK}\Hom{\cA}{Y}{X}/\radHom{\cA}{Y}{X}=\dimdivalg{Y}\delta_{Y,X}$ since $X$ and $Y$ are indecomposable.

Thus
\[
\sum_{Y\in\indec{\cA}}\dim_{\bK}(\Yonfun{\cA}{X}/\rad[2]{\cA}{\Yonfun{\cA}{X}})(Y)=\sum_{Y\in\indec{\cA}}\dimdivalg{Y}(\Gabmatentry{Y,X}+\delta_{Y,X}),\]
and similarly
\[\sum_{Y\in\indec{\cA}}\dim_{\bK}(\opYonfun{\cA}{X}/\rad[2]{\cA}{\opYonfun{\cA}{X}})(Y)=\sum_{Y\in\indec{\cA}}\dimdivalg{Y}(\Gabmatentry{X,Y}+\delta_{Y,X}).\]
Now the local finiteness condition at $X$ is equivalent to the finiteness of the two sums on the left of these expressions, while finiteness of the two sums on the right is equivalent to conditions \ref{p:loc-finite-quiver-dX-finite}--\ref{p:loc-finite-quiver-cXY-zero}, giving the result.
\end{proof}

In the case that $\divalg{X}=\bK$ for all $X\in\indec{\cA}$, such as if $\bK$ is algebraically closed, it is natural to associate to $\cA$ the \emph{Gabriel quiver} $\Gabquiv{\cA}$, with vertex set $\indec{\cA}$ and $\Gabmatentry{X,Y}=\dim_{\bK}\irr{\cA}{X}{Y}$ arrows from $Y$ to $X$. The matrix $\Gabmat{\cA}$ is thus the adjacency matrix of this Gabriel quiver (in a convention compatible with Fomin--Zelevinsky's for exchange matrices of quivers, as appearing below, whereby the arrows leaving vertex $X$ are indicated in the column labelled $X$, rather than the row). By Proposition~\ref{p:loc-finite-quiver}, local finiteness of $\cA$ is then equivalent to local finiteness of $\Gabquiv{\cA}$, where a quiver is locally finite if each of its vertices is incident with finally many arrows.

\begin{remark}
\label{r:Gab-quiv}
The terminology here stems from the fact that if $\cA$ is radically pseudocompact and additively finite with basic additive generator $E$, then $Q(\cA)$ is the usual Gabriel quiver of the radically pseudocompact algebra $A=\op{\End{\cA}{E}}$, with $\proj{A}\simeq\cA$. (This motivates the orientation of arrows in $Q(\cA)$ being opposite to the direction of morphisms in $\cA$.) While the construction of $Q(\cA)$ makes sense without the assumption that $\divalg{X}=\bK$ for all $X$, under this assumption the algebra $A$ is isomorphic to the complete path algebra of $Q(\cA)$ over $\bK$, modulo a closed ideal contained in the square of the arrow ideal.
\end{remark}

\begin{definition}\label{d:exch-mat}
Let $\cC$ be a Krull--Schmidt cluster category, and let $\cT\ctsubcat\cC$. The set of isomorphism classes of indecomposable non-projective objects of $\cT$ is naturally identified with $\indec{\stab{\cT}}$, where $\stab{\cT}$ denotes the image of $\cT$ in the stable category $\stab{\cC}$. If $\irr{\cT}{X}{Y}$ has finite rank over $\divalg{X}$ and $\op{\divalg{Y}}$ whenever $X,Y\in\indec{\cT}$ and at least one of $X$ or $Y$ lies in $\indec{\stab{\cT}}$, then we define the \emph{exchange matrix} $\exchmat{\cT}$ to be the $\indec{\cT}\times\indec{\stab{\cT}}$ matrix with entries $(\exchmatentry{X,Y})_{X\in \indec \cT, Y\in \indec \stab{\cT}}$, where
\begin{equation}
\label{eq:exch-mat-def}
\exchmatentry{X,Y}=\rank_{\divalg{X}}\irr{\cT}{X}{Y}-\rank_{\op{\divalg{X}}}\irr{\cT}{Y}{X}.
\end{equation}
We say that $\cT$ \emph{has an exchange matrix} when the finite rank assumptions necessary to define $B_{\cT}$ are satisfied.
\end{definition}

Any locally finite cluster-tilting subcategory $\cT\ctsubcat\cC$ has an exchange matrix by Proposition~\ref{p:loc-finite-quiver}.
If $\cC$ is skew-symmetric (Definition~\ref{d:skew-symmetric}), then $B_{\stab{\cT}}=(\exchmatentry{X,Y})_{X,Y\in \indec \stab{\cT}}$ is a skew-symmetric matrix; this is the reason for the terminology.
This upper $\indec{\stab{\cT}}\times\indec{\stab{\cT}}$ submatrix of $B_{\cT}$ is always well-defined (i.e.\ has finite entries) by the assumption that $\stab{\cC}$ is Hom-finite, noting that $(\cP)(X,Y)\subseteq\radHom[2]{\cC}{X}{Y}$ for $X,Y\in\indec{\stab{\cT}}$.

The entries of $\exchmat{\cT}$ are related to those of $\Gabmat{\cT}$ by the formula
\begin{equation}
\label{eq:exch-mat-vs-Gab-mat}
\dimdivalg{X}\exchmatentry{X,Y}=\dimdivalg{X}\Gabmatentry{X,Y}-\dimdivalg{Y}\Gabmatentry{Y,X},
\end{equation}
where $X\in \indec \cT$ and $Y\in \indec \stab{\cT}$, recalling that \[\dimdivalg{X}\rank_{\op{\divalg{X}}}\irr{\cT}{Y}{X}=\dim_{\bK}\irr{\cT}{Y}{X}=\dimdivalg{Y}\rank_{\divalg{Y}}\irr{\cT}{Y}{X}.\] 
In particular, $\exchmatentry{X,X}=0$ for any $X\in \indec \stab{\cT}$ and $\dimdivalg{X}\exchmatentry{X,Y}=-\dimdivalg{Y}\exchmatentry{Y,X}$.
The $\indec{\stab{\cT}}\times\indec{\stab{\cT}}$ matrix with $(X,Y)$-th entry $\dimdivalg{X}\exchmatentry{X,Y}$ is therefore skew-symmetric; in other words, the $\indec{\stab{\cT}}\times\indec{\stab{\cT}}$ submatrix of $B_{\cT}$ is skew-symmetrizable.

\begin{definition}
\label{d:no-loops}
Let $\cC$ be a Krull--Schmidt cluster category, and let $\cT\ctsubcat\cC$. 
We say $\cT$ has \emph{no loops} at $X\in \indec{\cT}$ if $\irr{\cT}{X}{X}=0$.  We say $\cT$ has \emph{no $2$-cycles} at $X$ if for every $Y\in\indec{\cT}$, either $\irr{\cT}{X}{Y}=0$ or $\irr{\cT}{Y}{X}=0$.
We further say that $\cT$ has no loops (respectively, no $2$-cycles) if it has no loops at $X$ (resp., no $2$-cycles at $X$) for any $X\in \indec{\stab{\cT}}$.
\end{definition}

\begin{proposition}
\label{p:exch-mat}
If $\cT$ has an exchange matrix and has no loops or $2$-cycles, then
\[\dimdivalg{X}\exchmatentry{X,Y}=\begin{cases}\dimdivalg{X}\Gabmatentry{X,Y}&\text{if}\ \exchmatentry{X,Y}\geq0,\\-\dimdivalg{Y}\Gabmatentry{Y,X}&\text{if}\ \exchmatentry{X,Y}<0\end{cases}\]
for any $X\in\indec{\cT}$ and $Y\in\indec{\stab{\cT}}$.
\end{proposition}

\begin{proof}
If $\cT$ has no loops or $2$-cycles, then only one term on the right-hand side of \eqref{eq:exch-mat-vs-Gab-mat} can be non-zero.
The result then follows since all of $\dimdivalg{X}$, $\dimdivalg{Y}$, $\Gabmatentry{X,Y}$ and $\Gabmatentry{Y,X}$ are dimensions of vector spaces, ergo non-negative.
\end{proof}

\begin{remark}
If $\divalg{X}=\bK$ for all $X\in\indec{\cT}$ (in particular, if $\cC$ is skew-symmetric), so that it makes sense to discuss the Gabriel quiver $\Gabquiv{\cT}$, then we partition the vertices of this quiver into frozen vertices, corresponding to indecomposable projectives in $\cT$, and mutable vertices, corresponding to elements of $\indec{\stab{\cT}}$.
Then $\cT$ has no loops or $2$-cycles if and only if the quiver $\Gabquiv{\cT}$ has no loops or $2$-cycles \emph{at its mutable vertices}, and in this case $B_{\cT}$ is the usual exchange matrix associated to the (ice) quiver $\Gabquiv{\cT}$.
\end{remark}

\begin{definition}
For $a\in \integ$, set $[a]_{+}=\max\{a,0\}$ and $[a]_{-}=\max\{0,-a\}$.
\end{definition}

\begin{corollary}\label{c:exch-mat-plus-minus}
If $\cT$ has an exchange matrix and has no loops or $2$-cycles, then $[\exchmatentry{X,Y}]_{+}=\Gabmatentry{X,Y}$ and $[\exchmatentry{X,Y}]_{-}=\frac{\dimdivalg{Y}}{\dimdivalg{X}}\Gabmatentry{Y,X}$. \qed
\end{corollary}

We note too that $[\exchmatentry{Y,X}]_{+}=\Gabmatentry{Y,X}=\frac{\dimdivalg{X}}{\dimdivalg{Y}}[\exchmatentry{X,Y}]_{-}$, and similarly $[\exchmatentry{Y,X}]_{-}=\frac{\dimdivalg{X}}{\dimdivalg{Y}}[\exchmatentry{X,Y}]_{+}$, so that $[\exchmatentry{X,Y}]_{\pm}=\frac{\dimdivalg{Y}}{\dimdivalg{X}}[\exchmatentry{Y,X}]_{\mp}$.

\begin{definition}
\label{d:mutable}
Let $\cC$ be a Krull--Schmidt cluster category, let $\cT\ctsubcat\cC$ and let $T\in\indec{\cT}$. We write $\cT\setminus T$ as shorthand for the additively closed subcategory with indecomposables $\indec{\cT}\setminus\{T\}$. We say that $T$ is \emph{mutable} in $\cT$ if
\begin{enumerate}
\item\label{d:cluster-structure-mutation} there is an indecomposable object $\mut{\cT}{T}\in\cC$, not isomorphic to $T$, such that $\mut{T}{\cT}\defeq \add(\cT\setminus T \union \{\mut{\cT}{T}\})$ is cluster-tilting, and
\item\label{d:cluster-structure-exch-seq}there are non-split conflations
\begin{equation}
\label{eq:exchange-confs}
\begin{tikzcd}
\mut{\cT}{T}\arrow[infl]{r}&\exchmon{\cT}{T}{+}\arrow[defl]{r}{\varphi^+}&T\arrow[confl]{r}&,
\end{tikzcd}\quad
\begin{tikzcd}
T\arrow[infl]{r}{\varphi^-}&\exchmon{\cT}{T}{-}\arrow[defl]{r}&\mut{\cT}{T}\arrow[confl]{r}&,
\end{tikzcd}
\end{equation}
which we call \emph{exchange conflations}, such that $\varphi^+$ and $\varphi^-$ are, respectively, right and left $(\cT\setminus T)$-approximations of $T$.
\end{enumerate}
In this situation, $(T,\mut{\cT}{T})$ is called an \emph{exchange pair}, the cluster-tilting subcategory $\mu_T\cT$ is called the \emph{mutation} of $\cT$ at $T$ (noting that $\mu_{\cT}T$ is unique up to isomorphism by \ref{d:cluster-structure-exch-seq}, so $\mu_T\cT$ is well-defined), and the \emph{mutation class} of $\cT$ consists of those cluster-tilting subcategories obtainable via some sequence of mutations starting at $\cT$. Cluster-tilting subcategories in this mutation class, and their objects, are said to be \emph{reachable} from $\cT$. We write $\exch{\cT}$ for the set of mutable indecomposable objects of $\cT$.
\end{definition}

The notation $\mu_{\cT}T$, for the mutation of $T$ in the cluster-tilting subcategory $\cT$, should not be confused with the similar notation $\mu_{T}\cT$, for the mutation of the cluster-tilting subcategory $\cT$ at the indecomposable $T$.

If $P\in\indec\cC$ is projective-injective, then it is an object of $\indec{\cT}$ for all $\cT\ctsubcat\cC$, but it is never mutable, since it cannot satisfy either part of  Definition~\ref{d:mutable}.
Thus $\exch{\cT}\subseteq\indec{\stab{\cT}}$, the latter set identifying with the indecomposable non-projective objects of $\cT$ as above.

\begin{definition}
\label{d:max-mut}
We say that $\cT\ctsubcat\cC$ is \emph{maximally mutable} if $\exch{\cT}=\indec{\stab{\cT}}$.
\end{definition}

\begin{proposition}
\label{p:mut-v-lf}
If $\cC$ is a Krull--Schmidt cluster category and $\cT\ctsubcat\cC$, then $\cT$ is locally finite at all $T\in\exch{\cT}$.
\end{proposition}
\begin{proof}
Since $\cT$ contains all projective-injective objects of $\cC$, but $T\in\exch{\cT}$ cannot be such an object, every endomorphism of $T$ factoring over a projective-injective lies in $\radHom[2]{\cT}{T}{T}$. It follows that $\Hom{\cT}{T}{T}/\radHom[2]{\cT}{T}{T}$ is a quotient of $\stabHom{\cC}{T}{T}$, and is hence finite-dimensional since $\cC$ is a cluster category.

Now let $U\in\indec{\cT}$ with $U\ne T$. Since $T\in\exch{\cT}$, there is a right $(\cT\setminus T)$-approximation $\varphi^+\colon\exchmon{\cT}{T}{+}\to T$. In particular, $\varphi^+\in\radHom{\cT}{\exchmon{\cT}{T}{+}}{T}$, since it does not split, and every morphism $U\to T$ factors over $\varphi^+$ since $U\in\cT\setminus T$.
As a result, $\dim_{\bK}\Hom{\cT}{U}{T}/\radHom[2]{\cT}{U}{T}\leq\dim_{\bK}\Hom{\cT}{U}{\exchmon{\cT}{T}{+}}/\radHom{\cT}{U}{\exchmon{\cT}{T}{+}}$.
But the latter dimension counts the number of appearances of $U$ in a direct sum decomposition of $\exchmon{\cT}{T}{+}$. Since $\cC$ is Krull--Schmidt, this number is finite, and even zero for all but finitely many $U\in\indec{\cU}$, and hence $\Yonfun{\cT}{T}/\rad[2]{\cT}\Yonfun{\cT}{T}\in\fd{\cT}$.
An analogous argument using the left approximation $\varphi^-\colon T\to\exchmon{\cT}{T}{-}$ demonstrates that $\opYonfun{\cT}{T}/\rad[2]{\cT}\opYonfun{\cT}{T}\in\fd{\op{\cT}}$, completing the proof. 
\end{proof}

\begin{corollary}
\label{c:KS-max-mut}
In a Krull--Schmidt cluster category $\cC$, any maximally mutable cluster-tilting subcategory has an exchange matrix.\qed
\end{corollary}

\begin{proposition}
\label{p:mut-v-ff}
If $\cC$ is a Krull--Schmidt cluster category, $\cT\ctsubcat\cC$, and $T\in\indec{\cT}$ is not projective, then $T$ is mutable in $\cT$ if and only if $\cT\setminus T$ is functorially finite in $\cC$.
\end{proposition}
\begin{proof}
The forward implication is a direct consequence of Definition~\ref{d:mutable}\ref{d:cluster-structure-exch-seq}.
Conversely, assume that $\cD=\cT\setminus T$ is functorially finite in $\cT$. In the case that $\cC$ is a triangulated category, \ref{d:cluster-structure-mutation} follows from work of Iyama and Yoshino \cite[Thm.~5.3]{IyamaYoshino} once we show that $\cD$ is an almost-complete $2$-cluster-tilting subcategory in the sense of \cite[Def.~5.2]{IyamaYoshino}. We first observe that the autoequivalence $\mathbb{S}_2\colon\mathcal{C}\to\mathcal{C}$ referred to in \cite{IyamaYoshino} is the identity in our case, since $\mathcal{C}$ is $2$-Calabi--Yau. Thus $\cD$ is automatically closed under this functor, and $\indec{\cT}\setminus\indec{\cD}=\{T\}$ is a singleton (i.e.\ a single $\mathbb{S}_2$-orbit) by construction. Since $\cD$ is functorially finite in $\cT$, and $\cT$ is functorially finite in $\cC$ since it is cluster-tilting, it follows that $\cD$ is functorially finite in $\cC$, and hence is almost complete $2$-cluster-tilting.

Since $(\cT,\mu_T\cT)$ and $(\mu_T\cT,\cT)$ are both $\cD$-mutation pairs by \cite[Thm.~5.3]{IyamaYoshino}, property \ref{d:cluster-structure-exch-seq} follows from Iyama--Yoshino's results on such pairs \cite[\S2]{IyamaYoshino} (especially \cite[Prop.~2.6]{IyamaYoshino}).

Now if $\cC$ is a general cluster category, the above argument shows that properties \ref{d:cluster-structure-mutation} and \ref{d:cluster-structure-exch-seq} hold for the stable category $\stab{\cC}$. Property \ref{d:cluster-structure-mutation} for $\cC$ then follows by applying Lemma~\ref{l:ct-bijection} in the case that $\cC/\cP=\stab{\cC}$. For property \ref{d:cluster-structure-exch-seq}, it follows from the definition of the triangulated structure on $\stab{\cC}$ that there are conflations
\[\begin{tikzcd}
\mut{\cT}{T}\arrow[infl]{r}&\exchmon{\cT}{T}{+}\arrow[defl]{r}{\varphi^+}&T\arrow[confl]{r}&\phantom{},
\end{tikzcd}\quad
\begin{tikzcd}
T\arrow[infl]{r}{\varphi^-}&\exchmon{\cT}{T}{-}\arrow[defl]{r}&\mut{\cT}{T}\arrow[confl]{r}&\phantom{},
\end{tikzcd}\]
in $\cC$ which project to the exchange triangles in $\stab{\cC}$, and we claim that these are the desired exchange conflations.

Minimality of $\varphi^+$ follows from the fact that the class of $\varphi^+$ in $\stab{\cC}$ is minimal, together with the fact that $\mut{\cT}{T}$ is indecomposable non-projective (and so in particular has no projective direct summands). Let $f\colon T'\to T$ be a morphism with $T'\in\cT\setminus T$. Since $\varphi^+$ projects to a right $(\cT\setminus T)$-approximation in $\stab{\cC}$, there exists a morphism $f'\colon T'\to T$ such that $f-\varphi^+f'$ factors over a projective object in $\cC$. So write $f-\varphi^+f'=qp$, where $p\colon T'\to P$, $q\colon P\to T$ and $P$ is projective. Since $\varphi^+$ is a deflation in $\cC$ and $P$ is projective, there exists $q'\colon P\to\exchmon{\cT}{T}{+}$ such that $q=\varphi^+q'$. It follows that $f=\varphi^+(f'+q'p)$ factors over $\varphi^+$, and so $\varphi^+$ is a right $(\cT\setminus T)$-approximation as required. The proof that $\varphi^-$ is a minimal left $(\cT\setminus T)$-approximation is similar.
\end{proof}

\begin{proposition}
\label{p:mutability-stab}
Let $\cC$ be a Krull--Schmidt cluster category, $\cP\subseteq\cC$ a full additively closed subcategory of projective-injective objects with associated quotient $\pi\colon\cC\to\cC/\cP$, and $\cT\ctsubcat\cC$, so $\pi\cT\ctsubcat\cC/\cP$. Then $T\in\exch\cT$ if and only if $T\in\exch{\pi\cT}$.
\end{proposition}
\begin{proof}
If $T\in\exch{\cT}$, then $T\notin\cP$, and so $T\in\indec{\pi\cT}$. Moreover, $\cT\setminus T$ is functorially finite in $T$ by Proposition~\ref{p:mut-v-ff}.
Therefore $\pi\cT\setminus T$ is functorially finite in $\pi T$, since the required approximations can be obtained by projection from $\cC$, and so $T$ is mutable in $\pi\cT$ by Proposition~\ref{p:mut-v-ff} again.

For the converse implication, we may compute $\mut{\pi\cT}{T}\in\cC/\cP$.
Then $\add(\cT\setminus T\cup\{\mut{\pi\cT}{T}\})$ corresponds to $\mut{T}{(\pi{\cT})}$ under the bijection of Lemma~\ref{l:ct-bijection}, and hence is cluster-tilting as required by Definition~\ref{d:mutable}\ref{d:cluster-structure-mutation}.
The required exchange conflations are obtained by lifting those in $\cC/\cP$, exploiting the fact that $\cC$ and $\cC/\cP$ have the same extension spaces.
Indeed, let
\[\begin{tikzcd}
\mut{\pi\cT}{T}\arrow[infl]{r}&\exchmon{\pi\cT}{T}{+}\oplus P_T\arrow[defl]{r}{\varphi^+}&T\arrow[confl]{r}&\phantom{}
\end{tikzcd}\]
be a conflation in $\cC$ obtained by lifting an exchange conflation in $\cC/\cP$, so $P_T\in\cP$.
Then any morphism $P\to T$ with $P\in\cP$ factors over $\varphi^+$ since this morphism is a deflation and $P$ is projective (see Proposition~\ref{p:extri-les}).
Given this, the fact that $\varphi^+$ is a right $(\cT\setminus T)$-approximation follows from the fact that it projects to a right $(\pi\cT\setminus T)$-approximation in $\cC/\cP$.
The statement for the other exchange conflation is proved similarly, using that objects of $\cP$ are injective.
\end{proof}
\begin{corollary}
\label{c:max-mutable}
If $\cC$ is a Krull--Schmidt cluster category, then $\cT\ctsubcat\cC$ is maximally mutable if and only if $\stab{\cT}\ctsubcat\stab{\cC}$ is maximally mutable.\qed
\end{corollary}

Given some additional assumptions local to $T\in\exch{\cT}$, we may describe the middle terms of the exchange conflations precisely, as follows.

\begin{proposition}\label{p:decomp-exch-terms}
Let $\cC$ be a compact cluster category, $\cT \ctsubcat \cC$, and $T\in\exch{\cT}$.
If $\cT$  has no loop at $T$, then the middle terms of the exchange conflations have the following decompositions into indecomposable objects:
\[
\exchmon{\cT}{T}{+}  = \bigdsum_{U\in \indec \cT\setminus T} U^{\Gabmatentry{U,T}}, \quad
\exchmon{\cT}{T}{-}  = \bigdsum_{U\in \indec \cT\setminus T} U^{\frac{\dimdivalg{T}}{\dimdivalg{U}}\Gabmatentry{T,U}}.
\]
In particular, $\exchmatentry{U,T}^\cT=[\exchmon{\cT}{T}{+}:U]-[\exchmon{\cT}{T}{-}:U]$.
If additionally there is no $2$-cycle at $T$, then we even have $\exchmon{\cT}{T}{+}=\bigdsum_{U\in \indec \cT\setminus T} U^{[\exchmatentry{U,T}]_{+}}$ and $\exchmon{\cT}{T}{-}=\bigdsum_{U\in \indec \cT\setminus T} U^{[\exchmatentry{U,T}]_{-}}$.
\end{proposition}

\begin{proof}
Since $\cT$ is locally finite at $T$ by Proposition~\ref{p:mut-v-lf}, we may apply Lemma~\ref{l:sink-map-construction} to see that the object $T$ has a source map in $\cT$ with codomain $\bigdsum_{U\in \indec \cT\setminus T} U^{\frac{\dimdivalg{T}}{\dimdivalg{U}}\Gabmatentry{T,U}}$ and a sink map in $\cT$ with domain $\bigdsum_{U\in \indec \cT\setminus T} U^{\Gabmatentry{U,T}}$.
Here we use again that $\dimdivalg{U}\rank_{\op{\divalg{U}}}\irr{\cT}{T}{U}=\dim_{\bK}\irr{\cT}{T}{U}=\dimdivalg{T}\rank_{\divalg{T}}\irr{\cT}{T}{U}$ to write the codomain of the source map in terms of the Cartan matrix, this argument also showing that the exponent is a non-negative integer (indeed, it is the $(U,T)$-entry for the Cartan matrix of $\op{\cT}$).
Since there is no loop at $T$, these maps are also minimal left and right $(\cT\setminus T)$-approximations, and so are isomorphic to $\exchmon{\cT}{T}{-}$ and $\exchmon{\cT}{T}{+}$ respectively, as required.

Now $[\exchmon{\cT}{T}{+}:U]-[\exchmon{\cT}{T}{-}:U]=\Gabmatentry{U,T}-\frac{\dimdivalg{T}}{\dimdivalg{U}}\Gabmatentry{T,U}=\exchmatentry{U,T}^\cU$ by \eqref{eq:exch-mat-vs-Gab-mat}, and the final statement follows by Corollary~\ref{c:exch-mat-plus-minus}.
\end{proof}

\begin{corollary}\label{c:exch-mat-at-mut}
In the setting of Proposition~\ref{p:decomp-exch-terms}, we have
\[ \exchmatentry{U,\mut{\cT}{T}}^{\mut{T}{\cT}}=-\exchmatentry{U,T}^{\cT}. \]
\end{corollary}

\begin{proof}
The exchange conflations for the pair $T$ and $\mut{\cT}{T}$ are \[\begin{tikzcd}
\mut{\cT}{T}\arrow[infl]{r}{\psi^-}&\exchmon{\cT}{T}{+}\arrow[defl]{r}{\varphi^+}&T\arrow[confl]{r}&\phantom{},
\end{tikzcd}\quad
\begin{tikzcd}
T\arrow[infl]{r}{\varphi^-}&\exchmon{\cT}{T}{-}\arrow[defl]{r}{\psi^+}&\mut{\cT}{T}\arrow[confl]{r}&\phantom{}.
\end{tikzcd}\]
Since $\Ext{1}{\cC}{T}{T'}=0$ for any $T'\in\mu_T\cT\setminus \mu_\cT T = \cT\setminus T$, the map $\psi^-$ is a left $(\mu_T\cT\setminus \mu_\cT T)$-approximation of $\mu_\cT T$. Similarly, $\psi^+$ is a right $(\mu_T\cT\setminus \mu_\cT T)$-approximation, and hence $\exchmon{\mut{T}{\cT}}{(\mut{\cT}{T})}{\pm}=\exchmon{\cT}{T}{\mp}$.
Calculating $\exchmatentry{U,T}^{\cT}$ and $\exchmatentry{U,\mut{\cT}{T}}^{\mut{T}{\cT}}$ using the formula from Proposition~\ref{p:decomp-exch-terms} then gives the result.	
\end{proof}

The following lemma, which will be useful in Section~\ref{s:clust-char}, applies in particular to the case that $\cC$ is a compact cluster category, $\cT\ctsubcat\cC$ and $(X,Y)=(T,\mut{\cT}{T})$ for $T\in\exch{\cT}$ such that $\cT$ has no loop at $T$.

\begin{lemma}
\label{l:rk1-unique-middle-term}
Let $\cC$ be a pseudocompact $\bK$-linear category and let $X,Y\in\cC$. Assume that $X$ is indecomposable, $\dimdivalg{X}<\infty$, and $\rank_{\divalg{X}}\Ext{1}{\cC}{X}{Y}=1$, where the $\divalg{X}$-structure derives from a choice of splitting as in Proposition~\ref{p:Wedderburn-splitting}.
Then any $\epsilon_1,\epsilon_2\in\Ext{1}{\cC}{X}{Y}\setminus\{0\}$ realise isomorphic conflations in $\cC$.
\end{lemma}

\begin{proof}
Let $\epsilon_1,\epsilon_2\in\Ext{1}{\cC}{X}{Y}\setminus\{0\}$.
Since $\rank_{\divalg{X}}\Ext{1}{\cC}{X}{Y}=1$, we have $\epsilon_1=\phi\cdot\epsilon_2$ for some $\phi\in\divalg{X}$; in particular $\phi$ is non-zero, hence invertible, in the division algebra $\divalg{X}$.
Letting $\overline{\phi}\in\op{\End{\cC}{X}}$ be the image of $\phi$ under our chosen splitting, there is a map
\begin{equation}
\label{eq:conf-map}
\begin{tikzcd}
Y\arrow[infl]{r}\arrow[equal]{d}&Z_1\arrow[defl]{r}\arrow{d}{\zeta}&X\arrow{d}{\overline{\phi}}\arrow[confl]{r}{\epsilon_1}&\phantom{}\\
Y\arrow[infl]{r}&Z_2\arrow[defl]{r}&X\arrow[confl]{r}{\epsilon_2}&\phantom{}
\end{tikzcd}
\end{equation}
of conflations.
Letting $\overline{\psi}\in\op{\End{\cC}{X}}$ be the image of $\phi^{-1}$ under our chosen splitting, we have $\overline{\psi}\circ\overline{\phi}=1_X+\alpha$ for $\alpha\in\rad{}\op{\End{\cC}{X}}$, hence this composition is invertible.
Thus $\overline{\phi}$ is an isomorphism, and so \eqref{eq:conf-map} is an isomorphism of conflations by \cite[Cor.~3.6]{NakaokaPalu}.
\end{proof}

For our cluster categories to decategorify to Fomin--Zelevinsky's cluster algebras, we will need them to have a cluster structure as in \cite[\S II.1]{BIRS1}, \cite[Def.~2.4]{FuKeller}. The conditions in Definition~\ref{d:mutable} are precisely those used in these sources to define a \emph{weak} cluster structure, as follows.

\begin{definition}
\label{d:weak-cluster-structure}
Let $\cC$ be a Krull--Schmidt cluster category. We say that $\cC$ has a \emph{weak cluster structure} if $\cT$ is maximally mutable for every $\cT\ctsubcat\cC$.
\end{definition}

By Corollary~\ref{c:max-mutable}, a Krull--Schmidt cluster category $\cC$ has a weak cluster structure if and only if its stable category $\stab{\cC}$ does.

In the next definition, we follow \cite[Def.~6.2]{PresslandPostnikov} by restricting the definitions from \cite{BIRS1,FuKeller} to a single mutation class of cluster-tilting subcategories.

\begin{definition}
\label{d:cluster-structure}
Let $\cC$ be a Krull--Schmidt cluster category and let $\cT\ctsubcat\cC$. We say that $(\cC,\cT)$ has a \emph{cluster structure} if $\cC$ has a weak cluster structure and additionally
\begin{enumerate}
\item\label{d:cluster-structure-no-loops} if $\cU$ is reachable from $\cT$, then $\cU$ has no loops or $2$-cycles, and
\item\label{d:cluster-structure-FZ-mutation} if $\cU$ is reachable from $\cT$ and $U\in\exch{\cU}$, then the exchange matrix $\exchmat{\mu_U\cU}$ agrees with the Fomin--Zelevinsky mutation $\mu_U\exchmat{\cU}$ of $\exchmat{\cU}$ at the column indexed by $U$.
\end{enumerate}
If $T$ is a cluster-tilting object, then the additive closure $\add{T}$ is a cluster-tilting subcategory, and we say that $(\cC,T)$ has a cluster structure if $(\cC,\add{T})$ does.
\end{definition}

In Definition~\ref{d:cluster-structure}\ref{d:cluster-structure-FZ-mutation}, the exchange matrices are well-defined by Corollary~\ref{c:KS-max-mut}.
Moreover, the fact that $U\in\exch{\cU}$ means that $\cU$ is locally finite at $U$ by Proposition~\ref{p:mut-v-lf}, so $\sum_{V\in\indec{\cU}}\Gabmatentry{U,V}<\infty$ and $\sum_{V\in\indec{\cU}}\Gabmatentry{V,U}<\infty$.
Together with Definition~\ref{d:cluster-structure}\ref{d:cluster-structure-no-loops}, this ensures that Definition~\ref{d:cluster-structure}\ref{d:cluster-structure-FZ-mutation} makes sense, i.e.\ that the computation of the Fomin--Zelevinsky mutation involves only finite sums.

While it is convenient for us to define cluster structures as above (in part to aid comparison with the original definition by Buan--Iyama--Reiten--Scott \cite{BIRS1}), some of the required properties are implied by only very mild additional assumptions on $\cC$.
For example, we have already seen in Proposition~\ref{p:mut-v-ff} that the property of having a weak cluster structure reduces for cluster categories to the statement that $\cT\setminus T$ is functorially finite in $\cT$ for any cluster-tilting subcategory $\cT$ and any $T\in\indec{\stab{\cT}}$.

\begin{definition}
We say that $\cC$ has \emph{finite rank} if $\stab{\cT}$ is additively finite for all $\cT\ctsubcat\cC$.
\end{definition}

It will turn out (Corollary~\ref{c:clusters-same-size}) that all cluster-tilting subcategories of $\stab{\cC}$ have the same cardinality, so in fact $\cC$ has finite rank if $\stab{\cT}$ is additively finite for some $\cT\ctsubcat\cC$.

\begin{corollary}
\label{c:weak-clust-struct}
Let $\cC$ be a Krull--Schmidt cluster category. If either
\begin{enumerate}
\item\label{c:weak-clust-struct-finite} $\cC$ has finite rank, or
\item\label{c:weak-clust-struct-compl} the cluster-tilting subcategories of $\stab{\cC}$ are locally finite and have no loops,
\end{enumerate}
then $\cC$ has a weak cluster structure.
\end{corollary}
\begin{proof}
We may reduce to the stable category $\stab{\cC}$ by Proposition~\ref{p:mutability-stab}, and then use the characterisation of mutability from Proposition~\ref{p:mut-v-ff}.
In case \ref{c:weak-clust-struct-finite}, each category $\stab{\cT}\setminus T$ is additively finite, hence functorially finite in the Hom-finite category $\stab{\cT}$.
Since $\stab{\cC}$ is Hom-finite and hence compact, in case \ref{c:weak-clust-struct-compl} we may use Lemma~\ref{l:sink-map-construction} to see that each $T\in\indec{\stab{\cT}}$ admits a sink map and a source map in $\stab{\cT}$. Because $\stab{\cT}$ has no loops, the minimal sink and source maps from this lemma are also right and left $(\stab{\cT}\setminus T)$-approximations respectively, so $\stab{\cT}\setminus T$ is functorially finite in $\cT$ as required.
\end{proof}

\begin{remark}
Between Proposition~\ref{p:mut-v-lf} and Corollary~\ref{c:weak-clust-struct}, we have shown that in a compact cluster category $\cC$ with $\cT\ctsubcat\cC$, a non-projective object $T\in\indec{\cT}$ such that $\cT$ has no loop at $T$ is mutable if and only if $\stab{\cT}$ is locally finite at $T$.
One can remove the no loop condition at the cost of more technical assumptions, allowing Lemma~\ref{l:approx-construction} to be used in place of Lemma~\ref{l:sink-map-construction} to construct the necessary approximations.
\end{remark}

There are also various natural assumptions on a pair $(\cC,\cT)$ which imply that this pair has a cluster structure; see for example \cite[\S5]{BIRS2}, \cite[\S5]{Presslandmutation} and \cite{Presslandcorrigendum}. A consequence of Theorem~\ref{t:exch-mat-mutation-clust-str} below is that \ref{d:cluster-structure-no-loops} implies \ref{d:cluster-structure-FZ-mutation} in Definition~\ref{d:cluster-structure} when $\cC$ is either compact or skew-symmetric.

While we need cluster structures to enable us to ultimately link back to cluster algebras and related constructions, many of our categorical results (such as those in Section~\ref{s:mutation}) will not require them.

\subsection{Modules over cluster-tilting subcategories}\label{ss:mods-over-cts}

Given a cluster category $\cC$ with $\cT\ctsubcat\cC$, we will often be led to consider $\cT$-modules, or representations of $\cT$.
These are, by definition, contravariant functors on $\cT$ with values in the category $\Mod{\bK}$ of (arbitrary) $\bK$-vector spaces.
Background on representations of additive categories in general may be found in Section~\ref{s:modules}.
The category of all $\cT$-modules is denoted by $\Mod{\cT}$, and we will be interested in three important subcategories:
\begin{enumerate}
\item $\lfd{\cT}$, the category of locally finite-dimensional modules, defined to be those with values in the category $\fd{\bK}$ of finite-dimensional vector spaces,
\item $\fpmod{\cT}$, the category of finitely presented $\cT$-modules (Definition~\ref{d:fp}), and
\item $\fd{\cT}$, the category of finite-dimensional $\cT$-modules (Definition~\ref{d:fd}).
When $\cT$ is Krull--Schmidt, these are the locally finite-dimensional $\cT$-modules which are zero on all but finitely many objects in $\indec{\cT}$.
\end{enumerate}

It is immediate from the definition that $\fd{\cT}\subseteq\lfd{\cT}$, with equality if $\cT$ is additively finite.
Moreover, $\cT$ is Hom-finite if and only if the projective $\cT$-modules are locally finite-dimensional, if and only if $\fpmod{\cT}\subseteq\lfd{\cT}$.

If $\cC/\cP$ is a partial stabilisation of $\cC$, then there is also a natural inclusion $\Mod(\cT/\cP)\subseteq\Mod{\cT}$ by viewing $\cT/\cP$-modules as $\cT$-modules which vanish on $\cP\subseteq\cT$.
One sees directly from the definitions that this restricts to inclusions $\lfd(\cT/\cP)\subseteq\lfd{\cT}$ and $\fd(\cT/\cP)\subseteq\fd{\cT}$.
The equivalent statement for finitely presented modules is only slightly more involved.

\begin{proposition}
\label{p:fpmod-inclusion}
Let $\cC$ be a cluster category and $\cT\ctsubcat\cC$.
If $\cP\subseteq\cC$ is a contravariantly finite subcategory of $\cT$, then the inclusion $\Mod{\cT/\cP}\subseteq\Mod{\cT}$ restricts to an inclusion $\fpmod(\cT/\cP)\subseteq\fpmod{\cT}$.
\end{proposition}

\begin{proof}
Let $T\in\cT$, and let $p\colon P\to T$ be a right $\cP$-approximation of $T$.
Then there is an exact sequence
\[\begin{tikzcd}
\Yonfun{\cT}{P}\arrow{r}{\Yonfun{\cT}f}&\Yonfun{\cT}T\arrow{r}&\Yonfun{\cT/\cP}T\arrow{r}&0,
\end{tikzcd}\]
since, by definition, the image of $\Yonfun{\cT}f$ is the subfunctor of $\Yonfun{\cT}T=\Hom{\cC}{\blank}{T}|_{\cT}$ consisting of maps factoring over $\cP$.
Thus the projective $(\cT/\cP)$-modules $\Yonfun{\cT/\cP}{T}$ lie in $\fpmod{\cT}$ when viewed as $\cT$-modules.
The result follows since $\fpmod{\cT}$ is closed under taking cokernels.
\end{proof}

Proposition~\ref{p:fpmod-inclusion} applies in particular when $\cP$ is the full subcategory of all projective-injective objects of $\cC$, which is functorially finite since $\cC$ is a Frobenius extriangulated category.
We thus always have an inclusion $\fpmod{\stab{\cT}}\subseteq\fpmod{\cT}$, for $\stab{\cT}=\cT/\cP$ the image of $\cT$ in the triangulated stable category $\stab{\cC}$.
While $\fpmod{\stab{\cT}}$ is always an abelian category \cite[Prop.~2.1(a)]{KellerReiten}, this may not be the case for $\fpmod{\cT}$.

The next statement has been proved several times by various authors (see the references in the proof) in the case that $\cC$ is a triangulated category.
For general cluster categories, we simply reduce the statement to this case and apply their results.
Given some of the results of Section~\ref{s:mutation}, the statement may also be deduced from \cite[Thm.~1.1]{ZSLWW}.

\begin{proposition}\label{p:equiv-to-mod}
Let $\cC$ be a cluster category, let $\cT\ctsubcat\cC$, and let $\stab{\cT}$ be the image of $\cT$ in the stable category $\stab{\cC}$.
Then there is an equivalence
\[\Extfun{\cT}\colon\cC/\cT\simeq\fpmod{\stab{\cT}},\ \]
defined by $\Extfun{\cT}X=\Ext{1}{\cC}{\blank}{X}|_{\cT}$.
\end{proposition}

\begin{proof}
A priori, $\Extfun{\cT}$ is a functor $\cC\to\Mod{\cT}$. However, since $\cT$ is cluster-tilting we have $\Extfun{\cT}{T}=0$ for all $T\in\cT$, and so $\Extfun{\cT}$ factors over the quotient $\cC\to\cC/\cT$.
Moreover, for any $X\in\cC$, the $\cT$-module $\Extfun{\cT}{X}$ vanishes on the projective objects of $\cC$, and so is a $\stab{\cT}$-module.
Thus we obtain an induced functor $\Extfun{\cT}\colon\cC/\cT\to\Mod{\stab{\cT}}$.
Since $\cT$ contains all projective objects of $\cC$, there is a natural equivalence $\cC/\cT=\stab{\cC}/\stab{\cT}$, and work of various authors (e.g.\ \cite[Cor.~4.4]{KoenigZhu}, see also \cite[Prop.~2.1]{KellerReiten}, \cite[Thm.~A]{BMR}, \cite{IyamaYoshino}) shows that $\Extfun{\cT}\colon\stab{\cC}/\stab{\cT}\to\Mod{\stab{\cT}}$ is an equivalence onto its image $\fpmod{\stab{\cT}}$, as claimed.
\end{proof}

\begin{corollary}
\label{c:fp-fcp}
Let $\cC$ be a cluster category and $\cT\ctsubcat\cC$. Then any $M\in\fpmod{\stab{\cT}}$ is a finitely copresented $\cT$-module, by which we mean that there is an exact sequence
\begin{equation}
\label{eq:copres}
\begin{tikzcd}
0\arrow{r}&M\arrow{r}&\dual{(\opYonfun{\cT}T_0)}\arrow{r}&\dual{(\opYonfun{\cT}T_1)}
\end{tikzcd}
\end{equation}
with $T_i\in\cT$.
\end{corollary}
\begin{proof}
By Proposition~\ref{p:equiv-to-mod}, we have $M=\Extfun{\cT}X$ for some $X\in\cC$.
Applying Proposition~\ref{p:equiv-to-mod} to $\op{\cC}$ and $\op{\cT}$, which are also a cluster category and cluster-tilting subcategory respectively, we see that $\Extfun{\op{\cT}}X\in\fpmod{\op{\stab{\cT}}}$.
By Proposition~\ref{p:fpmod-inclusion} we may thus choose a presentation
\[\begin{tikzcd}
\Yonfun{\op{\cT}}T_1\arrow{r}&\Yonfun{\op{\cT}}T_0\arrow{r}&\Extfun{\op{\cT}}X\arrow{r}&0.
\end{tikzcd}\]
The result then follows by duality, observing that $\Yonfun{\op{\cT}}T_i=\opYonfun{\cT}T_i$ by definition, and that $\dual{(\Extfun{\op{\cT}}X)}=\Extfun{\cT}X=M$ since $\stab{\cC}$ is $2$-Calabi--Yau.
\end{proof}

In practice, we will typically use Proposition~\ref{p:equiv-to-mod} to replace a $\stab{\cT}$-module by an object of $\cC$, unique up to summands in $\cT$, via the inverse of the equivalence $\Extfun{\cT}$.
To do this, we will need the relevant module to be finitely presented, for which the following results will be useful.
Recall that $\exch \cT$ is the set of mutable indecomposable objects of $\cT$.

\begin{lemma}\label{l:props-of-K0-mod-T}
Let $\cC$ be a Krull--Schmidt cluster category and $\cT\ctsubcat\cC$.
Then
\begin{enumerate}
\item The set $\{\simpmod{\cT}{T}: T\in\indec\cT\}$ of simple functors (Definition~\ref{d:simple-functor}) is a complete set of representatives of the isomorphism classes of simple $\cT$-modules,
\item if $T\in \exch \cT$, then $\Extfun{\cT}(\mut{\cT}{T})\in\fd{\cT}\cap\fpmod{\cT}$, and
\item\label{l:props-mod-T-simple-eq-E} there is no loop at $T\in \exch \cT$ if and only if $\Extfun{\cT}(\mut{\cT}{T})=\simpmod{\cT}{T}$.
\end{enumerate}
\end{lemma}

\begin{proof} {\ }
\begin{enumerate}
\item This is Proposition~\ref{p:KS-simples}.
\item We have $\Extfun{\cT}{\mut{\cT}{T}}\in\fpmod{\cT}$ by Proposition~\ref{p:equiv-to-mod}.
Since $\Ext{1}{\cC}{T'}{T}=0$ for any $T'\in\indec{\cT\setminus T}$, finite-dimensionality of $\Extfun{\cT}(\mut{\cT}{T})$ follows from that of $\Ext{1}{\cT}{T}{\mut{\cT}{T}}=\Hom{\stab{\cC}}{T}{\Sigma\mut{\cT}{T}}$, which follows from Hom-finiteness of $\stab{\cC}$.
\item Applying $\Yonfun{\cT}$ to the exchange conflation $\mut{\cT}{T}\infl \exchmon{\cT}{T}{+}\stackrel{\varphi^{+}}{\defl} T \confl$ in $\cC$, which exists since $T$ is mutable, we obtain the exact sequence
\begin{equation}\label{eq:ET-pres}
\begin{tikzcd} \projmod{\cT}{\exchmon{\cT}{T}{+}}\arrow{r}{\Yonfun{\cT}\varphi^{+}}& \projmod{\cT}{T}\arrow{r}& \Extfun{\cT}(\mut{\cT}{T})\arrow{r}& 0, \end{tikzcd}
\end{equation}
of $\cT$-modules, using that $\Extfun{\cT}\exchmon{\cT}{T}{+}=0$ because $\cT$ is cluster-tilting.
Thus $\Extfun{\cT}(\mut{\cT}{T})=\simpmod{\cT}{T}:=\projmod{\cT}{T}/\radHom{\cT}{\blank}{T}$ if and only if the image of $\Yonfun{\cT}\varphi^+$ is equal to $\radHom{\cT}{\blank}{T}$.

Recall that $\varphi^{+}$ is a right $(\cT\intersection \mut{T}{\cT})$-approximation of $T$, so if $U\in\indec{\cT\intersection\mut{T}{\cT}}$, the image of $\Yonfun{\cT}{\varphi^+}$ evaluates on $U$ to $\Hom{\cT}{U}{T}=\radHom{\cT}{U}{T}$.
Thus we need only consider the evaluation at $T$.

If $\cT$ has no loop at $T$, then any morphism in $\radHom{\cT}{T}{T}$ factors over an object of $\cT\intersection\mut{T}{\cT}=\cT\setminus T$, hence over $\varphi^{+}$, and $\Yonfun{\cT}\varphi^+$ has the desired image.
Conversely, if the image of $\Yonfun{\cT}\varphi^{+}$ evaluates to $\radHom{\cT}{T}{T}$, then every morphism in $\radHom{\cT}{T}{T}$ factors over $T_{\cT}^+\in\cT\setminus T$, and hence $\cT$ has no loop at $T$.
\qedhere
\end{enumerate}
\end{proof}

A consequence of Lemma~\ref{l:props-of-K0-mod-T}\ref{l:props-mod-T-simple-eq-E} is that if $\cT$ has no loop at $T\in\exch{\cT}$, then $\simpmod{\cT}{T}$ is finitely presented over both $\stab{\cT}$ and $\cT$.
In fact, this is true even if there is a loop at $T$, by combining the next result with \cite[Prop.~4]{KellerReiten} (see Proposition~\ref{p:KR-pdim3} below), which shows that $\fpmod{\stab{\cT}}\subseteq\fpmod{\cT}$.

\begin{proposition}
\label{p:fp-simp}
Let $\cC$ be a cluster category and let $T\in\exch{\stab{\cT}}$.
Then $\simpmod{\cT}{T}\in\fpmod{\stab{\cT}}$.
\end{proposition}
\begin{proof}
Since $\stab{\cC}$ is Hom-finite, it is in particular compact, and $\stab{\cT}$ is locally finite at $T$ by Proposition~\ref{p:mut-v-lf}.
Hence $T$ admits a sink map in $\stab{\cT}$ by Lemma~\ref{l:sink-map-construction}, which is equivalent to the statement that $\simpmod{\cT}{T}\in\fpmod{\stab{\cT}}$ by Proposition~\ref{p:sink-v-fp-simp}.
\end{proof}

\begin{corollary}
\label{c:fp=fd}
Let $\cC$ be a cluster category and $\cT\ctsubcat\cC$.
Then if $\stab{\cT}$ is maximally mutable, we have $\fd{\stab{\cT}}\subseteq\fpmod{\stab{\cT}}$.
In particular, any $M\in\fd{\stab{\cT}}$ is isomorphic to $\Extfun{\cT}X$ for some $X\in\cC$.
\end{corollary}
\begin{proof}
By Proposition~\ref{p:fp-simp}, we have $\simpmod{\cT}{T}\in\fpmod{\stab{\cT}}$ for all $T\in\exch{\cT}=\indec{\stab{\cT}}$.
Since $\stab{\cT}$ is Krull--Schmidt (being Hom-finite), the $\simpmod{\cT}{T}$ for $T\in\indec{\stab{\cT}}$ are a complete set of representatives of simple $\stab{\cT}$-modules by Proposition~\ref{p:KS-simples}, and so $\fd{\stab{\cT}}\subseteq\fpmod{\stab{\cT}}$ by the Jordan--Hölder theorem and horseshoe lemma.
The final statement is then a direct consequence of Proposition~\ref{p:equiv-to-mod}.
\end{proof}

Corollary~\ref{c:fp=fd} applies whenever $\stab{\cC}$ has a weak cluster structure, for example if it has finite rank (see Corollary~\ref{c:weak-clust-struct}\ref{c:weak-clust-struct-finite}).
In the finite rank case we have $\lfd{\stab{\cT}}=\fd{\stab{\cT}}$, and so it follows from Corollary~\ref{c:fp=fd} that in fact $\fd{\stab{\cT}}=\fpmod{\stab{\cT}}=\lfd{\stab{\cT}}$.

\sectionbreak
\section{Indices and coindices}
\label{s:mutation}

In what follows, we let $\cC$ be a cluster category, and discuss how to relate its various cluster-tilting subcategories via the \emph{index} and \emph{coindex}, which are certain isomorphisms between their Grothendieck groups.
The values of these maps on indecomposable rigid objects are categorical analogues of $\mathbf{g}$-vectors in cluster theory, since they behave in the same way under mutation (Theorem~\ref{t:c-vec-mut-formula}).
Via a natural duality, we will also define adjoint maps, which are to $\mathbf{c}$-vectors what the index and coindex are to $\mathbf{g}$-vectors; that is, they produce cluster-theoretic $\mathbf{c}$-vectors when evaluated on the appropriate objects, in this case simple modules.
Lastly, we show how these and other important concepts are linked by looking at projective resolutions of certain modules.

While our claims relating the index and coindex maps and their adjoints to $\mathbf{g}$- and $\mathbf{c}$-vectors will not be fully justified until Section~\ref{s:exch-mat}, they have already been established in many special cases, such as for triangulated categories.
We will give direct homological proofs of several properties of these maps which are necessary for this relationship to $\mathbf{g}$-vectors and $\mathbf{c}$-vectors to hold, see in particular Section~\ref{s:sign-coherence}.

We will often prove a result in the case that $\cC$ is exact and deduce it for general categories via partial stabilisation, so this process is also explained in detail.

\subsection{Definitions and first properties}\label{s:defs-and-props-of-ind-coind}

Let $\Kgp{\addcat{\cC}}$ denote the Grothendieck group of $\cC$ as an \emph{additive} category (i.e.\ ignoring the given extriangulated structure and using the split exact structure instead); we use this notation instead of $\Kgpsplit{\cC}$ \cite{JorgensenPalu} so that $\Kgp{\blank}$ consistently means `Grothendieck group with respect to the natural structure'.
The inclusion $\cT\to\cC$ induces a homomorphism 
\begin{equation}\label{eq:inj-T-C-add} \iota_{\cT}^{\addcat{\cC}}\colon\Kgp{\cT}\to\Kgp{\addcat{\cC}}. \end{equation}
Note that $\cT$ possesses the structure of an extriangulated category (inherited from that of $\cC$) but with all conflations split, so the notation $\Kgp{\cT}$ is unambiguous.

With this in mind, for an object $X$ of a category $\curly{X}$, we will use $[X]$ to refer to the class of $X$ in $\Kgp{\curly{X}}$---it will usually be clear from the context which meaning is intended, and many (but not all) of our Grothendieck groups will be of additive categories having no further extriangulated structure, in which case $[X]$ is nothing but the isomorphism class of $X$. When $X$ can be viewed as an object in several different categories, and there is a risk of confusion, we will indicate the relevant category via a subscript, writing $[X]_{\cA}\in\Kgp{\cA}$. For example, $X\in\cT\ctsubcat \cC$, has classes $[X]_{\cT}$, $[X]_{\cC}$ and $[X]_{\addcat{\cC}}$ in three different Grothendieck groups. 

Given a cluster category $\cC$ and $\cT\ctsubcat \cC$, the restricted Yoneda functor $\Yonfun{\cT}\colon\cC\to\Mod{\cT}$ defined by $X\mapsto\Hom{\cC}{\blank}{X}|_{\cT}$ restricts to the Yoneda equivalence $\cT\isoto\proj{\cT}$  and so induces an isomorphism of Grothendieck groups. In what follows, we will usually prefer to work in $\Kgp{\cT}$ rather than the isomorphic group $\Kgp{\proj{\cT}}$.

\begin{definition}\label{d:Yoneda-iso}
We denote by
\begin{equation}\label{eq:Yoneda-iso} h^{\cT}\colon\Kgp{\cT}\isoto\Kgp{\proj{\cT}},\ h^{\cT}[T]=[\projmod{\cT}{T}]
\end{equation}
the isomorphism of Grothendieck groups induced by the Yoneda equivalence.
\end{definition}

Next we recall the definition of the index and coindex of an object of a cluster category $\cC$ with respect to a cluster-tilting subcategory $\cT$.
This builds on work of a number of authors, principally Jørgensen--Palu \cite{JorgensenPalu} in the triangulated case and Fu--Keller \cite{FuKeller} in the exact case, with further references in both.
These definitions are foundational for cluster category theory, as they encode the $\cT$-approximations of objects of $\cC$ and, as we will see, they recover tropicalised cluster algebra mutation.

\begin{proposition}
\label{p:ind-well-def}
Let $\cC$ be a cluster category with cluster-tilting subcategory $\cT$, and let $X\in\cC$. Then there exists a conflation
\begin{equation}
\label{eq:index-seq}
\begin{tikzcd}
\rightker{\cT}{X}\arrow[infl]{r}&\rightapp{\cT}{X}\arrow[defl]{r}{\varphi}&{X}\arrow[confl]{r}&\phantom{}
\end{tikzcd}
\end{equation}
such that $\rightker{\cT}{X},\rightapp{\cT}{X}\in \cT$. Moreover, the value of $[{\rightapp{\cT}{X}}]-[{\rightker{\cT}{X}}]\in \Kgp{\cT}$ is independent of the choice of conflation \eqref{eq:index-seq}.
\end{proposition}

\begin{proof}
The existence of a conflation of the form \eqref{eq:index-seq}, obtained by choosing the map $\varphi$ to be a right $\cT$-approximation, follows from the argument in the proof of Proposition~\ref{p:ct-Yon-fp}.

Let $K\infl R\defl X\confl$ be another conflation of the form \eqref{eq:index-seq}. Since $\cC$ is algebraic, we may choose a Frobenius exact category $\cE$ such that $\cC\simeq\cE/\cP$ for a full and additively closed subcategory $\cP$ of projective-injective objects. By Lemma~\ref{l:ct-bijection}, there is $\widehat{\cT}\ctsubcat\cE$ with $\pi\widehat{\cT}=\cT$.
Lifting to $\cE$, we find admissible short exact sequences
\[\begin{tikzcd}[row sep=0pt]
0\arrow{r}&\rightker{\cT}{X}\arrow{r}&\rightapp{\cT}{X}\oplus P\arrow{r}&{X}\arrow{r}&0,\\
0\arrow{r}&K\arrow{r}&R\arrow{r}\oplus Q&{X}\arrow{r}&0,
\end{tikzcd}\]
with $P,Q\in\cE$ projective-injective, so we have short exact sequences
\[\begin{tikzcd}[row sep=0pt]
0\arrow{r}&\Yonfun{\widehat{\cT}}(\rightker{\cT}{X})\arrow{r}&\Yonfun{\widehat{\cT}}(\rightapp{\cT}{X}\oplus P)\arrow{r}&\Yonfun{\widehat{\cT}}{X}\arrow{r}&0,\\
0\arrow{r}&\Yonfun{\widehat{\cT}}(K)\arrow{r}&\Yonfun{\widehat{\cT}}(R\arrow{r}\oplus Q)&\Yonfun{\widehat{\cT}}{X}\arrow{r}&0
\end{tikzcd}\]
of $\widehat{\cT}$-modules, using that $\rightker{\cT}{X}$ and $K$ lie in $\widehat{\cT}$ for exactness on the right. Each is a projective presentation of $\Yonfun{\widehat{\cT}}X$, and so we have
\[\rightker{\cT}{X}\oplus P\oplus K \iso R\oplus Q\oplus\rightker{\cT}{X}\]
by Schanuel's lemma. Projecting back to $\cC$, we find that $\rightker{\cT}{X}\oplus K \iso R\oplus\rightker{\cT}{X}$, and hence $[\rightapp{\cT}{X}]-[\rightker{\cT}{X}]=[R]-[K]$ in $\Kgp{\cT}$.
\end{proof}

\begin{remark}
An alternative proof that $[{\rightapp{\cT}{X}}]-[{\rightker{\cT}{X}}]$ is independent of the choice of conflation is to use \cite[Prop.~5.1]{FGPPP} to realise conflations of the form \eqref{eq:index-seq} as projective resolutions of $X$ in the appropriate relative extriangulated structure on $\cC$, in which objects of $\cT$ are projective. Then the result follows from Schanuel's lemma for extriangulated categories \cite{Yin-Schanuel} (see also \cite{MatRos} for exact categories).
\end{remark}

\begin{definition}
\label{d:index}
Let $\cC$ be a cluster category and $\cT\ctsubcat\cC$. For each $X\in\cC$, choose a conflation as in \eqref{eq:index-seq}, and define the \emph{index} of $X$ in $\cC$ with respect to $\cT$ to be
\[ \ind{\cC}{\cT}(X)=[{\rightapp{\cT}{X}}]-[{\rightker{\cT}{X}}]\in \Kgp{\cT}. \]
We call \eqref{eq:index-seq} a \textit{$\cT$-index conflation} (or sequence, or triangle, if appropriate) for $X$.
\end{definition}

Proposition~\ref{p:ind-well-def} shows that the index is well-defined.
We also use the dual construction, beginning with a conflation
\begin{equation}
\label{eq:coindex-seq}
\begin{tikzcd}
X\arrow[infl]{r}&\leftapp{\cT}{X}\arrow[defl]{r}&\leftcok{\cT}{X}\arrow[confl]{r}&\phantom{}
\end{tikzcd}
\end{equation}
with $\leftapp{\cT}{X},\leftcok{\cT}{X}\in\cT$, in which the map $X\to \leftapp{\cT}{X}$ is a left $\cT$-approximation of $X$.

\begin{definition}\label{d:coindex} Given a conflation \eqref{eq:coindex-seq}, which we call a \emph{$\cT$-coindex conflation} for $X$, we define the \emph{coindex} of $X$ in $\cC$ with respect to $\cT$ to be
\[ \coind{\cC}{\cT}(X)=[{\leftapp{\cT}{X}}]-[{\leftcok{\cT}{X}}]\in\Kgp{\cT}. \]
\end{definition}

In particular, if $T\in\cT\ctsubcat\cC$ is mutable, the exchange conflations \eqref{eq:exchange-confs} are $\cT$-index and $\cT$-coindex conflations for $\mut{\cT}{T}$, implying that
\begin{align} \ind{\cC}{\cT}{(\mut{\cT}{T})}&=[{\rightapp{\cT}{(\mut{\cT}{T})}}]-[{T}]=[\exchmon{\cT}{T}{-}]-[T]\label{eq:ind-on-mut-T},\\
\coind{\cC}{\cT}{(\mut{\cT}{T})}&=[{\leftapp{\cT}{(\mut{\cT}{T})}}]-[{T}]=[\exchmon{\cT}{T}{+}]-[T].\label{eq:coind-on-mut-T}
\end{align}

\begin{remark}
\label{r:ind-op}
An alternative perspective on the coindex is that it is the index for the cluster category $\op{\cC}$, as follows.
A $\cT$-index conflation \eqref{eq:coindex-seq} for $X$ may be viewed as a conflation in $\op{\cC}$, with morphisms in the opposite direction, where it is a $\op{\cT}$-coindex conflation for $X$.
Thus for any $X\in\cC$, we have
\[\coind{\cC}{\cT}(X)=\ind{\op{\cC}}{\op{\cT}}(X).\]
Here we identify $\Kgp{\cT}=\Kgp{\op{\cT}}$, both groups having the same generators and relations.
\end{remark}

The next result is immediate, since the process of constructing $\cT$-approximations is additive on split exact sequences.

\begin{remark}
\label{r:no-common-summands}
We call a $\cT$-index conflation \emph{minimal} if the right approximation $\varphi\colon\rightapp{\cT}{X}\onto X$ is a minimal map (and adopt similar terminology for $\cT$-coindex conflations).
The objects $\rightapp{\cT}{X}$ and $\rightker{\cT}{X}$ in a minimal $\cT$-index conflation are determined up to (non-unique) isomorphism by $X$, but in much of the paper we will not need to assume our $\cT$-index conflations are minimal.
However, a minimal $\cT$-index conflation does have one useful extra property when $\cC$ is Krull--Schmidt and $X$ is rigid (i.e.\ $\Ext{1}{\cC}{X}{X}=0$), namely that the objects $\rightapp{\cT}{X}$ and $\rightker{\cT}{X}$ have no direct summands in common.
Indeed, this is true in $\stab{\cC}$ by \cite[Prop.~2.1]{DehyKeller}, and so any common summands must be projective-injective in $\cC$.
But $\rightker{\cT}X$ cannot have projective-injective summands since these would split, contradicting minimality of $\varphi$.
The analogous statement for minimal $\cT$-coindex conflations follows by considering $\op{\cC}$ and using Remark~\ref{r:ind-op}.
\end{remark}

\begin{proposition}\label{p:ind-coind-homs} For $\cT\ctsubcat\cC$, index and coindex with respect to $\cT$  define group homomorphisms
\[
\ind{\cC}{\cT}  \colon \Kgp{\addcat{\cC}} \to \Kgp{\cT}, \quad
\coind{\cC}{\cT}  \colon \Kgp{\addcat{\cC}} \to \Kgp{\cT}. \tag*{\qed}
\]
\end{proposition}

\begin{remark}\label{r:cind-not-additive}
It is crucial in the above that we use $\Kgp{\addcat{\cC}}$ and not $\Kgp{\cC}$ as the domain: indeed, the index and coindex are \emph{not} additive on conflations in general.
Which conflations they are additive on is an important and non-trivial question, considered by the second author and others in \cite{FGPPP}.
We return to this question ourselves in Proposition~\ref{p:beta-add-confl}.
\end{remark}

Any bounded complex of finitely generated projective $\cT$-modules has a class in $\Kgp{\proj{\cT}}$ given by the alternating sum of classes of its terms. When $\cC$ is an exact category, we can compute $\ind{\cC}{\cT}[X]$ directly in this language, since it is related by the Yoneda isomorphism to the class of a projective resolution of a particular $\cT$-module.

\begin{proposition}
\label{p:ind-proj-res}
If $\cC$ is an exact cluster category, then $\Yonfun{\cT}X$ has projective dimension~$1$, and $h^{\cT}\ind{\cC}{\cT}[X]\in\Kgp{\proj{\cT}}$ is the class of any projective resolution of $\Yonfun{\cT}(X)$.
\end{proposition}

\begin{proof}
Recall from Proposition~\ref{p:ct-Yon-fp} that $\Yonfun{\cT}X$ has projective presentation
\begin{equation}
\label{eq:ind-proj-res}
\begin{tikzcd}
\projmod{\cT}{\rightker{\cT}{X}}\arrow{r}&\projmod{\cT}{\rightapp{\cT}{X}}\arrow{r}& \Yonfun{\cT}X\arrow{r}&0,
\end{tikzcd}
\end{equation}
obtained by applying $\Yonfun{\cT}$ to a $\cT$-index conflation \eqref{eq:index-seq}.
If $\cC$ is exact, then conflations are short exact sequences, and the leftmost map in \eqref{eq:ind-proj-res} is a monomorphism since $\Yonfun{\cT}$ is left exact. This makes \eqref{eq:ind-proj-res} a projective resolution of the $\cT$-module $\Yonfun{\cT}X$, which thus has projective dimension $1$, and by construction the class of this resolution in $\Kgp{\proj{\cT}}$ is precisely $h^{\cT}\ind{\cC}{\cT}{[X]}\in\Kgp{\proj\cT}$. 
Since any other bounded projective resolution of $\Yonfun{\cT}X$ is homotopic to this one, it has the same class in $\Kgp{\proj{\cT}}$.
\end{proof}

\begin{remark}\label{r:coind-proj-res}
The analogue for the coindex is that the isomorphism $h_{\cT}\colon\Kgp{\cT}\isoto\Kgp{\proj{\op{\cT}}}$, induced from the contravariant Yoneda functor $\opYonfun{\cT}=\Yonfun{\op{\cT}}$, takes $\coind{\cC}{\cT}{[X]}$ to the class of a projective resolution of $\opYonfun{\cT}{X}\in\fpmod{\op{\cT}}$ (cf.~Remark~\ref{r:ind-op}).
Later, in Proposition~\ref{p:beta-proj-res}, we will relate a projective resolution of the $\cT$-module $\Extfun{\cT}X$ to both the index and coindex of $X$.
\end{remark}

The Grothendieck group $\Kgp{\cC}$ of the extriangulated category $\cC$ is naturally a quotient of $\Kgp{\addcat{\cC}}$, since there is a split conflation $X\infl Y\defl Z\confl$ in $\cC$ whenever $Y\cong X\oplus Z$. Hence there is a quotient map
\begin{equation}\label{eq:quot-C-add-C}
\pi_{\addcat{\cC}}^{\cC}\colon \Kgp{\addcat{\cC}}\to \Kgp{\cC},\ \pi_{\addcat{\cC}}^{\cC}([X]_{\addcat{\cC}})=[X]_{\cC},
\end{equation}
with kernel generated by all relations coming from (not necessarily split) conflations.
For $\cC$ a cluster category and $\cT \ctsubcat \cC$, we therefore obtain a homomorphism
\begin{equation}\label{eq:proj-T-C}
\pi_{\cT}^{\cC}\colon \Kgp{{\cT}} \to \Kgp{\cC},\ \pi_{\cT}^{\cC}([{T}]_{\cT})=[T]_{\cC}    
\end{equation}
by pre-composing $\pi_{\addcat{\cC}}^{\cC}$ with $\iota_{\cT}^{\addcat{\cC}}\colon\Kgp{\cT}\to\Kgp{\addcat{\cC}}$.

\begin{lemma}\label{l:pi-T-ind-is-pi-C} We have $\pi_{\cT}^{\cC} \circ \ind{\cC}{\cT}=\pi_{\addcat{\cC}}^{\cC}=\pi_{\cT}^{\cC} \circ \coind{\cC}{\cT}$.
\end{lemma}

\begin{proof} We give the details to demonstrate how evaluating expressions in the different Grothendieck groups yields equalities of this type. Since $\rightker{\cT}{X}\infl\rightapp{\cT}{X}\defl X\confl$ is a conflation in $\cC$, we have $[X]_{\cC}=[\rightapp{\cT}{X}]_{\cC}-[\rightker{\cT}{X}]_{\cC}$ in $\Kgp{\cC}$. By definition, $\ind{\cC}{\cT}{[X]_{\addcat{\cC}}}=[{\rightapp{\cT}{X}}]_{\cT}-[{\rightker{\cT}{X}}]_{\cT}$. Hence,
\[\pi_{\cT}^{\cC}(\ind{\cC}{\cT}{[X]_{\addcat{\cC}}})  = \pi_{\cT}^{\cC}([{\rightapp{\cT}{X}}]_{\cT}-[{\rightker{\cT}{X}}]_{\cT})  = [\rightapp{\cT}{X}]_{\cC}-[\rightker{\cT}{X}]_{\cC} = [X]_{\cC} = \pi_{\addcat{\cC}}^{\cC}([X]_{\addcat{\cC}}).\]
From this and the similar argument for coindex, we conclude the result.
\end{proof}

\begin{corollary}\label{c:pi-T-surj} The homomorphism $\pi_{\cT}^{\cC}$ is surjective. \qed
\end{corollary}

\begin{corollary}\label{c:coind-minus-ind-in-ker-pi-T} For all $X$, $\coind{\cC}{\cT}[X]-\ind{\cC}{\cT}[X]\in \ker \pi_{\cT}^{\cC}$. \qed
\end{corollary}

The importance of the index (and coindex) for the theory of cluster categories is emphasised by the following key result, essentially due to Dehy and Keller.
Recall that an object $X$ in an extriangulated category $\cC$ is \emph{rigid} if $\Ext{1}{\cC}{X}{X}=0$.

\begin{proposition}
\label{p:rigid-index}
Let $\cC$ be a Krull--Schmidt cluster category, let $\cT\ctsubcat\cC$ and let $X,X'\in\cC$ be rigid objects.
Then $\ind{\cC}{\cT}[X]=\ind{\cC}{\cT}[X']$ if and only if $X\iso X'$.
\end{proposition}
\begin{proof}
The non-trivial implication is the forward one.
When $\cC$ is exact, this is a result of Fu--Keller \cite[\S4]{FuKeller} (proved by reducing to the case that $\cC$ is triangulated, for which the result is due to Dehy--Keller \cite[\S2.3]{DehyKeller}).
While Fu--Keller have Hom-finiteness as a standing assumption, the proof of \cite[Lem.~4.2]{FuKeller} does not rely on this.

By further examination of the proof of the aforementioned lemma, we also see that it is compatible with partial stabilisation from the exact case so, by algebraicity, we obtain the statement for Krull--Schmidt (extriangulated) cluster categories in the full generality we consider here.
\end{proof}

The following lemma will turn out to be rather powerful: it expresses an equality between two alternating sums of dimensions of Ext-spaces whose terms involve the objects in a $\cT$-index \eqref{eq:index-seq} and $\cT$-coindex conflation \eqref{eq:coindex-seq} for $X$.  More specifically, recall that
\[
\ind{\cC}{\cT}[X]  =[{\rightapp{\cT}{X}}]-[{\rightker{\cT}{X}}], \quad
\coind{\cC}{\cT}[X]  =[{\leftapp{\cT}{X}}]-[{\leftcok{\cT}{X}}],\]
so that
$ \coind{\cC}{\cT}[X]-\ind{\cC}{\cT}[X]=[{\leftapp{\cT}{X}}]-[{\leftcok{\cT}{X}}]+[{\rightker{\cT}{X}}]-[{\rightapp{\cT}{X}}]$.

\begin{lemma}
\label{l:ind-coind-adjointness}
Let $\cC$ be a cluster category and $\cT \ctsubcat \cC$. Then for any $X,Y\in\cC$ we have
{\small\begin{multline*}
\ext{1}{\cC}{\leftapp{\cT}{X}}{Y}-\ext{1}{\cC}{\leftcok{\cT}{X}}{Y}+\ext{1}{\cC}{\rightker{\cT}{X}}{Y}-\ext{1}{\cC}{\rightapp{\cT}{X}}{Y}\\
=\ext{1}{\cC}{X}{\leftcok{\cT}{Y}}-\ext{1}{\cC}{X}{\leftapp{\cT}{Y}}+\ext{1}{\cC}{X}{\rightapp{\cT}{Y}}-\ext{1}{\cC}{X}{\rightker{\cT}{Y}}.
\end{multline*}}If moreover $\Ext{1}{\cC}{X}{Y}=0$, then this is a consequence of the stronger equalities
{\small\begin{align*}
\ext{1}{\cC}{X}{\leftapp{\cT}{Y}}-\ext{1}{\cC}{X}{\leftcok{\cT}{Y}}&=\ext{1}{\cC}{\rightapp{\cT}{X}}{Y}-\ext{1}{\cC}{\rightker{\cT}{X}}{Y},\\
\ext{1}{\cC}{\leftapp{\cT}{X}}{Y}-\ext{1}{\cC}{\leftcok{\cT}{X}}{Y}&=\ext{1}{\cC}{X}{\rightapp{\cT}{Y}}-\ext{1}{\cC}{X}{\rightker{\cT}{Y}}.
\end{align*}}
\end{lemma}

\begin{proof}
Throughout the proof, we drop $\cT$ in subscripts, writing $KX=\rightker{\cT}{X}$ and so on.
Since $\cC$ is algebraic, we may pick a Frobenius exact category $\cE$ such that $\cC=\cE/\cP$. For any $T\in\cT$, lifting the conflations \eqref{eq:index-seq} and \eqref{eq:coindex-seq} to $\cE$ and applying Hom-functors yields exact sequences
\[
\begin{tikzcd}[column sep=10pt,row sep=0pt,ampersand replacement=\&,font=\footnotesize]
0\arrow{r}\&\Hom{\cE}{CX}{T}\arrow{r}\&\Hom{\cE}{LX}{T}\arrow{r}\&\Hom{\cE}{RX}{T}\arrow{r}\&\Hom{\cE}{KX}{T}\arrow{r}\&\Ext{1}{\cC}{X}{T}\arrow{r}\&0,\\
0\arrow{r}\&\Hom{\cE}{T}{KY}\arrow{r}\&\Hom{\cE}{T}{RY}\arrow{r}\&\Hom{\cE}{T}{LY}\arrow{r}\&\Hom{\cE}{T}{CY}\arrow{r}\&\Ext{1}{\cC}{T}{Y}\arrow{r}\&0,
\end{tikzcd}
\]
noting that $\operatorname{Ext}^1_{\cE}=\operatorname{Ext}^1_{\cC}$; it is important here that we take Hom-functors in the exact category $\cE$ for exactness at the left in the above sequences.
Combining these sequences, for various values of $T\in\cT$, we construct the double complex
\[
\begin{tikzcd}[column sep=10pt,row sep=10pt,ampersand replacement=\&,font=\small]
\&0\arrow{d}\&0\arrow{d}\&0\arrow{d}\&0\arrow{d}\&\\
0\arrow{r}\&\Hom{\cE}{CX}{KY}\arrow{r}\arrow{d}\&\Hom{\cE}{CX}{RY}\arrow{r}\arrow{d}\&\Hom{\cE}{CX}{LY}\arrow{r}\arrow{d}\&\Hom{\cE}{CX}{CY}\arrow{r}\arrow{d}\&0\\
0\arrow{r}\&\Hom{\cE}{LX}{KY}\arrow{r}\arrow{d}\&\Hom{\cE}{LX}{RY}\arrow{r}\arrow{d}\&\Hom{\cE}{LX}{LY}\arrow{r}\arrow{d}\&\Hom{\cE}{LX}{CY}\arrow{r}\arrow{d}\&0\\
0\arrow{r}\&\Hom{\cE}{RX}{KY}\arrow{r}\arrow{d}\&\Hom{\cE}{RX}{RY}\arrow{r}\arrow{d}\&\Hom{\cE}{RX}{LY}\arrow{r}\arrow{d}\&\Hom{\cE}{RX}{CY}\arrow{r}\arrow{d}\&0\\
0\arrow{r}\&\Hom{\cE}{KX}{KY}\arrow{r}\arrow{d}\&\Hom{\cE}{KX}{RY}\arrow{r}\arrow{d}\&\Hom{\cE}{KX}{LY}\arrow{r}\arrow{d}\&\Hom{\cE}{KX}{CY}\arrow{r}\arrow{d}\&0\\
\&0\&0\&0\&0\&\end{tikzcd}
\]
The rows and columns of this complex are not exact, but rather there is cohomology (given by an extension group) on the right-most non-zero column (using the horizontal derivatives) and on the lowest non-zero row (using the vertical derivatives).
Indeed, differentiating horizontally, we obtain the complex
\[\begin{tikzcd}[column sep=17pt]
0\arrow{r}&\Ext{1}{\cC}{CX}{Y}\arrow{r}&\Ext{1}{\cC}{LX}{Y}\arrow{r}&\Ext{1}{\cC}{RX}{Y}\arrow{r}&\Ext{1}{\cC}{KX}{Y}\arrow{r}&0,
\end{tikzcd}\]
and differentiating vertically we obtain the complex
\[\begin{tikzcd}[column sep=17pt]
0\arrow{r}&\Ext{1}{\cC}{X}{KY}\arrow{r}&\Ext{1}{\cC}{X}{RY}\arrow{r}&\Ext{1}{\cC}{X}{LY}\arrow{r}&\Ext{1}{\cC}{X}{CY}\arrow{r}&0.
\end{tikzcd}\]
These complexes are again not exact, but both have cohomology equal to the total cohomology of the double complex we started with (both of the spectral sequences of this double complex converging on the second page).
Thus, the alternating sum of the dimensions of the terms of these two complexes coincides, being the alternating sum of dimensions of their common cohomology groups, and we obtain the first desired equality.

Now the middle map in each of these sequences factors over $\Ext{1}{\cC}{X}{Y}$, since the middle map in each column of the double complex factors over $\Hom{\cE}{X}{T}$, and the middle map in each row over $\Hom{\cT}{T}{Y}$, for the relevant $T\in\cT$.
Thus, if this extension space is zero, the two relevant complexes of extension spaces split as a direct sum, with non-zero terms of the summands in different degrees.
In this case we can compare only the left-hand half of each complex, and conclude that
\[\ext{1}{\cC}{RX}{Y}-\ext{1}{\cC}{KX}{Y}=\ext{1}{\cC}{X}{LY}-\ext{1}{\cC}{X}{CY},\]
as claimed.
The second such truncated equality may be obtained similarly by comparing the right-hand halves of the relevant complexes, or deduced from the first by using the fact that the stable category $\stab{\cC}$ is $2$-Calabi--Yau to swap the role of $X$ and $Y$.
\end{proof}

\begin{remark}
\label{r:magic-lemma}
Exploiting the fact that $\cT$ is split exact, a useful shorthand for the result of Lemma~\ref{l:ind-coind-adjointness} is
\[\extbig{1}{\cC}{(\coind{\cC}{\cT}-\ind{\cC}{\cT})[X]}{[Y]}=\extbig{1}{\cC}{[X]}{(\ind{\cC}{\cT}-\coind{\cC}{\cT})[Y]},\]
or, when $\Ext{1}{\cC}{X}{Y}=0$,
\begin{align*}
\ext{1}{\cC}{\coind{\cC}{\cT}[X]}{[Y]}&=\ext{1}{\cC}{[X]}{\ind{\cC}{\cT}[Y]},\\
\ext{1}{\cC}{\ind{\cC}{\cT}[X]}{[Y]}&=\ext{1}{\cC}{[X]}{\coind{\cC}{\cT}[Y]}.
\end{align*}
\end{remark}

\subsection{Stabilisation}\label{s:part-stab}

In this subsection, we consider the relationship between a cluster category $\cC$ and a partial stabilisation $\cC/\cP$.
While $\cC$ and $\cC/\cP$ have the same objects and extension groups, and hence the same cluster-tilting subcategories, we will write $\cT/\cP$ for the cluster-tilting subcategory in $\cC/\cP$ with the same objects as $\cT\ctsubcat\cC$, to be clear about which category we are working in.
This will also help to distinguish the Grothendieck groups $\Kgp{\cT}$ and $\Kgp{\cT/\cP}$, which are non-isomorphic when $\cP\ne0$.
There is, however, a surjective homomorphism
\begin{equation}
\label{eq:T-stab-hom}
\pi_{\cT}^{\cT/\cP}\colon\Kgp{\cT}\to\Kgp{\cT/\cP},\ \pi_{\cT}^{\cT/\cP}[T]_{\cT}=[T]_{\cT/\cP}
\end{equation}
induced by the restriction $\pi_{\cT}^{\cT/\cP}\colon\cT\to\cT/\cP$ of the quotient functor $\cC\to\cC/\cP$.
These surjections allow us to relate the index and coindex maps for $\cT\ctsubcat\cC$ to the index and coindex for $\cT/\cP\ctsubcat\cC/\cP$, as follows.

\begin{proposition}
\label{p:stabind}
Let $\cC$ be a cluster category, let $\cT\ctsubcat\cC$, and let $\cP$ be a full and additively closed subcategory of projective objects in $\cC$.
Write 
\begin{equation}\label{eq:C-add-stab-hom}  \pi_{\addcat{\cC}}^{\addcat{(\cC/\cP)}}\colon\Kgp{\addcat{\cC}}\to\Kgpbig{\addcat{(\cC/\cP)}},\  [X]_{\addcat{\cC}}=[X]_{\addcat{(\cC/\cP)}}
\end{equation}
for the natural projection.
Then
\[
\ind{\cC/\cP}{\cT/\cP}\circ\pi_{\addcat{\cC}}^{\addcat{(\cC/\cP)}}=\pi_{\cT}^{\cT/\cP}\circ\ind{\cC}{\cT},\quad
\coind{\cC/\cP}{\cT/\cP}\circ\pi_{\addcat{\cC}}^{\addcat{(\cC/\cP)}}=\pi_{\cT}^{\cT/\cP}\circ\coind{\cC}{\cT}.
\]
\end{proposition}
\begin{proof}
Let $X\in\cC$, and choose a $\cT$-index conflation \eqref{eq:index-seq} for $X$, so that $\ind{\cC}{\cT}{([X]_{\addcat{\cC}})}=[{\rightapp{\cT}{X}}]_{\cT}-[{\rightker{\cT}{X}}]_{\cT}$.
Projecting this conflation to $\cC/\cP$ gives a $\cT/\cP$-index conflation for $X$, and so $\ind{\cC/\cP}{\cT/\cP}{([X]_{\addcat{(\cC/\cP)}})}=[{\rightapp{\cT}{X}}]_{\cT/\cP}-[\rightker{\cT}{X}]_{\cT/\cP}$.
Since $\pi_{\cT}^{\cT/\cP}[T]_{\cT}=[T]_{\cT/\cP}$ for all $T\in\cT$, the result follows.
The proof for coindices is completely analogous.
\end{proof}

\begin{remark}
\label{rem:ind-abbrev1}
When $\cC/\cP$ is the stable category $\stab{\cC}$, we will abbreviate $\stabind{\cC}{\cT}\defeq\ind{\cC/\cP}{\cT/\cP}$,
and similarly for the coindex.
\end{remark}

We will also use later the fact that
\begin{equation}
\label{eq:stab-restrict} 
\iota_{\cT/\cP}^{\addcat{(\cC/\cP)}}\circ \pi_{\cT}^{\cT/\cP}=\pi_{\addcat{\cC}}^{\addcat{(\cC/\cP)}}\circ \iota_{\cT}^{\addcat{\cC}}
\end{equation}
where $\iota_{\cT}^{\addcat{\cC}}$ is the homomorphism induced by $\cT\inj \addcat{\cC}$. Together with Proposition~\ref{p:stabind}, this means that we have commutative diagrams
\begin{equation}
\label{eq:C-add-T-stable-comm-diags}
\begin{tikzcd}[column sep=50pt]
\Kgp{\addcat{\cC}} \arrow{r}{\cind{\cC}{\cT}} \arrow[->>]{d}[left]{\pi_{\addcat{\cC}}^{\addcat{(\cC/\cP)}}} & \Kgp{\cT} \arrow[->>]{d}{\pi_{\cT}^{\cT/\cP}}\\
\Kgpbig{\addcat{(\cC/\cP)}} \arrow{r}[below]{\cind{\cC/\cP}{\cT/\cP}} & \Kgp{\cT/\cP}
\end{tikzcd}\qquad
\begin{tikzcd}[column sep=40pt]
\Kgp{\addcat{\cC}}  \arrow[->>]{d}[left]{\pi_{\addcat{\cC}}^{\addcat{(\cC/\cP)}}} & \Kgp{\cT} \arrow[->]{l}[above]{\iota_{\cT}^{\addcat{\cC}}} \arrow[->>]{d}{\pi_{\cT}^{\cT/\cP}} \\
\Kgpbig{\addcat{(\cC/\cP)}}  & \Kgp{\cT/\cP} \arrow[->]{l}{\iota_{\cT/\cP}^{\addcat{(\cC/\cP)}}}
\end{tikzcd}
\end{equation}

\subsection{Duality}\label{s:duality}

In this section and subsequently, we will see that there is a duality between a cluster-tilting subcategory and modules over it. This relationship is at the heart of much of what follows.
For us, $\cT$ is associated with the $\cA$-side, and the category of modules over $\cT$ with the $\cX$-side, in the Fock--Goncharov philosophy.

For any additive category $\cT$, viewed as a split exact category, there is a bilinear form $\ip{\blank}{\blank}\colon\Kgp{\lfd{\cT}}\times\Kgp{\cT}\to\integ$ given by
\begin{equation}
\label{eq:fundamental-pairing}
\ip{[M]}{[T]}=\dim_{\bK} M(T)
\end{equation}
for objects $M\in \lfd \cT$ and $T\in \cT$ (and extended linearly to differences of classes of objects).
Linearity in $[M]$ uses that evaluation on $T\in\cT$ is an exact functor $\lfd{\cT}\to\fpmod{\bK}$, while linearity in $[T]$ uses that every additive functor on $\cT$ is exact, since $\cT$ has no non-split conflations.

\begin{definition}
\label{d:canform}
The \emph{numerical Grothendieck group} $\Kgpnum{\lfd{\cT}}$ is the quotient of $\Kgp{\lfd{\cT}}$ by the kernel of the form \eqref{eq:fundamental-pairing}, that is,
\[\Kgpnum{\lfd{\cT}}=\Kgp{\lfd{\cT}}/\{v\in\Kgp{\lfd{\cT}}:\ip{v}{[T]}=0\text{ for all }T\in\cT\}.\]
We write $\canform{\blank}{\blank}{\cT}\colon\Kgpnum{\lfd{\cT}}\times\Kgp{\cT}\to\integ$ for the form induced from \eqref{eq:fundamental-pairing}, so that in particular we still have $\canform{[M]}{[T]}{\cT}=\dim_{\bK} M(T)$ when evaluating on classes of objects.
\end{definition}

Recall that $\divalg{T}=\op{\End{\cT}{T}}/\rad{}{\op{\End{\cT}{T}}}$, and $\dimdivalg{T}=\dim_{\bK}\divalg{T}$ when this dimension is finite.
If $\cT$ is Krull--Schmidt and $T\in\indec{\cT}$, then by definition we have $\simpmod{\cT}{T}(T)=\divalg{T}$ (as vector spaces), so in particular $\dimdivalg{T}=\dim_{\bK}\simpmod{\cT}{T}(T)$.
\begin{proposition}\label{p:K-duality}
Let $\cT$ be a Krull--Schmidt $\bK$-linear category such that $\simpmod{\cT}{T}\in\fd{\cT}$ for every $T\in \indec \cT$. Then the pairing $\canform{\blank}{\blank}{\cT}\colon\Kgpnum{\lfd {\cT}}\times\Kgp{{\cT}}\to\integ$ is non-degenerate. 
That is, writing $\dual{(\blank)}=\Hom{\integ}{\blank}{\integ}$, the induced maps
\begin{align*}
\pdual{\cT} & \colon \Kgp{{\cT}}\to\dual{\Kgpnum{\lfd{\cT}}},\ \pdual{\cT}[T]=\canform{\blank}{[T]}{\cT} \\
\sdual{\cT} & \colon \Kgpnum{\lfd{\cT}}\to\dual{\Kgp{{\cT}}},\ \sdual{\cT}[M]=\canform{[M]}{\blank}{\cT}
\end{align*}
are injective. Moreover, computing adjoints with respect to the evaluation forms, we have that $\adj{(\pdual{\cT})}=\sdual{\cT}$ and $\adj{(\sdual{\cT})}=\pdual{\cT}$. 
\end{proposition}

\begin{proof}
The fact that $\cT$ is Krull--Schmidt means that $\indec{\cT}$ indexes a $\integ$-basis $\{[{T}]\}$ of $\Kgp{\cT}$.
We have $\simpmod{\cT}{U}(T)=0$ when $T\not\cong U$, whereas $\dim_{\bK}\simpmod{\cT}{T}(T)=\dimdivalg{T}\ne0$.
Thus any $v\in\Kgp{\cT}$ satisfies
\[v=\sum_{T\in\indec{\cT}}\frac{1}{\dimdivalg{T}}\canform{[\simpmod{\cT}{T}]}{v}{\cT}[T],\]
and injectivity of $\pdual{\cT}$ follows from this.
On the other hand, $\Kgpnum{\lfd{\cT}}$ is defined precisely in order to make $\sdual{\cT}$ injective.
Since $\pdual{\cT}$ and $\sdual{\cT}$ are obtained from the same non-degenerate bilinear form, they are adjoint by \eqref{eq:evform-genform} and Corollary~\ref{c:identify-adj}.
\end{proof}

\begin{remark}
One reason for using $\Kgp{\cT}$, rather than the isomorphic group $\Kgp{\proj{\cT}}$, is the simple description of the form \eqref{eq:fundamental-pairing}, which tells us that the functional $\canform{[M]}{\blank}{\cT}$ is the dimension vector of $M$.
Under the isomorphism $h^{\cT}\colon\Kgp{\cT}\to\Kgp{\proj{\cT}}$, the corresponding form is
\[\canform{[M]}{[P]}{\cT}=\dim\Hom{\cT}{P}{M},\]
since $\Hom{\cT}{\projmod{\cT}{T}}{M}=M(T)$ by Yoneda's lemma.
\end{remark}

In the context of Proposition~\ref{p:K-duality}, assume further that $\cT$ is additively finite, so $\lfd{\cT}=\fd{\cT}$.
Then, by the Jordan--Hölder theorem, the classes of simple modules $\simpmod{\cT}{T}$ are a basis for $\Kgp{\fd{\cT}}$, and so in this case the form $\ip{\blank}{\blank}\colon\Kgp{\fd{\cT}}\times\Kgp{\cT}\to\integ$ is already non-degenerate.
In particular, $\Kgp{\fd{\cT}}=\Kgpnum{\fd{\cT}}$ and $\ip{\blank}{\blank}=\canform{\blank}{\blank}{\cT}$, so the numerical Grothendieck group construction is not needed.

\begin{lemma}
\label{l:compact-module-division}
In the context of Proposition~\ref{p:K-duality}, assume that $\cT$ is pseudocompact (Definition~\ref{d:pseudocompact}), and let $M\in{\lfd{\cT}}$.
Then $\dimdivalg{T}\divides\dim_{\bK}M(T)$ for all $T\in\indec{\cT}$.
\end{lemma}
\begin{proof}
Since $M$ is a $\cT$-module, $M(T)$ is an $\op{\End{\cT}{T}}$-module.
By Proposition~\ref{p:Wedderburn-splitting}, we may therefore choose a (non-canonical) $\divalg{T}$-module structure on $M(T)$, which must be free since $\divalg{T}$ is a division algebra, and so $\dimdivalg{T}\divides\dim_{\bK}M(T)$.
\end{proof}

\begin{proposition}
\label{p:im-delta}
In the setting of Proposition~\ref{p:K-duality}, assume either that $\cT$ is pseudocompact or that $\dimdivalg{T}=1$ for all $T\in\indec{\cT}$. Then
\begin{align*}
\image(\sdual{\cT})&=\{\psi\in\dual{\Kgp{\cT}}:\text{$d_T$ divides $\psi[T]$ for all $T\in\indec{\cT}$}\}.
\end{align*}
In particular, $\image(\sdual{\cT})\tensor_{\integ} \bK\iso \dual{\Kgp{\cT}}\tensor_{\integ} \bK$.
\end{proposition}

\begin{proof}
Let $M\in\lfd{\cT}$ and $T\in\indec{\cT}$.
Then
$\sdual{\cT}[M](T)=\canform{[M]}{[T]}{\cT}=\dim_{\bK}M(T)$
is divisible by $\dimdivalg{T}$, either by Lemma~\ref{l:compact-module-division} or because $\dimdivalg{T}=1$.
This shows that $\image(\sdual{\cT})$ is contained in the claimed subspace of $\dual{\Kgp{\cT}}$.

Conversely, let $\psi\in\dual{\Kgp{\cT}}$ be such that $\dimdivalg{T}\divides\psi[T]$ for all $T\in\indec{\cT}$, and let $n_T=\psi[T]/\dimdivalg{T}\in\integ$.
Let $M=\prod_{T\in\indec{\cT}}(\simpmod{\cT}{T})^{n_T}$.
Then for any $T\in\indec{\cT}$, we have $M(T)=(\simpmod{\cT}{T}(T))^{n_T}$, since $\simpmod{\cT}{T}(U)=0$ whenever $U\ne T$, and so
$\dim_{\bK}M(T)=n_T\dimdivalg{T}=\psi[T]$.
In particular, $M\in\lfd{\cT}$, and $\psi=\sdual{\cT}[M]$ is in the image of $\sdual{\cT}$, as required.
\end{proof}

\begin{corollary}
\label{c:Kgpnum-prod}
In the setting of Proposition~\ref{p:im-delta}, $\Kgpnum{\lfd{\cT}}\iso\prod_{T\in\indec{\cT}}\integ[\simpmod{\cT}{T}]$.\qed
\end{corollary}

The preceding results will be used below primarily in the case that $\cT\ctsubcat\cC$ and $\cC$ is a compact or skew-symmetric cluster category, in which case the assumptions of Proposition~\ref{p:im-delta} are satisfied by definition.
They may also be applied to $\stab{\cT}\ctsubcat\stab{\cC}$ for any cluster category $\cC$, since $\stab{\cT}$ is Hom-finite (because by assumption $\stab{\cC}$ is) and therefore pseudocompact.

\begin{corollary}\label{c:ss-implies-isos}
If $\cC$ is a skew-symmetric cluster category, then $\sdual{\cT}$ is an isomorphism for any $\cT\ctsubcat\cC$. \qed
\end{corollary}

\begin{remark}
As explained in Section~\ref{ss:intro-cl-ensembles}, in the construction of cluster algebras from a seed datum, see e.g.\ \cite{GHKK}, the starting point includes a pair of dual lattices $\seedlat{N}$ and $\seedlat{M}$, together with a finite index sublattice $\seedlat{N}^\circ\leq \seedlat{N}$, so that $\seedlat{M}$ is naturally a finite index sublattice of $\seedlat{M}^\circ\defeq\dual{(\seedlat{N}^\circ)}$.

For us, the role of $\seedlat{N}^\circ$ is played by $\Kgpnum{\lfd{\cT}}$, and that of $\seedlat{M}$ by $\Kgp{\cT}$. Thus, $\seedlat{N}$ corresponds to $\dual{\Kgp{\cT}}$, which in the finite rank case contains $\Kgpnum{\lfd{\cT}}=\Kgp{\fd{\cT}}$ as a finite index sublattice by Proposition~\ref{p:im-delta}. Similarly, $\Kgp{\cT}$ is a finite index sublattice of $\dual{\Kgp{\fd{\cT}}}$, corresponding to $\seedlat{M}^\circ$.
\end{remark}

Let $\cC$ be a Krull--Schmidt cluster category with $\cT\ctsubcat\cC$, and let $\cP$ be a full and additively closed subcategory of projective-injective objects.
Let $\pi_{\cT}^{\cT/\cP}  \colon \Kgp{\cT}\to \Kgp{\cT/\cP}$ be induced by the quotient functor, as in Section~\ref{s:part-stab}.
Treating $M\in\lfd{\cT/\cP}$ as a $\cT$-module vanishing on $\cP$, as in Section~\ref{ss:mods-over-cts}, yields an inclusion $\Kgp{\lfd(\cT/\cP)}\to \Kgp{\lfd \cT}$.
Since the vector space $M(T)$, for $T\in\indec{\cT}$, does not depend on whether we view $M$ as a $(\cT/\cP)$-module or a $\cT$-module, this induces a further inclusion
\[\iota_{\cT/\cP}^{\cT}  \colon \Kgpnum{\lfd(\cT/\cP)}\to \Kgpnum{\lfd \cT}.\]
Since $\cC/\cP$ is also Krull--Schmidt cluster category (Proposition~\ref{p:partialstab-KS}), there is a form
\[ \canform{\blank}{\blank}{\cT/\cP}\colon\Kgpnum{\lfd{(\cT/\cP)}}\times\Kgp{\cT/\cP}\to\integ \]
as in Definition~\ref{d:canform}.
We may check that for $T\in \cT$ and $M\in \lfd(\cT/\cP)$, we have
\begin{equation}\label{eq:pi-p-iota-s-adjoint} \canform{\iota_{\cT/\cP}^{\cT}[M]}{[T]}{\cT}=\dim M(T)=\canform{[M]}{\pi_{\cT}^{\cT/\cP}[T]}{\cT/\cP}, \end{equation}
and therefore $\adj{(\pi_{\cT}^{\cT/\cP})}=\iota_{\cT/\cP}^{\cT}$ and $\adj{(\iota_{\cT/\cP}^{\cT})}=\pi_{\cT}^{\cT/\cP}$. 

\begin{remark}
\label{rem:pi-iota-shorthand}
We will be most interested in the case that $\cP$ consists of all projective objects, so that $\cC/\cP=\stab{\cC}$ is the triangulated stable category, and when we are in that context we will abbreviate
\[\pproj{\cT}\defeq\pi_{\cT}^{\cT/\cP},\quad\sinc{\cT}\defeq\iota_{\cT/\cP}^{\cT}.\]
The superscripts `$\mathrm{p}$' and `$\mathrm{s}$' appearing here stand for `projective' and `simple' respectively, and we will elaborate on the reasons for this choice below.
\end{remark}

\subsection{Relating cluster-tilting subcategories}

To transport data between different cluster-tilting subcategories, we restrict the index and coindex maps to these subcategories.
For $\cU\ctsubcat\cC$, recall that there is a map $\iota_{\cU}^{\addcat{\cC}}\colon\Kgp{\cU}\to\Kgp{\cC^{\add}}$ arising from the inclusion of categories, which we can postcompose with $\ind{\cC}{\cT}$ for a cluster-tilting subcategory $\cT$ to obtain a map 
\begin{equation}\label{eq:ind-T-T'} \ind{\cU}{\cT}\defeq \ind{\cC}{\cT}\circ \iota_{\cU}^{\addcat{\cC}}\colon\Kgp{\cU}\to\Kgp{\cT}.
\end{equation}
Completely analogously, we obtain another map
\begin{equation}\label{eq:coind-T-T'} \coind{\cU}{\cT}\defeq \coind{\cC}{\cT}\circ \iota_{\cU}^{\addcat{\cC}}\colon\Kgp{\cU}\to\Kgp{\cT},
\end{equation}
with the same domain and codomain.
In particular, $\ind{\cT}{\cT}=\coind{\cT}{\cT}=\id_{\Kgp{\cT}}$.  

Given $\cTU\ctsubcat\cC$, there are four potentially different endomorphisms of $\Kgp{\cT}$ taking the form of compositions $\Kgp{\cT}\to\Kgp{\cU}\to\Kgp{\cT}$ with each map being either an index or a coindex.
 For later purposes, we would like to compute these maps explicitly.
 Two of them turn out to be the identity: this was shown by Dehy--Keller \cite{DehyKeller} in the triangulated case, and their argument adapts to our setting in a straightforward way, as follows.

\begin{proposition}
\label{p:ind-coind-inverse}
Let $\cC$ be a cluster category and $\cTU\ctsubcat\cC$. Then $\ind{\cT}{\cU}\colon \Kgp{{\cT}}\isoto\Kgp{{\cU}}$ is an isomorphism, with inverse $\coind{\cU}{\cT}$.
\end{proposition}
\begin{proof}
Let $T\in\cT$, and let
\begin{equation}
\label{eq:conflation-f}
\begin{tikzcd}
\rightker{\cU}{T}\arrow[infl]{r}{f}&\rightapp{\cU}{T}\arrow[defl]{r}&T\arrow[confl]{r} & \phantom{}
\end{tikzcd}
\end{equation}
be a $\cU$-index conflation of $T$, so $\ind{\cT}{\cU}[{T}]=[{\rightapp{\cU}{T}}]-[{\rightker{\cU}{T}}]$. Pick a $\cT$-coindex conflation
\begin{equation}
\label{eq:conflation-g}
\begin{tikzcd}
\rightapp{\cU}{T}\arrow[infl]{r}{g}&\leftapp{\cT}{\rightapp{\cU}{T}}\arrow[defl]{r}&\leftcok{\cT}{\rightapp{\cU}{T}} \arrow[confl]{r} & \phantom{}
\end{tikzcd}
\end{equation}
of $\rightapp{\cU}{T}$, so $g$ is a left $\cT$-approximation of $\rightapp{\cU}{T}$, and $\coind{\cU}{\cT}[{\rightapp{\cU}{T}}]=[{\leftapp{\cT}{\rightapp{\cU}{T}}}]-[{\leftcok{\cT}{\rightapp{\cU}{T}}}]$.
Now we claim that $gf\colon \rightker{\cU}{T}\to\leftapp{\cT}{\rightapp{\cU}{T}}$ is a left $\cT$-approximation of $\rightker{\cU}{T}$.
Indeed, if $T'\in\cT$, it follows from the vanishing of $\Ext{1}{\cC}{T}{T'}$ that any morphism $h\colon \rightker{\cU}{T}\to T'$ factors through $f$.
Since $g$ is a left $\cT$-approximation of $\rightapp{\cU}{T}$, we see that $h$ even factors through $gf$, as required.
Since both $f$ and $g$ are inflations, it follows from the definition of an extriangulated category (axiom (ET4) in \cite[Def.~2.12]{NakaokaPalu}) that $gf$ is also an inflation.
Thus there is a conflation
\begin{equation}
\label{eq:conflation-gf}
\begin{tikzcd}
\rightker{\cU}{T}\arrow[infl]{r}{gf}&\leftapp{\cT}{\rightapp{\cU}{T}}\arrow[defl]{r}&\leftcok{\cT}{\rightker{\cU}{T}} \arrow[confl]{r} & \phantom{}
\end{tikzcd}
\end{equation}
in which $\leftcok{\cT}{\rightker{\cU}{T}}\in\cT$ since $\cT$ is cluster-tilting and $gf$ is a left $\cT$-approximation. We thus have $\coind{\cU}{\cT}{[{\rightker{\cU}{T}}]}=[{\leftapp{\cT}{\rightapp{\cU}{T}}}]-[{\leftcok{\cT}{\rightker{\cU}{T}}}]$, and wish to show that
\begin{align*}
\coind{\cU}{\cT}{\ind{\cT}{\cU}{[{T}}]}&=\coind{\cU}{\cT}{([{\rightapp{\cU}{T}}]-[{\rightker{\cU}{T}}])}\\
&=[{\leftapp{\cT}{\rightapp{\cU}{T}}}]-[{\leftcok{\cT}{\rightapp{\cU}{T}}}]-[{\leftapp{\cT}{\rightapp{\cU}{T}}}]+[{\leftcok{\cT}{\rightker{\cU}{T}}}]\\
&=[{\leftcok{\cT}{\rightker{\cU}{T}}}]-[{\leftcok{\cT}{\rightapp{\cU}{T}}}]
\end{align*}
is equal to $[{T}]$. But applying \cite[Lem.~3.14]{NakaokaPalu} to the conflations \eqref{eq:conflation-f}, \eqref{eq:conflation-g} and \eqref{eq:conflation-gf} yields a conflation $T\infl\leftcok{\cT}{\rightker{\cU}{T}}\defl\leftcok{\cT}{\rightapp{\cU}{T}}\confl$,
and the result follows.
\end{proof}

\begin{corollary}\label{c:clusters-same-size}
For any Krull--Schmidt cluster category $\cC$, if $\cTU\ctsubcat\cC$ then $\indec \cT$ and $\indec \cU$ have the same cardinality; in particular, if one cluster-tilting subcategory is additively finite, then all are.
If $\cC$ has a weak cluster structure, then $\exch{\cT}$ and $\exch{\cU}$ have the same cardinality.
\end{corollary}

\begin{proof}
Since $\cT$ is Krull--Schmidt, the cardinality of $\indec{\cT}$ is equal to the rank of $\Kgp{\cT}$, and so the result follows directly from Proposition~\ref{p:ind-coind-inverse}.
Then since $\cC$ having a weak cluster structure precisely means that $\exch{\cT}=\indec{\stab{\cT}}$ for any $\cT\ctsubcat\cC$, and $\indec{\cT}\setminus \indec{\stab{\cT}}$ is the set of isoclasses of indecomposable projective-injectives, independent of $\cT$, the second claim follows.
\end{proof}

\begin{definition}
The \emph{rank} of a cluster category $\cC$ is the common cardinality of $\indec{\stab{\cT}}$ for $\cT\ctsubcat\cC$.
\end{definition}

\begin{remark}
In making this definition, we use that $\stab{\cT}\ctsubcat\stab{\cC}$ and $\stab{\cC}$ is Krull--Schmidt since it is Hom-finite.
It is compatible with the definition of the rank of a cluster algebra in \cite[Def.~2.1.6]{Marsh-book}, for example, where it refers to the number of \emph{mutable} variables in a cluster, ignoring any frozen variables.
\end{remark}

In contrast to the classical theory of cluster algebras, in which definitions are made iteratively via mutations, the index and coindex isomorphisms are defined directly for \emph{any} pair of cluster-tilting subcategories.
This means that, until one attempts to decategorify, the question of reachability---whether any two cluster-tilting subcategories are linked by a sequence of mutations---does not arise.

\begin{example}\label{ex:markov-cl-cat} Let $\cC$ be the (triangulated) cluster category associated to the Markov quiver with its Labardini potential \cite{Labardini-QPs1}.
The Markov quiver has three vertices (say, $1$, $2$, $3$) and a pair of arrows from $i$ to $i+1$ (modulo $3$) for each $i$, forming a `double oriented cycle', and the potential can be found in \cite[Ex.~4.3]{Plamondon-Generic}.

As shown in loc.\ cit., $\cC$ has two mutation classes of cluster-tilting subcategories.
In particular, choosing a root cluster-tilting subcategory $\cTroot$, its shift $\Sigma\cTroot$ lies in the other class, as it is not reachable by iterated mutations from $\cTroot$.
The argument presented in \cite{Plamondon-Generic} is another neat use of indices: one checks that the sum of coefficients of the index of a cluster-tilting object is a mutation invariant, but this sum is $3$ for the cluster-tilting object generating $\cTroot$ and $-3$ for the cluster-tilting object generating $\Sigma\cTroot$.

However, so far we treat $\Sigma\cTroot$ no differently from any other cluster-tilting subcategory.
For example, with respect to the natural bases (with the natural bijection between them), the map $\ind{\cTroot}{\Sigma\cTroot}\colon\Kgp{\cTroot}\to \Kgp{\Sigma\cTroot}$ is represented by the matrix $-I_{3}$.
\end{example}

The next proposition shows that the index and coindex maps behave well under partial stabilisation.

\begin{proposition}\label{p:stable-ind-coind}
If $\cC$ is a cluster category, $\cTU\ctsubcat\cC$, and $\cP$ is a full and additively closed subcategory of projectives in $\cC$, then we have
\[\cind{\cT/\cP}{\cU/\cP}\circ \pi_{\cT}^{\cT/\cP} =\pi_{\cU}^{\cU/\cP}\circ \cind{\cT}{\cU}.\]
\end{proposition}

\begin{proof} By the definitions, \eqref{eq:stab-restrict} and Proposition~\ref{p:stabind} we have
\begin{align*}
\cind{\cT/\cP}{\cU/\cP}\circ\pi_{\cT}^{\cT/\cP}&=\cind{\addcat{(\cC/\cP)}}{\cU/\cP}\circ\iota_{\cT/\cP}^{\addcat{(\cC/\cP)}}\circ\pi_{\cT}^{\cT/\cP}\\
&=\cind{\addcat{(\cC/\cP)}}{\cU/\cP}\circ\pi_{\addcat{\cC}}^{\addcat{(\cC/\cP)}}\circ\iota_{\cT}^{\addcat{\cC}}\\
&=\pi_{\cU}^{\cU/\cP}\circ\cind{\addcat{\cC}}{\cU}\circ\iota_{\cT}^{\addcat{\cC}}\\
&=\pi_{\cU}^{\cU/\cP}\circ\cind{\cT}{\cU}.\qedhere
\end{align*}
\end{proof}

\begin{remark}
In the case that $\cC/\cP=\stab{\cC}$, we abbreviate
$\stabcind{\cT}{\cU}\defeq\cind{\cT/\cP}{\cU/\cP}$.
Together with the abbreviations from Remark~\ref{rem:pi-iota-shorthand}, the statement of Proposition~\ref{p:stable-ind-coind} becomes
\[\stabcind{\cT}{\cU}\circ\pproj{\cT}=\pproj{\cU}\circ\cind{\cT}{\cU}.\]
\end{remark}

We would like to have analogues of the index and coindex maps relating the numerical Grothendieck groups $\Kgpnum{\lfd{\cT}}$ and $\Kgpnum{\lfd{\cU}}$ for $\cTU\ctsubcat\cC$.
We will obtain these via adjunction, so we need to assume that $\cC$ is compact or skew-symmetric.

\begin{lemma}
\label{l:cind-divisibility}
Let $\cC$ be a compact or skew-symmetric cluster category, and let $\cTU\ctsubcat\cC$.
Then for any $T\in\indec{\cT}$ and $M\in\lfd{\cU}$, we have
\[\dimdivalg{T}\divides \canform{[M]}{\cind{\cT}{\cU}[T]}{\cU}.\]
\end{lemma}

\begin{proof}
We give the proof for the index, that for the coindex being similar.
If $\cC$ is skew-symmetric, then $\dimdivalg{T}=1$ by definition, and so there is nothing to prove.
So assume $\cC$ is compact, and choose a \emph{minimal} $\cU$-index conflation
\[\begin{tikzcd}
\rightker{\cU}{T}\arrow[infl]{r}&\rightapp{\cU}{T}\arrow[defl]{r}{\varphi}&T\arrow[confl]{r} & \phantom{}
\end{tikzcd}\]
for $T$.
We may then compute
\[\canform{[M]}{\ind{\cT}{\cU}[T]}{\cU}=\sum_{U\in\indec{\cU}}\dim_{\bK}M(U)([\rightapp{\cU}{T}:U]-[\rightker{\cU}{T}:U]).\]
Since $\dimdivalg{U}|\dim_{\bK}M(U)$ by Lemma~\ref{l:compact-module-division}, it is therefore sufficient to show that
\[\dimdivalg{T}\divides\dimdivalg{U}([\rightapp{\cU}{T}:U]-[\rightker{\cU}{T}:U])\]
for each $U\in\indec{\cU}$.
Since $\dimdivalg{T}\divides \dimdivalg{U}[\rightapp{\cU}{T}:U]$ by Corollary~\ref{c:approx-count}, it is even enough to show that $\dimdivalg{T} \divides \dimdivalg{U}[\rightker{\cU}{T}:U]$.

In the stable category $\stab{\cC}$, there is a triangle
\[\begin{tikzcd}
\Sigma^{-1}T\arrow{r}{\psi}&\rightker{\cU}{T}\arrow{r}&\rightapp{\cU}{T}\arrow{r}{\varphi}&T.
\end{tikzcd}\]
Because $\rightker{\cU}{T},\rightapp{\cU}{T}\in\cU$, the map $\psi$ is a left $\cU$-approximation of $\Sigma^{-1}T$ in the cluster category $\stab{\cC}$, and it is minimal because $\varphi$ is.
Moreover, the multiplicity $[\rightker{\cU}{T}:U]$ for $U\in\indec{\stab{\cU}}$ is the same in either $\stab{\cC}$ or $\cC$; the two multiplicities may only differ if $U$ is projective, but in this case both are zero by minimality of $\varphi$.

Now either $T$ is projective, so $\rightker{\cU}{T}=0$ and there is nothing to prove, or both $T$ and $\Sigma^{-1}T$ are indecomposable in $\stab{\cC}$.
In the second case, we have $\dimdivalg{\Sigma^{-1}T}\divides \dimdivalg{U}[\rightker{\cU}{T}:U]$ by Corollary~\ref{c:approx-count} again, but $\dimdivalg{\Sigma^{-1}T}=\dimdivalg{T}$ because $\Sigma$ is an autoequivalence.
\end{proof}

Recall from Proposition~\ref{p:K-duality} that the non-degenerate form $\canform{\blank}{\blank}{\cT}$ induces an injection $\sdual{\cT}\colon\Kgpnum{\lfd{\cT}}\to\dual{\Kgp{\cT}}$ for each $\cT\ctsubcat\cC$ in a Krull--Schmidt cluster category $\cC$.

\begin{proposition}
\label{p:ind-coind-bar}
Let $\cC$ be a compact or skew-symmetric cluster category and let $\cTU\ctsubcat\cC$. Then for all $M\in\lfd{\cU}$, we have
$\dual{(\cind{\cT}{\cU})}\circ\sdual{\cU}[M]\in\image(\sdual{\cT})$.
\end{proposition}

\begin{proof}
We compute
$\dual{(\cind{\cT}{\cU})}\circ\sdual{\cU}[M]=\canform{[M]}{\cind{\cT}{\cU}(\blank)}{\cU}$.
Evaluating on $[T]$ for $T\in\indec\cT$ gives $\canform{[M]}{\cind{\cT}{\cU}[T]}{\cU}$, which is divisible by $\dimdivalg{T}$ by Lemma~\ref{l:cind-divisibility}.
The functional $\canform{[M]}{\cind{\cT}{\cU}(\blank)}{\cU}$ thus lies in $\image(\sdual{\cT})$ by the characterisation of this image in Proposition~\ref{p:im-delta}.
\end{proof}

We may thus apply Proposition~\ref{p:adjunction} to take adjoints of the index and coindex isomorphisms.

\begin{definition}\label{d:ind-bar-def}
For a compact or skew-symmetric cluster category $\cC$ and $\cTU\ctsubcat\cC$, we define 
\begin{align*}
\coindbar{\cU}{\cT}  =\adj{(\ind{\cT}{\cU})}&\colon \Kgpnum{\lfd \cU} \to \Kgpnum{\lfd \cT}, \\
\indbar{\cU}{\cT} =\adj{(\coind{\cT}{\cU})}&\colon \Kgpnum{\lfd \cU} \to \Kgpnum{\lfd \cT},
\end{align*}
by taking adjoints with respect to $\canform{\blank}{\blank}{\cT}$ and $\canform{\blank}{\blank}{\cU}$.
\end{definition}

Here, adjunction tells us that
\begin{equation}\label{eq:ind-coindbar-adj}
\canform{\coindbar{\cU}{\cT}[M]}{[T]}{\cT}=\canform{[M]}{\ind{\cT}{\cU}[T]}{\cU},
\end{equation}
for all $T\in\cT$ and $M\in\lfd{\cU}$, cf.\ Proposition~\ref{p:adjunction}, and similarly for $\indbar{}{}$ and $\coind{}{}$.

\begin{remark}
\label{r:adj-restr}
Analysing the construction of the adjoint in Proposition~\ref{p:adjunction}, we see that we can also take adjoints to $\cind{\cT}{\cU}$ using the standard form $\canform{\blank}{\blank}{\cT}$ for $\cT$, and the restricted form $\canform{(\blank)|_{\fd{\cU}}}{\blank}{\cU}\colon\Kgp{\fd{\cU}}\times\Kgp{\cU}\to\integ$ for $\cU$.
The resulting adjoints are given simply by the restrictions of $\cindbar{\cU}{\cT}$ to $\Kgp{\fd{\cU}}\leq\Kgpnum{\lfd{\cU}}$.
If one attempts to restrict both forms, the adjoints exist if and only if the restrictions $\cindbar{\cU}{\cT}\colon\Kgp{\fd{\cU}}\to\Kgpnum{\lfd{\cT}}$ take values in $\Kgp{\fd{\cT}}\leq\Kgpnum{\lfd{\cT}}$ (in which case the adjoints are precisely these maps, with appropriately restricted codomain), but this is not always the case.
\end{remark}

\begin{proposition}
\label{p:indbar-coindbar-inverse}
Let $\cC$ be a compact or skew-symmetric cluster category and $\cTU\ctsubcat\cC$. Then $\indbar{\cU}{\cT}\colon \Kgpnum{\lfd{\cU}}\isoto\Kgpnum{\lfd{\cT}}$ is an isomorphism with inverse $\coindbar{\cT}{\cU}$.
\end{proposition}
\begin{proof}
Let $M\in\lfd{\cU}$, and $U\in\cU$.
Then
\begin{align*}
\canform{\indbar{\cT}{\cU}\coindbar{\cU}{\cT}[M]}{[U]}{\cU}&=\canform{\coindbar{\cU}{\cT}[M]}{\coind{\cU}{\cT}[U]}{\cT}\\
&=\canform{[M]}{\ind{\cT}{\cU}\coind{\cU}{\cT}[U]}{\cU}\\
&=\canform{[M]}{[U]}{\cU}
\end{align*}
by adjunction and Proposition~\ref{p:ind-coind-inverse}.
Since $\canform{\blank}{\blank}{\cU}$ is non-degenerate, it follows that $\indbar{\cT}{\cU}\circ\coindbar{\cU}{\cT}=\id_{\Kgpnum{\lfd{\cU}}}$.
The analogous calculation for the other composition gives the result.
\end{proof}

\begin{remark}
As we have already remarked, and will return to below, the map $\ind{\cU}{\cT}$ gives the categorical analogue of $\mathbf{g}$-vectors in cluster theory (as does $\coind{\cU}{\cT}$, for a different convention).
The maps $\indbar{\cU}{\cT}$ and $\coindbar{\cU}{\cT}$ will provide $\mathbf{c}$-vectors under the corresponding conventions.
Indeed, it is the tropical duality between $\mathbf{g}$-vectors and $\mathbf{c}$-vectors \cite{NakanishiZelevinsky} that motivates the definition of $\cindbar{\cU}{\cT}$.
This connection will be made more precise in Theorem~\ref{t:c-vec-mut-formula}.
\end{remark}

Given a Krull--Schmidt cluster category $\cC$ with $\cT\ctsubcat\cC$, and $\cP$ a full and additively closed subcategory of projectives, there is an injection $\iota_{\cT/\cP}^{\cT}\colon \Kgpnum{\lfd(\cT/\cP)}\to \Kgpnum{\lfd \cT}$ and a surjection $\pi_{\cT}^{\cT/\cP}\colon\Kgp{\cT}\to\Kgp{\cT/\cP}$ (see Section~\ref{s:duality}), and moreover $\iota_{\cT/\cP}^{\cT}=\adj{(\pi_{\cT}^{\cT/\cP})}$  \eqref{eq:pi-p-iota-s-adjoint}.
The next statement is adjoint to Proposition~\ref{p:stable-ind-coind}.

\begin{proposition}\label{p:iota-coind}
Let $\cC$ be a compact or skew-symmetric cluster category, and let $\cP$ be a full and additively closed subcategory of projectives.  Then
\[\iota_{\cT/\cP}^{\cT}\circ\cindbar{\cU/\cP}{\cT/\cP}=\cindbar{\cU}{\cT}\circ\iota_{\cU/\cP}^{\cU}\]
for any $\cTU\ctsubcat\cC$.\qed
\end{proposition}

\begin{remark}
In the usual way, when $\cC/\cP=\stab{\cC}$ we abbreviate
$\stabcindbar{\cT'}{\cT}\defeq\cindbar{\cT'/\cP}{\cT/\cP}$.
While the notation in Proposition~\ref{p:iota-coind} is a little heavy, it amounts to the fact that if we view each $\Kgpnum{\lfd{\cT/\cP}}$ as a subgroup of $\Kgpnum{\lfd{\cT}}$ in the natural way, then the maps $\cindbar{\cU}{\cT}$ restrict to the maps $\cindbar{\cU/\cP}{\cT/\cU}$ between these subgroups.
\end{remark}

By Proposition~\ref{p:equiv-to-mod}, each $\cT\ctsubcat\cC$ determines a functor $\Extfun{\cT}\colon\cC\to\fpmod{\stab{\cT}}\subseteq\lfd{\stab{\cT}}$, with $\Extfun{\cT}X=\Ext{1}{\cC}{\blank}{X}|_{\cT}$, recalling for the inclusion of categories that $\stab{\cT}$ is Hom-finite.
The next lemma, another application of Lemma~\ref{l:ind-coind-adjointness}, demonstrates the extent to which the induced maps $\Extfun{\cT}\colon\Kgp{\cC^{\add}}\to\Kgpnum{\lfd\stab{\cT}}$ commute with the index and coindex.

\begin{lemma}
\label{ind-coind-ext}
Let $\cC$ be a compact or skew-symmetric cluster category and let $\cTU\ctsubcat\cC$.
Then for any $X\in\cC$, we have
\[ (\stabcoindbar{\cU}{\cT}-\stabindbar{\cU}{\cT})[\Extfun{\cU}X]=\Extfun{\cT}((\coind{\cC}{\cU}-\ind{\cC}{\cU})[X]).\]
If $X\in\cT$, so $\ind{\cC}{\cU}[X]=\ind{\cT}{\cU}[{X}]$, and similarly for the coindex, then we even have
\[\stabcindbar{\cU}{\cT}[\Extfun{\cU}X]=\Extfun{\cT}(\cind{\cT}{\cU}[{X}]).\]
\end{lemma}

\begin{proof}
By Proposition~\ref{p:K-duality}, we need only check that the equality holds after applying the injective map $\sdual{\cT}$ to each side; this gives us functions on $\Kgp{{\cT}}$, which we compare by evaluating on some class $[T]$.
Evaluating the right-hand side is straightforward, and gives
\[\ext{1}{\cC}{T}{\rightker{\cU}{X}}-\ext{1}{\cC}{T}{\rightapp{\cU}{X}}+\ext{1}{\cC}{T}{\leftapp{\cU}{X}}-\ext{1}{\cC}{T}{\leftcok{\cU}{X}}.\]
On the left-hand side, unpacking the definitions tells us that we should precompose the function $[\Extfun{\cU}X]=\ext{1}{\cC}{\blank}{X}$ on $\cU$ with $\ind{\cT}{\cU}$ and $\coind{\cT}{\cU}$, and then take the difference---evaluating the resulting function on $[{T}]$ gives
\[\ext{1}{\cC}{\leftcok{\cU}{T}}{X}-\ext{1}{\cC}{\leftapp{\cU}{T}}{X}+\ext{1}{\cC}{\rightapp{\cU}{T}}{X}-\ext{1}{\cC}{\rightker{\cU}{T}}{X}.\]
By Lemma~\ref{l:ind-coind-adjointness}, these two values agree.
The statements for $X\in\cT$ are proved similarly, again using Lemma~\ref{l:ind-coind-adjointness}, noting that in this case $\Ext{1}{\cC}{T}{X}=0$.
\end{proof}

The value of $\cindbar{\cU}{\cT}$ on simple $\cU$-modules is of course particularly important for calculations.
At first we do not exclude loops or $2$-cycles, but the slightly simpler statements under this extra assumption are given in Corollary~\ref{c:one-step-X-mut-from-root} below.

\begin{lemma}\label{l:d-of-mutant}
Let $\cC$ be a compact or skew-symmetric cluster category, let $\cT\ctsubcat\cC$, and let $T\in\exch{\cT}$.
Then $\dimdivalg{T}=\dimdivalg{\mut{\cT}{T}}$.
\end{lemma}
\begin{proof}
Set $\cT'=\mut{T}\cT$ and $T'=\mut{\cT}T$.
By Lemma~\ref{l:cind-divisibility} and \eqref{eq:ind-on-mut-T}, we have 
\[
\dimdivalg{T'}\divides \canform{[\simpmod{\cT}{T}]}{\ind{\cT}{\cT}[T']}{\cT}=\canform{[\simpmod{\cT}{T}]}{[\rightapp{\cT}{T'}]-[T]}{\cT}
=\canform{[\simpmod{\cT}{T}]}{[\exchmon{\cT}{T}{-}]-[T]}{\cT}
=-\dimdivalg{T},
\]
recalling that $\exchmon{\cT}{T}{-}$ is a $(\cT\setminus T)$-approximation of $T$ and hence has no summand isomorphic to $T$.
By a symmetric argument, $\dimdivalg{T}\divides-\dimdivalg{T'}$.
Since both $\dimdivalg{T}$ and $\dimdivalg{T'}$ are positive integers, they must therefore be equal.
\end{proof}

\begin{remark}
An alternative argument, requiring only that $\cC$ is Krull--Schmidt, is to use Iyama--Yoshino's construction \cite{IyamaYoshino} of $\mut{\cT}{T}$ as the shift of $T$ in an appropriate triangulated subquotient of $\cC$.
This strategy requires showing that passing to this subquotient does not change the values of either $\dimdivalg{T}$ or $\dimdivalg{\mut{\cT}{T}}$, but this can be done.

There is also a simple argument under the assumption that there is no loop at either $T\in\cT$ or $\mut{\cT}{T}\in\mut{T}{\cT}$.
Indeed, in this situation we have $\dimdivalg{T}=\dim_{\bK}\Ext{1}{\cC}{T}{\mut{\cT}{T}}$ and $\dimdivalg{\mut{\cT}{T}}=\dim_{\bK}\Ext{1}{\cC}{\mut{\cT}{T}}{T}$ by Lemma~\ref{l:props-of-K0-mod-T}\ref{l:props-mod-T-simple-eq-E}, but these dimensions are equal since $\cC$ is stably 2-Calabi--Yau.
\end{remark}

\begin{corollary}
\label{c:no-loops-mutates}
Let $\cC$ be a Krull--Schmidt cluster category, let $\cT\ctsubcat\cC$, and let $T\in\exch{\cT}$.
Then there is no loop at $T\in\cT$ if and only if there is no loop at $\mut{\cT}{T}\in\mut{T}{\cT}$.
\end{corollary}
\begin{proof}
Abbreviating $T'=\mut{\cT}{T}$ and $\cT'=\mut{T}{\cT}$, we have
\[\Extfun{\cT}{T'}(T)=\Ext{1}{\cC}{T}{T'}=\dual{\Ext{1}{\cC}{T'}{T}}=\dual{\Extfun{\cT'}{T}(T')},\]
since $\cC$ is stably $2$-Calabi--Yau, and so $\dim_{\bK}\Extfun{\cT}{T'}(T)=\dim_{\bK}\Extfun{\cT}{T}(T')$.
Since $\dimdivalg{T}=\dimdivalg{T'}$ by Lemma~\ref{l:d-of-mutant}, it follows that $\rank_{\divalg{T}}\Extfun{\cT}{T'}(T)=\rank_{\divalg{T'}}\Extfun{\cT}{T}(T')$, these values being obtained by dividing the $\bK$-dimensions by $\dimdivalg{T}$ and $\dimdivalg{T'}$ respectively.
Since $\Extfun{\cT}{T'}=\simpmod{\cT}{T}$ if and only if $\rank_{\divalg{T}}\Extfun{\cT}{T'}(T)=1$, and similarly with the roles of $(\cT,T)$ and $(\cT',T')$ swapped, $\Extfun{\cT}{T'}$ is simple if and only if $\Extfun{\cT'}{T}$ is simple.
By Lemma~\ref{l:props-of-K0-mod-T}\ref{l:props-mod-T-simple-eq-E}, the simplicity of these respective functors is equivalent to the respective no loop conditions.
\end{proof}

The following proposition gives the analogous expressions to \eqref{eq:ind-on-mut-T} and \eqref{eq:coind-on-mut-T}.

\begin{proposition}\label{p:one-step-X-mut-from-root}
Let $\cC$ be a compact or skew-symmetric cluster category, $\cT\ctsubcat\cC$ and $T\in \exch{\cT}$.  Let $\cT'=\mut{T}{\cT}$ with associated exchange conflations 
\[\begin{tikzcd}
\mut{\cT}{T}\arrow[infl]{r}&\exchmon{\cT}{T}{+}\arrow[defl]{r}&T\arrow[confl]{r}&,
\end{tikzcd}\quad
\begin{tikzcd}
T\arrow[infl]{r}&\exchmon{\cT}{T}{-}\arrow[defl]{r}&\mut{\cT}{T}\arrow[confl]{r}&,
\end{tikzcd}\]
and let $U\in\indec(\cT'\setminus\mut{\cT}{T})$. Then in $\Kgp{\fd{\cT}}$, we have
\begin{enumerate}
\item\label{p:one-step-X-mut-from-root-indbar-at-mut} $\indbar{\cT'}{\cT}[\simpmod{\cT'}{\mut{\cT}{T}}]=-[\simpmod{\cT}{T}]$,
\item\label{p:one-step-X-mut-from-root-indbar-away-from-mut} $\indbar{\cT'}{\cT}[\simpmod{\cT'}{U}]=[\simpmod{\cT}{U}]+\dimdivalg{T}^{-1}\canform{[\simpmod{\cT'}{U}]}{[\exchmon{\cT}{T}{-}]}{\cT'}[\simpmod{\cT}{T}]$,
\item\label{p:one-step-X-mut-from-root-coindbar-at-mut} $\coindbar{\cT'}{\cT}[\simpmod{\cT'}{\mut{\cT}{T}}]=-[\simpmod{\cT}{T}]$, and
\item\label{p:one-step-X-mut-from-root-coindbar-away-from-mut} $\coindbar{\cT'}{\cT}[\simpmod{\cT'}{U}]=[\simpmod{\cT}{U}]+\dimdivalg{T}^{-1}\canform{[\simpmod{\cT'}{U}]}{[\exchmon{\cT}{T}{+}]}{\cT'}[\simpmod{\cT}{T}]$.
\end{enumerate}
\end{proposition}

\begin{proof}
Using the definitions, we compute that for $U\in\indec{\cT'}$ we have
\begin{align*} \indbar{\cT'}{\cT}[\simpmod{\cT'}{U}]&=(\delta_{\cT}^{s})^{-1}\bigl(\sum\nolimits_{V \in \indec{\cT}} \canform{[\simpmod{\cT'}{U}]}{\coind{\cT}{\cT'}[V]}{\cT'}[V]^{*}\bigr), \\
\coindbar{\cT'}{\cT}[\simpmod{\cT'}{U}]&=(\delta_{\cT}^{s})^{-1}\bigl(\sum\nolimits_{V \in \indec{\cT}} \canform{[\simpmod{\cT'}{U}]}{\ind{\cT}{\cT'}[V]}{\cT'}[V]^{*}\bigr),
\end{align*}
where $\canform{\blank}{\blank}{\cT'}$ is the non-degenerate form of Proposition~\ref{p:K-duality}. Now let $V\in\indec{\cT}$ and $U'\in\indec{\cT'}$. If $V\neq T$, then $V\in \cT \intersection \cT'$, so $\ind{\cT}{\cT'}[V]=\coind{\cT}{\cT'}[V]=[V]$ and we have 
\begin{equation}
\label{eq:simp-vs-cind} \canform{[\simpmod{\cT'}{U'}]}{\cind{\cT}{\cT'}[V]}{\cT'}=\canform{[\simpmod{\cT'}{U'}]}{[V]}{\cT'}=\delta_{U'V}\dimdivalg{U'}.
\end{equation}
On the other hand, if $V=T$, we compute using \eqref{eq:ind-on-mut-T} and \eqref{eq:coind-on-mut-T} that
\begin{align*} \canform{[\simpmod{\cT'}{U'}]}{\ind{\cT}{\cT'}[T]}{\cT'} & =\canform{[\simpmod{\cT'}{U'}]}{[\rightapp{\cT'}{T}]-[\mut{\cT}{T}]}{\cT'} 
=\canform{[\simpmod{\cT'}{U'}]}{[\exchmon{\cT}{T}{+}]-[\mut{\cT}{T}]}{\cT'}\\
\canform{[\simpmod{\cT'}{U'}]}{\coind{\cT}{\cT'}[T]}{\cT'} & =\canform{[\simpmod{\cT'}{U'}]}{[\leftapp{\cT'}{T}]-[\mut{\cT}{T}]}{\cT'} 
=\canform{[\simpmod{\cT'}{U'}]}{[\exchmon{\cT}{T}{-}]-[\mut{\cT}{T}]}{\cT'}.\end{align*}
Now for $U=\mut{\cT}{T}$,    
\begin{align*} \indbar{\cT'}{\cT}[\simpmod{\cT'}{\mut{\cT}{T}}] & =(\delta_{\cT}^{s})^{-1}\bigl(\sum\nolimits_{V \in \indec{\cT}} \canform{[\simpmod{\cT'}{\mut{\cT}{T}}]}{\coind{\cT}{\cT'}[V]}{\cT'}[V]^{*} \bigr)\\
& = (\delta_{\cT}^{s})^{-1}(\canform{[\simpmod{\cT'}{\mut{\cT}{T}}]}{\coind{\cT}{\cT'}{[T]}}{\cT'}[T]^{*}) \\
& = (\delta_{\cT}^{s})^{-1}(-\dimdivalg{\mut{\cT}{T}}[T]^{*}) \\
& = (\delta_{\cT}^{s})^{-1}(-\dimdivalg{T}[T]^{*}) \\
& = -[\simpmod{\cT}{T}]
\end{align*}
by \eqref{eq:simp-vs-cind} and Lemma~\ref{l:d-of-mutant} (and the calculation in its proof).
Repeating the computation for $\coindbar{}{}$, the only change is the appearance of $T_{\cT}^{+}$ in place of $T_{\cT}^{-}$, and the rest of the argument is identical. This gives us \ref{p:one-step-X-mut-from-root-indbar-at-mut} and \ref{p:one-step-X-mut-from-root-coindbar-at-mut}.

For the remaining identities, we have that for $U\neq \mut{\cT}{T}$,
\begin{align*} \indbar{\cT'}{\cT}[\simpmod{\cT'}{U}] & =(\delta_{\cT}^{s})^{-1}\bigl(\sum\nolimits_{V \in \indec{\cT}} \canform{[\simpmod{\cT'}{U}]}{\coind{\cT}{\cT'}[V]}{\cT'}[V]^{*}\bigr) \\
& = (\delta_{\cT}^{s})^{-1}(d_{U}[U]^{*}+\canform{[\simpmod{\cT'}{U}]}{[\exchmon{\cT}{T}{-}]-[\mut{\cT}{T}]}{\cT'}[T]^{*}) \\
&=[\simpmod{\cT}{U}]+(\delta_{\cT}^{s})^{-1}(\canform{[\simpmod{\cT'}{U}]}{[\exchmon{\cT}{T}{-}]}{\cT'}[T]^{*}) \\
& = [\simpmod{\cT}{U}]+d_{T}^{-1}\canform{[\simpmod{\cT'}{U}]}{[\exchmon{\cT}{T}{-}]}{\cT'}[\simpmod{\cT}{T}].
\end{align*}
The $\coindbar{}{}$ computation is the same but with $T_{\cT}^{+}$ instead of $T_{\cT}^{+}$, giving \ref{p:one-step-X-mut-from-root-indbar-away-from-mut} and \ref{p:one-step-X-mut-from-root-coindbar-away-from-mut}.
\end{proof}

\begin{proposition}\label{p:no-loops-simp}
Let $\cC$ be a compact cluster category, $\cT\ctsubcat\cC$ and assume that $\cT$ has no loop or $2$-cycle at $T\in\exch{\cT}$, and let $\cT'=\mut{T}{\cT}$.
Then
\[\dimdivalg{T}^{-1}\canform{[\simpmod{\cT'}{U}]}{[\exchmon{\cT}{T}{\mp}]}{\cT'}=[\exchmatentry{T,U}^{\cT}]_{\pm}.\]
\end{proposition}
\begin{proof}
By Proposition~\ref{p:decomp-exch-terms} and Corollary~\ref{c:exch-mat-at-mut}, no loop or $2$-cycle at $T$ implies that
\begin{align*}
\dimdivalg{T}^{-1}\canform{[\simpmod{\cT'}{U}]}{[\exchmon{\cT}{T}{\mp}]}{\cT'} & = \dimdivalg{T}^{-1}\canform{[\simpmod{\cT'}{U}]}{[\exchmon{\mut{T}{\cT}}{(\mut{\cT}{T})}{\pm}]}{\cT'} \\
& = \dimdivalg{T}^{-1}(\dimdivalg{U}[\exchmatentry{U,\mut{\cT}{T}}^{\mut{T}{\cT}}]_{\pm}) \\
& = \dimdivalg{T}^{-1}\dimdivalg{U}[\exchmatentry{U,T}^{\cT}]_{\mp} \\
& = [\exchmatentry{T,U}^{\cT}]_{\pm}. \qedhere
\end{align*}
\end{proof}

Combining the two previous results gives us categorical analogues of the formulæ
for the two (signed) tropical mutations of Fock--Goncharov \cite[Eq.~(7)]{FockGoncharov}, as follows.

\begin{corollary}\label{c:one-step-X-mut-from-root}
In the setting of Proposition~\ref{p:one-step-X-mut-from-root}, if $\cC$ is compact and $\cT$ has no loop or $2$-cycle at $T\in\exch{\cT}$ then
\begin{enumerate}
\item\label{c:one-step-X-mut-from-root-indbar-at-mut} $\indbar{\cT'}{\cT}[\simpmod{\cT'}{\mut{\cT}{T}}]=-[\simpmod{\cT}{T}]$,
\item\label{c:one-step-X-mut-from-root-indbar-away-from-mut} for $U\neq \mut{\cT}{T}$ indecomposable, $\indbar{\cT'}{\cT}[\simpmod{\cT'}{U}]=[\simpmod{\cT}{U}]+[\exchmatentry{T,U}^{\cT}]_{+}[\simpmod{\cT}{T}]$,
\item\label{c:one-step-X-mut-from-root-coindbar-at-mut} $\coindbar{\cT'}{\cT}[\simpmod{\cT'}{\mut{\cT}{T}}]=-[\simpmod{\cT}{T}]$, and
\item\label{c:one-step-X-mut-from-root-coindbar-away-from-mut} for $U\neq \mut{\cT}{T}$ indecomposable, $\coindbar{\cT'}{\cT}[\simpmod{\cT'}{U}]=[\simpmod{\cT}{U}]+[\exchmatentry{T,U}^{\cT}]_{-}[\simpmod{\cT}{T}]$.\qed
\end{enumerate}
\end{corollary}

\subsection{Sign-coherence}\label{s:sign-coherence}

An important phenomenon in cluster theory is the sign-coherence (in two dual senses) of the $\mathbf{g}$-vectors and $\mathbf{c}$-vectors.
Given our claim (still to be fully justified) that the values of the index and coindex maps and their adjoints are the homological analogues of these vectors, these maps should exhibit matching sign-coherence properties.
We now establish this, adapting arguments of Jørgensen--Yakimov \cite{JorgensenYakimov}.

\begin{definition}\label{d:g-vectors} Let $\cC$ be a cluster category and $\cTU \ctsubcat \cC$.  Define
\[ 
\gvecplus{\cT}{\cU}  =\{ \ind{\cU}{\cT}[U] : U\in \indec{\cU} \},\quad
\gvecminus{\cT}{\cU}  =\{ \coind{\cU}{\cT}[U] : U\in \indec{\cU} \}. 
\]

\end{definition}

By Proposition~\ref{p:ind-coind-inverse}, $\gvecplus{\cT}{\cU}$ and $\gvecminus{\cT}{\cU}$ are bases for $\Kgp{\cT}$, being the images of the standard basis for $\Kgp{\cU}$ under the isomorphisms $\ind{\cU}{\cT}$ and $\coind{\cU}{\cT}$.

\begin{proposition}
\label{p:g-vector-sign-coherence}The sets $\gvecplus{\cT}{\cU}$ and $\gvecminus{\cT}{\cU}$ are sign-coherent. That is, for each $T\in \indec{\cT}$ and any $U,V\in\indec{\cU}$, we have
\begin{align*} \canform{[\simpmod{\cT}{T}]}{\ind{\cU}{\cT}[U]}{\cT} \geq 0 &\iff \canform{[\simpmod{\cT}{T}]}{\ind{\cU}{\cT}[V]}{\cT} \geq 0, \\
\canform{[\simpmod{\cT}{T}]}{\coind{\cU}{\cT}[U]}{\cT} \geq 0 &\iff \canform{[\simpmod{\cT}{T}]}{\coind{\cU}{\cT}[V]}{\cT} \geq 0.
\end{align*}
\end{proposition}

\begin{proof} We follow the proof in \cite{DehyKeller} in our language. Let
\[\begin{tikzcd} \rightker{\cT}{U}\arrow[infl]{r}&\rightapp{\cT}{U}\arrow[defl]{r}&U\arrow[confl]{r} &,
\end{tikzcd}\quad\begin{tikzcd}
\rightker{\cT}{V}\arrow[infl]{r}&\rightapp{\cT}{V}\arrow[defl]{r}&V\arrow[confl]{r} & \phantom{}
\end{tikzcd}\]
be minimal $\cT$-index conflations of $U$ and $V$, so that $\ind{\cU}{\cT}[U]=[\rightapp{\cT}{U}]-[\rightker{\cT}{U}]$, and similarly for $V$.
The deflation in the $\cT$-index conflation
\[\begin{tikzcd} \rightker{\cT}{U}\dsum \rightker{\cT}{V}\arrow[infl]{r}&\rightapp{\cT}{U}\dsum \rightapp{\cT}{V}\arrow[defl]{r}&U\dsum V\arrow[confl]{r} & \phantom{}
\end{tikzcd}\]
for $U\oplus V$ is again minimal and so, since $U\oplus V$ is rigid, $\rightker{\cT}{U}\dsum \rightker{\cT}{V}$ and $\rightapp{\cT}{U}\dsum \rightapp{\cT}{V}$ have no common summands by Remark~\ref{r:no-common-summands}.

Therefore, if $\canform{[\simpmod{\cT}{T}]}{[\rightker{\cT}{U}]}{\cT}>0$ we must have $\canform{[\simpmod{\cT}{T}]}{[\rightapp{\cT}{V}]}{\cT}=0$, and similarly for the other combinations, and the result follows for $\gvecplus{\cT}{\cU}$.
The corresponding argument using $\cT$-coindex conflations yields the result for $\gvecminus{\cT}{\cU}$.
\end{proof}

\begin{definition}
Let $\cC$ be a compact or skew-symmetric cluster category and let $\cTU \ctsubcat \cC$.
Define
\[\cvecplus{\cT}{\cU}  =\bigl\{ \indbar{\cU}{\cT}[\simpmod{\cU}{U}] : U\in \indec{\cU} \bigr\}, \quad
\cvecminus{\cT}{\cU}  =\bigl\{ \coindbar{\cU}{\cT}[\simpmod{\cU}{U}] : U\in \indec{\cU}\bigr\}.
\]
\end{definition}

\begin{corollary}\label{c:sign-coh-c-vectors} For any $\cU\ctsubcat\cC$ and $U\in \indec \cU$, the vectors $\indbar{\cU}{\cT}[\simpmod{\cU}{U}]$ and $\coindbar{\cU}{\cT}[\simpmod{\cU}{U}]$ are sign-coherent. That is, for any $T,T'\in\indec{\cT}$, we have
\begin{align*}
\canform{\indbar{\cU}{\cT}[\simpmod{\cU}{U}]}{[T]}{\cT} \geq 0 &\iff \canform{\indbar{\cU}{\cT}[\simpmod{\cU}{U}]}{[T']}{\cT} \geq 0,\\
\canform{\coindbar{\cU}{\cT}[\simpmod{\cU}{U}]}{[T]}{\cT} \geq 0 &\iff \canform{\coindbar{\cU}{\cT}[\simpmod{\cU}{U}]}{[T']}{\cT} \geq 0.
\end{align*}
\end{corollary}

\begin{proof} By \eqref{eq:ind-coindbar-adj}, we have
$ \canform{\indbar{\cU}{\cT}[\simpmod{\cU}{U}]}{[T]}{\cT} = \canform{[\simpmod{\cU}{U}]}{\coind{\cT}{\cU}{[T]}}{\cU}$,
so the result follows immediately from Proposition~\ref{p:g-vector-sign-coherence}.
The argument for $\coindbar{}{}$ is completely parallel.
\end{proof}

\begin{remark}
The sign-coherence properties of $\gvecset{\pm}{\cT}{\cU}$ and $\cvecset{\pm}{\cT}{\cU}$ are different, although we use the same terminology (as is typical).
In each case, one can write the elements of the sets as vectors using the appropriate standard basis (indecomposables in $\cT$ for $\gvecset{\pm}{\cT}{\cU}$ and the simple $\cT$-modules for $\cvecset{\pm}{\cT}{\cU}$), and consider the matrix with these vectors as columns.
The sets $\gvecset{\pm}{\cT}{\cU}$ are then \emph{row} sign-coherent, in the sense that every row of this matrix has either all non-negative or all non-positive entries, whereas $\cvecset{\pm}{\cT}{\cU}$ is \emph{column} sign-coherent, the analogous condition on the columns of the matrix.
In particular, the column sign-coherence of $\cvecset{\pm}{\cT}{\cU}$ is a property of its individual elements, whereas the row sign-coherence of $\gvecset{\pm}{\cT}{\cU}$ is a property of the entire set.
\end{remark}

It also follows from Proposition~\ref{p:ind-coind-inverse} and the adjunction of $\coind{}{}{}$ and $\indbar{}{}{}$ that
\[ \canform{\indbar{\cU}{\cT}(\blank)}{\ind{\cU}{\cT}(\blank)}{\cT}=\canform{\blank}{\coind{\cT}{\cU}(\ind{\cU}{\cT}(\blank))}{\cU}=\canform{\blank}{\blank}{\cU}, \]
and in particular that
\[ \canform{\indbar{\cU}{\cT}{[\simpmod{\cU}{U}]}}{\ind{\cU}{\cT}{[V]}}{\cT}=\canform{\simpmod{\cU}{U}}{[V]}{\cU}=\dimdivalg{U}\delta_{UV}. \]

\begin{remark}
\label{r:scattering}
From these sets, one can start to build cones and fans, as in the theory developed by Bridgeland \cite{Bridgeland}, Gross--Hacking--Keel--Kontsevich \cite{GHKK} and others.
Indeed, some parts follow immediately from the above.
For simplicity, we assume that $\cC$ has finite rank, so that in particular $\lfd \cT=\fd \cT$.

Take as ambient spaces $\Kgp{\cT}\tensor \real$ and $\Kgp{\fd \cT}\tensor \real$ and consider the initial cones $G_{\cT}(\cT)\defeq \real_{+}\gvecplus{\cT}{\cT}$ and $C_{\cT}(\cT)\defeq\real_{+}\cvecplus{\cT}{\cT}$, which are the strongly convex rational polyhedral cones spanned by classes of objects in $\cT$ and $\fd \cT$ respectively.
We have many other cones $G_{\cT}(\cU)=\real_{+}\gvecplus{\cT}{\cU}$ and $C_{\cT}(\cU)=\real_{+}\cvecplus{\cT}{\cU}$, one for each cluster-tilting subcategory.
Since the $G_{\cT}(\cU)$ are given by taking images of the initial cone $G_{\cT}(\cT)$ under the isomorphisms given by indices, and $\gvecplus{\cT}{\cU}$ is sign-coherent as above, the $G_{\cT}(\cU)$ are also convex rational polyhedral cones.

The interiors of two cones $G_{\cT}(\cU)$ and $G_{\cT}(\cV)$ intersect if and only if $\cU=\cV$: indeed, given an object $X$ corresponding to a lattice point in the intersection of the interiors, we have $\add(X)=\cU$ and $\add(X)=\cV$, so $\cU=\cV$.
A similar argument shows that $G_{\cT}(\cU)\cap G_{\cT}(\cV)=\real_{+}\Kgp{\cU\cap\cV}$, and so in particular the intersection of any two cones of this form contains $\real_{+}\Kgp{\mathcal{P}}$, for $\mathcal{P}$ the full subcategory of projective-injective objects in $\cC$.

When $\cC$ is triangulated, one may also define
\[ C_{\cT}(\cU)^{\circ}=\real_{+}\{ m \in \Kgp{\fd \cT} \mid \canform{m}{\adj{\beta}_{\cT}C_{\cT}(\cU)}{\cT}\geq 0 \} \]
where $\adj{\beta}_{\cT}$ is the adjoint of a map of lattices $\beta_{\cT}\colon \Kgp{\fd \cT}\to \Kgp{\cT}$ which will be introduced in Section~\ref{s:cat-ex-mat}, and categorifies the exchange matrix (or the map $p^*$ in the cluster ensemble).
Properties of this map, in particular Proposition~\ref{p:beta-maps-cones}, mean that if we extend scalars in $\beta_{\cT}$ to obtain a map $\beta_{\cT}^{\real}$, we have $(\beta_{\cT}^{\real})^{-1}G_{\cT}(\cU)=C_{\cT}(\cU)^{\circ}$.
This relationship between the cones is the starting point for studying cluster algebras and varieties via scattering diagrams and wall-crossing.

The remainder of the theory of scattering diagrams, most notably the functions to attach to the walls, is not so elementary.
Nevertheless, we believe the results herein can be further developed to provide natural categorical expressions for these quantities, and proofs of their key properties.
This is, of course, closely related to the topic of cluster characters, which we address in Section~\ref{s:clust-char}.
\end{remark}

\subsection{Projective resolutions}\label{s:proj-resol}

In this subsection we define a map $p_{\cT}\colon\Kgp{\fpmod{\stab{\cT}}}\to\Kgp{\cT}$ for $\cT\ctsubcat\cC$, closely related to the process of taking projective resolutions.
We will use the lifting technique outlined in Section~\ref{ss:clust-cats}, wherein we use the algebraicity in the definition of a cluster category to ensure the existence of an exact cluster category $\cE$ and a full and additively closed subcategory $\cP$ of projective objects in $\cE$ such that $\cE/\cP\simeq \cC$; this is Proposition~\ref{p:cl-cat-exact-lift}.  
Let $\widehat{\cT}\ctsubcat\cE$ be the cluster-tilting subcategory corresponding to $\cT$ under the bijection of Lemma~\ref{l:ct-bijection}.
Then $\cT$ is a quotient of $\widehat{\cT}$, and we naturally identify $\cT$-modules with $\widehat{\cT}$-modules vanishing on $\cP$.
Write $\per{\widehat{\cT}}$ for the category of perfect complexes of $\widehat{\cT}$-modules.
The following result of Keller--Reiten is key to the construction.

\begin{proposition}[{\cite[Prop.~4(c)]{KellerReiten}}]
\label{p:KR-pdim3}
Let $\cC$, $\cT\ctsubcat \cC$, $\cE$ and $\widehat{\cT}$ be as above.
Then any finitely presented $\stab{\cT}$-module lies in $\per{\widehat{\cT}}$ when considered as a $\widehat{\cT}$-module.\qed
\end{proposition}

Let $\cC$ be a cluster category and $\cT\ctsubcat\cC$.
For $\cE$ and $\widehat{\cT}$ as above, we obtain a fully faithful functor $\fpmod{\stab{\cT}}\to\per{\widehat{\cT}}$ by Proposition~\ref{p:KR-pdim3}.
Since $\stab{\cE}=\stab{\cC}$ is Hom-finite, the same argument as in \cite[Prop.~3.2(a)]{FuKeller} shows that this functor induces a natural map  $\Kgp{\fpmod{\stab{\cT}}}\to\Kgp{\per{\widehat{\cT}}}=\Kgp{\proj{\widehat{\cT}}}$.

\begin{definition}
\label{d:proj-res-map}
For $\cC$, $\cT$, $\cE$ and $\widehat{\cT}$ as above, write
$p_{\widehat{\cT}}\colon\Kgp{\fpmod{\stab{\cT}}}\to\Kgp{\widehat{\cT}}$
for the composition $\Kgp{\fpmod{\stab{\cT}}}\to\Kgp{\proj{\widehat{\cT}}}\isoto\Kgp{\widehat{\cT}}$ of the above natural map with the inverse of the Yoneda isomorphism $h^{\widehat{\cT}}\colon\Kgp{\widehat{\cT}}\isoto\Kgp{\proj{\widehat{\cT}}}$.
Define
\[p_{\cT}=\pi_{\widehat{\cT}}^{\cT}\circ p_{\widehat{\cT}}\colon\Kgp{\fpmod{\stab{\cT}}}\to\Kgp{\cT},\]
where $\pi_{\widehat{\cT}}^{\cT}\colon\Kgp{\widehat{\cT}}\to\Kgp{\cT}$ is the natural projection.
\end{definition}

\begin{proposition}
\label{p:beta-proj-res}
Let $\cC$ be a cluster category, and let $\cT\ctsubcat\cC$. Then any $M\in\fpmod{\stab{\cT}}$ is isomorphic to $\Extfun\cT X$ for some $X\in\cC$ (Corollary~\ref{c:fp=fd}), and for any such $X$ we have
\begin{equation}
\label{eq:p-vs-ind-coind}
p_{\cT}[M]=\ind{\cC}{\cT}[X]-\coind{\cC}{\cT}[X].
\end{equation}
In particular, $p_{\cT}$ depends only on $\cT\ctsubcat\cC$, and not on the choice of exact lift $\cE$.
\end{proposition}

\begin{proof}
By Proposition~\ref{p:stabind}, it suffices to prove the identity in the case that $\cC$ is exact.
Just as in the proof of Proposition~\ref{p:ind-proj-res}, we may do this by exhibiting a single projective resolution with the correct class. Let
\[\begin{tikzcd}[row sep=0pt]
0\arrow{r}&\rightker{\cT}{X}\arrow{r}&\rightapp{\cT}{X}\arrow{r}&X\arrow{r}&0,\\
0\arrow{r}&X\arrow{r}&\leftapp{\cT}{X}\arrow{r}&\leftcok{\cT}{X}\arrow{r}&0
\end{tikzcd}\]
be $\cT$-index and $\cT$-coindex sequences for $X$, so that $\ind{\cC}{\cT}[X]=[{\rightapp{\cT}{X}}]-[{\rightker{\cT}{X}}]$ and $\coind{\cC}{\cT}[X]=[{\leftapp{\cT}{X}}]-[{\leftcok{\cT}{X}}]$.
Applying $\Yonfun{\cT}$ to these sequences yields
\[\begin{tikzcd}[row sep=0pt]
0\arrow{r}&\projmod{\cT}{\rightker{\cT}{X}}\arrow{r}&\projmod{\cT}{\rightapp{\cT}{X}}\arrow{r}&\Yonfun{\cT}{X}\arrow{r}&0,\\
0\arrow{r}&\Yonfun{\cT}X\arrow{r}&\projmod{\cT}{\leftapp{\cT}{X}}\arrow{r}&\projmod{\cT}{\leftcok{\cT}{X}}\arrow{r}&\Extfun{\cT}X=M\arrow{r}&0.
\end{tikzcd}\]
Taking the cup product, we obtain a projective resolution
\begin{equation}
\label{eq:pdim3}
\begin{tikzcd}[column sep=22pt]
0\arrow{r}&\projmod{\cT}{\rightker{\cT}{X}}\arrow{r}&\projmod{\cT}{\rightapp{\cT}{X}}\arrow{r}&\projmod{\cT}{\leftapp{\cT}{X}}\arrow{r}&\projmod{\cT}{\leftcok{\cT}{X}}\arrow{r}&M\arrow{r}&0
\end{tikzcd}
\end{equation}
of $M$, whose class in $\Kgp{\proj{\cT}}$ is
\[h^{\cT}p_{\cT}[M]=[\projmod{\cT}{\leftcok{\cT}{X}}]-[\projmod{\cT}{\leftapp{\cT}{X}}]+[\projmod{\cT}{\rightapp{\cT}{X}}]-[\projmod{\cT}{\rightker{\cT}{X}}] = h^{\cT}(\ind{\cC}{\cT}[X]-\coind{\cC}{\cT}[X]).
\]
The identity \eqref{eq:p-vs-ind-coind} follows since $h^{\cT}$ is an isomorphism.
\end{proof}

\begin{remark}
\label{r:pdim3}
From \eqref{eq:pdim3}, we see that $M\in\fpmod{\cT}$ has projective dimension at most $3$ as a $\widehat{\cT}$-module, where $\widehat{\cT}$ is the lift of $\cT$ to an exact category $\cE$ with partial stabilisation $\cC$.
This is also a by-product of Keller--Reiten's proof of Proposition~\ref{p:KR-pdim3}.

The reason we do not use the right-hand side of \eqref{eq:p-vs-ind-coind} as the definition of $p_{\cT}$ is that the description in Definition~\ref{d:proj-res-map} makes it clearer that $p_{\cT}$ is well-defined as a function of $[M]$: while it follows from Proposition~\ref{p:equiv-to-mod} that the expression in \eqref{eq:p-vs-ind-coind} does not depend on the choice of $X$ with $[M]=[\Extfun{\cT}X]$, it is less clear that this expression is additive on exact sequences in $\fpmod{\stab{\cT}}$.
\end{remark}

Let $\cC$ be a Krull--Schmidt cluster category, and fix $\cT\ctsubcat\cC$.
As in Section~\ref{s:sign-coherence}, any $\cU\ctsubcat\cC$ determines a basis $\{\ind{\cU}{\cT}[U] \mid U\in\indec{\cU}\}$ of $\Kgp{\cT}$, typically different from the standard basis $\{[T] \mid T\in\indec{\cT}\}$, and analogous statements hold for $\coind{\cU}{\cT}$ and the adjoint maps, when these exist.
This situation is reminiscent of that of classical tilting theory \cite{BrennerButler}, and indeed this is more than simply an analogy, at least for exact cluster categories, as the next result shows.
The argument is essentially due to Geiß--Leclerc--Schröer \cite[Thm.~10.2]{GLS-Unipotent}, who adapt it from work of Iyama \cite[Thm.~5.3.2]{Iyama-AusCor}.
For convenience, we give a full proof in our notation and level of generality.

\begin{theorem}\label{t:coind-tilt}
Let $\cC$ be an exact cluster category with cluster-tilting objects $T$ and $T'$, and write $A=\op{\End{\cC}{T}}$ and $A'=\op{\End{\cC}{T'}}$.
Then $\bT=\Hom{\cC}{T}{T'}$ is a tilting $A$-module with $\op{\End{A}{\bT}}\cong A'$.

Moreover, the isomorphism $\Kgp{\per{A}}\isoto\Kgp{\per{A'}}$ induced from the equivalence $\RHom{A}{\bT}{\blank}\colon\per{A}\isoto\per{A'}$ coincides with $\coind{T}{T'}$ under the natural identification
$\Kgp{\per{A}}=\Kgp{\proj{A}}=\Kgp{\add{T}}$
and the corresponding identification for $A'$ and $T'$.
\end{theorem}
\begin{proof}
Choosing a $T$-index sequence
\begin{equation}
\label{eq:index-seq2}
\begin{tikzcd}
0\arrow{r}& KT'\arrow{r}& RT'\arrow{r}& T'\arrow{r}& 0
\end{tikzcd}
\end{equation}
for $T'$ and applying $\Hom{\cE}{T}{\blank}$, we obtain the projective resolution
\[\begin{tikzcd}
0\arrow{r}&\Hom{\cC}{T}{KT'}\arrow{r}&\Hom{\cC}{T}{RT'}\arrow{r}&\bT\arrow{r}&0.
\end{tikzcd}\]
This shows that $\bT$ has projective dimension at most $1$, and hence that $\Ext{i}{A}{\bT}{\bT}=0$ for all $i>1$.
Applying $\Hom{A}{\blank}{\bT}$ to this resolution we obtain the sequence
\[
\begin{tikzcd}[column sep=8.9pt,ampersand replacement=\&,font=\footnotesize]
0\arrow{r}\&\Hom{A}{\bT}{\bT}\arrow{r}\&\Hom{A}{\Hom{\cC}{T}{RT'}}{\bT}\arrow{r}{\varphi}\&\Hom{A}{\Hom{\cC}{T}{KT'}}{\bT}\arrow{r}\&\Ext{1}{A}{\bT}{\bT}\arrow{r}\&0.
\end{tikzcd}\]
Here $\varphi$ is related by a Yoneda equivalence to the map $\Hom{\cC}{RT'}{T'}\to\Hom{\cC}{KT'}{T'}$ induced from the sequence \eqref{eq:index-seq2}, which is an epimorphism since $T'$ is rigid.
Thus $\varphi$ is also an epimorphism, and $\Ext{1}{A}{\bT}{\bT}=0$. In the same way, we conclude that $\Hom{A}{\bT}{\bT}=\Ker{\varphi}$ is isomorphic (as a vector space) to the kernel of the map $\Hom{\cC}{RT'}{T'}\to\Hom{\cC}{KT'}{T'}$, which is $\Hom{\cC}{T'}{T'}=A'$. One can check that the induced isomorphism $A'\to\Hom{A}{\bT}{\bT}$ is $f\mapsto(f\circ\blank)$, and so is also an isomorphism of algebras.

Finally, choosing a $T'$-coindex sequence $0\to T\to L'T\to C'T\to 0$
for $T$ and applying $\Hom{\cC}{T}{\blank}$ yields an exact sequence
\[\begin{tikzcd}
0\arrow{r}& A\arrow{r}&\Hom{\cC}{T}{L'T}\arrow{r}&\Hom{\cC}{T}{C'T}\arrow{r}&0
\end{tikzcd}\]
in which the middle and right-hand terms are in $\add{\bT}$, so $\bT$ is tilting.

Now let $U\in\add{T}$ and write $P=\Hom{\cC}{T}{U}$ for the corresponding projective $A$-module.
By a similar argument to that above involving $\varphi$, there is a map $\Hom{A}{\Hom{\cC}{T}{RT'}}{P}\to\Hom{A}{\Hom{\cC}{T}{KT'}}{P}$ which is related by the Yoneda isomorphism to $\Hom{\cC}{RT'}{U}\to\Hom{\cC}{KT'}{U}$.
There are thus induced isomorphisms between the kernels of these two maps, $\Hom{A}{\bT}{P}$ and $\Hom{\cC}{T'}{U}=\Yonfun{T'}U$, and between their cokernels, $\Ext{1}{A}{\bT}{P}$ and $\Ext{1}{\cC}{T'}{U}=\Extfun{T'}U$. Since $\bT$ has projective dimension at most $1$, we also have $\Ext{i}{A}{\bT}{P}=0$ for all $i>1$.

Thus, the isomorphism $\Kgp{\proj{A}}\to\Kgp{\proj{A'}}$ induced by $\RHom{A}{\bT}{\blank}$ takes $[P]=[\Yonfun{T}U]$ to $[\RHom{A}{\bT}{P}]=[\Yonfun{T'}U]-[\Extfun{T'}U]$.
By Propositions \ref{p:ind-proj-res} and \ref{p:beta-proj-res}, this corresponds under Yoneda to the isomorphism $\Kgp{\add{T}}\to\Kgp{\add{T'}}$ taking $[U]$ to
\[\ind{T}{T'}[U]-(\ind{T}{T'}[U]-\coind{T}{T'}[U])=\coind{T}{T'}[U].\qedhere\]
\end{proof}

\begin{corollary}
In the context of Theorem~\ref{t:coind-tilt}, the map $\ind{T'}{T}$ is induced from the equivalence $\bT\Ltens{A'}\blank\colon\per{A'}\to\per{A}$.
\end{corollary}
\begin{proof}
As $\bT\Ltens{A'}\blank$ is quasi-inverse to $\RHom{A}{\bT}{\blank}$ \cite[Thm.~3.3]{Rickard}, this follows from Proposition~\ref{p:ind-coind-inverse}.
\end{proof}

\sectionbreak
\section{Exchange matrices}
\label{s:exch-mat}

\subsection{The exchange matrix as a linear map}
\label{s:cat-ex-mat}

In this section we give the desired categorification of the exchange matrix, in a basis-free fashion, as a linear map between Grothendieck groups.
In order to match existing sign conventions, we define this as the negative of the map $p_{\cT}$.
With some extra assumptions---most notably that $\stab{\cT}$ is maximally mutable, so we may restrict $\beta_{\cT}$ to the subgroup $\Kgp{\fd{\stab{\cT}}}\leq\Kgp{\fpmod{\stab{\cT}}}$ (Proposition~\ref{p:fd-sub-fpmod})---we may decategorify to a matrix (Proposition~\ref{p:beta-exch-mat}).
The assumption that $\stab{\cT}$ is maximally mutable is needed frequently throughout this section, although it is mild: it holds whenever $\stab{\cC}$ (which is always Krull--Schmidt) has a weak cluster structure, in particular in the situations of Corollary~\ref{c:weak-clust-struct}.

\begin{definition}\label{d:beta} Let $\cC$ be a cluster category. For each $\cT\ctsubcat \cC$, define $\beta_{\cT}\colon \Kgp{\fpmod\stab{\cT}}\to \Kgp{\cT}$ to be the map $-p_{\cT}$ (see Definition~\ref{d:proj-res-map}).
\end{definition}

In particular, Proposition~\ref{p:beta-proj-res} tells us that for any $X\in\cC$, we have
\[\beta_{\cT}[\Extfun{\cT}{X}]=\coind{\cC}{\cT}[X]-\ind{\cC}{\cT}[X].\]

\begin{remark}\label{r:magic-lemma-for-beta}
We may now reinterpret Remark~\ref{r:magic-lemma}, itself derived from Lemma~\ref{l:ind-coind-adjointness}.
The first equality there implies, after restricting to $\cU\ctsubcat\cT$, that
\[  \adj{(\beta_{\cT}\Extfun{\cT}|_{\cU})}=-\Extfun{\cU}\beta_{\cT}\colon\Kgp{\fpmod\stab{\cT}}\to\Kgp{\fpmod\stab{\cT}}.\]
As usual, the adjoint is with respect to the forms $\canform{\blank}{\blank}{\cT}$ and $\canform{\blank}{\blank}{\cU}$.
\end{remark}

\begin{proposition}
\label{p:fd-sub-fpmod}
Let $\cC$ be a Krull--Schmidt cluster category, and assume $\cT\ctsubcat\cC$ is maximally mutable.
Then there is a natural injective map $\iota\colon\Kgp{\fd{\stab{\cT}}}\to\Kgp{\fpmod{\stab{\cT}}}$ and a further injection $\image{\iota}\hookrightarrow\Kgpnum{\lfd{\stab{\cT}}}$.
\end{proposition}
\begin{proof}
We have inclusions $\fd{\stab{\cT}}\subseteq\fpmod{\stab{\cT}}\subseteq\lfd{\stab{\cT}}$, the first by Corollary~\ref{c:fp=fd}, and the second since $\stab{\cC}$, and hence $\stab{\cT}$, is Hom-finite.
There are thus induced morphisms $\Kgp{\fd{\stab{\cT}}}\to\Kgp{\fpmod{\stab{\cT}}}\to\Kgpnum{\lfd{\stab{\cT}}}$, the composition of which is itself induced from the inclusion $\fd{\stab{\cT}}\subseteq\lfd{\stab{\cT}}$.
Our statement reduces to the injectivity of this inclusion, which holds because $\Kgp{\fd{\stab{\cT}}}$ is spanned by the classes of simple $\cT$-modules, whose images in $\Kgpnum{\lfd{\cT}}$ are linearly independent by Proposition~\ref{p:K-duality}.
\end{proof}

Below, we will treat the natural map $\iota$ from Proposition~\ref{p:fd-sub-fpmod} as an inclusion.

\begin{proposition}
\label{p:beta-exch-mat}
Let $\cC$ be a compact cluster category, and assume $\cT\ctsubcat \cC$ has no loops and is such that $\stab{\cT}$ is maximally mutable.
With respect to the natural bases $\{[\simpmod{\cT}{T}]\mid T\in\indec{\stab{\cT}}\}$ and $\{[T] \mid T\in\indec{\cT}\}$ of $\Kgp{\fd{\stab{\cT}}}$ and $\Kgp{\cT}$ respectively, the matrix of $\beta_{\cT}|_{\Kgp{\fd{\stab{\cT}}}}$ is the exchange matrix $\exchmat{\cT}$ (Definition~\ref{d:exch-mat}).
\end{proposition}

\begin{proof}
Let $T\in\indec{\cT}$.
Since $\cT$ has no loop at $T$, we have $\simpmod{\cT}{T}=\Extfun{\cT}(\mut{\cT}{T})$ by Lemma~\ref{l:props-of-K0-mod-T}\ref{l:props-mod-T-simple-eq-E}.
From Proposition~\ref{p:beta-proj-res}, equations~\eqref{eq:ind-on-mut-T} and \eqref{eq:coind-on-mut-T}, Proposition~\ref{p:decomp-exch-terms} and finally \eqref{eq:exch-mat-vs-Gab-mat}, we have
\begin{align*}
\beta_{\cT}[\simpmod{\cT}{T}] & = \beta_{\cT}[\Extfun{\cT}(\mut{\cT}{T})]
\\
& = \coind{\cC}{\cT}[\mut{\cT}{T}]-\ind{\cC}{\cT}[\mut{\cT}{T}] \\
& = [\exchmon{\cT}{T}{+}]-[\exchmon{\cT}{T}{-}] \\
& = \sum_{U\in \indec \cT} (\Gabmatentry{U,T}-\tfrac{\dimdivalg{T}}{\dimdivalg{U}}\Gabmatentry{T,U})[U] \\
& = \sum_{U\in \indec \cT} \exchmatentry{U,T}[U]. \qedhere
\end{align*}
\end{proof}

\begin{remark}
\label{r:beta-exch-mat-loops}
Whenever $T\in\exch{\stab{\cT}}$, the $\cT$-module $\Extfun{\cT}(\mut{\cT}{T})$ is non-zero, supported only on $T$, and of finite rank over $\divalg{T}$, and so $[\Extfun{\cT}(\mut{\cT}{T})]\in\integ_{>0}[\simpmod{\cT}{T}]$.
The core of the argument in the proof of Proposition~\ref{p:beta-exch-mat} thus also applies to cases in which $\cT$ does have loops, to yield the weaker result that the exchange matrix $B_{\cT}$ is obtained from the matrix of $\beta_{\cT}|_{\Kgp{\fd{\stab{\cT}}}}$ (in the given bases) by multiplying the columns by appropriate positive integers.
\end{remark}

In several places, especially in the quantum setting, the rank of the exchange matrix is relevant to considerations: Proposition~\ref{p:beta-exch-mat} shows (in the context in which it applies) that the exchange matrix having full rank is equivalent to the injectivity of $\beta_{\cT}|_{\Kgp{\fd{\stab{\cT}}}}$.
Below we will often have the assumption that $\stab{\cT}$ is maximally mutable, and will treat $\beta_{\cT}$ as a map on $\Kgp{\fd{\stab{\cT}}}$ without making this restriction explicit.

If $M\in\lfd{\stab{\cT}}$, then the dual module is $\dual{M}\in\lfd{\op{\stab{\cT}}}$ with $\dual{M}(T)=\dual{M(T)}$, and we may canonically identify $\Kgp{\cT}=\Kgp{\op{\cT}}$ as in Remark~\ref{r:ind-op}.

\begin{proposition}
\label{p:beta-op}
Let $M\in\fpmod{\stab{\cT}}\subseteq\lfd{\stab{\cT}}$. Then $\beta_{\op{\cT}}[\dual{M}]=-\beta_{\cT}[M]$.
\end{proposition}
\begin{proof}
Write $M=\Extfun{\cT}X$ for $X\in\cC$.
Because $\stab{\cC}$ is $2$-Calabi--Yau, we have $\dual{(\Extfun{\cT}X)}=\Extfun{\op{\cT}}X$ for any $X\in\cC$; in particular, we have $\dual{(\Extfun{\cT}X)}\in\fpmod{\op{\stab{\cT}}}$ by Proposition~\ref{p:equiv-to-mod}, and so we may apply $\beta_{\op{\cT}}$.
Doing so, we find that
\begin{align*}
\beta_{\op{\cT}}[\dual{(\Extfun{\cT}X)}]&=\coind{(\op{\cC})}{\op{\cT}}{[X]}-\ind{(\op{\cC})}{\op{\cT}}{[X]}\\
&=\ind{\cC}{\cT}{[X]}-\coind{\cC}{\cT}{[X]}\\
&=-\beta_{\cT}[\Extfun{\cT}X],
\end{align*}
recalling from Remark~\ref{r:ind-op} that indices and coindices swap in the opposite category.
\end{proof}

Using the homomorphism $\beta_{\cT}$ and the pairing $\canform{\blank}{\blank}{\cT}\colon\Kgpnum{\lfd\cT}\times\Kgp{\cT}\to\integ$ (Definition~\ref{d:canform}), we have an induced bilinear form as follows.

\begin{definition}\label{d:s-form} Let $\cC$ be a cluster category and $\cT\ctsubcat \cC$.  Define $\sform{\blank}{\blank}{\cT}\colon \Kgpnum{\lfd \cT} \cross \Kgp{\fpmod \stab{\cT}} \to \integ$ by $\sform{\blank}{\blank}{\cT}=\canform{\blank}{\beta_{\cT}(\blank)}{\cT}$.
\end{definition}

\begin{proposition}
\label{p:s-form-matrix}
Let $\cC$ be a cluster category and let $\cT\ctsubcat\cC$. Then for $U\in\indec{\cT}$ and $V\in\exch{\cT}$, we have 
\begin{equation}
\label{eq:s-gram-matrix}
\sform{[\simpmod{\cT}{U}]}{[\simpmod{\cT}{V}]}{\cT}=\dimdivalg{U}\exchmatentry{U,V}.
\end{equation}
\end{proposition}
\begin{proof}
We calculate
\begin{align*}
\sform{[\simpmod{\cT}{U}]}{[\simpmod{\cT}{V}]}{\cT} & = \canform{[\simpmod{\cT}{U}]}{\beta_{\cT}[\simpmod{\cT}{V}]}{\cT} \\
& =\bigcanform{[\simpmod{\cT}{U}]}{{\sum}_{W\in\indec{\cT}} \exchmatentry{W,V}[W]}{\cT} \\
& = \exchmatentry{U,V}\canform{[\simpmod{\cT}{U}]}{[U]}{\cT} \\
& = \dimdivalg{U}\exchmatentry{U,V}.\qedhere 
\end{align*}
\end{proof}
When $\cT$ is maximally mutable, so that $\Kgp{\fd{\stab{\cT}}}\leq\Kgp{\fpmod{\stab{\cT}}}$, the content of Proposition~\ref{p:s-form-matrix} is that the Gram matrix of the restricted form $\sform{\blank}{\blank}{\cT}\colon\Kgpnum{\lfd{\cT}}\times\Kgp{\fd{\stab{\cT}}}$ with respect to the classes of simple modules is equal to $D_{\cT}\exchmat{\cT}$, where $D_{\cT}$ is the diagonal matrix with diagonal entries $\dimdivalg{T}$ for $T\in\indec{\cT}$.
Indeed, since
\begin{equation}
\label{eq:beta-s-form}
\sform{[M]}{[N]}{\cT}=\canform{[M]}{\beta_{\cT}[N]}{\cT}=(\pdual{\cT}\circ\beta_{\cT}[N])[M],
\end{equation}
we see that $\sform{\blank}{\blank}{\cT}$ corresponds to the map $\pdual{\cT}\circ \beta_{\cT}\colon \Kgp{\fpmod \stab{\cT}}\to \dual{\Kgpnum{\lfd \cT}}$.

\begin{lemma}\label{l:stable-s-form} Let $\cC$ be a cluster category and $\cT\ctsubcat\cC$.
Then
$ \sform{\blank}{\blank}{\stab{\cT}}=\sform{\sinc{\cT}(\blank)}{\blank}{\cT}$.
\end{lemma}

\begin{proof} 
By construction (or by Propositions~\ref{p:stabind} and \ref{p:beta-proj-res}), $\beta_{\stab{\cT}}=\pproj{\cT}\circ\beta_{\cT}$, so~\eqref{eq:pi-p-iota-s-adjoint} yields
\[
\sform{\blank}{\blank}{\stab{\cT}}=\canform{\blank}{\beta_{\stab{\cT}}(\blank)}{\stab{\cT}}
=\canform{\blank}{\pproj{\cT}\circ \beta_{\cT}(\blank)}{\stab{\cT}}
=\canform{\sinc{\cT}(\blank)}{\beta_{\cT}(\blank)}{\cT}
=\sform{\sinc{\cT}(\blank)}{\blank}{\cT}.\qedhere
\]
\end{proof}

\begin{lemma}
\label{l:s-form-skew-sym}
When $\cC$ is a triangulated cluster category and $\cT\ctsubcat \cC$ is maximally mutable, the restricted form $\sform{\blank}{\blank}{\cT}\colon\Kgp{\fd{\cT}}\times\Kgp{\fd{\cT}}$ is skew-symmetric.
\end{lemma}

\begin{proof}
Let $M,N\in\fd{\cT}\subseteq\fpmod{\cT}$, and choose $X,Y\in\cC$ with $M=\Extfun{\cT}{X}$ and $N=\Extfun{\cT}{Y}$.
Then by Proposition~\ref{p:beta-proj-res}, we have
\[\sform{[M]}{[N]}{\cT}=\canform{[M]}{\beta_{\cT}[N]}{\cT}=\canform{[M]}{\coind{\cC}{\cT}[Y]-\ind{\cC}{\cT}[Y]}{\cT}.\]
Extending the shorthand in Remark~\ref{r:magic-lemma}, this form evaluates to
\begin{align*}
\dim M(\coind{\cC}{\cT}{[Y]}-\ind{\cC}{\cT}{[Y]})&=\ext{1}{\cC}{\coind{\cC}{\cT}{[Y]}-\ind{\cC}{\cT}{[Y]}}{[X]}\\
&=\ext{1}{\cC}{[X]}{\coind{\cC}{\cT}{[Y]}-\ind{\cC}{\cT}{[Y]}}\\
&=\ext{1}{\cC}{\ind{\cC}{\cT}{[X]}-\coind{\cC}{\cT}{[X]}}{[Y]}\\
&=-\sform{[N]}{[M]}{\cT}
\end{align*}
as required, where the third equality uses Lemma~\ref{l:ind-coind-adjointness}.
\end{proof}

We say that a map $\varphi\colon V\to\dual{V}$, for $V$ a free $\integ$-module, is \emph{skew-symmetric} if $\adj{\varphi}=-\varphi$, where the adjoint is with respect to the evaluation form $\evform{\blank}{\blank}$.
This is equivalent to $\varphi$ being represented by a skew-symmetric matrix with respect to a pair of dual bases of $V$ and $\dual{V}$, and to skew-symmetry of the form on $V$ defined by $\varphi$.

\begin{corollary}
\label{c:beta-skew-symmetric}
Let $\cC$ be a cluster category and let $\cT\ctsubcat\cC$ be maximally mutable. Then $\pdual{\stab{\cT}}\circ\beta_{\stab{\cT}}\colon\Kgp{\fd{\stab{\cT}}}\to\dual{\Kgp{\fd{\stab{\cT}}}}$ is skew-symmetric. \qed
\end{corollary}

\begin{remark}
\label{r:s-form-rewrite}
When $\cC$ is an exact cluster category, so $p_{\cT}[N]=-\beta_{\cT}[N]$ is (under Yoneda) the class of a projective resolution of $N$ as a $\cT$-module, it follows when $\stab{\cT}$ is maximally mutable that
\begin{equation}
\label{eq:s-form-rewrite}
-\sform{[M]}{[N]}{\cT}=\sum_{i=0}^3(-1)^i\ext{i}{\cT}{N}{M}
\end{equation}
is the Euler pairing of the $\cT$-modules $N$ and $M$ (noting that $N\in\fpmod{\stab{\cT}}$ has projective dimension at most $3$ by Remark~\ref{r:pdim3}).

If we further assume $N\in\fd\stab{\cT}$ and $M\in\fpmod{\cT}$, we may then use the relative Calabi--Yau property of $\fpmod{\cT}$ \cite[Prop.~4(c)]{KellerReiten} to write
\begin{multline}
\label{eq:s-form-rewrite2}
\sform{[M]}{[N]}{\cT}=-\hom{\cT}{M}{N}+\ext{1}{\cT}{M}{N}\\-\ext{1}{\cT}{N}{M}+\hom{\cT}{N}{M}.
\end{multline}
If $\cP$ is a full additive subcategory of projective objects in $\cC$, the category of modules for $\cT/\cP\subseteq\cC/\cP$ may be viewed as the full extension-closed subcategory of $\cT$-modules which vanish on $\cP$.
This means that for $M,N\in\Mod{\cT/\cP}$ we have
\[
\Hom{\cT}{M}{N}=\Hom{\cT/\cP}{M}{N},\quad
\Ext{1}{\cT}{M}{N}=\Ext{1}{\cT/\cP}{M}{N},
\]
and so \eqref{eq:s-form-rewrite2} actually holds in any algebraic cluster category.
This is not the case for \eqref{eq:s-form-rewrite}, since we generally do not have $\Ext{i}{\cT}{M}{N}=\Ext{i}{\cT/\cP}{M}{N}$ for $i\geq2$ (cf.~\cite{APT}).
\end{remark}

Recall from Corollary~\ref{c:coind-minus-ind-in-ker-pi-T} that $\beta_{\cT}[\Extfun{\cT}{X}]=\coind{\cC}{\cT}[X]-\ind{\cC}{\cT}[X]\in \ker \pi_{\cT}^{\cC}$ for all $X$, where $\pi_{\cT}^{\cC}[T]_{\cT}=[T]_{\cC}$ as in \eqref{eq:proj-T-C}.
The following result, generalising Palu \cite[Thm.~10]{Palu-Groth-gp} for triangulated categories and the authors \cite[Thm.~3.12]{GradedFrobeniusClusterCategories} for exact categories, strengthens this by showing that these elements in fact generate $\ker{\pi}_{\cT}^{\cC}$.

\begin{theorem}
\label{t:ker-pi-T}
Let $\cC$ be a Krull--Schmidt cluster category, and let $\cT\ctsubcat\cC$.
Then
\[\begin{tikzcd}
\Kgp{\fpmod{\stab{\cT}}}\arrow{r}{\beta_{\cT}}&\Kgp{\cT}\arrow{r}{\pi_{\cT}^{\cC}}&\Kgp{\cC}\arrow{r}&0
\end{tikzcd}\]
is an exact sequence.
\end{theorem}

\begin{proof}
Let $\cE$ be an exact cluster category, let $\cP$ be a full and additively closed subcategory of projective objects in $\cE$ such that $\cE/\cP\simeq \cC$, and let $\widehat{\cT}\ctsubcat\cE$ be the lift of $\cT$.
In particular, $\stab{\widehat{\cT}}=\stab{\cT}$.
Consider the commutative diagram
\begin{equation}
\label{eq:partial-stab-diag}
\begin{tikzcd}[row sep=15pt]
&\Kgp{\cP}\arrow[equal]{r}\arrow{d}&\Kgp{\cP}\arrow{d}\\
\Kgp{\fpmod{\stab{\widehat{\cT}}}}\arrow{r}{\beta_{\widehat{\cT}}}\arrow[equal]{d}&\Kgp{\widehat{\cT}}\arrow{r}{\pi_{\widehat{\cT}}^{\cE}}\arrow{d}&\Kgp{\cE}\arrow{r}\arrow{d}&0\\
\Kgp{\fpmod{\stab{\cT}}}\arrow{r}{\beta_{\cT}}\arrow{r}&\Kgp{\cT}\arrow{r}{\pi_{\cT}^{\cC}}\arrow{d}&\Kgp{\cC}\arrow{r}\arrow{d}&0\\
&0&0
\end{tikzcd}
\end{equation}
in which the vertical arrows (and the horizontal ones labelled by some decoration of $\pi$) are the natural maps taking the class of an object in one category to its class in a second category to which it also belongs.

To see that the columns of \eqref{eq:partial-stab-diag} are exact, first note that the Grothendieck groups $\Kgp{\cP}$ and $\Kgp{\widehat{\cT}}$ of (split) exact categories identify with those of the bounded homotopy categories $\bhcat{\cP}$ and $\bhcat{\widehat{\cT}}$ respectively.
Hence, there is an exact sequence
\[\begin{tikzcd}
\Kgp{\cP}\arrow{r}&\Kgp{\widehat{\cT}}\arrow{r}&\Kgp{\bhcat{\widehat{\cT}}/\bhcat{\cP}}\arrow{r}&0,
\end{tikzcd}\]
which identifies with the middle column of \eqref{eq:partial-stab-diag} as in the proof of \cite[Lem.~9]{Palu-Groth-gp}.
Similarly, $\Kgp{\cE}$ identifies with the Grothendieck group $\Kgp{\bdcat{\cE}}$ of the bounded derived category of $\cE$, so that there is an exact sequence
\[\begin{tikzcd}
\Kgp{\cP}\arrow{r}&\Kgp{\cE}\arrow{r}&\Kgp{\bdcat{\cE}/\bhcat{\cP}}\arrow{r}&0.
\end{tikzcd}\]
But by \cite[Prop.~3.5, Thm.~3.23]{Chen-exact-dg-2}, the Grothendieck group $\Kgp{\bdcat{\cE}/\bhcat{\cP}}$ is naturally isomorphic to that of the algebraic extriangulated category $\cC\simeq\cE/\cP$ (which has a connective dg enhancement by \cite[Rem.~4.25]{ChenThesis}), yielding exactness of the right-hand column in \eqref{eq:partial-stab-diag}.

Now the middle row of \eqref{eq:partial-stab-diag} is exact by \cite[Lem.~2]{Palu-Groth-gp} (see also \cite[Proof of Thm.~3.12]{GradedFrobeniusClusterCategories}), and so a diagram chase shows that the lower row is also exact, as required.
\end{proof}

\begin{remark}
For each $T\in\exch{\cT}$, the cup product
\begin{equation}
\label{eq:2-ar-seq}
\begin{tikzcd}
T \arrow[infl]{r}& \exchmon{\cT}{T}{-} \arrow{r}& \exchmon{\cT}{T}{+} \arrow[defl]{r}& T
\end{tikzcd}
\end{equation}
of the two exchange conflations for $T$ lies entirely in $\cT$, and
\begin{equation}
\label{eq:mutrel} [T]-[\exchmon{\cT}{T}{-}]+[\exchmon{\cT}{T}{+}]-[T]=[\exchmon{\cT}{T}{+}]-[\exchmon{\cT}{T}{-}]=0
\end{equation}
is a relation in $\Kgp{\cC}$. For each simple $\stab{\cT}$-module $\simpmod{\cT}{T}=\Extfun{\cT}(\mut{\cT}{T})$, with $T\in\indec{\stab{\cT}}$, we calculate using Proposition~\ref{p:beta-proj-res}, together with \eqref{eq:ind-on-mut-T} and \eqref{eq:coind-on-mut-T}, that
\begin{equation}\label{eq:beta-ST-equals-R-minus-L}  \beta_{\cT}[\simpmod{\cT}{T}]=\coind{\cC}{\cT}[\mut{\cT}{T}]-\ind{\cC}{\cT}[\mut{\cT}{T}]=[\exchmon{\cT}{T}{-}]-[\exchmon{\cT}{T}{+}].
\end{equation}
If $\fd{\stab{\cT}}=\fpmod{\stab{\cT}}$ (e.g.\ if $\stab{\cT}$ is additively finite and maximally mutable), then the classes $[\simpmod{\cT}{T}]$, for $T\in\indec{\stab{\cT}}$, generate $\Kgp{\fpmod{\stab{\cT}}}$.
It then follows from Theorem~\ref{t:ker-pi-T} that the relations \eqref{eq:mutrel}, coming from exchange conflations, generate $\ker \pi_{\cT}^{\cC}$, cf.\ \cite[Thm.~4.4]{Murphy} (a minor correction to \cite[Thm.~10]{Palu}) and \cite[Thm.~3.12(a)]{GradedFrobeniusClusterCategories}.

\end{remark}

In some cases, we are also able to identify the kernel of $\beta_{\cT}$.

\begin{proposition}\label{p:ker-beta} Let $\cC$ be a cluster category and $\cT\ctsubcat\cC$.
\begin{enumerate}
\item\label{p:ker-beta-exact} If $\cC$ is exact and Hom-finite, then $\beta_{\cT}$ is injective.
\item\label{p:ker-beta-triang} If $\cC$ is triangulated and $\cT$ is additively finite, then $\ker{\beta_{\cT}}\tensor_{\integ} \bK$ is isomorphic to $\dual{\Kgp{\cC}}\tensor_{\integ} \bK$.
If $\cC$ is also skew-symmetric, then $\ker{\beta_{\cT}}\iso\dual{\Kgp{\cC}}$.
\end{enumerate}
\end{proposition}

\begin{proof}\mbox{}
\begin{enumerate}
\item See Corollary~\ref{c:Hom-finite-exact-full-rank} below, or \cite[Rem.~4.5]{FuKeller} (whose assumption that $\gldim{\cT}<\infty$ is a consequence of our assumptions on $\cC$).
\item Abbreviating $\beta=\beta_{\cT}$ and $\pi=\pi_{\cT}^{\cC}$, by Theorem~\ref{t:ker-pi-T} we have an exact sequence
\begin{equation}\label{eq:exact-seq-with-K} \begin{tikzcd}
0 \arrow{r} & \ker{\beta_{\cT}} \arrow{r}{i} & \Kgp{\fd{\cT}} \arrow{r}{\beta} & \Kgp{\cT} \arrow{r}{\pi} & \Kgp{\cC} \arrow{r}&0.
\end{tikzcd}\end{equation}
Take the dual sequence
\begin{equation}\label{eq:exact-seq-dual} \begin{tikzcd}[column sep=23pt]
0 \arrow{r} & \dual{\Kgp{\cC}} \arrow{r}{\dual{\pi}} & \dual{\Kgp{\cT}} \arrow{r}{\dual{\beta}} & \dual{\Kgp{\fd \cT}} \arrow{r}{\dual{i}} & \dual{(\ker{\beta_{\cT}})} \arrow{r}&0
\end{tikzcd}\end{equation}
given by applying $\Hom{\integ}{\blank}{\integ}$ to \eqref{eq:exact-seq-with-K}.
This dual sequence need not be exact at $\dual{\Kgp{\fd{\cT}}}$, since $\Kgp{\cC}$ need not be free, but it is exact elsewhere since the other Grothendieck groups appearing in \eqref{eq:exact-seq-with-K}, as well as the image of $\beta$, are free.
This uses the fact that $\cC$ is a triangulated cluster category, and hence Krull--Schmidt.
In particular, $\dual{\pi}$ is the kernel of $\dual{\beta}$.

We now form a commutative diagram
\begin{equation}
\label{eq:duality-diag}
\begin{tikzcd}
0\arrow{r}&\ker{\beta_{\cT}}\arrow{r}{-i}\arrow{d}{\kappa}&\Kgp{\fd{\cT}}\arrow{r}{-\beta}\arrow{d}{\sdual{\cT}}&\Kgp{\cT}\arrow{d}{\pdual{\cT}}\\
0 \arrow{r} & \dual{\Kgp{\cC}} \arrow{r}{\dual{\pi}} & \dual{\Kgp{\cT}} \arrow{r}{\dual{\beta}} & \dual{\Kgp{\fd \cT}}.
\end{tikzcd}
\end{equation}
Indeed, the right-hand square commutes by Corollary~\ref{c:beta-skew-symmetric}.
Since $\dual{\pi}$ is the kernel of $\dual{\beta}$, there is an induced map $\kappa$ such that the left-hand square commutes.

Recall from Proposition~\ref{p:K-duality} that $\sdual{\cT}$ is injective, and hence so is $\kappa$.
Because $\ker{\beta_{\cT}}$ and $\dual{\Kgp{\cC}}$ have the same rank by \eqref{eq:exact-seq-with-K}, it follows that $\ker{\beta_{\cT}}\tensor_{\integ} \bK\iso \dual{\Kgp{\cC}}\tensor_{\integ} \bK$.
Since $\cT$ is additively finite, $\sdual{\cT}$ and $\pdual{\cT}$ are isomorphisms when $\cC$ is skew-symmetric, and hence so is $\kappa$.\qedhere
\end{enumerate}
\end{proof}

\begin{remark}
It is not true for a general cluster category $\cC$ that $\ker \beta_{\cT}\iso \dual{\Kgp{\stab{\cC}}}$: a counterexample is provided by the module category of the preprojective algebra of type $\mathsf{A}_2$, which is an exact cluster category.
In this example, $\beta_{\cT}$ is injective by Proposition~\ref{p:ker-beta}\ref{p:ker-beta-exact}, but $\rank\Kgp{\stab{\cC}}=1$.
\end{remark}

The following proposition translates a result of Melo--Nájera Chávez \cite[Cor.~3.4]{MNC} into our language; we also give a brief proof to demonstrate the use of adjunction to deduce results such as this.
This result uses the sets $\gvecplus{\cU}{\cT}$ and $\cvecplus{\cU}{\cT}$, as well as the cones $G_{\cT}(\cU)=\real_{+}\gvecplus{\cT}{\cU}\subseteq\Kgp{\cT}\tensor_\integ \real$ and $C_{\cT}(\cU)=\real_+\cvecplus{\cT}{\cU}\subseteq\Kgp{\fd \cT}\tensor_\integ \real$, from Section~\ref{s:sign-coherence}.
We also use the set
\[ C_{\cT}(U)^{\circ}=\real_{+}\{ m \in \Kgp{\fd \cT} \mid \canform{m}{\adj{\beta}_{\cT}C_{\cT}(\cU)}{\cT}\geq 0 \} \]
from Remark~\ref{r:scattering}.

\begin{proposition}
\label{p:beta-maps-cones}
Let $\cC$ be a triangulated cluster category with finite rank and let $\cTU \ctsubcat \cC$.
Then for $\beta_{\cT}^{\real}=\beta_{\cT}\tensor\real\colon \Kgp{\fd \cT}\tensor \real \to \Kgp{\cT}\tensor \real$, we have
\[ (\beta_{\cT}^{\real})^{-1}G_{\cT}(\cU)=C_{\cT}(\cU)^{\circ}. \]
\end{proposition}

\begin{proof}
Let $w\in (\beta_{\cT}^{\real})^{-1}G_{\cT}(\cU)$.
By $\real_+$-linearity, it suffices to consider the case that $\beta_{\cT}^{\real}w=\ind{\cU}{\cT}[V]$ for some $V\in \indec \cU$.
Then, for $U\in \indec \cU$, we have
\begin{equation}
\label{eq:ind-indbar-duality}
\canform{w}{\adj{\beta}_{\cT}\indbar{\cU}{\cT}[\simpmod{\cU}{U}]}{\cT}=\canform{\indbar{\cU}{\cT}[\simpmod{\cU}{U}]}{\beta_{\cT}w}{\cT}=\canform{\indbar{\cU}{\cT}[\simpmod{\cU}{U}]}{\ind{\cU}{\cT}[V]}{\cT}=\dimdivalg{U}\delta_{U,V}\geq0
\end{equation}
and hence $w\in C_{\cT}(\cU)^{\circ}$.
Conversely, for $w\in C_{\cT}(\cU)^{\circ}$ and $U\in \indec \cU$,
\begin{equation}
\label{eq:indbar-beta}
\canform{\indbar{\cU}{\cT}[\simpmod{\cU}{U}]}{\beta_{\cT}w}{\cT} =\canform{w}{\adj{\beta}_{\cT}\indbar{\cU}{\cT}[\simpmod{\cU}{U}]}{\cT}\geq0.
\end{equation}
Recall from Proposition~\ref{p:ind-coind-inverse} that $\{ \ind{\cU}{\cT}[U] \mid U\in \indec \cU \}$ is a basis for $\Kgp{\cT}\tensor_{\integ} \real$.
By \eqref{eq:ind-indbar-duality}, the left-hand side of \eqref{eq:indbar-beta} is a positive multiple of the coefficient of $\ind{\cU}{\cT}{[U]}$ in an expression for $\beta_{\cT}w$.
Thus, this coefficient is non-negative, and $\beta_{\cT}w\in G_{\cT}(\cU)$.
\end{proof}

To conclude this subsection, we return to the question of the additivity (or otherwise) of the index and coindex on conflations in $\cC$, and see that $\beta_{\cT}$ can be used to measure this.

\begin{proposition}
\label{p:beta-add-confl}
Let $\cC$ be a cluster category with cluster-tilting subcategory $\cT$.
For any conflation $X\stackrel{i}{\infl} Y\stackrel{p}{\defl} Z \confl$ in $\cC$, applying the restricted Yoneda functors $\Yonfun{\cT}$ and $\opYonfun{\cT}$ yields exact sequences
\[\begin{tikzcd}[row sep=5pt]
\Yonfun{\cT}X\arrow{r}&\Yonfun{\cT}Y\arrow{r}&\Yonfun{\cT}Z\arrow{r}&M\arrow{r}&0\\
\opYonfun{\cT}Z\arrow{r}&\opYonfun{\cT}Y\arrow{r}&\opYonfun{\cT}X\arrow{r}&N\arrow{r}&0
\end{tikzcd}\]
of $\cT$ and $\op{\cT}$-modules respectively, by defining $M=\coker{\Yonfun{\cT}}p$ and $N=\coker{\opYonfun{\cT}}i$.
Then
\begin{align*}
\ind{\cC}{\cT}[X]+\ind{\cC}{\cT}[Z]&=\ind{\cC}{\cT}[Y]-\beta_{\cT}[M],\\
\coind{\cC}{\cT}[X]+\coind{\cC}{\cT}[Z]&=\coind{\cC}{\cT}[Y]+\beta_{\cT}[\dual{N}].\\
\end{align*}
\end{proposition}

\begin{proof}
We first observe that $M$ is a submodule of $\Extfun{\cT}{X}$, while $\dual{N}$ is a quotient of $\Extfun{\cT}{Z}$, and so these modules lie in $\fd{\stab{\cT}}$, as is necessary to apply $\beta_{\cT}$ to their classes.
The claimed formulæ thus make sense, and, in the now familiar way, it suffices to prove them in the case that $\cC$ is exact, with the general case then following by partial stabilisation.

When $\cC$ is exact, we even have an exact sequence
\begin{equation}
\label{eq:Yonfun-conflation}
\begin{tikzcd}
0\arrow{r}&\Yonfun{\cT}X\arrow{r}&\Yonfun{\cT}Y\arrow{r}&\Yonfun{\cT}Z\arrow{r}&M\arrow{r}&0.
\end{tikzcd}
\end{equation}
Recall that $-\beta_{\cT}[M]$ corresponds to the class of a projective resolution of $M$ under the natural isomorphism $\Kgp{\cT}=\Kgp{\proj{\cT}}$.
Moreover, by Proposition~\ref{p:ind-proj-res}, the index $\ind{\cC}{\cT}[X]$ of any $X\in\cC$ corresponds under this isomorphism to the class of any projective resolution of the $\cT$-module $\Yonfun{\cT}{X}$.
Thus, applying the horseshoe lemma to \eqref{eq:Yonfun-conflation}, we have
\[
\ind{\cC}{\cT}[X]-\ind{\cC}{\cT}[Y]+\ind{\cC}{\cT}[Z]-(-\beta_{\cT}[M])=0
\]
and thus the first identity.

We may now deduce the second identity by applying the first to $\op{\cC}$.
So doing, we find that
\[
\ind{\op{\cC}}{\op{\cT}}[X]+\ind{\op{\cC}}{\op{\cT}}[Z]=\ind{\op{\cC}}{\op{\cT}}[Y]-\beta_{\op{\cT}}[N]=0.
\]
Now $\ind{\op{\cC}}{\op{\cT}}=\coind{\cC}{\cT}$ by Remark~\ref{r:ind-op}, and $\beta_{\op{\cT}}[N]=-\beta_{\cT}[\dual{N}]$ by Proposition~\ref{p:beta-op}, which gives the result.
\end{proof}

\begin{corollary}[cf.~{\cite[Prop.~2.2]{Palu}}]
\label{p:ind-coind-additive-if-map-zero}
Let $\cC$ be a cluster category, $\cT \ctsubcat \cC$ and $X\stackrel{i}{\infl} Y\stackrel{p}{\defl} Z \confl$ a conflation in $\cC$.
\begin{enumerate}
\item If $\Extfun{\cT}i=\Ext{1}{\cC}{\blank}{i}|_{\cT}$ is injective then the index is additive on this conflation, i.e.
\[ \ind{\cC}{\cT}[X]+\ind{\cC}{\cT}[Z]=\ind{\cC}{\cT}[Y].\]
\item If $\Ext{1}{\cC}{p}{\blank}|_{\cT}$ is injective then the coindex is additive on this conflation, i.e.
\[ \coind{\cC}{\cT}[X]+\coind{\cC}{\cT}[Z]=\coind{\cC}{\cT}[Y]. \] 
\end{enumerate}
\end{corollary}
\begin{proof}
This follows from Proposition~\ref{p:beta-add-confl}, since $M=\Ker\Extfun{\cT}{i}$ and $N=\Ker\Ext{1}{\cC}{p}{\blank}|_{\cT}$ by the long-exact sequence of extension groups.
\end{proof}

\subsection{Compositions of indices and coindices}
\label{ss:compositions}

We continue to build up a calculus for the index and coindex maps and their adjoints. 
Given three cluster-tilting subcategories $\cTU$ and $\cV$ of a cluster category $\cC$, the various possible compositions of index and coindex provide four maps $\Kgp{\cV}\to\Kgp{\cT}$ factoring over $\Kgp{\cU}$.
The goal of this subsection is to compare these compositions to the direct maps $\ind{\cV}{\cT},\coind{\cV}{\cT}\colon\Kgp{\cV}\to\Kgp{\cT}$.
This will in particular allow us to fully justify our claim from Section~\ref{s:mutation} that these maps are the counterparts to $\mathbf{g}$-vectors and $\mathbf{c}$-vectors.

In what follows, it may be helpful to think of $\cT$ as a `root', i.e.\ a fixed initial cluster-tilting subcategory, and $\cU$, $\cV$ being two other cluster-tilting subcategories away from the root.
The case where $\cU$ and $\cV$ are related by a single mutation will be of particular interest, but we do not assume this (or even that $\cU$ and $\cV$ are related by a longer sequence of mutations) in general.

Given $X\in\cC$ and $\cU\ctsubcat\cC$, we may choose $\cU$-coindex and $\cU$-index conflations
\begin{equation}
\label{eq:ind-coind-for-error}
\begin{tikzcd}[row sep=0pt]
X\arrow[infl]{r}{i_{L}^{X}}&\leftapp{\cU}{X}\arrow[defl]{r}{p_{L}^{X}}&\leftcok{\cU}{X}\arrow[confl]{r}&,\end{tikzcd}\quad\begin{tikzcd}
\rightker{\cU}{X}\arrow[infl]{r}{i_{R}^{X}}&\rightapp{\cU}{X}\arrow[defl]{r}{p_{R}^{X}}&X\arrow[confl]{r}&\phantom{}
\end{tikzcd}
\end{equation}
for $X$, and consider the resulting exact sequences
\begin{equation}
\label{eq:lr-seqs}
\begin{tikzcd}[column sep=40pt, row sep=0pt]
\Ext{1}{\cC}{\blank}{X}\arrow{r}{\Ext{1}{\cC}{\blank}{i_{L}^{X}}}&\Ext{1}{\cC}{\blank}{\leftapp{\cU}{X}}\arrow{r}{\Ext{1}{\cC}{\blank}{p_{L}^{X}}}&\Ext{1}{\cC}{\blank}{\leftcok{\cU}{X}}, \\
\Ext{1}{\cC}{\blank}{\rightker{\cU}{X}}\arrow{r}{\Ext{1}{\cC}{\blank}{i_{R}^{X}}}&\Ext{1}{\cC}{\blank}{\rightapp{\cU}{X}}\arrow{r}{\Ext{1}{\cC}{\blank}{p_{R}^{X}}}&\Ext{1}{\cC}{\blank}{X}, 
\end{tikzcd}
\end{equation}
of functors. 
We obtain four $\cC$-modules, defined by
\begin{equation}
\label{eq:errors}
\begin{aligned}
\lefterror{1}{\cU}{X} & =\Ker{\Ext{1}{\cC}{\blank}{i_{L}^{X}}}, &\quad
\lefterror{2}{\cU}{X} & =\Coker{\Ext{1}{\cC}{\blank}{p_{L}^{X}}},\\
\righterror{1}{\cU}{X} & =\Ker{\Ext{1}{\cC}{\blank}{i_{R}^{X}}}, &\quad
\righterror{2}{\cU}{X} & =\Coker{\Ext{1}{\cC}{\blank}{p_{R}^{X}}}.
\end{aligned}
\end{equation}
These functors depend only on $X$ and $\cU$, and not on the choice of conflations \eqref{eq:ind-coind-for-error}.
Since $\Ext{1}{\cC}{P}{\blank}=0$ for any projective-injective $P$, each descends naturally to a $\stab{\cC}$-module. 
Moreover, since $\Ext{1}{\cC}{\blank}{P}=0$ for any projective-injective $P$, these $\stab{\cT}$-modules can be computed from conflations \eqref{eq:ind-coind-for-error} taken in any partial stabilisation $\cC/\cP$.
In particular, each of $\lefterror{i}{\cU}{X}$ and $\righterror{i}{\cU}{X}$ depends only on the class of $X$ in the stable category $\stab{\cU}$, and can be computed from $\cU$-coindex and $\cU$-index triangles in this stable category.

Restricting the $\stab{\cC}$-modules \eqref{eq:errors} to $\cT\ctsubcat\cC$ gives four $\stab{\cT}$-modules.
These restrictions are finitely presented since $\fpmod{\stab{\cT}}$ is abelian \cite[Prop.~2.1(a)]{KellerReiten} and $\Extfun{\cT}{Y}$ is finitely presented for all $Y\in\cC$ (Proposition~\ref{p:equiv-to-mod}), and they are also locally finite-dimensional since $\stab{\cC}$ is Hom-finite.
In some cases, they are even finite-dimensional.

\begin{proposition}
\label{p:error-term-fd}
Let $\cC$ be a cluster category and let $\cTU,\cV\ctsubcat\cC$.
\begin{enumerate}
\item\label{p:error-term-fd-add-finite} If $\stab{\cC}$ has finite rank, then $\lefterror{i}{\cU}X|_{\cT},\righterror{i}{\cU}X|_{\cT}\in\fd{\stab{\cT}}$ for all $X\in\cC$.
\item\label{p:error-term-fd-V-reachable} If $\cV$ is reachable from $\cT$, then $\lefterror{1}{\cU}V|_{\cT},\righterror{2}{\cU}V|_{\cT}\in\fd{\stab{\cT}}$ for all $V\in\cV$.
\item\label{p:error-term-fd-U-reachable} If $\cU$ is reachable from $\cT$, then $\lefterror{2}{\cU}X|_{\cT},\righterror{1}{\cU}X|_{\cT}\in\fd{\stab{\cT}}$ for all $X\in\cC$.
\end{enumerate}
\end{proposition}

\begin{proof}
In all cases, we use the fact that $\lefterror{i}{\cU}X,\righterror{i}{\cU}X\in\lfd{\stab{\cC}}$ since $\stab{\cC}$ is Hom-finite.
This also implies that $\stab{\cT}$ is Krull--Schmidt, and so to check that these functors restrict to $\fd{\stab{\cT}}$ it is enough to check that they are supported on finitely many objects of $\indec{\stab{\cT}}$.
In particular, \ref{p:error-term-fd-add-finite} is now immediate, since in this case $\indec{\stab{\cT}}$ is a finite set.
\begin{enumerate}
\setcounter{enumi}{1}
\item If $\cV$ is reachable from $\cT$, then $\cT\setminus\cV$ is additively finite.
Since $\Extfun{\cT}{V}\in\lfd{\stab{\cT}}$ is supported on $\cT\setminus\cV$, it is finite-dimensional.
It follows that $\lefterror{1}{\cU}V|_{\cT},\righterror{2}{\cU}V|_{\cT}\in\fd{\stab{\cT}}$, being a submodule and quotient module respectively of $\Extfun{\cT}{V}$.
\item As for \ref{p:error-term-fd-V-reachable}, $\Extfun{\cT}\leftcok{\cU}{V}\in\fd{\stab{\cT}}$ because $\cT\setminus\cU$ is additively finite, and hence so is the quotient module $\lefterror{2}{\cU}X$.
Similarly, $\righterror{1}{\cU}{X}$ is a submodule of $\Extfun{\cT}{\rightker{\cU}{X}}\in\fd{\stab{\cT}}$.\qedhere
\end{enumerate}
\end{proof}

\begin{remark}\label{r:add-error-vanish} If $X\in \cU$, then the conflations \eqref{eq:ind-coind-for-error} split and so $\lefterror{i}{\cU}{X}=0=\righterror{i}{\cU}{X}$.
\end{remark}

We may also consider the $\op{\stab{\cC}}$-modules $\lefterror{i}{\op{\cU}}{X}$ and $\righterror{i}{\op{\cU}}{X}$, computed in the cluster category $\op{\cC}$ with respect to the cluster-tilting subcategory $\op{\cU}$.

\begin{proposition}
\label{p:error-op}
There are $\stab{\cC}$-module isomorphisms
\[\begin{aligned}
\lefterror{1}{\cU}{X}&\iso\dual{(\righterror{2}{\op{\cU}}{X})},&\quad
\lefterror{2}{\cU}{X}\iso\dual{(\righterror{1}{\op{\cU}}{X})},\\
\righterror{1}{\cU}{X}&\iso\dual{(\lefterror{2}{\op{\cU}}{X})},&\quad
\righterror{2}{\cU}{X}\iso\dual{(\lefterror{1}{\op{\cU}}{X})}.
\end{aligned}\]
\end{proposition}

\begin{proof}
We establish the first two isomorphisms; the others then follow by swapping the roles of $\cU$ and $\op{\cU}$ and using Hom-finiteness of $\stab{\cC}$ to remove double duals.
Fix $X$, and write $i_{L}$ for $i_{L}^{X}$, etc.
Because a $\cU$-coindex conflation for $X$ in $\cC$
becomes a $\op{\cU}$-index conflation for $X$ in $\op{\cC}$, we have
\begin{align}
\begin{split}
\label{eq:op-errors}
\righterror{1}{\op{\cU}}{X}&=\Ker{\Ext{1}{\op{\cC}}{\blank}{\op{p_L}}}=\Ker{\Ext{1}{\cC}{p_L}{\blank}},\\ \righterror{2}{\op{\cU}}{X}&=\Coker{\Ext{1}{\op{\cC}}{\blank}{\op{i_L}}}=\Coker{\Ext{1}{\cC}{i_L}{\blank}}.
\end{split}
\end{align}
Since $\stab{\cC}$ is $2$-Calabi--Yau, we have
\[\begin{tikzcd}[column sep=45pt,row sep=10pt]
\dual{\Ext{1}{\cC}{X}{\blank}}\arrow{r}{\dual{\Ext{1}{\cC}{i_{L}}{\blank}}}\arrow[equal]{d}&\dual{\Ext{1}{\cC}{\leftapp{\cU}{X}}{\blank}}\arrow{r}{\dual{\Ext{1}{\cC}{p_{L}}{\blank}}}\arrow[equal]{d}&\dual{\Ext{1}{\cC}{\leftcok{\cU}{X}}{\blank}}\arrow[equal]{d}\\
\Ext{1}{\cC}{\blank}{X}\arrow{r}{\Ext{1}{\cC}{\blank}{i_{L}}}&\Ext{1}{\cC}{\blank}{\leftapp{\cU}{X}}\arrow{r}{\Ext{1}{\cC}{\blank}{p_{L}}}&\Ext{1}{\cC}{\blank}{\leftcok{\cU}{X}}
\end{tikzcd}\]
The result now follows from this diagram, by comparing \eqref{eq:op-errors} to \eqref{eq:errors}.
\end{proof}

Using the functors \eqref{eq:errors}, and the classes of their values in $\Kgp{\fpmod{\stab{\cT}}}$, we can quantify precisely the failure of index and coindex to be additive on an index or coindex conflation, as follows.
Recall that $\ind{\cU}{\cT}=\ind{\cC}{\cT}|_{\cU}$.

\begin{theorem}\label{t:ind-coind-additive-error-terms}
Let $\cC$ be a cluster category and $\cTU\ctsubcat \cC$.  Then for any $X\in \cC$,
\begin{enumerate}
\item\label{t:ind-additive-error-terms} $\ind{\cC}{\cT}[X]=\ind{\cU}{\cT}[\leftapp{\cU}{X}]-\ind{\cU}{\cT}[\leftcok{\cU}{X}]-\beta_{\cT}[\lefterror{1}{\cU}{X}|_{\cT}]$,
\item\label{t:ind-additive-error-terms-2} $\ind{\cC}{\cT}[X]=\ind{\cU}{\cT}[\rightapp{\cU}{X}]-\ind{\cU}{\cT}[\rightker{\cU}{X}]-\beta_{\cT}[\righterror{1}{\cU}{X}|_{\cT}]$,
\item\label{t:coind-additive-error-terms}
$\coind{\cC}{\cT}[X]=\coind{\cU}{\cT}[\leftapp{\cU}{X}]-\coind{\cU}{\cT}[\leftcok{\cU}{X}]+\beta_{\cT}[\lefterror{2}{\cU}{X}|_{\cT}]$, and
\item\label{t:coind-additive-error-terms-2}
$\coind{\cC}{\cT}[X]=\coind{\cU}{\cT}[\rightapp{\cU}{X}]-\coind{\cU}{\cT}[\rightker{\cU}{X}]+\beta_{\cT}[\righterror{2}{\cU}{X}|_{\cT}]$.
\end{enumerate}
\end{theorem}

\begin{proof}
Each identity follows from applying one of the two identities from Proposition~\ref{p:beta-add-confl} to either a $\cU$-index or $\cU$-coindex sequence for $X$.
We use here the fact, immediate from the definition and the long-exact sequence of extension groups, that $\lefterror{1}{\cU}X|_{\cT}=\Coker\Yonfun{\cT}{p_L^X}$ and $\righterror{\cU}{1}X|_{\cT}=\Coker\Yonfun{\cT}{p_R^X}$.
For the identities involving the coindex, we also use Proposition~\ref{p:error-op} to rewrite the argument of $\beta_{\cT}$.
\end{proof}

\begin{remark}
Theorem~\ref{t:ind-coind-additive-error-terms} categorifies (and generalises, for categorifiable cluster algebras) an identity of Nakanishi and Zelevinsky \cite[Eq.~2.5]{NakanishiZelevinsky}.
\end{remark}

For any $\cT,\cV\ctsubcat\cC$, we obtain maps
\[\begin{aligned}
\lefterrormap{i}{\cU}{\cV}{\cT}\colon\Kgp{\cV}&\to\Kgp{\fpmod{\stab{\cT}}}, &\quad
\righterrormap{i}{\cU}{\cV}{\cT}\colon\Kgp{\cV}&\to\Kgp{\fpmod{\stab{\cT}}},\\
[V]&\mapsto[\lefterror{i}{\cU}V|_{\cT}], &\quad
[V]&\mapsto[\righterror{i}{\cU}V|_{\cT}]
\end{aligned}\]
for $i=1,2$, which are well-defined since all conflations in $\cV$ split.

\begin{corollary}\label{c:ind-coind-mult-error-terms}
Let $\cC$ be a cluster category and $\cTU,\cV\ctsubcat \cC$.  Then 
\begin{enumerate}
\item\label{c:ind-coind-mult-error-terms-1} 
$\ind{\cV}{\cT}=\ind{\cU}{\cT}\circ \coind{\cV}{\cU}-\beta_{\cT}\circ \lefterrormap{1}{\cU}{\cV}{\cT}$,
\item\label{c:ind-coind-mult-error-terms-2}
$\ind{\cV}{\cT}=\ind{\cU}{\cT}\circ \ind{\cV}{\cU}-\beta_{\cT}\circ \righterrormap{1}{\cU}{\cV}{\cT}$,
\item\label{c:ind-coind-mult-error-terms-3}
$\coind{\cV}{\cT}=\coind{\cU}{\cT}\circ \coind{\cV}{\cU}+\beta_{\cT}\circ \lefterrormap{2}{\cU}{\cV}{\cT}$ and
\item\label{c:ind-coind-mult-error-terms-4}
$\coind{\cV}{\cT}=\coind{\cU}{\cT}\circ \ind{\cV}{\cU}+\beta_{\cT}\circ \righterrormap{2}{\cU}{\cV}{\cT}$.
\end{enumerate}
\end{corollary}

\begin{proof} This follows immediately from Theorem~\ref{t:ind-coind-additive-error-terms}, recalling that $ \ind{\cV}{\cU}[V]=[\rightapp{\cU}{V}]-[\rightker{\cU}{V}]$ and $ \coind{\cV}{\cU}[V]=[\leftapp{\cU}{V}]-[\leftcok{\cU}{V}]$.
\end{proof}

\begin{remark}\label{r:mult-error-vanish} Together with Remark~\ref{r:add-error-vanish}, the identities in Corollary~\ref{c:ind-coind-mult-error-terms} imply that if $V\in \cU$ (e.g\ if $\cV=\mut{U}{\cU}$ and $V\in \cU\setminus U$), the relevant index and coindex maps in fact compose transitively, i.e.\ $\ind{\cV}{\cT}[V]=\ind{\cU}{\cT}\ind{\cV}{\cU}[V]$.
\end{remark}

The functors \eqref{eq:errors} also play a role in describing the adjoint maps $\indbar{\cT}{\cU}$ and $\coindbar{\cT}{\cU}$.
To see this, we first need to describe the values of the $\cC$-modules $\lefterror{i}{\cU}{X}$ and $\righterror{i}{\cU}{X}$ at an object $Y\in\cC$ using $\cU$-index and $\cU$-coindex conflations for $Y$, instead of for $X$.

\begin{lemma}
\label{l:error-terms-modT}
For any $X,Y\in\cC$ and $\cU\ctsubcat\cC$, let $\cU(Y,X)$ denote the subspace of $\Hom{\cC}{Y}{X}$ consisting of maps factoring over $\cU$.
Then
\begin{enumerate}
\item\label{p:error-terms-modT-r1} $\righterror{1}{\cU}{X}=\Hom{\cC}{\blank}{X}/\cU(\blank,X)$ and
\item\label{p:error-terms-modT-l3} $\lefterror{2}{\cU}{X}=\dual{(\Hom{\cC}{X}{\blank}/\cU(X,\blank))}$.
\end{enumerate}
\end{lemma}

\begin{proof}
Directly from the definition, $\righterror{1}{\cU}{X}=\Ker\Ext{1}{\cC}{\blank}{i_R^X}=\Coker\Hom{\cC}{\blank}{p_R^X}$ for \[\begin{tikzcd}
\rightker{\cU}{X}\arrow[infl]{r}{i_R^X}&\rightapp{\cU}{X}\arrow[defl]{r}{p_R^X}&X
\end{tikzcd}\]
a $\cU$-index sequence.
A map $Y\to X$ in $\cC$ factors over $\cU$ if and only if it factors over the $\cU$-approximation $p_R^X$, and so the image of $\Hom{\cC}{\blank}{p_R^X}$ is $\cU(\blank,X)$.
This gives \ref{p:error-terms-modT-r1}, and \ref{p:error-terms-modT-l3} then follows by applying \ref{p:error-terms-modT-r1} to $\op{\cC}$ and using Proposition~\ref{p:error-op}.
\end{proof}

\begin{proposition}
\label{p:error-alt}
Let $\cC$ be a cluster category, let $\cU\ctsubcat\cC$, and let $X$. For each morphism $f\colon Y\to Z$ in $\cC$, choose a $\cU$-coindex and a $\cU$-index conflation of $Y$ and $Z$, together with lifts of $f$ to maps of these conflations.
\[\begin{tikzcd}Y\arrow[infl]{r}{i_L^Y}\arrow{d}{f}&\leftapp{\cU}{Y}\arrow[defl]{r}{p_{L}^{Y}}\arrow{d}{g_L}&\leftcok{\cU}{Y}\arrow[confl]{r}{\delta_L^Y}\arrow{d}{h_L}&\phantom{}\\
Z\arrow[infl]{r}{i_L^Z}&\leftapp{\cU}{Z}\arrow[defl]{r}{p_{L}^Z}&\leftcok{\cU}{Z}\arrow[confl]{r}{\delta_L^Z}&\phantom{}\end{tikzcd}\quad
\begin{tikzcd}\rightker{\cU}{Y}\arrow[infl]{r}{i_{R}^{Y}}\arrow{d}{h_R}&\rightapp{\cU}{Y}\arrow[defl]{r}{p_R^Y}\arrow{d}{g_R}&Y\arrow[confl]{r}{\delta_R^Y}\arrow{d}{f}&\phantom{}\\
\rightker{\cU}{Z}\arrow[infl]{r}{i_R^Z}&\rightapp{\cU}{Z}\arrow[defl]{r}{p_R^Z}&Z\arrow[confl]{r}{\delta_R^Z}&\phantom{}\end{tikzcd}\]
Then the $\stab{\cC}$-modules $\lefterror{i}{\cU}{X}$ and $\righterror{i}{\cU}{X}$ evaluate on $f$ as follows:
\begin{enumerate}
\item
\label{p:error-alt-r1-l1}
$\begin{aligned}[t]
\righterror{1}{\cU}{X}(f)&=h_L^*\colon\Ker{\Ext{1}{\cC}{p_{L}^{Z}}{X}}\to\Ker{\Ext{1}{\cC}{p_{L}^{Y}}{X}},\\
\lefterror{1}{\cU}{X}(f)&=g_L^*\colon\Coker{\Ext{1}{\cC}{p_{L}^{Z}}{X}}\to\Coker{\Ext{1}{\cC}{p_{L}^{Y}}{X}},
\end{aligned}$
\item\label{p:error-alt-r2-l2}
$\begin{aligned}[t] \lefterror{2}{\cU}{X}(f)&=h_R^*\colon\Coker{\Ext{1}{\cC}{i_{R}^{Z}}{X}}\to\Coker{\Ext{1}{\cC}{i_{R}^{Y}}{X}},\\
\righterror{2}{\cU}{X}(f)&=g_R^*\colon\Ker{\Ext{1}{\cC}{i_{R}^{Z}}{X}}\to\Ker{\Ext{1}{\cC}{i_{R}^{Y}}{X}}. \end{aligned}$
\end{enumerate}
\end{proposition}
\begin{proof}
For the statement concerning $\righterror{1}{\cU}{X}$, consider the commutative diagram
\begin{equation}
\label{eq:righterror-diag}
\begin{tikzcd}[column sep=18pt]
\Hom{\cC}{Z}{\rightapp{\cU}{X}}\arrow{rr}{\Hom{}{Z}{p_R^X}}\arrow{d}{f^*}&&\Hom{\cC}{Z}{X}\arrow{r}{(\delta_L^Z)^*}\arrow{d}{f^*}&\Ext{1}{\cC}{\leftcok{\cU}{Z}}{X}\arrow{r}{(p_L^Z)^*}\arrow{d}{h_L^*}&\Ext{1}{\cC}{\leftapp{\cU}{Z}}{X}\arrow{d}{g_L^*}\\
\Hom{\cC}{Y}{\rightapp{\cU}{X}}\arrow{rr}{\Hom{}{Y}{p_R^X}}&&\Hom{\cC}{Y}{X}\arrow{r}{(\delta_L^Y)^*}&\Ext{1}{\cC}{\leftcok{\cU}{Y}}{X}\arrow{r}{(p_L^Y)^*}&\Ext{1}{\cC}{\leftapp{\cU}{Y}}{X},
\end{tikzcd}
\end{equation}
in which the right-hand pair of squares comes from the long exact sequences obtained by applying $\Hom{\cC}{\blank}{X}$ to the $\cU$-coindex conflations of $Y$ and $Z$, and the chosen map between them.
In particular, the rows are exact at $\Ext{1}{\cC}{\leftcok{\cU}{Y}}{X}$ and $\Ext{1}{\cC}{\leftcok{\cU}{Z}}{X}$, so $\Ker(p_L^Y)^*=\image(\delta_L^Y)^*$, and similar for $Z$.

We have $\righterror{1}{\cU}{X}(f)=f^*\colon\Coker\Hom{\cC}{Z}{p_R^X}\to\Coker\Hom{\cC}{Y}{p_R^X}$ as in Lemma~\ref{l:error-terms-modT}, so to obtain our desired statement, it is enough to show that the rows of \eqref{eq:righterror-diag} are also exact at $\Hom{\cC}{Y}{X}$ and $\Hom{\cC}{Z}{Y}$, so that $\image(\delta_L^Y)^*=\Coker\Hom{\cC}{Y}{p_R^X}$ and similarly for $Z$.
From the long exact sequence obtained by applying $\Hom{\cC}{\blank}{X}$ to the $\cU$-coindex conflation for $Y$, we have $\ker(\delta_L^Y)^*=\image(\Hom{\cC}{i_L^Y}{X})$.
But since $i_L^Y$ is a left $\cU$-approximation of $Y$, this image is precisely $\cU(Y,X)$, which is also the image of $\Hom{\cC}{Y}{p_R^X}$ as in Lemma~\ref{l:error-terms-modT}.
Repeating this argument for $Z$, we see that \eqref{eq:righterror-diag} has exact rows, as required.

Similar reasoning shows that the statement for $\lefterror{1}{\cU}{X}$ reduces to exactness of the rows of the commutative diagram
\begin{equation}
\label{eq:lefterror-diag}
\begin{tikzcd}[column sep=19pt]
\Ext{1}{\cC}{\leftcok{\cU}{Z}}{X}\arrow{r}{(p_L^Z)^*}\arrow{d}{h_L^*}&\Ext{1}{\cC}{\leftapp{\cU}{Z}}{X}\arrow{r}{(i_L^Z)^*}\arrow{d}{g_L^*}&\Ext{1}{\cC}{Z}{X}\arrow{rr}{\Ext{1}{\cC}{Z}{i_L^X}}\arrow{d}{f^*}&&\Ext{1}{\cC}{Z}{\leftapp{\cU}{X}}\arrow{d}{f^*}\\
\Ext{1}{\cC}{\leftcok{\cU}{Y}}{X}\arrow{r}{(p_L^Y)^*}&\Ext{1}{\cC}{\leftapp{\cU}{Y}}{X}\arrow{r}{(i_L^Y)^*}&\Ext{1}{\cC}{Y}{X}\arrow{rr}{\Ext{1}{\cC}{Y}{i_L^X}}&&\Ext{1}{\cC}{Y}{\leftapp{\cU}{X}},
\end{tikzcd}
\end{equation}
at $\Ext{1}{\cC}{Y}{X}$ and $\Ext{1}{\cC}{Z}{X}$.
The map $\Ext{1}{\cC}{Y}{i_L^X}\circ(i_L^Y)^*$ factors over $\Ext{1}{\cC}{\leftapp{\cU}{Y}}{\leftapp{\cU}{X}}$, but this is zero since both arguments are in the cluster-tilting subcategory $\cU$.
This (and the same argument for $Z$) means that the rows of \eqref{eq:lefterror-diag} are complexes.

It remains to show that $\Ker\Ext{1}{\cC}{Y}{i_L^X}=\Ker\stabHom{\cC}{Y}{\Sigma i_L^X}\subseteq\image(i_L^Y)^*$, the argument for $Z$ again being the same.
If $(\Sigma i_L^X)\circ\varphi=0$ for a morphism $\varphi\in\stabHom{\cC}{Y}{\Sigma X}=\Ext{1}{\cC}{Y}{X}$, then $\varphi$ factors over $\delta_L^X\colon\leftcok{\cU}{X}\to\Sigma X$.
Since $\leftcok{\cU}{X}\in\cU$, this means that $\varphi$ further factors over the left $\cU$-approximation $i_L^Y$, i.e.\ that $\varphi\in\image(i_L^Y)^*$.
Thus \eqref{eq:lefterror-diag} has exact rows, completing the proof of \ref{p:error-alt-r1-l1}. Statement \ref{p:error-alt-r2-l2} follows by applying \ref{p:error-alt-r1-l1} to $\op{\cC}$.
\end{proof}

We may also use the $\stab{\cT}$-modules $\lefterror{i}{\cU}{X}|_{\cT}$ and $\righterror{i}{\cU}{X}|_{\cT}$ to give an alternative description of the maps $\stabcindbar{\cU}{\cT}\colon\Kgpnum{\lfd{\stab{\cU}}}\to\Kgpnum{\lfd{\stab{\cT}}}$ from Definition~\ref{d:ind-bar-def}, adjoint to $\stabcind{\cT}{\cU}\colon\Kgp{\stab{\cT}}\to\Kgp{\stab{\cU}}$, as well as a `lift' of these maps to functions $\Kgp{\fpmod{\stab{\cU}}}\to\Kgp{\fpmod{\stab{\cT}}}$.
We note here that while $\fpmod{\stab{\cU}}\subseteq\lfd{\stab{\cU}}$ (and similarly for $\cT$) since $\stab{\cC}$ is Hom-finite, the induced map $\Kgp{\fpmod{\cU}}\to\Kgpnum{\lfd{\cU}}$ need not be injective, hence referring to a lift rather than a restriction.

\begin{definition}
Let $\cC$ be a cluster category and let $\cTU\ctsubcat{\cC}$.
For $X\in\cC$, define
\[\stabIndbar{\cU}{\cT}(\Extfun{\cU}{X})\defeq[\lefterror{1}{\cU}{X}|_{\cT}]-[\righterror{1}{\cU}{X}|_{\cT}],\quad
\stabCoindbar{\cU}{\cT}(\Extfun{\cU}{X})\defeq[\lefterror{2}{\cU}{X}|_{\cT}]-[\righterror{2}{\cU}{X}|_{\cT}]\]
in $\Kgp{\fpmod{\stab{\cT}}}$, recalling that $\lefterror{i}{\cU}{X}|_{\cT}$ and $\righterror{i}{\cU}{X}|_{\cT}$ are finitely presented $\stab{\cT}$-modules.
\end{definition}

\begin{proposition}
\label{p:stabIndbar-well-def}
The maps $\stabIndbar{\cU}{\cT}$ and $\stabCoindbar{\cU}{\cT}$ induce homomorphisms $\Kgp{\fpmod{\stab{\cU}}}\to\Kgp{\fpmod{\stab{\cT}}}$.
\end{proposition}

\begin{proof}
Given an exact sequence
\begin{equation}
\label{eq:E-exact-seq}
\begin{tikzcd}
0\arrow{r}&\Extfun{\cU}X\arrow{r}{\Extfun{\cU}{i}}&\Extfun{\cU}{Y}\arrow{r}{\Extfun{\cU}{p}}&\Extfun{\cU}{Z}\arrow{r}&0
\end{tikzcd}
\end{equation}
in $\fpmod{\stab{\cU}}$ (cf.~Proposition~\ref{p:equiv-to-mod}) and $T\in\cT$, we may construct the diagram
\[\begin{tikzcd}[column sep=22pt,row sep=13pt]
&&0\arrow{d}&0\arrow{d}&&\\
0\arrow{r}&\righterror{1}{\cU}{X}(T)\arrow{r}&\Ext{1}{\cC}{\leftcok{\cU}{T}}{X}\arrow{d}\arrow{r}&\Ext{1}{\cC}{\leftapp{\cU}{T}}{X}\arrow{d}\arrow{r}&\lefterror{1}{\cU}{X}(T)\arrow{r}&0\\
0\arrow{r}&\righterror{1}{\cU}{Y}(T)\arrow{r}&\Ext{1}{\cC}{\leftcok{\cU}{T}}{Y}\arrow{d}\arrow{r}&\Ext{1}{\cC}{\leftapp{\cU}{T}}{Y}\arrow{d}\arrow{r}&\lefterror{1}{\cU}{Y}(T)\arrow{r}&0\\
0\arrow{r}&\righterror{1}{\cU}{Z}(T)\arrow{r}&\Ext{1}{\cC}{\leftcok{\cU}{T}}{Z}\arrow{d}\arrow{r}&\Ext{1}{\cC}{\leftapp{\cU}{T}}{Z}\arrow{d}\arrow{r}&\lefterror{1}{\cU}{Z}(T)\arrow{r}&0\\
&&0&0&&
\end{tikzcd}\]
in which the rows are exact by Proposition~\ref{p:error-alt}\ref{p:error-alt-r1-l1}, and the columns are obtained by evaluating \eqref{eq:E-exact-seq}, and are hence also exact.
The snake lemma thus gives us an exact sequence
\begin{equation}
\label{eq:error-term-seq}
\begin{tikzcd}[column sep=10pt]
0\arrow{r}&\righterror{1}{\cU}{X}(T)\arrow{r}&\righterror{1}{\cU}{Y}(T)\arrow{r}{\varphi}&\righterror{1}{\cU}{Z}(T)\arrow{r}{\delta}&\lefterror{1}{\cU}{X}(T)\arrow{r}{\psi}&\lefterror{1}{\cU}{Y}(T)\arrow{r}&\lefterror{1}{\cU}{Z}(T)\arrow{r}&0.
\end{tikzcd}
\end{equation}
Now one may check that each morphism in this sequence is functorial in $T$.
For example, for a map $f\colon T'\to T$ in $\cT$ and a choice of lift $h_L\colon\leftcok{\cU}{T'}\to\leftcok{\cU}{T}$ as in Proposition~\ref{p:error-alt}, we have the diagram
\[\begin{tikzcd}
\righterror{1}{\cU}{Y}(T)\arrow[hookrightarrow]{rr}\arrow{dr}{\righterror{1}{\cU}{Y}(f)}\arrow{dd}{\varphi_T}&&\Ext{1}{\cC}{\leftcok{\cU}T}{Y}\arrow[near end]{dd}{\Ext{1}{\cC}{\leftcok{\cU}{T}}{p}}\arrow{dr}{\Ext{1}{\cC}{h_L}{Y}}\\
&\righterror{1}{\cU}{Y}(T')&&\arrow[from=ll,crossing over,hookrightarrow]\Ext{1}{\cC}{\leftcok{\cU}{T'}}{Y}\arrow{dd}{\Ext{1}{\cC}{\leftcok{\cU}{T'}}{p}}\\
\righterror{1}{\cU}{Z}(T)\arrow[hookrightarrow]{rr}\arrow{dr}{\righterror{1}{\cU}{Z}(f)}&&\Ext{1}{\cC}{\leftcok{\cU}{T}}{Z}\arrow{dr}{\Ext{1}{\cC}{h_L}{Z}}\\
&\righterror{1}{\cU}{Z}(T')\arrow[from=uu,crossing over,"\varphi_{T'}",near start]\arrow[hookrightarrow]{rr}&&\Ext{1}{\cC}{\leftcok{\cU}{T'}}{Z}
\end{tikzcd}\]
in which the front and back faces commute by construction, the right-hand face commutes by bifunctoriality of $\Ext{1}{\cC}{\blank}{\blank}$, and the upper and lower faces commute by Proposition~\ref{p:error-alt}\ref{p:error-alt-r1-l1}.
Functoriality of $\varphi$ is commutativity of the left-hand face, which follows since the horizontal arrows are inclusions.
Similar reasoning gives functoriality of all morphisms in \eqref{eq:error-term-seq} except $\delta$, but this follows since $\delta$ is the composition of the cokernel of $\varphi$ with the kernel of $\psi$.

The exact sequence \eqref{eq:error-term-seq} thus shows that in $\Kgp{\fpmod{\stab{\cU}}}$ we have
\begin{align*}
0&=[\lefterror{1}{\cU}{Z}|_{\cT}]-[\lefterror{1}{\cU}{Y}|_{\cT}]+[\lefterror{1}{\cU}{X}|_{\cT}]-[\righterror{1}{\cU}{Z}|_{\cT}]+[\righterror{1}{\cU}{Y}|_{\cT}]-[\righterror{1}{\cU}{X}|_{\cT}]\\
&=\stabIndbar{\cU}{\cT}(\Extfun{\cU}{Z})-\stabIndbar{\cU}{\cT}(\Extfun{\cU}{Y})+\stabIndbar{\cU}{\cT}(\Extfun{\cU}{X}),
\end{align*}
as required.
The statement for $\stabCoindbar{\cU}{\cT}$ is then deduced from $\op{\cC}$ in the usual way.
\end{proof}

\begin{corollary}
\label{c:indbar-lift}
Let $\cC$ be a compact or skew-symmetric cluster category and $\cTU\ctsubcat\cC$.
Then there are commutative diagrams
\[\begin{tikzcd}
\Kgp{\fpmod{\stab{\cU}}}\arrow{r}\arrow{d}{\stabIndbar{\cU}{\cT}}&\Kgpnum{\lfd{\stab{\cU}}}\arrow{d}{\stabindbar{\cU}{\cT}}\\
\Kgp{\fpmod{\stab{\cT}}}\arrow{r}\arrow{r}&\Kgpnum{\lfd{\stab{\cT}}}
\end{tikzcd}\qquad
\begin{tikzcd}
\Kgp{\fpmod{\stab{\cU}}}\arrow{r}\arrow{d}{\stabCoindbar{\cU}{\cT}}&\Kgpnum{\lfd{\stab{\cU}}}\arrow{d}{\stabcoindbar{\cU}{\cT}}\\
\Kgp{\fpmod{\stab{\cT}}}\arrow{r}\arrow{r}&\Kgpnum{\lfd{\stab{\cT}}}
\end{tikzcd}\]
in which the horizontal arrows are induced from the inclusion $\fpmod{\stab{\cU}}\subseteq\lfd{\stab{\cU}}$, and the analogous inclusion for $\cT$.
\end{corollary}
\begin{proof}
By using the non-degenerate bilinear form $\canform{\blank}{\blank}{\cT}$ from Proposition~\ref{p:K-duality}, we see that commutativity of the left-hand square is equivalent to the statement that
\begin{align*}
\canform{[\lefterror{1}{\cU}{X}|_{\cT}]-[\righterror{1}{\cU}{X}|_{\cT}]}{[T]}{\cT}&=\canform{\indbar{\cU}{\cT}[\Extfun{\cU}{X}]}{[T]}{\cT}\\
&=\canform{[\Extfun{\cU}{X}]}{\coind{\cT}{\cU}[T]}{\cU}\\
&=\ext{1}{\cC}{\leftapp{\cU}{T}}{X}-\ext{1}{\cC}{\leftcok{\cU}{T}}{X}
\end{align*}
for all $T\in\cT$, where $T\infl \leftapp{\cU}{T}\defl \leftcok{\cU}{T}\confl$ is a $\cU$-coindex conflation for $T$; here $[\lefterror{1}{\cU}{X}|_{\cT}]$ and $[\righterror{1}{\cU}{X}|_{\cT}]$ refer to classes in $\Kgpnum{\lfd{\stab{\cT}}}$.
Taking $Y=T$ in Proposition~\ref{p:error-alt}\ref{p:error-alt-r1-l1}, we have an exact sequence
\begin{equation}
\label{eq:l1-r1-seq}
\begin{tikzcd}[column sep=20pt]
0\arrow{r}&\righterror{1}{\cU}{X}(T)\arrow{r}&\Ext{1}{\cC}{\leftcok{\cU}{T}}{X}\arrow{r}&\Ext{1}{\cC}{\leftapp{\cU}{T}}{X}\arrow{r}&\lefterror{1}{\cU}{X}(T)\arrow{r}&0,
\end{tikzcd}
\end{equation}
from which the required identity follows.
Commutativity of the right-hand square follows analogously from Proposition~\ref{p:error-alt}\ref{p:error-alt-r2-l2} (since $\cC$ is stably $2$-Calabi--Yau).
\end{proof}

\begin{remark}
The distinction between $\stabIndbar{\cU}{\cT}$ and $\stabindbar{\cU}{\cT}$ (and similarly for the coindex) is only relevant in the infinite rank case: if $\cC$ has finite rank, then the horizontal arrows in Corollary~\ref{c:indbar-lift} are isomorphisms.
\end{remark}

\begin{corollary}
\label{c:indbar-on-fd}
Let $\cC$ be a compact or skew-symmetric cluster category, let $\cTU\ctsubcat\cC$, and assume that
$\cU$ is maximally mutable. Then $\stabcindbar{\cU}{\cT}$ restricts to a map $\Kgp{\fd{\stab{\cU}}}\to\Kgp{\fd{\stab{\cT}}}$ if either $\cC$ has finite rank or
$\cU$ is reachable from $\cT$.
In either case, if $\cT$ is also maximally mutable then we have commutative diagrams
\[\begin{tikzcd}
\Kgp{\fd{\stab{\cU}}}\arrow{r}\arrow{d}{\stabindbar{\cU}{\cT}}&\Kgp{\fpmod{\stab{\cU}}}\arrow{d}{\stabIndbar{\cU}{\cT}}\\
\Kgp{\fd{\stab{\cT}}}\arrow{r}\arrow{r}&\Kgp{\fpmod{\stab{\cT}}}
\end{tikzcd}\qquad
\begin{tikzcd}
\Kgp{\fd{\stab{\cU}}}\arrow{r}\arrow{d}{\stabcoindbar{\cU}{\cT}}&\Kgp{\fpmod{\stab{\cU}}}\arrow{d}{\stabCoindbar{\cU}{\cT}}\\
\Kgp{\fd{\stab{\cT}}}\arrow{r}\arrow{r}&\Kgp{\fpmod{\stab{\cT}}}
\end{tikzcd}\]
in which the horizontal arrows are injective; that is, $\stabIndbar{\cU}{\cT}$ and $\stabindbar{\cU}{\cT}$ have the same restriction to $\Kgp{\fd{\stab{\cU}}}\to\Kgp{\fd{\stab{\cT}}}$ (and similarly for the coindex).
\end{corollary}
\begin{proof}
For $\cU$ maximally mutable we have $\fd{\stab{\cU}}\subseteq\fpmod{\stab{\cU}}$, and so every $M\in\fd{\stab{\cU}}$ has the form $\Extfun{\cU}X$ for some $X\in\cC$.
By Corollary~\ref{c:indbar-lift}, we have $\stabcindbar{\cU}{\cT}[M]\in\Kgp{\fd{\stab{\cT}}}$ provided $\lefterror{i}{\cU}X|_{\cT},\righterror{i}{\cU}X|_{\cT}\in\fd{\stab{\cT}}$.
By Proposition~\ref{p:error-term-fd}\ref{p:error-term-fd-add-finite}, this is the case if $\cC$ has finite rank.

If $\cU$ is reachable from $\cT$, then $\lefterror{2}{\cU}X|_{\cT},\righterror{1}{\cU}X|_{\cT}\in\fd{\stab{\cT}}$ by Proposition~\ref{p:error-term-fd}\ref{p:error-term-fd-U-reachable}.
Since each $\mut{U}{\cU}$ is (by definition) reachable from $\cU$, and hence reachable from $\cT$, we also have $\lefterror{1}{\cU}{(\mut{\cU}{U})}|_{\cT},\righterror{2}{\cU}{(\mut{\cU}{U})}|_{\cT}\in\fd{\stab{\cT}}$ by Proposition~\ref{p:error-term-fd}\ref{p:error-term-fd-V-reachable}.
Thus $\stabcindbar{\cU}{\cT}$ takes the classes $[\Extfun{\cU}{(\mut{\cU}{U})}]$ into $\Kgp{\fd{\stab{\cT}}}$.
Since $\Extfun{\cU}{(\mut{\cU}{U})}\in\fd{\cU}$ is supported only on $U\in\indec{\cU}$, we have $[\Extfun{\cU}{(\mut{\cU}{U})}]=\delta_U[\simpmod{\cU}{U}]$ for some $\delta_U\in\nat$, and so in fact $\stabcindbar{\cU}{\cT}$ takes the classes $[\simpmod{\cU}{U}]$, for $U\in\exch{\cU}$, into $\Kgp{\fd{\stab{\cT}}}$.
Since $\cU$ is maximally mutable, these classes span $\Kgp{\fd{\stab{\cU}}}$, and so $\stabcindbar{\cU}{\cT}$ has the desired restriction.

If both $\cT$ and $\cU$ are maximally mutable, the horizontal arrows in the given diagrams exist by Corollary~\ref{c:fp=fd} and are injective by Proposition~\ref{p:fd-sub-fpmod}.
The image of $\Kgp{\fd{\stab{\cT}}}\to\Kgp{\fpmod{\stab{\cT}}}$ includes into $\Kgpnum{\lfd{\stab{\cT}}}$ by the same proposition.
The arguments above then show that, under our assumptions, both $\stabIndbar{\cU}{\cT}$ and $\stabCoindbar{\cU}{\cT}$ take values in this image.
The commutativity of the squares thus follows from Corollary~\ref{c:indbar-lift}.
\end{proof}

\begin{corollary}\label{c:l-r-duality}
Let $\cTU,\cV\ctsubcat\cC$ and consider the maps $\lefterrormap{i}{\cU}{\cT}{\cV}\colon\Kgp{\cT}\to\Kgp{\fpmod\stab{\cV}}$ for $i=1,2$.
Computing adjoints $\Kgp{\stab{\cV}}\to\Kgpnum{\lfd{\cT}}$ with respect to $\canform{\blank}{\blank}{\cT}$ and the form $\ip{\blank}{\blank}\colon\Kgp{\fpmod{\stab{\cV}}}\times\Kgp{\cV}\to\integ$ induced from \eqref{eq:fundamental-pairing}, we have
\[
\adj{\bigl(\lefterrormap{1}{\cU}{\cT}{\cV}\bigr)}=\righterrormap{2}{\cU}{\cV}{\cT},\quad
\adj{\bigl(\lefterrormap{2}{\cU}{\cT}{\cV}\bigr)}=\righterrormap{1}{\cU}{\cV}{\cT},\quad
\adj{\bigl(\righterrormap{1}{\cU}{\cT}{\cV}\bigr)}=\lefterrormap{2}{\cU}{\cV}{\cT},\quad
\adj{\bigl(\righterrormap{2}{\cU}{\cT}{\cV}\bigr)}=\lefterrormap{1}{\cU}{\cV}{\cT}.
\]
\end{corollary}

\begin{proof}
Let $T\in\cT$ and $V\in\cV$, and choose a $\cU$-coindex conflation $V\infl\leftapp{\cU}{V}\defl\leftcok{\cU}{V}\confl$ for $V$.
Then we compute using Proposition~\ref{p:error-alt}\ref{p:error-alt-r1-l1} that
\begin{align*}
\ip{\lefterrormap{1}{\cU}{\cT}{\cV}[T]}{[V]}
&=\dim\Coker\bigl(\Ext{1}{\cC}{\leftcok{\cU}{V}}{T}\to\Ext{1}{\cC}{\leftapp{\cU}{V}}{T}\big)\\
&=\dim\Ker\bigl(\dual{\Ext{1}{\cC}{\leftapp{\cU}{V}}{T}}\to\dual{\Ext{1}{\cC}{\leftcok{\cU}{V}}{T}}\big)\\
&=\dim\Ker\bigl(\Ext{1}{\cC}{T}{\leftapp{\cU}{V}}\to\Ext{1}{\cC}{T}{\leftcok{\cU}{V}}\bigr)\\
&=\canform{\righterrormap{2}{\cU}{\cV}{\cT}[V]}{[T]}{\cV}.
\end{align*}
The first equality then follows from Corollary~\ref{c:identify-adj}, since $\canform{\blank}{\blank}{\cV}$ is non-degenerate.
The argument for the other equalities is completely analogous.
(As the above calculation indicates, the second pair of equalities are not immediate consequences of the first pair, because of the asymmetry in the choice of forms used to take adjoints.)
\end{proof}

The next corollary is essentially the adjoint of Corollary~\ref{c:ind-coind-mult-error-terms}, although this is more subtle than it may first appear.
Since the adjoints $\cindbar{}{}$ are defined uniformly using the forms $\canform{\blank}{\blank}{\cT}$ for various choices of cluster-tilting subcategory $\cT$, we deduce that
\[\adj{(\ind{\cU}{\cV}\circ \coind{\cT}{\cU})}=\indbar{\cU}{\cT}\circ \coindbar{\cV}{\cU},\]
and similarly for other choices of compositions; that is, the composition of adjoints is the adjoint of the composition.
As in Remark~\ref{r:adj-restr}, restricting this composition to $\Kgp{\fd{\cV}}\leq\Kgpnum{\lfd{\cV}}$ is equivalent to restricting the form $\canform{\blank}{\blank}{\cV}$ to $\Kgp{\fd{\cV}}$ in its first argument when calculating the adjoints.

To deal with the remaining terms of the identities in Corollary~\ref{c:ind-coind-mult-error-terms}, we may use the adjoints of $\lefterrormap{i}{\cU}{\cT}{\cV}$ and $\righterrormap{i}{\cU}{\cT}{\cV}$ constructed in Corollary~\ref{c:l-r-duality}, meaning it remains to take an adjoint of $\beta_{\cV}$. This must be done using forms which are both compatible with the desired composition, and with those used to take adjoints of the other terms in the equation.
In order to do so, we assume that $\cV$ is maximally mutable, so that we may consider the restriction $\beta_{\cV}\colon\Kgp{\fd{\stab{\cV}}}\to\Kgp{\cV}$.
We then wish to take an adjoint $\adj{\beta}_{\cV}\colon\Kgp{\fd{\cV}}\to\Kgp{\stab{\cV}}$ with respect to the forms $\canform{\blank|_{\Kgp{\fd{\stab{\cV}}}}}{\blank}{\stab{\cV}}$ and $\canform{\blank|_{\Kgp{\fd{\cV}}}}{\blank}{\cV}$.
This will then be compatible with the forms used in calculating adjoints to $\lefterrormap{i}{\cU}{\cT}{\cV}$ and $\righterrormap{i}{\cU}{\cT}{\cV}$ in Corollary~\ref{c:l-r-duality}, in the sense that (taking $\lefterrormap{1}{\cU}{\cT}{\cV}$ as a representative example) the composition $\adj{\beta}_{\cV}\circ\adj{\bigl(\lefterrormap{1}{\cU}{\cT}{\cV}\bigr)}$ is well-defined and equal to $\adj{(\lefterrormap{1}{\cU}{\cT}{\cV}\circ\beta_{\cV})}$.
To prove the existence of this adjoint, we must impose some minor additional assumptions.

\begin{lemma}
Let $\cC$ be a cluster category, and let $\cV\ctsubcat\cC$ be maximally mutable and locally finite.
Then $\beta_{\cV}\colon\Kgp{\fd{\stab{\cV}}}\to\Kgp{\cV}$ admits an adjoint $\adj{\beta}_{\cV}\colon\Kgp{\fd{\cV}}\to\Kgp{\stab{\cV}}$ with respect to the forms $\canform{\blank|_{\Kgp{\fd{\stab{\cV}}}}}{\blank}{\stab{\cV}}$ and $\canform{\blank|_{\Kgp{\fd{\cV}}}}{\blank}{\cV}$.
\end{lemma}
\begin{proof}
The map $\beta_{\cV}\colon\Kgp{\fpmod{\stab{\cV}}}\to\Kgp{\cV}$ restricts to $\Kgp{\fd{\stab{\cV}}}$ by the assumption that $\cV$ is maximally mutable.
The map $\Kgp{\stab{\cV}}\to\dual{\Kgp{\fd{\stab{\cV}}}}$ induced from $\canform{\blank|_{\Kgp{\fd{\stab{\cV}}}}}{\blank}{\stab{\cV}}$ is injective (because this form is non-degenerate), and so by Proposition~\ref{p:adjunction} it suffices to show that the functional $\canform{[M]}{\beta_{\cV}(\blank)}{\cV}$ is in its image, for any $M\in\fd{\cV}$.
That is, we have to show that $n_V\defeq\canform{[M]}{\beta_{\cV}[\simpmod{\cV}{V}]}{\cV}=0$ for all but finitely many $V\in\indec{\stab{\cV}}$, in which case we will have $\adj{\beta}_{\cV}[M]=\bigl[\bigoplus_{V\in\indec{\stab{\cV}}}V^{n_V}\bigr]\in\Kgp{\stab{\cV}}$.

By Proposition~\ref{p:beta-exch-mat} and Remark~\ref{r:beta-exch-mat-loops}, we have that $\beta_{\cV}[\simpmod{\cV}{V}]$ is a (positive integer) multiple of $\sum_{U\in \indec \cV} \exchmatentry{U,V}[U]$, where $B_{\cV}=(b_{U,V})$ is the exchange matrix of $\cV$.
Now $M\in\fd{\cV}$, so there are finitely many $U\in\indec{\cV}$ for which $M(U)\ne0$.
Since $\cV$ is locally finite, for each such $U$ there are finitely many $V$ for which $b_{U,V}\ne0$.
In this way, we see that
\[\canform{[M]}{\beta_{\cV}[\simpmod{\cV}{V}]}{\cV}=\sum_{U\in\indec{\cV}}b_{U,V}\cdot\dim_{\bK}M(U)\]
is zero for all but finitely many $V$, as required.
\end{proof}

\begin{corollary}\label{c:indbar-coindbar-mult-error-terms}
Let $\cC$ be a compact cluster category and $\cTU,\cV\ctsubcat \cC$.
Assume that $\cV$ is maximally mutable and locally finite.
Then
\begin{enumerate}
\item\label{c:indbar-coindbar-mult-error-terms-1} 
$\coindbar{\cV}{\cT}=\indbar{\cU}{\cT}\circ \coindbar{\cV}{\cU}-\righterrormap{2}{\cU}{\cV}{\cT}\circ \adj{\beta}_{\cV}$,
\item\label{c:indbar-coindbar-mult-error-terms-2}
$\coindbar{\cV}{\cT}=\coindbar{\cU}{\cT}\circ \coindbar{\cV}{\cU}-\lefterrormap{2}{\cU}{\cV}{\cT}\circ \adj{\beta}_{\cV}$,
\item\label{c:indbar-coindbar-mult-error-terms-3}
$\indbar{\cV}{\cT}=\indbar{\cU}{\cT}\circ \indbar{\cV}{\cU}+\righterrormap{1}{\cU}{\cV}{\cT}\circ \adj{\beta}_{\cV}$, and
\item\label{c:indbar-coindbar-mult-error-terms-4}
$\indbar{\cV}{\cT}=\coindbar{\cU}{\cT}\circ\indbar{\cV}{\cU}+\lefterrormap{1}{\cU}{\cV}{\cT}\circ \adj{\beta}_{\cV}$.
\end{enumerate}
on $\Kgp{\fd{\cV}}\subseteq\Kgpnum{\lfd{\cV}}$.
\end{corollary}

\begin{proof}
Let $M\in\fd{\cV}$.
To emphasise the choice of forms used to take the adjoints, we also choose $T\in\cT$, and calculate
\begin{align*}
\canform{\coindbar{\cV}{\cT}[M]}{[T]}{\cT}&=\canform{[M]}{\ind{\cT}{\cV}[T]}{\cV}\\
&=\canform{[M]}{\ind{\cU}{\cV}\coind{\cT}{\cU}[T]}{\cV}-\canform{[M]}{\beta_{\cV} \lefterrormap{1}{\cU}{\cT}{\cV}[T]}{\cV}\\
&=\canform{\coindbar{\cV}{\cU}[M]}{\coind{\cT}{\cU}[T]}{\cU}-\canform{\lefterrormap{1}{\cU}{\cT}{\cV}[T]}{\adj{\beta}_{\cV}[M]}{\stab{\cV}}\\
&=\canform{\indbar{\cU}{\cT}\coindbar{\cV}{\cU}[M]}{[T]}{\cT}-\canform{\righterrormap{2}{\cU}{\cV}{\cT}\adj{\beta}_{\cV}[M]}{[T]}{\cT}\\
&=\canform{(\indbar{\cU}{\cT}\circ\coindbar{\cV}{\cU}-\righterrormap{2}{\cU}{\cV}{\cT}\circ\adj{\beta}_{\cV})[M]}{[T]}{\cT},
\end{align*}
using Corollary~\ref{c:ind-coind-mult-error-terms} for the second equality, and Corollary~\ref{c:l-r-duality} for the fourth.
Identity \ref{c:indbar-coindbar-mult-error-terms-1} then follows since $\canform{\blank}{\blank}{\cT}$ is non-degenerate, and the others are proved similarly.
\end{proof}

As already remarked, we are particularly interested in the case that $\cV=\mut{U}{\cU}$ for some $U\in \cU$, since it is by analysing this case that we may finally prove that $\cind{\cU}{\cT}$ and $\cindbar{\cU}{\cT}$ mutate as $\mathbf{g}$-vectors and $\mathbf{c}$-vectors. 

\begin{theorem}\label{t:one-step-X-mut}
Let $\cC$ be a compact cluster category and $\cTU \ctsubcat \cC$.  Let $U\in\exch{\cU}$, write $\cU'=\mut{U}{\cU}$, and assume that $\cU'$ is maximally mutable and locally finite. Let
\[ \begin{tikzcd}U\arrow[infl]{r}{i_{L}^{U}}& \exchmon{\cU}{U}{-} \arrow[defl]{r}{p_{L}^{U}}&\mut{\cU}{U}\arrow[confl]{r}&,\end{tikzcd}\quad
\begin{tikzcd}\mut{\cU}{U} \arrow[infl]{r}{i_{R}^{U}}&\exchmon{\cU}{U}{+}\arrow[defl]{r}{p_{R}^{U}}&U\arrow[confl]{r}&\phantom{}\end{tikzcd} \]
be the exchange conflations for $U$, and let $V\in\indec(\cU'\setminus\mut{\cU}{U})$. Then
\begin{enumerate}
\item\label{t:one-step-X-mut-indbar-at-mut} $\indbar{\cU'}{\cT}[\simpmod{\cU'}{\mut{\cU}{U}}]=-\indbar{\cU}{\cT}[\simpmod{\cU}{U}]$,
\item\label{t:one-step-X-mut-indbar-away-from-mut}
$\indbar{\cU'}{\cT}[\simpmod{\cU'}{V}]=\indbar{\cU}{\cT}[\simpmod{\cU}{V}]+\dimdivalg{U}^{-1}\canform{[\simpmod{\cU'}{V}]}{[U_{\cU}^{-}]}{\cU'}\indbar{\cU}{\cT}[\simpmod{\cU}{U}]+\exchmatentry{U,V}^{\cU}[\Ker{\Extfun{\cT}i_{L}^{U}}]$, 		\item\label{t:one-step-X-mut-coindbar-at-mut} $\coindbar{\cU'}{\cT}[\simpmod{\cU'}{\mut{\cU}{U}}]=-\coindbar{\cU}{\cT}[\simpmod{\cU}{U}]$, and
\item\label{t:one-step-X-mut-coindbar-away-from-mut} $\coindbar{\cU'}{\cT}[\simpmod{\cU'}{V}]=\coindbar{\cU}{\cT}[\simpmod{\cU}{V}]+\dimdivalg{U}^{-1}\canform{[\simpmod{\cU'}{V}]}{[U_{\cU}^{+}]}{\cU'}\coindbar{\cU}{\cT}[\simpmod{\cU}{U}]+\exchmatentry{U,V}^{\cU}[\Coker{\Extfun{\cT}p_{R}^{U}}].$
\end{enumerate}
\end{theorem}

\begin{proof}
Recall from Proposition~\ref{p:one-step-X-mut-from-root}\ref{p:one-step-X-mut-from-root-indbar-at-mut} and   \ref{p:one-step-X-mut-from-root-coindbar-at-mut} that \[ \indbar{\cU'}{\cU}[\simpmod{\cU'}{\mut{\cU}{U}}]=-[\simpmod{\cU}{U}]=\coindbar{\cU'}{\cU}[\simpmod{\cU'}{\mut{\cU}{U}}]. \]
By Corollary~\ref{c:indbar-coindbar-mult-error-terms}\ref{c:indbar-coindbar-mult-error-terms-3}, we have
\begin{align*}
\indbar{\cU'}{\cT}[\simpmod{\cU'}{\mut{\cU}{U}}] & = \indbar{\cU}{\cT}\indbar{\cU'}{\cU}[\simpmod{\cU'}{\mut{\cU}{U}}]+\righterrormap{1}{\cU}{\cU'}{\cT}\adj{\beta}_{\cU'}[\simpmod{\cU'}{\mut{\cU}{U}}] \\ 
& = -\indbar{\cU}{\cT}[\simpmod{\cU}{U}]+\righterrormap{1}{\cU}{\cU'}{\cT}\adj{\beta}_{\cU'}[\simpmod{\cU'}{\mut{\cU}{U}}],
\end{align*}
and similarly for $\coindbar{\mut{\cU}{U}}{\cT}[\simpmod{\cU'}{\mut{\cU}{U}}]$, so to verify \ref{t:one-step-X-mut-indbar-at-mut} and \ref{t:one-step-X-mut-coindbar-at-mut} it suffices to show that \begin{equation}\label{eq:one-step-X-mut-at-mut-error-vanish} 
\righterrormap{1}{\cU}{\cU'}{\cT}\adj{\beta}_{\cU'}[\simpmod{\cU'}{\mut{\cU}{U}}]=0=\lefterrormap{2}{\cU}{\cU'}{\cT}\adj{\beta}_{\cU'}[\simpmod{\cU'}{\mut{\cU}{U}}]. \end{equation}
Since $\sform{\blank}{\blank}{\cU'}$ is skew-symmetric, we have \[\canform{[\simpmod{\cU'}{\mut{\cU}{U}}]}{\adj{\beta}_{\cU'}[\simpmod{\cU'}{\mut{\cU}{U}}]}{\cU'}=\canform{[\simpmod{\cU'}{\mut{\cU}{U}}]}{\beta_{\cU'}[\simpmod{\cU'}{\mut{\cU}{U}}]}{\cU'}=\sform{[\simpmod{\cU'}{\mut{\cU}{U}}]}{[\simpmod{\cU'}{\mut{\cU}{U}}]}{\cU'}=0 \] and hence $\adj{\beta}_{\cU'}[\simpmod{\cU'}{\mut{\cU}{U}}]\in \Kgp{\cU'\setminus \mut{\cU}{U}}$.
Since $\lefterror{2}{\cU}{W}=0=\righterror{1}{\cU}{W}$ for $W\in\indec(\cU'\setminus \mut{\cU}{U})=\indec(\cU\setminus U)$, as in Remark~\ref{r:add-error-vanish}, we therefore have \eqref{eq:one-step-X-mut-at-mut-error-vanish} as required.

Next, recall from Proposition~\ref{p:one-step-X-mut-from-root}\ref{p:one-step-X-mut-from-root-indbar-away-from-mut} that for $V\in\indec(\cU'\setminus \mut{\cU}{U})$ mutable,
\begin{align*} \indbar{\cU'}{\cU}[\simpmod{\cU'}{V}]&=[\simpmod{\cU}{V}]+\dimdivalg{U}^{-1}\canform{[\simpmod{\cU'}{V}]}{[U_{\cU}^{-}]}{\cU'}[\simpmod{\cU}{U}].
\end{align*}
Then by Corollary~\ref{c:indbar-coindbar-mult-error-terms}\ref{c:indbar-coindbar-mult-error-terms-3}, we have
\begin{equation}
\label{eq:indbar-away-step-1}
\indbar{\cU'}{\cT}[\simpmod{\cU'}{V}] =\indbar{\cU}{\cT}[\simpmod{\cU}{V}]+d_{U}^{-1}\canform{[\simpmod{\cU'}{V}]}{[U_{\cU}^{-}]}{\cU'}\indbar{\cU}{\cT}[\simpmod{\cU}{U}]+\righterrormap{1}{\cU}{\cU'}{\cT}\adj{\beta}_{\cU'}[\simpmod{\cU'}{V}].
\end{equation}
Since $\righterrormap{1}{\cU}{\cU'}{\cT}[U']=0$ for $U\in\indec{(\cU'\setminus\mu_\cU U)}$ as above, and the coefficient of $[\mut{\cU}{U}]$ in $\adj{\beta}_{\cU'}[\simpmod{\cU'}{V}]$ is $\dimdivalg{U}^{-1}\canform{[\simpmod{\cU'}{\mut{\cU}{U}}]}{\adj{\beta}_{\cU'}[\simpmod{\cU'}{V}]}{\cU'}$ by \eqref{eq:s-gram-matrix}, the third term of \eqref{eq:indbar-away-step-1} is
\begin{align*}
\righterrormap{1}{\cU}{\cU'}{\cT}\adj{\beta}_{\cU'}[\simpmod{\cU'}{V}]&=\dimdivalg{U}^{-1}\canform{[\simpmod{\cU'}{\mut{\cU}{U}}]}{\adj{\beta}_{\cU'}[\simpmod{\cU'}{V}]}{\cU'}\righterrormap{1}{\cU}{\cU'}{\cT}[\mut{\cU}{U}]\\
&=\dimdivalg{U}^{-1}\canform{[\simpmod{\cU'}{V}]}{\beta_{\cU'}[\simpmod{\cU'}{\mut{\cU}{U}}]}{\cU'}\righterrormap{1}{\cU}{\cU'}{\cT}[\mut{\cU}{U}]\\
&=\exchmatentry{U,V}^{\cU}\righterrormap{1}{\cU}{\cU'}{\cT}[\mut{\cU}{U}]\\
&=\exchmatentry{U,V}^{\cU}[\Ker{\Ext{1}{\cC}{\blank}{i_{R}^{\mut{\cU}{U}}}}|_{\cT}]\\
&=\exchmatentry{U,V}^{\cU}[\Ker{\Ext{1}{\cC}{\blank}{i_{L}^{U}}}|_{\cT}].
\end{align*}
Substituting back into \eqref{eq:indbar-away-step-1} gives \ref{t:one-step-X-mut-indbar-away-from-mut}, and the proof of \ref{t:one-step-X-mut-coindbar-away-from-mut} is completely analogous.
\end{proof}

Once again, we may use Proposition~\ref{p:no-loops-simp} to get slightly simpler expressions under mild additional assumptions.

\begin{corollary}
\label{c:one-step-X-mut}
In the context of Theorem~\ref{t:one-step-X-mut}, if $\cU$ has no loops or $2$-cycles at $U$ then, for $V\in\indec{\cU'}$,
\begin{enumerate}
\item\label{c:one-step-X-mut-indbar-at-mut} $\indbar{\cU'}{\cT}[\simpmod{\cU'}{\mut{\cU}{U}}]=-\indbar{\cU}{\cT}[\simpmod{\cU}{U}]$,
\item\label{c:one-step-X-mut-indbar-away-from-mut}
$\indbar{\cU'}{\cT}[\simpmod{\cU'}{V}]=\indbar{\cU}{\cT}[\simpmod{\cU}{V}]+[\exchmatentry{U,V}]_+\indbar{\cU}{\cT}[\simpmod{\cU}{U}]+\exchmatentry{U,V}^{\cU}[\Ker{\Ext{1}{\cC}{\blank}{i_{L}^{U}}}|_{\cT}]$, 
\item\label{c:one-step-X-mut-coindbar-at-mut} $\coindbar{\cU'}{\cT}[\simpmod{\cU'}{\mut{\cU}{U}}]=-\coindbar{\cU}{\cT}[\simpmod{\cU}{U}]$, and
\item\label{c:one-step-X-mut-coindbar-away-from-mut} $\coindbar{\cU'}{\cT}[\simpmod{\cU'}{V}]=\coindbar{\cU}{\cT}[\simpmod{\cU}{V}]+[\exchmatentry{U,V}]_-\coindbar{\cU}{\cT}[\simpmod{\cU}{U}]+\exchmatentry{U,V}^{\cU}[\Coker{\Ext{1}{\cC}{\blank}{p_{R}^{U}}}|_{\cT}]$.\qed
\end{enumerate}
\end{corollary}

\begin{theorem}
\label{t:c-vec-mut-formula}
Let $\cC$ be a compact cluster category with $\cTU\ctsubcat \cC$.
Let $U\in \indec \cU$ and assume there are no loops or 2-cycles at $U$.
Assume that $\mut{U}{\cU}$ is maximally mutable and locally finite, and let $V\in \indec \mut{U}{\cU}\setminus \mut{\cU}{U}$.
Then we have the formulæ
\begin{align}
\label{eq:g-vec-mut-formula}
\begin{split}
\ind{\mut{U}{\cU}}{\cT}[\mut{\cU}{U}] & = -\ind{\cU}{\cT}[U]+\Bigl(\sum_{W\in \cU\setminus U} [\exchmatentry{W,U}]_{-}\ind{\cU}{\cT}[W]\Bigr)-\beta_{\cT}\bigl[\stabindbar{\cU}{\cT}[\simpmod{\cU}{U}]\bigr]_-, \\
\ind{\mut{U}{\cU}}{\cT}[V] & =\ind{\cU}{\cT}[V]
\end{split}
\intertext{and}
\label{eq:c-vec-mut-formula}
\begin{split}
\indbar{\mut{U}{\cU}}{\cT}[\simpmod{\mut{U}{\cU}}{\mut{\cU}{U}}]& =-\indbar{\cU}{\cT}[\simpmod{\cU}{U}], \\ 
\indbar{\mut{U}{\cU}}{\cT}[\simpmod{\mut{U}{\cU}}{V}] & =\indbar{\cU}{\cT}[\simpmod{\cU}{V}]+[\exchmatentry{U,V}^{\cU}]_{+}\indbar{\cU}{\cT}[\simpmod{\cU}{U}]+\exchmatentry{U,V}^{\cU}\bigl[\stabindbar{\cU}{\cT}[\simpmod{\cU}{U}]\bigr]_{-}.
\end{split}
\end{align}	 
Consequently, if $(\cC,\cT)$ has a cluster structure, then the values of the index on indecomposable objects $U\in\cU$ with $\cU\ctsubcat\cC$ reachable from $\cT$, and its adjoint on simple modules, compute respectively the $\mathbf{g}$-vectors and $\mathbf{c}$-vectors of the cluster algebra with initial exchange matrix $B_{\cT}$.
\end{theorem}

\begin{proof}
For $T\in\cT$, choose a minimal $\cU$-coindex conflation $T\infl\leftapp{\cU}{T}\defl\leftcok{\cU}{T}$, so that $\leftapp{\cU}{T}$ and $\leftcok{\cU}{T}$ have no common direct summands by Remark~\ref{r:no-common-summands}.
Thus, if we take $U\in\exch{\cU}$, with $U'=\mut{\cU}{U}$ so that $\Extfun{\cU}U'=\simpmod{\cU}{U}$, then either $\Ext{1}{\cC}{\leftapp{\cU}{T}}{U'}=0$ or $\Ext{1}{\cC}{\leftcok{\cU}{T}}{U'}=0$, since these spaces are only non-zero when $\leftapp{\cU}{T}$, respectively $\leftcok{\cU}{T}$, has a summand isomorphic to $U$.
Comparing to \eqref{eq:l1-r1-seq}, it follows that either $\righterror{1}{\cU}{U'}(T)=0$ or $\lefterror{1}{\cU}{U'}(T)=0$; that is, the $\cT$-modules $\righterror{1}{\cU}{U'}$ and $\lefterror{1}{\cU}{U'}$ have disjoint support.
Now recalling from Corollary~\ref{c:indbar-lift} that $\stabindbar{\cU}{\cT}[\simpmod{\cU}{U}]=[\lefterror{1}{\cU}{U'}|_{\cT}]-[\righterror{1}{\cU}{U'}|_{\cT}]$, we see that
\[\bigl[\stabindbar{\cU}{\cT}[\simpmod{\cU}{U}]\bigr]_+=[\lefterror{1}{\cU}{U'}|_{\cT}],\quad
\bigl[\stabindbar{\cU}{\cT}[\simpmod{\cU}{U}]\bigr]_-=[\righterror{1}{\cU}{U'}|_{\cT}].\]
Thus, the identities in Corollary~\ref{c:one-step-X-mut}\ref{c:one-step-X-mut-indbar-at-mut}--\ref{c:one-step-X-mut-indbar-away-from-mut} are precisely the identities \eqref{eq:c-vec-mut-formula}.

Further abbreviating $\cU'=\mut{U}{\cU}$, we see from Proposition~\ref{p:decomp-exch-terms} and Corollary~\ref{c:ind-coind-mult-error-terms} that
\begin{align*}
\ind{\cU'}{\cT}[U'] & =(\ind{\cU}{\cT}\circ \ind{\cU'}{\cU})[U']-(\beta_{\cT}\circ \righterrormap{1}{\cU}{\cU'}{\cT})[U'] \\
& = \ind{\cU}{\cT}([U_{\cU}^{-}]-[U])-\beta_{\cT}[\righterror{1}{\cU}{U'}|_{\cT}] \\
& = -\ind{\cU}{\cT}[U]+\ind{\cU}{\cT}[U_{\cU}^{-}]-\beta_{\cT}\bigl[\stabindbar{\cU}{\cT}[\simpmod{\cU}{U}]\bigr]_- \\
& = -\ind{\cU}{\cT}[U]+\Bigl(\sum_{W\in \cU\setminus U} [\exchmatentry{W,U}]_{-}\ind{\cU}{\cT}[W]\Bigr)-\beta_{\cT}\bigl[\stabindbar{\cU}{\cT}[\simpmod{\cU}{U}]\bigr]_-.
\end{align*}
If $V\ne U'$, then $V\in\cU\cap\cU'$ and so $\ind{\cU'}{\cT}[V]=\ind{\cC}{\cT}[V]=\ind{\cU}{\cT}[V]$.
This gives \eqref{eq:g-vec-mut-formula}.

Comparing \eqref{eq:g-vec-mut-formula} to \cite[Eq.~4.13]{NakanishiBook} and \eqref{eq:c-vec-mut-formula} to \cite[Eq.~4.5]{NakanishiBook}, we see that the index and its adjoint satisfy the mutation formulæ for $\mathbf{g}$- and $\mathbf{c}$-vectors respectively; the assumption that $(\cC,\cT)$ has a cluster structure means that the assumption of no loop or $2$-cycle at $U\in\cU\ctsubcat\cC$ holds whenever $\cU$ is reachable from $\cT$, and so these formulæ are always valid.
Since these maps are the identity on the initial cluster-tilting subcategory $\cT$, and $\mathbf{g}$- and $\mathbf{c}$-vectors are the standard basis vectors on the initial seed, we conclude that the index and its adjoint indeed categorify the classical cluster algebra notions.
\end{proof}

As a consequence, claims about $\mathbf{g}$- and $\mathbf{c}$-vectors in a cluster algebra follow from the existence of a categorification.
For example, it is now immediate from Proposition~\ref{p:ind-coind-inverse} (as observed in the remark after Definition~\ref{d:g-vectors}) that the $\mathbf{g}$-vectors associated to a single cluster are $\integ$-linearly independent (\cite[Conj.~7.10]{FZ-CA2}, proved in \cite{DehyKeller} for $\cC$ triangulated and \cite[Thm.~5.5(b)]{FuKeller} for $\cC$ exact); see also Remark~\ref{r:g-vecs-distinguish-cl-mons}.

\begin{remark}
\label{r:gen-c-vec-mut}
The more general formulæ in Theorem~\ref{t:one-step-X-mut} show that the values of $\indbar{\cU}{\cT}$ on simple modules still mutate with $\cU$ according to the usual rules for $\mathbf{c}$-vectors, even when there are loops and $2$-cycles, but with a modified definition of the exchange matrix $B_{\cU}$ so that it involves the quantities $\dimdivalg{U}^{-1}\canform{[\simpmod{\cU'}{V}]}{\exchmon{U}{\cU}{\pm}}{\cU'}$.
Corresponding $\mathbf{g}$-vector mutation formulæ may then be deduced exactly as in the proof of Theorem~\ref{t:c-vec-mut-formula}.
\end{remark}

\subsection{Change of cluster-tilting subcategory}
\label{s:change-of-cts-beta}

The careful analysis of compositions of index and coindex maps from the previous subsection also allows us to prove the following theorem, which we will use to relate forms on different cluster tilting subcategories via these maps, and deduce (in the appropriate context) that our exchange matrices mutate in the expected way.

\begin{theorem}\label{t:exch-isos} Let $\cC$ be a cluster category such that $\stab{\cC}$ is compact or skew-symmetric.  Then we have commutative diagrams
\[\begin{tikzcd}
\Kgp{\fpmod \stab{\cU}} \arrow{r}{\beta_{\cU}} \arrow{d}[swap]{\stabIndbar{\cU}{\cT}} & \Kgp{\cU} \arrow{d}{\ind{\cU}{\cT}} \\
\Kgp{ \fpmod \stab{\cT}} \arrow{r}{\beta_{\cT}} & \Kgp{\cT} 
\end{tikzcd}\qquad
\begin{tikzcd}
\Kgp{\fpmod \stab{\cU}} \arrow{r}{\beta_{\cU}} \arrow{d}[swap]{\stabCoindbar{\cU}{\cT}} & \Kgp{\cU} \arrow{d}{\coind{\cU}{\cT}} \\
\Kgp{ \fpmod \stab{\cT}} \arrow{r}{\beta_{\cT}} & \Kgp{\cT}
\end{tikzcd}\]
for any $\cTU\ctsubcat\cC$.
\end{theorem}

\begin{proof}
Recall from Proposition~\ref{p:equiv-to-mod} that $\Kgp{\fpmod \underline{\cU}}$ is spanned by classes of modules of the form $\Extfun{\cU}X$, for $X\in\cC$. By Theorem~\ref{t:ind-coind-additive-error-terms}\ref{t:ind-additive-error-terms} and \ref{t:ind-additive-error-terms-2}, we have
\[\ind{\cU}{\cT}\coind{\cC}{\cU}[X]-\beta_{\cT}[\lefterror{1}{\cU}{X}|_{\cU}]=\ind{\cU}{\cT}\ind{\cC}{\cU}[X]-\beta_{\cT}[\righterror{1}{\cU}{X}|_{\cT}],\]
both being equal to $\ind{\cC}{\cT}[X]$. Since
$\stabIndbar{\cU}{\cT}[\Extfun{\cU}{X}]=[\lefterror{1}{\cU}{X}|_{\cT}]-[\righterror{1}{\cU}{X}|_{\cT}]$,
rearranging gives
\[\ind{\cU}{\cT}\coind{\cC}{\cU}[X]-\ind{\cU}{\cT}\ind{\cC}{\cU}[X]=\beta_{\cT}(\stabIndbar{\cU}{\cT}[\Extfun{\cU}X]).\]
But the left-hand side is, by definition, equal to
$\ind{\cU}{\cT}[\beta_{\cU}(\Extfun{\cU}X)]$.
Commutativity of the right-hand square is proved similarly, using Theorem~\ref{t:ind-coind-additive-error-terms}\ref{t:coind-additive-error-terms} and \ref{t:coind-additive-error-terms-2}.
\end{proof}

To be able to decategorify Theorem~\ref{t:exch-isos} to statements about matrices, we need to be able to restrict to Grothendieck groups of finite dimensional modules (with their bases of simples), which we may do using Corollary~\ref{c:indbar-on-fd}.

\begin{corollary}
\label{c:exch-isos}
If $\cT$ and $\cU$ are maximally mutable, then we have commutative diagrams
\[\begin{tikzcd}
\Kgp{\fd \stab{\cU}} \arrow{r}{\beta_{\cU}} \arrow{d}[swap]{\stabindbar{\cU}{\cT}} & \Kgp{\cU} \arrow{d}{\ind{\cU}{\cT}} \\
\Kgp{ \fd \stab{\cT}} \arrow{r}{\beta_{\cT}} & \Kgp{\cT}
\end{tikzcd}\qquad
\begin{tikzcd}
\Kgp{\fd \stab{\cU}} \arrow{r}{\beta_{\cU}} \arrow{d}[swap]{\stabcoindbar{\cU}{\cT}} & \Kgp{\cU} \arrow{d}{\coind{\cU}{\cT}} \\
\Kgp{ \fd \stab{\cT}} \arrow{r}{\beta_{\cT}} & \Kgp{\cT}
\end{tikzcd}\]
whenever either $\cC$ has finite rank or $\cU$ is reachable from $\cT$.\qed
\end{corollary}

\begin{corollary}
\label{c:rank-invariant}
If $\stab{\cC}$ is Krull--Schmidt and finite rank, then the rank of $\beta_{\cT}|_{\Kgp{\fd{\stab{\cT}}}}$ is the same for all $\cT\ctsubcat\cC$.
\end{corollary}
\begin{proof}
By Corollary~\ref{c:weak-clust-struct}\ref{c:weak-clust-struct-finite}, all $\cT\ctsubcat\cC$ are maximally mutable, meaning both that the statement makes sense (i.e.\ $\Kgp{\fd{\stab{\cT}}}\leq\Kgp{\fpmod{\stab{\cT}}}$) and that we have the commutative diagrams from Corollary~\ref{c:exch-isos}.
The statement then follows since $\cind{\cT}{\cU}$ and $\stabcindbar{\cT}{\cU}$ are isomorphisms by Propositions~\ref{p:ind-coind-inverse} and \ref{p:indbar-coindbar-inverse}.
\end{proof}

\begin{remark}\label{r:decategorification-of-square}
When $\cC$ is triangulated, expressing the commutativity of the left-hand square from Corollary~\ref{c:exch-isos} in matrix form recovers a formula due to Palu \cite[Thm.~12(a)]{Palu-Groth-gp}.

When $\cT$ and $\cU$ are related by a single mutation, Corollary~\ref{c:exch-isos} (in combination with Corollary~\ref{c:one-step-X-mut-from-root}) is the categorification of the matrix mutation formula as written by Gekhtman--Shapiro--Vainshtein \cite{GSV-Moscow} and described explicitly as a matrix product in the proof of \cite[Lem.~3.2]{BFZ-CA3}.  Namely, for $\epsilon\in \{\pm 1 \}$, $m=\card{\indec \cT}$ and $n=\card{\indec \stab{\cT}}$, we may define an $m\times m$ matrix $E_{\epsilon}(k)$ and an $n\times n$ matrix $F_{\epsilon}(k)$ with entries
\[ E_{\epsilon}(k)_{ij}=\begin{cases} \delta_{ij} & \text{if}\ j\neq k, \\ -1 & \text{if}\ i=j=k, \\ [-\epsilon b_{ik}]_+ & \text{if}\ i\neq j=k, \end{cases}\qquad
F_{\epsilon}(k)_{ij}=\begin{cases} \delta_{ij} & \text{if}\ i\neq k, \\ -1 & \text{if}\ i=j=k, \\ [\epsilon b_{kj}]_+ & \text{if}\ i=k\neq j. \end{cases}\]
Then the mutation of $B$ in the direction $k$ is given by
\[ \mu_{k}(B) = E_{\epsilon}(k)BF_{\epsilon}(k) \]
for either choice of $\epsilon$, this corresponding to the choice of square in Corollary~\ref{c:exch-isos}.
\end{remark}

Theorem~\ref{t:exch-isos} leads to the next collection of results, which show how the form $\sform{\blank}{\blank}{\cT}$ from Definition~\ref{d:s-form} behaves as we vary the cluster-tilting subcategory $\cT$. While we use the notation $\mu$ in the next definition, the cluster-tilting subcategories $\cT$ and $\cU$ do not need to be related by any sequence of mutations.

\begin{definition}\label{d:mutated-s-form} Let $\cC$ be a compact or skew-symmetric cluster category, and let $\cTU\ctsubcat\cC$. Define $\mu_{\cT}^{\cU}\sform{\blank}{\blank}{\cT}\colon\Kgpnum{\lfd \cU}\cross \Kgp{\fpmod \stab{\cU}}\to \integ$ by
\[ \mu_{\cT}^{\cU}\sform{\blank}{\blank}{\cT}=\sform{\indbar{\cU}{\cT}(\blank)}{\stabIndbar{\cU}{\cT}(\blank)}{\cT}. \]
\end{definition}

It turns out that this does not give us anything new, instead simply recovering the intrinsic form $\sform{\blank}{\blank}{\cU}$.

\begin{proposition}
\label{p:mut-of-Euler-form}
Let $\cC$ be a compact or skew-symmetric cluster category, and let $\cTU\ctsubcat\cC$. Then 
\begin{equation}
\label{eq:sform-invariance}
\mu_{\cT}^{\cU}\sform{\blank}{\blank}{\cT}=\sform{\blank}{\blank}{\cU}.
\end{equation}
\end{proposition}
\begin{proof}
Let $M\in\lfd{\cU}$ and $N\in \fpmod{\stab{\cU}}$. Then
\begin{align*}
\mu_{\cT}^{\cU}\sform{[M]}{[N]}{\cT}
&=\canform{\indbar{\cU}{\cT}[M]}{\beta_{\cT}\stabIndbar{\cU}{\cT}[N]}{\cT}\\
&=\canform{\indbar{\cU}{\cT}[M]}{\ind{\cU}{\cT}{\beta_{\cU}[N]}}{\cT}\\
&=\canform{\coindbar{\cT}{\cU}{\indbar{\cU}{\cT}[M]}}{\beta_{\cU}[N]}{\cU}\\
&=\canform{[M]}{\beta_{\cU}[N]}{\cU}\\
&=\sform{[M]}{[N]}{\cU}.
\end{align*}
Here the first equality is a pair of definitions. Subsequently, we apply Theorem~\ref{t:exch-isos}, \eqref{eq:ind-coindbar-adj}, Proposition~\ref{p:ind-coind-inverse} and the definition of $\sform{\blank}{\blank}{\cU}$.
\end{proof}

The next corollary is \emph{sign-invariance}, the statement that using $\coindbar{}{}$ instead of $\indbar{}{}$ to transfer the form yields the same answer.  The name for this property is derived from the matrix version recalled in Remark~\ref{r:decategorification-of-square}, whereby classical matrix mutation is expressed as multiplication by matrices $E_+$ and $F_+$ or $E_-$ and $F_-$, the choice of sign ultimately having no effect on the result.

\begin{corollary}\label{c:sign-invar-s-form}
Let $\cC$ be a compact or skew-symmetric cluster category, and let $\cTU\ctsubcat\cC$. Then $\mu_{\cT}^{\cU}\sform{\blank}{\blank}{\cT}=\sform{\coindbar{\cU}{\cT}(\blank)}{\stabCoindbar{\cU}{\cT}(\blank)}{\cT}$.
\end{corollary}
\begin{proof}
This follows from Proposition~\ref{p:mut-of-Euler-form} and Proposition~\ref{p:ind-coind-inverse}.
\end{proof}

A second corollary is that this transportation of forms is transitive, despite the fact that compositions of indices or coindices produce `error terms' as in Corollary~\ref{c:ind-coind-mult-error-terms}.

\begin{corollary}\label{c:s-form-mut-transitive}
Let $\cC$ be a compact or skew-symmetric cluster category and $\cTU,\cV\ctsubcat\cC$. Then $\mu_{\cU}^{\cV}\mu_{\cT}^{\cU}\sform{\blank}{\blank}{\cT}=\mu_{\cT}^{\cV}\sform{\blank}{\blank}{\cT}$.
\end{corollary}
\begin{proof}
By Proposition~\ref{p:mut-of-Euler-form}, both forms are equal to $\sform{\blank}{\blank}{\cV}$.
\end{proof}

This leads to mutation formulæ for our exchange matrices, recovering the familiar expressions from cluster theory, as follows.

\begin{theorem}
\label{t:exch-mat-mutation-clust-str}
Let $\cC$ be a compact cluster category and let $\cT\ctsubcat\cC$.
Assume $\cT$ has no loop or $2$-cycle at $T\in \exch{\cT}$, and let $U,V\in \indec \mut{T}{\cT}$ with $V$ non-projective.
Then
\begin{equation}
\label{eq:exch-mat-mutation-clust-str}
\exchmatentry{U,V}^{\mut{T}{\cT}}=\begin{cases} -\exchmatentry{T,V}^{\cT} & \text{if $U=\mut{\cT}{T}$},\\
-\exchmatentry{U,T}^{\cT}&\text{if $V=\mut{\cT}{T}$}, \\
\exchmatentry{U,V}^{\cT}+\exchmatentry{U,T}^{\cT}[\exchmatentry{T,V}^{\cT}]_{+}+[\exchmatentry{U,T}^{\cT}]_{-}\exchmatentry{T,V}^{\cT} & \text{otherwise.} \end{cases}
\end{equation} 
\end{theorem}

\begin{proof}
Recall from Corollary~\ref{c:exch-mat-at-mut} that $\exchmatentry{U,\mut{\cT}{T}}^{\mut{T}{\cT}}=-\exchmatentry{U,T}^{\cT}$.
Then \[ \exchmatentry{\mut{\cT}{T},V}^{\mut{T}{\cT}}=-\frac{\dimdivalg{V}}{\dimdivalg{\mut{\cT}{T}}}\exchmatentry{V,\mut{\cT}{T}}^{\mut{T}{\cT}}=-\frac{\dimdivalg{V}}{\dimdivalg{T}}\exchmatentry{V,\mut{\cT}{T}}^{\mut{T}{\cT}}=\frac{\dimdivalg{V}}{\dimdivalg{T}}\exchmatentry{V,T}^{\cT}=-\exchmatentry{T,V}^{\cT}, \]
using Lemma~\ref{l:d-of-mutant} for the second equality, and Corollary~\ref{c:exch-mat-at-mut} again for the third.
Now let $U\in\indec{\mut{T}{\cT}}$ and $V\in\exch{\mut{T}{\cT}}$ both be different from $\mut{\cT}{T}$.
By Corollary~\ref{c:indbar-on-fd}, we have $\stabindbar{\mut{T}{\cT}}{\cT}[\simpmod{\mut{T}{\cT}}{V}]\in\Kgp{\fd{\stab{\cT}}}$, and so we may calculate
\begin{align*}
\exchmatentry{U,V}^{\mut{T}{\cT}} & = \dimdivalg{U}^{-1}\sform{[\simpmod{\mut{T}{\cT}}{U}]}{[\simpmod{\mut{T}{\cT}}{V}]}{\mut{T}{\cT}} \\
& = \dimdivalg{U}^{-1}\sform{\indbar{\mut{T}{\cT}}{\cT}[\simpmod{\mut{T}{\cT}}{U}]}{\stabindbar{\mut{T}{\cT}}{\cT}[\simpmod{\mut{T}{\cT}}{V}]}{\cT} \\
& = \dimdivalg{U}^{-1}\sform{[\simpmod{\cT}{U}]+[\exchmatentry{T,U}^{\cT}]_{+}[\simpmod{\cT}{T}]}{[\simpmod{\cT}{V}]+[\exchmatentry{T,V}^{\cT}]_{+}[\simpmod{\cT}{T}]}{\cT} \\
& = \begin{multlined}[t]\dimdivalg{U}^{-1}\sform{[\simpmod{\cT}{U}]}{[\simpmod{\cT}{V}]}{\cT}+\dimdivalg{U}^{-1}[\exchmatentry{T,V}^{\cT}]_{+}\sform{[\simpmod{\cT}{U}]}{[\simpmod{\cT}{T}]}{\cT} +\dimdivalg{U}^{-1}[\exchmatentry{T,U}^{\cT}]_{+}\sform{[\simpmod{\cT}{T}]}{[\simpmod{\cT}{V}]}{\cT}\\+\dimdivalg{U}^{-1}[\exchmatentry{T,U}^{\cT}]_{+}[\exchmatentry{T,V}^{\cT}]_{+}\sform{[\simpmod{\cT}{T}]}{[\simpmod{\cT}{T}]}{\cT}\end{multlined} \\
& = \exchmatentry{U,V}^{\cT}+\exchmatentry{U,T}^{\cT}[\exchmatentry{T,V}^{\cT}]_{+}+\dimdivalg{U}^{-1}\dimdivalg{T}\exchmatentry{T,V}^{\cT}[\exchmatentry{T,U}^{\cT}]_{+} \\
& = \exchmatentry{U,V}^{\cT}+\exchmatentry{U,T}^{\cT}[\exchmatentry{T,V}^{\cT}]_{+}+[\exchmatentry{U,T}^{\cT}]_{-}\exchmatentry{T,V}^{\cT}
\end{align*}
using Proposition~\ref{p:mut-of-Euler-form} and Corollary~\ref{c:one-step-X-mut-from-root}.
\end{proof}

\sectionbreak
\section{Cluster characters}\label{s:clust-char}

In this section, we will first explain the construction of the usual cluster character, taking objects of a cluster category $\cC$ to Laurent polynomials, such that for suitable inputs one obtains $\Aside$-cluster variables.  We will refer to this as the $\Aside$-cluster character.  

We are working in a slightly more general framework and different notation than is in the extant literature, so for these reasons and also in order to better facilitate the new construction that follows, we will give the definition and proofs of properties of the $\Aside$-cluster character in a little detail.
In particular, we still do not assume that our perfect ground field $\bK$ is algebraically closed, although at this point we will need to assume that it admits an Euler--Poincaré characteristic in the sense of Definition~\ref{d:Euler-char}, as algebraically closed fields do.

Our main goal, however, is to introduce the corresponding $\cX$-cluster character, defined on module categories for the cluster-tilting categories of $\stab{\cC}$.
There are several distinctly different features of the $\cX$-cluster character, starting with its domain of definition, and we will highlight these as we go along.

\subsection{Quiver Grassmannians}
\label{s:qGrass}

\begin{definition}
Let $\cA$ be a Krull--Schmidt category.
For $M\in \lfd\cA$ and $[L]\in\Kgp{\fd\cA}$, the \emph{quiver Grassmannian} $\QGra{[L]}{M}$ is the (algebraic) moduli space whose points parametrise submodules $L'\leq M$ with $[L']=[L]\in \Kgp{\fd \cA}$.
\end{definition}

The space $\QGra{[L]}{M}$ is a subvariety of the product $\prod_{X\in\indec{\cA}}\QGra{\dim_{\bK}L(X)}{M(X)}$ of ordinary Grassmannians; this is a finite product because $L\in\fd{\cA}$, so $L(X)=0$ for all but finitely many $X\in\indec{\cA}$.
In particular, $\QGra{[L]}{M}$ is a projective variety.

\begin{definition}
\label{d:Euler-char}
An \emph{Euler--Poincaré characteristic} is a function $\chi\colon\Var_\bK\to\integ$ on the set of proper algebraic varieties over the (perfect) field $\bK$, having the following properties (cf.~\cite[Rem.~2.12]{Plamondon-CCs}):
\begin{enumerate}
\item\label{d:Euler-char-affine} $\chi(\mathbb{A}_\bK^n)=1$, for $\mathbb{A}_\bK^n$ the affine space of dimension $n\geq0$;
\item\label{d:Euler-char-union} $\chi(U)=\chi(V_1)+\chi(V_2)$ if $U=V_1\sqcup V_2$ for constructible subsets $V_1$ and $V_2$;
\item\label{d:Euler-char-fibres} if $f\colon U\to V$ is a surjective constructible map (e.g.\ a morphism of algebraic varieties) with $\chi(f^{-1}(v))=c$ independent of $v\in V$, then $\chi(U)=c\chi(V)$.
\end{enumerate}
\end{definition}

A consequence of Definition~\ref{d:Euler-char}\ref{d:Euler-char-fibres} is that $\chi(U\times V)=\chi(U)\chi(V)$ for any $U,V\in\Var_\bK$.

\begin{example}
If $\bK$ is algebraically closed, then the usual Euler--Poincaré characteristic, defined for example using étale or $\ell$-adic cohomology with compact support, satisfies the conditions from Definition~\ref{d:Euler-char}. For a general (perfect) field $\bK$, it is not clear that an Euler--Poincaré characteristic exists.
\end{example}

Since quiver Grassmannians are projective varieties, they are in particular proper, and so are suitable inputs for an Euler--Poincaré characteristic.

\begin{example}
If $\chi$ is an Euler--Poincaré characteristic, then $\chi(\mathbb{P}_{\bK}^n)=n+1$.
Indeed, for $n \geq1$ we may decompose $\mathbb{P}_{\bK}^n$ into a pair of constructible subsets, one isomorphic to $\mathbb{A}_{\bK}^n$ and the other to $\mathbb{P}_{\bK}^{n-1}$.
Since $\mathbb{P}_{\bK}^{0}=\mathbb{A}_{\bK}^0$, we obtain the desired result by induction.
\end{example}

From now on, we fix a choice of Euler--Poincaré characteristic $\chi$ (and so, implicitly, assume that the field $\bK$ admits one).
From \cite[\S1.6]{JorgensenPalu} and \cite{Palu2} (building on the foundational work of \cite{CalderoChapoton,CalderoKeller,Palu}) we have the following two key identities.

\begin{proposition}
\label{p:chidentities}
Let $\cC$ be a cluster category and let $\cT\ctsubcat\cC$.
\begin{enumerate}
\item\label{p:chidentities-split} For any $M,N\in\lfd{\stab{\cT}}$ and $L\in\fd{\stab{\cT}}$, we have
\begin{equation}\label{eq:chi-split}
\chi(\QGra{[L]}{M\oplus N})=\sum_{[H]+[K]=[L]} \chi(\QGra{[H]}{M})\chi(\QGra{[K]}{N}).
\end{equation}
\item\label{p:chidentities-non-split} 
Assume $\cC$ is compact or skew-symmetric, and let $X,Y\in\cC$ be indecomposable with $\rank_{\divalg{X}}\Ext{1}{\cC}{X}{Y}=1$.
Given non-split conflations $Y\stackrel{i_{1}}{\infl} Z\stackrel{p_{1}}{\defl} X\confl$ and $X\stackrel{i_{2}}{\infl} Z'\stackrel{p_{2}}{\defl} Y\confl$, we have
\begin{equation}\label{eq:chi-non-split} \chi(\QGra{[H]}{\Extfun{\cT}X})\chi(\QGra{[K]}{\Extfun{\cT}Y})= \chi(\curly{S}(Z)_{[H],[K]})+\chi(\curly{S}(Z')_{[H],[K]}) \end{equation}
for each $H,K\in\fd{\stab{\cT}}$, where 
\begin{align*} \curly{S}(Z)_{[H],[K]} & =\{ A\leq \Extfun{\cT}Z \mid [(\Extfun{\cT}i_{1})^{-1}A]=[H],\ [(\Extfun{\cT}p_{1})A]=[K]\}, \\
\curly{S}(Z')_{[H],[K]} & =\{ B\leq \Extfun{\cT}Z' \mid [(\Extfun{\cT}i_{2})^{-1}B]=[H],\ [(\Extfun{\cT}p_{2})B]=[K]\}.
\end{align*}
\end{enumerate}
\end{proposition}
\begin{proof}
The arguments from \cite[\S1.6]{JorgensenPalu} apply to prove \ref{p:chidentities-split} as written, and \ref{p:chidentities-non-split} in the skew-symmetric case, in which $\divalg{X}\cong\bK$ and so our assumption becomes $\dim_{\bK}\Ext{1}{\cC}{X}{Y}=1$.

The remaining part of \ref{p:chidentities-non-split}, with $\dim_{\bK}{\Ext{1}{\cC}{X}{Y}}=\dimdivalg{X}$, follows from the calculations in \cite{Palu2}, as we now sketch.
Following the notation of loc.\ cit., write $L([H],[K])=\bP\Ext{1}{\cC}{X}{Y}\times\QGra{[H]}{\Extfun{\cT}{X}}\times\QGra{[K]}{\Extfun{\cT}{Y}}$.
As in \cite[\S3]{Palu2}, there is a constructible map $W^Z_{X,Y}([H],[K])\to L([H],[K])$, for $W^Z_{X,Y}([H],[K])\defeq\bP\Ext{1}{\cC}{X}{Y}\times\curly{S}(Z)_{[H],[K]}$; we denote the image by $L_1([H],[K])$ and its complement by $L_2([H],[K])$, so that $\chi(L([H],[K]))=\chi(L_1([H],[K]))+\chi(L_2([H],[K]))$.
For comparison with \cite{Palu2}, our formulation simplifies because of Lemma~\ref{l:rk1-unique-middle-term}: up to isomorphism, $Z$ is the only possible middle term of the relevant conflations.
This also lets us avoid the assumption that $\stab{\cC}$ has constructible cones, needed for the more general results of \cite{Palu2}.

Now \cite[Lem.~3.2]{Palu2} (see also \cite[Lem.~3.11]{CalderoChapoton}) implies that \[\chi(L_1([H],[K]))=\chi(W^Z_{X,Y})=\chi(\bP\Ext{1}{\cC}{X}{Y})\chi(\curly{S}(Z)_{[H],[K]})=\dimdivalg{X}\cdot\chi(\curly{S}(Z)_{[H],[K]});\]
again there is no need to sum over $Z$ since in our situation there is only one option.
It further follows from \cite[Prop.~3.4]{Palu2} (see also \cite[Prop.~5]{CalderoKeller}) that \[\chi(L_2([H],[K]))=\chi(\bP\Ext{1}{\cC}{Y}{X})\chi(\curly{S}(Z')_{[H],[K]})=\dimdivalg{X}\cdot\chi(\curly{S}(Z')_{[H],[K]}),\]
using for the second equality that $\stab{\cC}$ is $2$-Calabi--Yau so that $\dim_{\bK}{\Ext{1}{\cC}{Y}{X}}=\dimdivalg{X}$.
Here we also apply Lemma~\ref{l:rk1-unique-middle-term} to $\op{\cC}$ to see that $Z'$ is the only possible middle term of an extension from $\bP\Ext{1}{\cC}{Y}{X}$, which has rank $1$ over $\op{\divalg{X}}$.
We thus have
\begin{align*}
\dimdivalg{X}\cdot\chi(\QGra{[H]}{\Extfun{\cT}{X}})\cdot\chi(\QGra{[K]}{\Extfun{\cT}{Y}})
&=\chi(L([H],[K]))\\
&=\dimdivalg{X}\cdot\chi(\curly{S}(Z)_{[H],[K]})+\dimdivalg{X}\cdot\chi(\curly{S}(Z')_{[H],[K]}),
\end{align*}
and so dividing through by $\dimdivalg{X}$ gives our desired result.
\end{proof}

\begin{lemma}[cf.~{\cite[Lem.~5.1]{Palu}}]\label{l:apres-Palu}
For a conflation $X\stackrel{i}{\infl} Z\stackrel{\pi}{\defl} Y\confl$ and
$L\in\curly{S}(Z)_{[H],[K]}$,
\begin{enumerate}
\item\label{l:apres-Palu-HKN} $[L]=[H]+[K]-[\Ker \Extfun{\cT}i]$ and
\item\label{l:apres-Palu-ind-plus-beta} $\ind{\cC}{\cT}{([X]+[Y])}+\beta_{\cT}([H]+[K])=\ind{\cC}{\cT}{[Z]}+\beta_{\cT}[L]$.
\end{enumerate}
\end{lemma}

\begin{proof}
Since $L\in\curly{S}(Z)_{[H],[K]}$, there is a commutative diagram
\[ \begin{tikzcd}[row sep=15pt]
0 \arrow{r}{} & \Ker \Extfun{\cT}i \arrow{r}{} & \Extfun{\cT}X \arrow{r}{\Extfun{\cT}i} & \Extfun{\cT}Z \arrow{r}{\Extfun{\cT}\pi} & \Extfun{\cT}Y \\
0 \arrow{r}{} & \Ker \Extfun{\cT}i \arrow{r}{} \arrow[equal]{u}{} & H=(\Extfun{\cT}i)^{-1}L \arrow{r}{} \arrow[hookrightarrow]{u}{} & L \arrow{r}{} \arrow[hookrightarrow]{u}{} & K=\Extfun{\cT}\pi L \arrow{r}{} \arrow[hookrightarrow]{u}{} & 0
\end{tikzcd}
\]
with exact rows, the lower of which gives \ref{l:apres-Palu-HKN}.
By Proposition~\ref{p:stabind}, it suffices to prove \ref{l:apres-Palu-ind-plus-beta} for $\cC$ exact. In this case, we deduce from the exact sequence
\[\begin{tikzcd}
0\arrow{r}& \Yonfun{\cT}{X} \arrow{r}&\Yonfun{\cT}{Z} \arrow{r}& \Yonfun{\cT}{Y} \arrow{r}& \Ker \Extfun{\cT}i \arrow{r}& 0,
\end{tikzcd}\]
together with Propositions~\ref{p:ind-proj-res} and \ref{p:beta-proj-res} that
\[ \ind{\cC}{\cT}[Z]=\ind{\cC}{\cT}[X]+\ind{\cC}{\cT}[Y]+\beta_{\cT}[\Ker \Extfun{\cT}i], \]
the terms in this expression being classes of projective resolutions of the terms in the sequence.
Since $[L]=[H]+[K]-[\Ker \Extfun{\cT}i]$ by \ref{l:apres-Palu-HKN}, we have that
\begin{align*}
\beta_{\cT}[L] & = \beta_{\cT}[H]+\beta_{\cT}[K]-\beta_{\cT}[\Ker \Extfun{\cT}i] \\
& = \beta_{\cT}[H]+\beta_{\cT}[K]+\ind{\cC}{\cT}[X]+\ind{\cC}{\cT}[Y]-\ind{\cC}{\cT}[Z],
\end{align*}
as required for \ref{l:apres-Palu-ind-plus-beta}.
\end{proof}

\begin{corollary}
For a conflation $X\stackrel{i}{\infl} Z\stackrel{\pi}{\defl} Y\confl$, we have
\[\QGra{[L]}{\Extfun{\cT}Z}=\bigsqcup_{[H]+[K]=[L]+[\Ker\Extfun{\cT}i]}\curly{S}(Z)_{[H],[K]},\]
and hence
\begin{equation}
\label{eq:chi-mid-term}
\chi\bigl(\QGra{[L]}{\Extfun{\cT}Z}\bigr)=\sum_{[H]+[K]=[L]+[\Ker{\Extfun{\cT}i}]}\chi\bigl(\curly{S}(Z)_{[H],[K]}\bigr).
\end{equation}
\end{corollary}
\begin{proof}
If $L\leq\Extfun{\cT}Z$ then $L\in\curly{S}(Z)_{[(\Extfun{\cT}i)^{-1}L],[\Extfun{\cT}\pi L]}$, so
\[\bigsqcup_{[L]}\QGra{[L]}{\Extfun{\cT}Z}=\bigsqcup_{[H],[K]}\curly{S}(Z)_{[H],[K]}.\]
By Lemma~\ref{l:apres-Palu}\ref{l:apres-Palu-HKN}, we have $\curly{S}(Z)_{[H],[K]}\subseteq\QGra{[L]}{\Extfun{\cT}Z}$ if and only if $[H]+[K]=[L]+[\Ker\Extfun{\cT}i]$, and the result follows.
\end{proof}

\subsection{\texorpdfstring{$\Fpoly$-polynomials}{F-polynomials}}
\label{s:Fpolys}

Let $\bK\Kgp{\cT}$ be the group algebra of $\Kgp{\cT}$, written as
\[ \bK\Kgp{\cT}=\operatorname{span}_{\bK}\{ a^{t} \mid t\in \Kgp{\cT} \} \]
with multiplication defined on basis elements by $a^{t}a^{u}=a^{t+u}$ and extended linearly.
Similarly,
\[\bK\Kgp{\fd{\stab{\cT}}}=\operatorname{span}_{\bK}\{ x^{m} \mid m\in \Kgp{\fd{\stab{\cT}}} \} \]
denotes the group algebra of $\Kgp{\fd{\stab{\cT}}}$.
We use the formal symbols `$a$' and `$x$' here to match $\Aside$ and $\Xside$.
For $\cT$ maximally mutable, the map $\beta_\cT\colon\Kgp{\fd{\stab{\cT}}}\to\Kgp{\cT}$ induces a map $(\beta_{\cT})_*\colon\bK\Kgp{\fd{\stab{\cT}}}\to\bK\Kgp{\cT}$ by
\[(\beta_{\cT})_*x^m=a^{\beta_{\cT}m}.\]
In the finite rank case, the $\Aside$-cluster character will take values in $\bK\Kgp{\cT}$, and its definition will involve $(\beta_{\cT})_*$, and the $\Xside$-cluster character will take values in the field of fractions of $\bK\Kgp{\fd{\stab{\cT}}}$.
To write down our cluster characters in the infinite rank case, we need to enlarge these algebras slightly and, for the $\Aside$-cluster character, impose a condition on $\beta_{\cT}$ so that it still induces a well-defined map.

Let $\bV$ be a free abelian group.
We write
\[\powser{\bK}{\bV}=\Bigl\{\sum_{v\in\bV}\lambda_vy^v\mid \lambda_v\in\bK\Bigr\}\]
for the set of (possibly) infinite linear combinations of formal symbols $y^v$, for $v\in\bV$, with coefficients in $\bK$.
The formal symbols are sometimes denoted with different letters; in particular, we will use $a^{v}$ in place of $y^{v}$ when $\mathbb{V}=\Kgp{\cT}$ and $x^{v}$ when $\mathbb{V}=\Kgp{\fd \stab{\cT}}$, compatible with their group algebras considered above.

While $\powser{\bK}{\bV}$ is a $\bK$-vector space, with addition and scalar multiplication defined termwise, attempting to define multiplication by
\begin{equation}
\label{eq:pseudo-poly-mult}
\Bigl(\sum_{u\in \bV}\lambda_{u}y^{u}\Bigr)\Bigl(\sum_{v\in\bV}\rho_{v}y^{v}\Bigr)=\sum_{w\in\mathbb{V}}\Bigl(\sum_{u+v=w}\lambda_{u}\rho_{v}\Bigr)y^{w}
\end{equation}
leads to the issue that the coefficient $\sum_{u+v=w}\lambda_{u}\rho_{v}$ in the result need not be a finite sum.
Nonetheless, \eqref{eq:pseudo-poly-mult} does define an algebra structure on various subspaces of $\powser{\bK}{\bV}$, such as the group algebra $\bK\bV$, which identifies naturally with the subset of finite linear combinations.

Let $\cB$ be a basis for $\bV$.  For $v\in \bV$ and $b\in \cB$, let $\ip{v}{b}\in \integ$ denote the coefficient of $b$ in the expansion of $v$ with respect to the basis $\cB$.
We give $\bV$ a poset structure with respect to $\cB$ by declaring that $v\leq_{\cB} w$ if $\ip{v}{b}\leq \ip{w}{b}$ for all $b\in \cB$.

\begin{definition}\label{d:pseudo-poly}
Given a free abelian group $\bV$ with basis $\cB$, a \emph{Laurent pseudo-polynomial} in $\bV$ is $p=\sum_{v\in\mathbb{V}} \lambda_{v}y^{v}\in\powser{\bK}{\mathbb{V}}$
such that
\begin{enumerate}
\item\label{d:pseudo-poly-lower-bound} there exists $v_0\in\bV$ such that $\lambda_v=0$ unless $v\geq_{\cB}v_0$, and
\item for every $b\in \cB$, the set $ \mathcal{K}(p,b)=\{ \ip{v}{b} \mid \lambda_{v}\neq 0 \} \subseteq \integ $ is bounded.
\end{enumerate}
Write $\kappa^+(p,b)\defeq\max \mathcal{K}(p,b)$, which we call the \emph{$b$-degree} of $p$, and $\kappa^-(p,b)\defeq\min\curly{K}(p,b)$.
We write $\Laurent{\mathbb{V}}$ for the set of Laurent pseudo-polynomials in $\mathbb{V}$ with respect to $\cB$.
We call $p\in\Laurent{\mathbb{V}}$ just a \emph{pseudo-polynomial} if we may take $v_0=0$ in \ref{d:pseudo-poly-lower-bound}.
\end{definition}

The condition in Definition~\ref{d:pseudo-poly-lower-bound} is equivalent to requiring that $\kappa^{-}(p,b)\geq0$ for all but finitely many $b$.
Indeed, one can then take $v_0=\sum_{b\in\cB}\min\{\kappa^-(p,b),0\}b$, a finite sum because of this property of the $\kappa^-(p,b)$.

A further consequence of the definition is that $\Laurent{\mathbb{V}}\subseteq\powser{\bK}{\integ(\cB\setminus\{b\})}[y^{\pm b}]$ for each $b\in\cB$.
That is, a Laurent pseudo-polynomial $p$ is an ordinary Laurent polynomial in any given $y^{b}$ (and even a polynomial for all but finitely many $b$), with coefficients being (possibly) infinite series in the remaining variables.

\begin{remark}
\label{r:pseudo-poly-finite}
There is a natural inclusion $\bK\mathbb{V}\subseteq\Laurent{\mathbb{V}}$, with the group algebra consisting of the finite sums in $\Laurent{\mathbb{V}}$. If $\cB$ is finite, this inclusion is even an equality.
The next result gives a natural algebra structure on $\Laurent{\mathbb{V}}$ (for general $\cB$), extending that on $\bK\mathbb{V}$.
In particular, the above inclusion makes $\Laurent{\mathbb{V}}$ into a $\bK\mathbb{V}$-module.
\end{remark}

\begin{lemma}\label{l:pseudo-poly-alg} Let $\bV$ be a free abelian group, $\cB$ a basis for $\bV$ and $\bK$ a field. Then $\Laurent{\bV}$ is a subspace of $\powser{\bK}{\bV}$, and a $\bK$-algebra with multiplication \eqref{eq:pseudo-poly-mult}.
\end{lemma}

\begin{proof}
It is straightforward to see that $\Laurent{\mathbb{V}}$ is closed under addition and scalar multiplication, for if $\mathcal{K}(p,b)$ and $\mathcal{K}(q,b)$ are bounded, or equivalently finite, so are $\mathcal{K}(p+q,b)\subseteq\mathcal{K}(p,b)\cup\mathcal{K}(q,b)$ and $\mathcal{K}(\alpha p,b)$, which is equal to $\mathcal{K}(p,b)$ if $\alpha\ne0$, and equal to $\{0\}$ otherwise.
Since $\kappa^-(p+q,b)\geq\min\{\kappa^-(p,b),\kappa^-(q,b)\}$ and $\kappa^-(\alpha p,b)\in\{0,\kappa^-(p,b)\}$, we have $\kappa^-(p+q,b),\kappa^-(\alpha p,b)\geq0$ for all but finitely many $b$, and so $p+q,\alpha p\in\Laurent{\mathbb{V}}$ as required.

Now let $p,q\in\Laurent{\bV}$, with $p=\sum_{u\geq_{\cB} u_0}\lambda_{u}y^{u}$ and $q=\sum_{v\geq_{\cB} v_0}\rho_{v}y^{v}$.
Now $u\geq_{\cB}u_0$ and $v\geq_{\cB}v_0$ implies that $u+v\geq_{\cB}u_0+v_0$.
Moreover, for a given $w\geq_{\cB}u_0+v_0$, finding $u\geq_{\cB}u_0$ and $v\geq_{\cB}v_0$ such that $u+v=w$ amounts to choosing, for each $b\in\cB$, the values $\ip{u}{b}$ and $\ip{v}{b}$, subject to the conditions $\ip{u_0}{b}\leq\ip{u}{b}$, $\ip{v_0}{b}\leq\ip{u}{b}$ and $\ip{u}{b}+\ip{v}{b}=\ip{w}{b}$.
In particular, this forces the inequalities
\begin{gather*}
\ip{u_0}{b}\leq \ip{u}{b}\leq\ip{w}{b}-\ip{v_0}{b},\\
\ip{v_0}{b}\leq\ip{v}{b}\leq\ip{w}{b}-\ip{u_0}{b}.
\end{gather*}
Because $\cB$ is a basis, we have $\ip{u_0}{b}=\ip{v_0}{b}=\ip{w}{b}=0$ for all but finitely many $b$, in which case the only solution is $\ip{u}{b}=\ip{v}{b}=0$.
In the remaining cases, the above inequalities leave only finitely many possibilities for $u$ and $v$.
Thus, the sum $\sum_{u+v=w} \lambda_{u}\rho_{v}$ is finite and hence \eqref{eq:pseudo-poly-mult} is a well-defined expression.

It remains to show that $pq\in \Laurent{\mathbb{V}}$.  But this follows by the same considerations as for ordinary polynomials: as above, if $u\geq_{\cB}u_0$ and $v\geq_\cB v_0$ then $u+v\geq_{\cB}u_0+v_0$, and moreover $\kappa^+(pq,b)=\kappa^+(p,b)+\kappa^+(q,b)$, as can be seen by multiplying any terms from $p$ and $q$ evidencing that $\kappa(p,b)$ and $\kappa(q,b)$ are the $b$-degrees of the respective pseudo-polynomials.
\end{proof}

\begin{remark}
\label{r:pseudo-poly-factor}
With this multiplication, a Laurent pseudo-polynomial $p$ may always be factored as $p=y^{v_0}p'$, where $v_0\in\mathbb{V}$ and $p'$ is a pseudo-polynomial.
\end{remark}

\begin{definition}
Let $\cC$ be a Krull--Schmidt cluster category and $\cT\ctsubcat\cC$.
Define $\Laurent{\Kgp{\fd{\stab{\cT}}}}$ as above, taking the basis $\cB$ to be that consisting of classes of simple modules $\simpmod{\cT}{T}$ for $T\in\indec{\cT}$, and define the \emph{$\Fpoly$-polynomial} of $M\in \lfd \stab{\cT}$ to be
\[ \Fpoly(M)=\sum_{[L]\in \Kgp{\fd \stab{\cT}}} \chi(\QGra{[L]}{M})x^{[L]} \in \Laurent{\Kgp{\fd{\stab{\cT}}}}.\]
\end{definition}

\begin{remark}
\label{r:Fpoly-props}
The $\Fpoly$-polynomial of $M$ is non-zero since the zero submodule of $M$ gives rise to at least one non-zero term.
If $[L]$ is not the class of a submodule of $M$, then $\QGra{[L]}{M}$ is empty and its Euler characteristic is zero.
Thus $\Fpoly(M)$ is a pseudo-polynomial, with $0=\kappa^-(\Fpoly(M),[\simpmod{\cT}{T}])$ and  $\kappa^+(\Fpoly(M),[\simpmod{\cT}{T}])\leq\dim_{\bK}M(T)$ for all $T\in\indec{\cT}$.
If $M\in\fd{\stab{\cT}}$, then $\Fpoly(M)\in\bK\Kgp{\fd{\stab{\cT}}}$ is an element of the ordinary group algebra, and if $M\iso N$ then $\Fpoly(M)=\Fpoly(N)$, as is visible from the formula.
\end{remark}

\begin{proposition}
\label{p:Fpoly-zero}
Let $\cC$ be a Krull--Schmidt cluster category, let $\cT\ctsubcat\cC$, and let $M\in\fpmod{\stab{\cT}}$.
Then $\Fpoly(M)=1$ if and only if $M=0$.
\end{proposition}
\begin{proof}
It follows directly from the definition that $\Fpoly(0)=1$.
Conversely, $M\in\fpmod{\stab{\cT}}$ is finitely copresented by Corollary~\ref{c:fp-fcp}, and hence if $M\ne0$ then $M$ has non-zero socle.
In particular, there exists $T\in\indec{\stab{\cT}}$ such that $\simpmod{\cT}{T}\hookrightarrow M$, leading to an $x^{[\simpmod{\cT}{T}]}$-term in $\Fpoly(M)\ne1$.
\end{proof}

\begin{proposition}
\label{p:Fpoly-rels}
Let $\cC$ be a Krull--Schmidt cluster category and let $\cT\ctsubcat\cC$.  
\begin{enumerate}
\item 
Let $M,N\in\lfd{\stab{\cT}}$.  Then
\begin{equation}\label{eq:F-poly-split} \Fpoly(M\oplus N)=\Fpoly(M)\Fpoly(N). \end{equation}
\item\label{p:F-poly-non-split} Assume $\cC$ is compact or skew-symmetric and let $X,Y\in\cC$ with $X$ indecomposable and  $\rank_{\divalg{X}}\Ext{1}{\cC}{X}{Y}=1$.	Then for non-split conflations $Y\stackrel{i_{1}}{\infl} Z\stackrel{p_{1}}{\defl} X\confl$ and $X\stackrel{i_{2}}{\infl} Z' \stackrel{p_{2}}{\defl} Y\confl$, we have
\begin{equation}\label{eq:F-poly-non-split}
\Fpoly(\Extfun{\cT}X)\Fpoly(\Extfun{\cT}Y)=x^{[\Ker \Extfun{\cT}i_{1}]}\Fpoly(\Extfun{\cT}Z)+x^{[\Ker \Extfun{\cT}i_{2}]}\Fpoly(\Extfun{\cT}Z').
\end{equation}
\end{enumerate}
\end{proposition}

\begin{proof} {\ }
\begin{enumerate}
\item Using \eqref{eq:chi-split}, we compute
\begin{align*}
\Fpoly(M\oplus N) & = \sum_{[L]} \chi(\QGra{[L]}{M\oplus N})x^{[L]} \\
& =\sum_{[L]}\Bigl( \sum_{[H]+[K]=[L]} \chi(\QGra{[H]}{M})\chi(\QGra{[K]}{N})\Bigr)x^{[L]} \\
& = \Bigl( \sum_{[H]} \chi(\QGra{[H]}{M})x^{[H]}\Bigr)\Bigl( \sum_{[K]} \chi(\QGra{[K]}{N})x^{[K]}\Bigr) \\
& = \Fpoly(M)\Fpoly(N).	
\end{align*}
\item In this case, we have	
{\small\begin{align*} \Fpoly(\Extfun{\cT}X)\Fpoly(\Extfun{\cT}Y) & = \biggl(\sum_{[H]}\chi(\QGra{[H]}{\Extfun{\cT}X})x^{[H]}\biggr)\biggl(\sum_{[K]}\chi(\QGra{[K]}{\Extfun{\cT}Y})x^{[K]}\biggr) \\
& = \sum_{[M]} \biggl(\sum_{[H]+[K]=[M]}  \chi(\QGra{[H]}{\Extfun{\cT}X})\chi(\QGra{[K]}{\Extfun{\cT}Y}\biggr)x^{[M]} \\
& \leftstackrel{\eqref{eq:chi-non-split}}{=} \sum_{[M]} \biggl(\sum_{[H]+[K]=[M]} \chi\bigl(\curly{S}(Z)_{[H],[K]}\bigr)+\chi\bigl(\curly{S}(Z')_{[H],[K]}\bigr)\biggr)x^{[M]} \\	
& =\begin{multlined}[t] \sum_{[M]} \biggl(\sum_{[H]+[K]=[M]} \chi\bigl(\curly{S}(Z)_{[H],[K]}\bigr)\biggr)x^{[M]}\\+\sum_{[M]}\biggl(\sum_{[H]+[K]=[M]}\chi\bigl(\curly{S}(Z')_{[H],[K]}\bigr)\biggr)x^{[M]}\end{multlined}\\
& =\begin{multlined}[t] \sum_{[L]} \biggl(\sum_{[H]+[K]=[L]+[\Extfun{\cT}i_1]} \chi\bigl(\curly{S}(Z)_{[H],[K]}\bigr)\biggr)x^{[L]+[\Extfun{\cT}i_1]}\\+\sum_{[L]}\biggl(\sum_{[H]+[K]=[L]+[\Extfun{\cT}i_2]}\chi\bigl(\curly{S}(Z')_{[H],[K]}\bigr)\biggr)x^{[L]+[\Extfun{\cT}i_2]}\end{multlined}\\
& \leftstackrel{\eqref{eq:chi-mid-term}}{=}x^{[\Extfun{\cT}i_1]}\sum_{[L]} \chi\bigl(\QGra{[L]}{\Extfun{\cT}Z}\bigr)x^{[L]}+x^{[\Extfun{\cT}i_2]}\sum_{[L]}\chi\bigl(\QGra{[L]}{\Extfun{\cT}Z'}\bigr)x^{[L]}\\
& = x^{[\Ker \Extfun{\cT}i_{1}]}\Fpoly(\Extfun{\cT}Z)+x^{[\Ker \Extfun{\cT}i_{2}]}\Fpoly(\Extfun{\cT}Z').\qedhere
\end{align*}}
\end{enumerate}
\end{proof}

The formula \eqref{eq:F-poly-non-split} applies in particular to $X=U$ and $Y=\mut{\cU}{U}$ when $\cU\ctsubcat\cC$ has no loop at $U$, and we may make the resulting formula even more explicit if there is also no 2-cycle.

\begin{theorem}\label{t:Fpoly-mutation}
Let $\cC$ be a compact cluster category with $\cTU\ctsubcat \cC$.
Let $U\in \indec \cU$ and assume there is no loop at $U$.
Then the exchange conflations $\mut{\cU}{U}\stackrel{i_{1}}{\infl} \exchmon{\cU}{U}{+}\stackrel{p_{1}}{\defl} U\confl$ and $U\stackrel{i_{2}}{\infl} \exchmon{\cU}{U}{-} \stackrel{p_{2}}{\defl} \mut{\cU}{U}\confl$ imply the relation
\begin{equation}
\label{eq:F-poly-non-split-one-step}
\Fpoly(\Extfun{\cT}U)\Fpoly(\Extfun{\cT}\mut{\cU}{U})=x^{[\stabindbar{\cU}{\cT}[\simpmod{\cU}{U}]]_{+}}\Fpoly(\Extfun{\cT}\exchmon{\cU}{U}{+})+x^{[\stabindbar{\cU}{\cT}[\simpmod{\cU}{U}]]_{-}}\Fpoly(\Extfun{\cT}\exchmon{\cU}{U}{-})
\end{equation}
between the $\Fpoly$-polynomials of the associated $\stab{\cT}$-modules.
If furthermore there is no 2-cycle at $U$, we have
\begin{equation}
\label{eq:F-poly-non-split-one-step-no-2-cycle}
\begin{split}
\Fpoly(\Extfun{\cT}U)\Fpoly(\Extfun{\cT}\mut{\cU}{U})=\begin{multlined}[t]x^{[\stabindbar{\cU}{\cT}[\simpmod{\cU}{U}]]_{+}}\prod_{V\in \indec \cU\setminus U} \Fpoly(\Extfun{\cT}V)^{[\exchmatentry{U,T}]_{+}}\\+x^{[\stabindbar{\cU}{\cT}[\simpmod{\cU}{U}]]_{-}}\prod_{V\in \indec \cU\setminus U} \Fpoly(\Extfun{\cT}V)^{[\exchmatentry{U,T}]_{-}}.\end{multlined}
\end{split}
\end{equation}
\end{theorem}

\begin{proof}
Applying Proposition~\ref{p:Fpoly-rels}\ref{p:F-poly-non-split} to the exchange conflations, we obtain
\begin{equation*}
\Fpoly(\Extfun{\cT}U)\Fpoly(\Extfun{\cT}\mut{\cU}{U})=x^{[\Ker \Extfun{\cT}i_{1}]}\Fpoly(\Extfun{\cT}\exchmon{\cU}{U}{+})+x^{[\Ker \Extfun{\cT}i_{2}]}\Fpoly(\Extfun{\cT}\exchmon{\cU}{U}{-}).
\end{equation*}
As explained in the proof of Theorem~\ref{t:c-vec-mut-formula}, we may identify the classes of the two kernels with the positive and negative parts of $\stabindbar{\cU}{\cT}[\simpmod{\cU}{U}]$, to obtain
\begin{equation*}
\Fpoly(\Extfun{\cT}U)\Fpoly(\Extfun{\cT}\mut{\cU}{U})=x^{[\stabindbar{\cU}{\cT}[\simpmod{\cU}{U}]]_{+}}\Fpoly(\Extfun{\cT}\exchmon{\cU}{U}{+})+x^{[\stabindbar{\cU}{\cT}[\simpmod{\cU}{U}]]_{-}}\Fpoly(\Extfun{\cT}\exchmon{\cU}{U}{-}).
\end{equation*}
If there is also no 2-cycle at $U$ then $\exchmon{\cU}{U}{+}=\bigdsum_{V\in \indec \cU} V^{[\exchmatentry{U,T}]_{+}}$ and $\exchmon{\cU}{U}{-}=\bigdsum_{V\in \indec \cU} V^{[\exchmatentry{U,T}]_{-}}$ by Proposition~\ref{p:decomp-exch-terms}.
Substituting into the above gives
\begin{align*}
\begin{split}\Fpoly(\Extfun{\cT}U)\Fpoly(\Extfun{\cT}\mut{\cU}{U}) & =x^{[\stabindbar{\cU}{\cT}[\simpmod{\cU}{U}]]_{+}}\Fpoly(\Extfun{\cT}\exchmon{\cU}{U}{+})+x^{[\stabindbar{\cU}{\cT}[\simpmod{\cU}{U}]]_{-}}\Fpoly(\Extfun{\cT}\exchmon{\cU}{U}{-}) \\
& =\begin{multlined}[t]x^{[\stabindbar{\cU}{\cT}[\simpmod{\cU}{U}]]_{+}}\prod_{V\in \indec \cU\setminus U} \Fpoly(\Extfun{\cT}V)^{[\exchmatentry{U,T}]_{+}}\\+x^{[\stabindbar{\cU}{\cT}[\simpmod{\cU}{U}]]_{-}}\prod_{V\in \indec \cU\setminus U} \Fpoly(\Extfun{\cT}V)^{[\exchmatentry{U,T}]_{-}}\end{multlined}
\end{split}
\end{align*}
as required.
\end{proof}

Given Theorem~\ref{t:c-vec-mut-formula}, we see that if $(\cC,\cT)$ has a cluster structure then we have recovered \cite[Eq.~4.20]{NakanishiBook}, and so the $\Fpoly$-polynomials of $\Extfun{\cT}X$, for $X\in\indec\cU$ and $\cU\ctsubcat\cC$ reachable from $T$, are precisely the $\Fpoly$-polynomials of the cluster algebra with initial exchange matrix $B_{\cT}$.

\subsection{\texorpdfstring{$\Aside$-cluster characters}{A-cluster characters}}

Let $\cC$ be a cluster category, and fix a cluster-tilting subcategory $\cT\ctsubcat\cC$.
We may attempt to define a map $(\beta_{\cT})_*\colon\Laurent{\Kgp{\fd{\stab{\cT}}}}\to\powser{\bK}{\Kgp{\cT}}$ by
\begin{equation}
\label{eq:beta-star}
(\beta_{\cT})_*\biggl(\sum_{v\in\Kgp{\fd{\stab{\cT}}}}\lambda_{v}x^{v}\biggr)=\sum_{v\in\Kgp{\fd{\stab{\cT}}}}\lambda_{v}a^{\beta_{\cT}(v)},
\end{equation}
extending the homomorphism $(\beta_{\cT})_*\colon\bK\Kgp{\fd{\stab{\cT}}}\to\bK\Kgp{\cT}$ of group algebras defined by the same formula.
However, similar to the issue with defining multiplication in $\powser{\bK}{\bV}$, for this to make sense we need the coefficient $\sum_{v:\beta_{\cT}(v)=w}\lambda_{[L]}$ of $a^{w}$ on the right-hand side of \eqref{eq:beta-star} to be a finite sum, which may not be the case without an extra condition on $\beta_{\cT}$.
Recall that the \emph{nullity} of $\beta_{\cT}$ is the rank of its kernel.

\begin{proposition}
\label{p:finite-nullity}
If $\beta_{\cT}$ has finite nullity, then for any $v_0\in\Kgp{\fd{\stab{\cT}}}$, $w\in\Kgp{\cT}$ and $M\in\lfd{\stab{\cT}}$, the set $\{v\in\Kgp{\fd{\stab{\cT}}}\mid v_0\leq v\leq[M],\ \beta_{\cT}(v)=w\}$ is finite.
In particular, the map $(\beta_{\cT})_*\colon\Laurent{\Kgp{\fd{\stab{\cT}}}}\to\powser{\bK}{\Kgp{\cT}}$ from \eqref{eq:beta-star} is well-defined.
\end{proposition}
\begin{proof}
If $v$ and $v'$ are elements of the given set, then $v-v'\in\ker(\beta_{\cT})$ and $v_0-[M]\leq v-v'\leq [M]-v_0$, so it suffices to show that there are finitely many $u\in\ker(\beta_{\cT})$ satisfying these inequalities.
Since $\ker(\beta_{\cT})\leq\Kgp{\fd{\stab{\cT}}}$ is finitely generated, there is a finite set $T_1,\dotsc,T_r\in\indec{\stab{\cT}}$ such that, for any $T\in\indec{\cT}\setminus\{T_1,\dotsc,T_r\}$, we have $\canform{v_0}{[T]}{\cT}=0$ and $\canform{u}{[T]}{\cT}=0$ for all $u\in\ker(\beta_{\cT})$. 
In particular, an element $u\in\ker(\beta_{\cT})$ is completely determined by the finitely many values $\canform{u}{[T_i]}{\cT}$, for $i=1,\dotsc,r$. If $v_0-[M]\leq u\leq[M]-v_0$ then $\canform{v_0-[M]}{[T_i]}{\cT}\leq\canform{u}{[T_i]}{\cT}\leq\canform{[M]+v_0}{[T_i]}{\cT}$ for each $i$, leaving only finitely many possibilities.
\end{proof}

\begin{remark}
\label{r:beta-star-product}
When $\beta_{\cT}$ has finite nullity, the image of $(\beta_{\cT})_*\colon\Laurent{\Kgp{\fd{\stab{\cT}}}}\to\powser{\bK}{\Kgp{\cT}}$ lies in a subset on which multiplication via \eqref{eq:pseudo-poly-mult} is well-defined; indeed, if $u,v\in\Kgp{\fd{\stab{\cT}}}$ then unpacking \eqref{eq:pseudo-poly-mult} gives
\[(\beta_{\cT})_*(u)\cdot(\beta_{\cT})_*(v)=(\beta_{\cT})_*(u\cdot v),\]
and the right-hand side is well-defined by Lemma~\ref{l:pseudo-poly-alg} and Proposition~\ref{p:finite-nullity}.
\end{remark}

The upshot of the preceding discussion is that $M\in\fpmod{\stab{\cT}}$ yields a well-defined element
\[(\beta_{\cT})_*\Fpoly(M)=\sum_{[L]\in\Kgp{\fd{\cT}}}\chi(\QGra{[L]}{M})a^{\beta_{\cT}[L]}\in\powser{\bK}{\Kgp{\cT}}\]
as long as either $M\in\fd{\stab{\cT}}$ or $\beta_{\cT}$ has finite nullity: if $\cC$ has finite rank, then both of these conditions are satisfied for any $M\in\fpmod{\cT}$ and any $\cT\ctsubcat\cC$.
Indeed, in the finite rank case $\Fpoly(M)\in\bK\Kgp{\fd\stab{\cT}}$ and $(\beta_{\cT})_*\Fpoly(M)\in\bK\Kgp{\cT}$ are elements of the ordinary group algebras of $\Kgp{\fd{\stab{\cT}}}$ and $\Kgp{\cT}$ respectively, which are isomorphic to algebras of Laurent polynomials.

\begin{definition}
Let $\cC$ be a cluster category, let $\cT\ctsubcat\cC$ and let $X\in\cC$.
Under the assumption that either $\Extfun{\cT}X\in\fd{\stab{\cT}}$ or that $\beta_{\cT}$ has finite nullity, we define the \emph{$\Aside$-cluster character} of $X$ with respect to $\cT$ to be
\begin{equation}\label{eq:clucha-A-alt} \clucha[\cT]{\Aside}(X)=a^{\ind{\cC}{\cT}[X]}(\beta_{\cT})_{*}\Fpoly(\Extfun{\cT}X)\in\powser{\bK}{\Kgp{\cT}}.
\end{equation}
\end{definition}

Spelling out the definition of $(\beta_{\cT})_*\Fpoly(\Extfun{\cT}{X})$, we see that
\begin{equation}
\label{eq:clucha-A}
\clucha[\cT]{\Aside}(X)=a^{\ind{\cC}{\cT}{[X]}}\sum_{[L]\in \Kgp{\fd \stab{\cT}}}\chi(\QGra{[L]}{\Extfun{\cT}X})a^{\beta_{\cT}[L]},
\end{equation}
cf.~\cite{CalderoChapoton,JorgensenPalu,Palu,Plamondon-ClustCat}.

\begin{remark}\label{r:clucha-A} {\ }
\begin{enumerate}
\item\label{r:clucha-A-finite-non-zero}
Just as for $\Fpoly$-polynomials (Remark~\ref{r:Fpoly-props}), the cluster character $\clucha[\cT]{\Aside}{X}$ is non-zero, and a finite sum whenever $\Extfun{\cT}X\in\fd{\stab{\cT}}$.
\item\label{r:clucha-A-isoclasses} If $X\iso Y$ then $\clucha[\cT]{\Aside}(X)=\clucha[\cT]{\Aside}(Y)$, as is visible from the formula.
\item The value of $\clucha[\cT]{\Aside}(X)$ when $\Extfun{\cT}X\in\fd\stab{\cT}$ is a Laurent polynomial (i.e.\ an element of the group algebra $\bK\Kgp{\cT}$), although the formula itself is emphatically not a simplified expression as an inverse monomial multiplied by a polynomial.
\item\label{r:clucha-A-trop} The submodules of $\Extfun{\cT}{X}$ are naturally a poset via inclusion, with unique minimal element $0$ and unique maximal element $\Extfun{\cT}{X}$. This gives \eqref{eq:clucha-A} a canonical minimal term, corresponding to $N=0$, which is $a^{\ind{\cC}{\cT}{[X]}}$.
If $\Extfun{\cT}{X}\in\fd{\stab{\cT}}$ then there is also a canonical maximal term, which is $a^{\coind{\cC}{\cT}{[X]}}$ since
$\beta_{\cT}[\Extfun{\cT}X]=\coind{\cC}{\cT}{[X]}-\ind{\cC}{\cT}{[X]}$.
This is a reappearance of $\ind{\cC}{\cT}{[X]}$ and $\coind{\cC}{\cT}{[X]}$ being associated to the two natural tropicalisations of a cluster algebra.

If $\Extfun{\cT}X$ is simple, for example when $X=\mut{\cT}{T}$ for $T\in\exch\cT$ and $\cT$ has no loop at $T$, these are the only two terms, recovering the usual exchange relation for mutation of the initial cluster variable $a^{[T]}$. For non-simple $\Extfun{\cT}X$, the cluster character is a more complicated interpolation between the minimal and maximal terms.
\end{enumerate}
\end{remark}

\begin{remark}
\label{r:inf-rank}
Depending on $\cC$, there are potentially several different natural ways of defining the $\Fpoly$-polynomial $\Fpoly(\Extfun{\cT}X)$, and by extension the cluster character $\clucha[\cT]{\Aside}(X)$, which coincide when $\Extfun{\cT}X\in\fd{\stab{\cT}}$ (in particular, if $\cC$ has finite rank or $X$ is a rigid object reachable from $\cT$).
In \cite{JorgensenPalu}, more general inputs to $\clucha[\cT]{\Aside}$ are not considered.
Similarly in \cite{Plamondon-ClustCat}, in which our assumption that $\stab{\cC}$ is Hom-finite is relaxed, the cluster character is still restricted to objects such that $\Extfun{\cT}X\in\fd{\stab{\cT}}$ is finite-dimensional.

Ideally, one would like to define $\Fpoly(M)$ as a sum over all finitely presented submodules of $M$, the classes of which lie in the natural domain of definition for $\beta_{\cT}$, so that one always has the maximal term $x^{[\Extfun{\cT}X]}$ (cf.~Remark~\ref{r:clucha-A}\ref{r:clucha-A-trop}).
However, in the generality in which we are working, it is not clear that a finitely presented submodule $L\leq M$ determines a projective variety $\QGra{[L]}{M}$ for us to take the Euler characteristic of: a priori, this would be a subvariety of the infinite product $\prod_{T\in\indec{\cT}}\QGra{\dim_{\bK}L(T)}{M(T)}$, in which infinitely many factors may be non-trivial.
This issue is addressed by Paquette--Yıldırım \cite[\S6]{PaqYil} when $\cC$ is the completed cluster category of a disc (which is \emph{not} a cluster category in our sense, since it is not $2$-Calabi--Yau), by showing that in this case a finitely presented submodule of $M\in\fpmod{\stab{\cT}}$ is fully determined by finitely many of its subspaces.
The properties of $\cC$ used in this case are quite technical, and it is not clear to us how generally they may hold.

A potential issue with our definition of the $\Fpoly$-polynomial is that it cannot distinguish $\stab{\cT}$-modules with the same collection of finite-dimensional submodules, which one might expect it to by analogy with cluster theory (see e.g.\ \cite[Thm.~6.1]{CKQ}), although it can at least distinguish non-zero modules, by Proposition~\ref{p:Fpoly-zero}.
As above, this issue does not arise if $\cC$ has finite rank or if one restricts to cluster characters of reachable rigid objects.
\end{remark}

We can now prove the following properties of the $\Aside$-cluster character, justifying its name, using the same methodology as \cite{JorgensenPalu}.

\begin{proposition}\label{p:props-of-A-cl-char}
Let $\cC$ be a cluster category and $\cT\ctsubcat\cC$. Then
\begin{enumerate}
\item\label{p:props-of-A-cl-char-split} for any $X,Y\in\cC$,
\begin{equation}\label{eq:clucha-A-split} \clucha[\cT]{\Aside}(X\oplus Y)=\clucha[\cT]{\Aside}(X)\clucha[\cT]{\Aside}(Y); \end{equation}
\item\label{p:props-of-A-cl-char-non-split} if $\cC$ is compact or skew-symmetric, $X,Y\in\cC$ are such that $X$ is indecomposable and $\rank_{\divalg{X}}\Ext{1}{\cC}{X}{Y}=1$, and there are non-split conflations $Y\infl Z\defl X\confl$ and $X\infl Z'\defl Y\confl$, then
\begin{equation}\label{eq:clucha-A-non-split} \clucha[\cT]{\Aside}(X)\clucha[\cT]{\Aside}(Y)=\clucha[\cT]{\Aside}(Z)+\clucha[\cT]{\Aside}(Z'). \end{equation}
\end{enumerate}
\end{proposition}

\begin{proof} {\ }
\begin{enumerate}
\item By \eqref{eq:F-poly-split} and the additivity of $\Extfun{\cT}$,
\begin{align*}
\clucha[\cT]{\Aside}(X\oplus Y) & = a^{\ind{\cC}{\cT}[X\oplus Y]}(\beta_{\cT})_{*}\Fpoly(\Extfun{\cT}X\oplus\Extfun{\cT}Y) \\
& = a^{\ind{\cC}{\cT}([X]+[Y])}(\beta_{\cT})_{*}(\Fpoly(\Extfun{\cT}X)\Fpoly(\Extfun{\cT}Y)) \\
& = a^{\ind{\cC}{\cT}[X]}a^{\ind{\cC}{\cT}[Y]}(\beta_{\cT})_{*}\Fpoly(\Extfun{\cT}X)(\beta_{\cT})_{*}\Fpoly(\Extfun{\cT}Y) \\
& = \clucha[\cT]{\Aside}(X)\clucha[\cT]{\Aside}(Y).
\end{align*}
\item We have
\begin{align*} \clucha[\cT]{\Aside}(X)\clucha[\cT]{\Aside}(Y) & = a^{\ind{\cC}{\cT}[X]}(\beta_{\cT})_{*}\Fpoly(\Extfun{\cT}X)a^{\ind{\cC}{\cT}[Y]}(\beta_{\cT})_{*}\Fpoly(\Extfun{\cT}Y) \\ 
& = a^{\ind{\cC}{\cT}[X]+\ind{\cC}{\cT}[Y]}(\beta_{\cT})_{*}(\Fpoly(\Extfun{\cT}X)\Fpoly(\Extfun{\cT}Y)) \\
& = a^{\ind{\cC}{\cT}[X]+\ind{\cC}{\cT}[Y]}(\beta_{\cT})_{*}\bigl(x^{[\Ker \Extfun{\cT}i_{1}]}\Fpoly(Z)+x^{[\Ker \Extfun{\cT}i_{2}]}\Fpoly(Z')\bigr) \\ 
& = \begin{multlined}[t]a^{\ind{\cC}{\cT}[X]+\ind{\cC}{\cT}[Y]+\beta_{\cT}[\Ker \Extfun{\cT}i_{1}]}(\beta_{\cT})_{*}\Fpoly(Z)\\
+a^{\ind{\cC}{\cT}[X]+\ind{\cC}{\cT}[Y]+\beta_{\cT}[\Ker \Extfun{\cT}i_{2}]}(\beta_{\cT})_{*}\Fpoly(Z')\end{multlined} \\ 
& = a^{\ind{\cC}{\cT}[Z]}(\beta_{\cT})_{*}\Fpoly(\Extfun{\cT}Z)+a^{\ind{\cC}{\cT}[Z']}(\beta_{\cT})_{*}\Fpoly(\Extfun{\cT}Z') \\ 
& = \clucha[\cT]{\Aside}(Z)+\clucha[\cT]{\Aside}(Z'),
\end{align*}
where the third equality uses \eqref{eq:F-poly-non-split} and the fifth equality uses (both parts of) Lemma~\ref{l:apres-Palu}. \qedhere
\end{enumerate}
\end{proof}

\begin{remark}
As suggested by the proof of Proposition~\ref{p:chidentities}, one can also obtain \eqref{eq:clucha-A-non-split} from \cite[Thm.~1.1]{Palu2}, using Lemma~\ref{l:rk1-unique-middle-term}.
While the product of two cluster characters (via \eqref{eq:pseudo-poly-mult}) is always well-defined because of Remark~\ref{r:beta-star-product}, the previous argument also gives an implicit proof of this statement for the two special cases considered.
\end{remark}

\begin{corollary}
\label{c:A-mutation}
Let $\cC$ be a compact or skew-symmetric cluster category, let $\cTU\ctsubcat\cC$ and let $U\in\exch{\cU}$.
If $\cU$ has no loop at $U$, we have
\[\clucha[\cT]{\Aside}(\mut{\cU}U)\clucha[\cT]{\Aside}(U)=\clucha[\cT]{\Aside}(\exchmon{\cU}{U}{+})+\clucha[\cT]{\Aside}(\exchmon{\cU}{U}{-}).\]
\end{corollary}
\begin{proof}
The assumption that $\cU$ has no loop at $U$ is equivalent to $\rank_{\divalg{U}}\Ext{1}{\cC}{U}{\mut{\cU}{U}}=1$ by Lemma~\ref{l:props-of-K0-mod-T}\ref{l:props-mod-T-simple-eq-E}, so this is a special case of Proposition~\ref{p:props-of-A-cl-char}\ref{p:props-of-A-cl-char-non-split}.
\end{proof}

\begin{remark}
Remark~\ref{r:clucha-A}\ref{r:clucha-A-isoclasses} and \eqref{eq:clucha-A-split} tell us that $\clucha[\cT]{\Aside}$ induces a well-defined monoid homomorphism on the positive cone in $\Kgp{\addcat{\cC}}$, consisting of classes of objects, justifying the term `character'.
By appropriately enlarging $\powser{\bK}{\Kgp{\cT}}$, we may then extend $\clucha[\cT]{\Aside}$ to a group homomorphism on $\Kgp{\addcat{\cC}}$ by defining its value on a difference of classes to be the ratio of the corresponding cluster characters, noting that these are non-zero by Remark~\ref{r:clucha-A}\ref{r:clucha-A-finite-non-zero}.
For example, when $\cC$ has finite rank, a suitable codomain for this extension of $\clucha[\cT]{\Aside}$ is the field of fractions of the group algebra $\bK\Kgp{\cT}$.
\end{remark}

The fundamental theorem for cluster characters now follows.
Recall that if $\mathscr{A}$ is a cluster algebra, a \emph{cluster monomial} is a monomial in the cluster variables of one cluster.

\begin{theorem}[\cite{BMRRT,FuKeller,PresslandPostnikov}]
\label{t:Aside-bijection}
Let $\cC$ be a compact or skew-symmetric cluster category, let $\cT\ctsubcat\cC$, and assume that $(\cC,\cT)$ has a cluster structure.
Then $\clucha[\cT]{\Aside}$ is a bijection between objects in cluster-tilting subcategories reachable from $\cT$ and cluster $\Aside$-monomials of the cluster algebra $\mathscr{A}$ with initial exchange matrix $B_{\cT}$.
Under this bijection, cluster (and frozen) variables correspond to indecomposable objects, with frozen variables corresponding to indecomposable projectives, and there is an induced bijection of cluster-tilting subcategories of $\cC$ reachable from $\cT$ and seeds of $\mathscr{A}$, commuting with mutations.
\end{theorem}

\begin{proof}
Our definitions align with those of \cite{PresslandPostnikov} and so the argument sketched in the proof of \cite[Thm.~6.10]{PresslandPostnikov} goes through for us too.
Notably, what is required to see that our cluster character transforms correctly to align with cluster variable mutation is contained in Remark~\ref{r:clucha-A}\ref{r:clucha-A-isoclasses} and Corollary~\ref{c:A-mutation}.
\end{proof}

\begin{remark}
To prove Theorem~\ref{t:Aside-bijection}, it is never necessary to consider the values of $\clucha[\cT]{\Aside}$ on $X\in\cC$ with $\Extfun{\cT}X$ infinite-dimensional, and so the issues discussed in Remark~\ref{r:inf-rank} do not arise.
Indeed, if $X\in\cU$ for $\cU\ctsubcat\cC$ reachable from $\cT$, then $\Extfun{\cT}X\in\fd{\stab{\cT}}$ since it is supported on the additively finite subcategory $\cT\setminus\cU$ (cf.~Proposition~\ref{p:error-term-fd}).
\end{remark}

The cluster character is compatible with partial stabilisation (\S \ref{s:part-stab}); let $\cC$ be a cluster category and $\cP$ a full additively closed subcategory of projective-injective objects.
Then by Proposition~\ref{p:partialstab}, the partial stabilisation $\cC/\cP$ is again a cluster category, and the quotient functor $\pi$ induces a bijection of cluster-tilting subcategories; as before, we write $\cT/\cP=\pi \cT$ for $\cT\ctsubcat\cC$.

The homomorphism $\pi_{\cT}^{\cT/\cP}\colon \Kgp{\cT}\to \Kgp{\cT/\cP}$ of Grothendieck groups induces a linear map $(\pi_{\cT}^{\cT/\cP})_*\colon \powser{\bK}{\Kgp{\cT}}\to\powser{\bK}{\Kgp{\cT/\cP}}$ with $(\pi_{\cT}^{\cT/\cP})_*(a^{t})=a^{\pi_{\cT}^{\cT/\cP}(t)}$, which respects the multiplication \eqref{eq:pseudo-poly-mult} when this is defined.
This map has the property that $(\pi_{\cT}^{\cT/\cP})_*(a^{[P]})=1$ for $P\in \cP$, and indeed the kernel of $(\pi_{\cT}^{\cT/\cP})_*$ is generated by such elements; that is, $(\pi_{\cT}^{\cT/\cP})_*$ encodes the setting of certain frozen variables to $1$.
This allows us to relate the cluster characters of $\cC$ and $\cC/\cP$ as follows.

\begin{proposition}\label{p:A-cl-cat-partial-stab}
With notation as above, we have the following commutative diagram:
\[
\begin{tikzcd}
\cC \arrow{r}{\clucha[\cT]{\Aside}} \arrow{d}[left]{\pi_{\cC}^{\cC/\cP}} & \powser{\bK}{\Kgp{\cT}} \arrow{d}{(\pi_{\cT}^{\cT/\cP})_*}\\
\cC/\cP \arrow{r}[below,yshift=-0.2em]{\clucha[\cT/\cP]{\Aside}} & \powser{\bK}{\Kgp{\cT/\cP}}
\end{tikzcd}
\]
\end{proposition}

\begin{proof} This follows immediately from the definitions and Proposition~\ref{p:stabind}.
\end{proof}

The next few results recover the (known) statement of \cite[Conj.~7.2]{FZ-CA4}, as in \cite[Thm.~5.5]{FuKeller}, and demonstrate how established techniques apply to cluster categories in our setting in order to prove statements about the values of their cluster characters.

\begin{definition}
Let $\bV$ be a free abelian group with basis $\cB$.
We say $p=\sum_{v\in\bV}\lambda_vy^v\in\powser{\bK}{\bV}$ is \emph{proper Laurent} if $v\not\geq_\cB0$ whenever $\lambda_v\ne0$; in other words, as a formal series in the variables $y^b$ for $b\in\cB$, every monomial of $p$ includes at least one of these variables with a negative degree.
\end{definition}

\begin{theorem}[{cf.\ \cite[Cor.~3.4]{CKLP}}]
\label{t:g-vectors-cl-mons}
Let $\cC$ be a Krull--Schmidt cluster category with $\cT\ctsubcat\cC$ maximally mutable.
Then for every rigid object $X\in\cC\setminus\cT$, the cluster character $\clucha[\cT]{\Aside}(X)$ is proper Laurent in $\powser{\bK}{\Kgp{\cT}}$ (with respect to the basis $\indec{\cT}$).
\end{theorem}

\begin{proof}
We follow the strategy of \cite[\S 3]{CKLP}.
The monomials of $\clucha[\cT]{\Aside}(U)$ have the form $a^{\ind{\cC}{\cT}[U]+\beta_{\cT}[L]}$ for $L\in \fd \stab{\cT}$ a finite-dimensional submodule of $\Extfun{\cT}U$, and we aim to show that this is a proper Laurent monomial when $U\in\cC\setminus\cT$ is rigid.

For $L=0$, we observe that $\canform{[\simpmod{\cT}{T}]}{\ind{\cC}{\cT}[U]}{\cT}\geq 0$ for all $T\in \indec \cT$ if and only if $U\in \cT$, so we are done in this case.
Assume now that $L\neq 0$ and consider $\canform{[L]}{\ind{\cC}{\cT}[U]+\beta_{\cT}[L]}{\cT}$.
By Lemma~\ref{l:s-form-skew-sym}, the form $\sform{\blank}{\blank}{\cT}$ is skew-symmetric (even if $\beta_{\stab{\cT}}$ is not), and so $\canform{[L]}{\beta_{\cT}[L]}{\cT}=\sform{[L]}{[L]}{\cT}=0$.
Thus, $\canform{[L]}{\ind{\cC}{\cT}[U]+\beta_{\cT}[L]}{\cT}=\canform{[L]}{\ind{\cC}{\cT}[U]}{\cT}$.

Since $\cT$ is maximally mutable, $\fd{\stab{\cT}}\subseteq\fpmod{\stab{\cT}}$ by Corollary~\ref{c:fp=fd}, and so by Proposition~\ref{p:equiv-to-mod} we may choose $V\in \cC$ such that $L=\Extfun{\cT}V$.
Let $\rightker{\cT}{U}\infl \rightapp{\cT}{U} \stackrel{p}{\defl} U\confl$ be a $\cT$-index conflation for $U$, so that
\[ \canform{[L]}{\ind{\cC}{\cT}[U]}{\cT}  = \dim \Ext{1}{\cC}{\rightapp{\cT}{U}}{V}-\dim \Ext{1}{\cC}{\rightker{\cT}{U}}{V}.	\]
Now, as in \eqref{eq:lr-seqs}, we have an exact sequence 
\begin{equation}
\label{eq:proper-LP-seq}
\begin{tikzcd}[column sep=20pt]
0\arrow{r}&\righterror{1}{\cT}{U}(V)\arrow{r}&\Ext{1}{\cC}{V}{\rightker{\cT}{U}}\arrow{r}&\Ext{1}{\cC}{V}{\rightapp{\cT}{U}}\arrow{rr}{\Ext{1}{\cC}{V}{p}}&&\Ext{1}{\cC}{V}{U}.
\end{tikzcd}
\end{equation}
We claim that $\Ext{1}{\cC}{V}{p}=0$.
To see this, we first pass to the stable category $\stab{\cC}$, in which $\Ext{1}{\cC}{V}{p}=\stabHom{\cC}{V}{\Sigma p}\colon\stabHom{\cC}{V}{\Sigma\rightapp{\cT}{U}}\to\stabHom{\cC}{V}{\Sigma U}$.
Since $p$ is a right $\cT$-approximation of $U$ (in $\cC$, and hence also in $\stab{\cC}$), its shift $\Sigma p$ is a right $\Sigma\cT$-approximation of $\Sigma U$.
Thus the image of $\stabHom{\cC}{V}{\Sigma p}$ is the set of morphisms $V\to \Sigma U$ in $\stab{\cC}$ which factor over $\Sigma\cT$.
But any such morphism is zero as in \cite{CKLP}: if $v\colon V\to U$ is a map such that $\Extfun{\cT}{v}\colon L\to\Extfun{\cT}{U}$ is the inclusion, so in particular injective, then its mapping cylinder is in the kernel of $\Extfun{\cT}$, which is the ideal $(\cT)$ of maps factoring over $\cT$.
Thus, if $f\colon V\to\Sigma U$ factors over $(\Sigma\cT)$, we have the commutative diagram
\[\begin{tikzcd}
\Sigma^{-1}U\arrow{r}&C\arrow{r}{c\in(\cT)}&V\arrow{d}[swap]{f\in(\Sigma\cT)}\arrow{r}{v}&U\arrow{dl}{g=0}\\
&&\Sigma U
\end{tikzcd}\]
in which the upper row is a triangle in $\stab{\cC}$.
Here $fc=0$ since there are no morphisms from $\cT$ to $\Sigma\cT$ (because $\cT$ is cluster-tilting, hence rigid), so there exists a map $g\colon U\to\Sigma U$ making the diagram commute, but $g=0$ since $U$ is rigid.
Thus, $f=0$, as required.

Consequently, by taking dimensions we deduce from \eqref{eq:proper-LP-seq} that \[ \dim \righterror{1}{\cT}{U}(V)+\canform{[L]}{\ind{\cC}{\cT}[U]}{\cT}=0. \] By Lemma~\ref{l:error-terms-modT}\ref{p:error-terms-modT-r1} and Proposition~\ref{p:equiv-to-mod}, we have \[\righterror{1}{\cT}{U}(V)=\Hom{\cC}{V}{U}/\cT(V,U)\iso \Hom{\stab{\cT}}{\Extfun{\cT}V}{\Extfun{\cT}U},\]
and this space is non-zero since $0\neq L=\Extfun{\cT}V\leq \Extfun{\cT}U$.  Hence, $\canform{[L]}{\ind{\cC}{\cT}[U]}{\cT}=-\dim \righterror{1}{\cT}{U}(V)<0$.
Decomposing $L\in\fd\stab{\cT}$ into its simple composition factors, we conclude that $\canform{[\simpmod{\cT}{T}]}{\ind{\cC}{\cT}[U]}{\cT}<0$ for some $T$ as required.
\end{proof}

\begin{remark}\label{r:g-vecs-distinguish-cl-mons}
We also see that cluster characters of different objects from cluster-tilting subcategories have different $\mathbf{g}$-vectors (\cite[Conj.~7.10]{FZ-CA2}, proved in the exact case in \cite[Thm.~5.5(b)]{FuKeller}).
Let $\clucha[\cT]{\Aside}(V)$ and $\clucha[\cT]{\Aside}(W)$ be distinct, where $V\in \cV$ and $W\in \curly{W}$ for some cluster-tilting subcategories $\cV,\curly{W}\ctsubcat \cC$.
The $\mathbf{g}$-vector of $\clucha[\cT]{\Aside}(V)$ is $\ind{\cU}{\cT}{[U]}=\ind{\cC}{\cT}{[V]}\in \gvecplus{\cT}{\cU}$ (see Definition~\ref{d:g-vectors}, Theorem~\ref{t:c-vec-mut-formula}), and similarly for $W$.

Assume for a contradiction that the $\mathbf{g}$-vectors of $\clucha[\cT]{\Aside}(V)$ and $\clucha[\cT]{\Aside}(W)$ are equal.
Since $V$ and $W$ are rigid, these quantities being equal implies that $V\iso W$ by Proposition~\ref{p:rigid-index}, so that the associated cluster monomials are not distinct, a contradiction.
\end{remark}

\begin{corollary}
If $\cC$ is a Krull--Schmidt cluster category and $\cT\ctsubcat\cC$ is such that $(\cC,\cT)$ has a cluster structure, then the cluster monomials of the cluster algebra $\mathscr{A}$ with initial exchange matrix $\beta_{\cT}$ are linearly independent.
\end{corollary}
\begin{proof}
By Theorems~\ref{t:Aside-bijection} and \ref{t:g-vectors-cl-mons}, the cluster monomials of $\mathscr{A}$ have the proper Laurent property with respect to any cluster, hence are linearly independent by \cite[Thm.~6.4]{CILF}.
\end{proof}

We suspect that the cluster characters of reachable rigid objects are linearly independent even if $(\cC,\cT)$ does not have a cluster structure.
To apply (the proof of) \cite[Thm.~6.4]{CILF} to this more general situation it would be necessary to realise the cluster characters $\clucha[\cT]{\cA}(X)$ and $\clucha[\cU]{\cA}(X)$ of one object with respect to two different cluster characters as expressions for `the same quantity' in two different coordinate systems; in the presence of a cluster structure, the bijection of Theorem~\ref{t:Aside-bijection} makes this possible, at least when $\cU$ is reachable from $\cT$ and $X$ is reachable rigid, since in this case the cluster characters are expressions for the same cluster monomial in two different clusters.
On the other hand, if $\beta_{\cT}$ is injective, we may circumvent this argument, and the reachability hypothesis, as follows.

\begin{proposition}
\label{p:li-full-rank}
Let $\cC$ be a Krull--Schmidt cluster category and $\cT\ctsubcat\cC$.
If $\beta_{\cT}$ is injective, then the cluster characters $\clucha[\cT]{\cA}(U)$ of rigid objects $U\in\cC$ are linearly independent.
\end{proposition}
\begin{proof}
Let $S=\{ \clucha[\cT]{\Aside}(U_{i}) \mid U_{i}\in \cU_{i},\ 1\leq i\leq r\}$ be pairwise distinct, so that in particular the $U_{i}$ are pairwise non-isomorphic.
Assume for a contradiction that $S$ is linearly dependent.
The $\Fpoly$-polynomials $\Fpoly(\Extfun{\cT}X)$ have constant term $1$, and the same is true of $(\beta_{\cT})_{*}\Fpoly(\Extfun{\cT}X)$ since $\beta_{\cT}$ is injective.
A linear dependence of $S$ would thus imply a linear dependence of the monomials $a^{\ind{\cC}{\cT}{[U_{i}]}}$, these being the minimal degree terms.
However, by Proposition~\ref{p:rigid-index}, the fact that the $U_{i}$ are pairwise non-isomorphic implies that the indices $\ind{\cC}{\cT}{[U_{i}]}$ are all distinct.
The corresponding monomials $a^{\ind{\cC}{\cT}{[U_{i}]}}$ are therefore linearly independent, and hence so is $S$.
\end{proof}

As indicated above, our cluster character is still defined when $(\cC,\cT)$ does not have a cluster structure, but its values may not be cluster variables.
They may still be interesting functions, however, as the following example indicates: we will leave a more general exploration of this phenomenon for future work.

\begin{example}
\label{eg:CNS-alg}
Consider the mesh category $\cC$, defined over $\complex$, with Auslander--Reiten quiver as follows: 
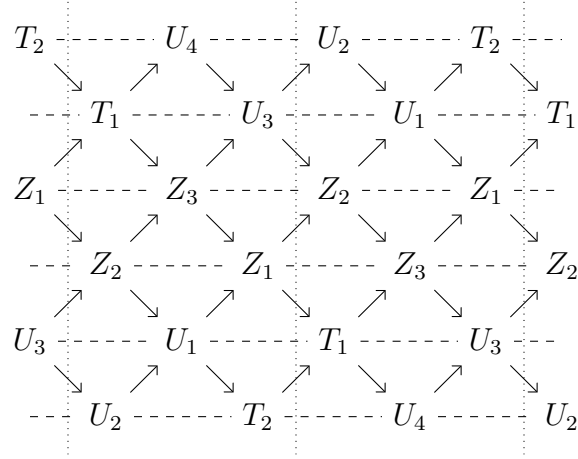
\[\begin{tikzpicture}
\foreach \c in {1,2,3,4}
{\foreach \r in {1,3,5}
{\coordinate (c\r-\c) at (2*\c-2,1-\r);
}
\foreach \r in {2,4,6}
{\coordinate (c\r-\c) at (2*\c-1,1-\r);
}
}
\foreach \r/\c/\l in {1/1/T_2, 1/2/U_4, 1/3/U_2, 1/4/T_2, 2/1/T_1, 2/2/U_3, 2/3/U_1, 2/4/T_1, 3/1/Z_1, 3/2/Z_3, 3/3/Z_2, 3/4/Z_1, 4/1/Z_2, 4/2/Z_1, 4/3/Z_3, 4/4/Z_2, 5/1/U_3, 5/2/U_1, 5/3/T_1, 5/4/U_3, 6/1/U_2, 6/2/T_2, 6/3/U_4, 6/4/U_2}
{\node (\r-\c) at (c\r-\c) {$\l$};
}
\foreach \r in {1,3,5}
{\foreach \c/\d in {1/2,2/3,3/4}
{\draw[dashed] (\r-\c) to (\r-\d);
}
\draw[dashed] (\r-4) to (7,1-\r);
}
\foreach \r in {2,4,6}
{\foreach \c/\d in {1/2,2/3,3/4}
{\draw[dashed] (\r-\c) to (\r-\d);
}
\draw[dashed] (0,1-\r) to (\r-1);
}
\foreach \c in {1,2,3,4}
{\foreach \r/\s in {1/2,3/4,5/6}
{\draw[-angle 90] (\r-\c) to (\s-\c);}
\foreach \r/\s in {3/2,5/4}
{\draw[-angle 90] (\r-\c) to (\s-\c);}
}
\foreach \c/\d in {1/2,2/3,3/4}
{\foreach \r/\s in {2/3,4/5}
{\draw[-angle 90] (\r-\c) to (\s-\d);}
\foreach \r/\s in {2/1,4/3,6/5}
{\draw[-angle 90] (\r-\c) to (\s-\d);}
}
\foreach \x in {0,3,6}
{\draw[dotted] (\x.5,.5) to (\x.5,-5.5);}
\end{tikzpicture}\]
Here the dotted lines are identified via a glide reflection (so two copies of a fundamental domain are visible).
One may check that $\cC$ is $2$-Calabi--Yau, for example by realising it as the orbit category for the action of $\langle\Sigma^3\rangle\iso\integ_3$ on the classical cluster category \cite{BMRRT} of type $\mathsf{A}_6$, and that $T=T_1\oplus T_2$ is a cluster-tilting object in $\cC$.
The quiver of $\op{\End{\cC}{T}}$ is
\[\begin{tikzpicture}
\node (1) at (0,0) {$T_1$};
\node (2) at (1.5,0) {$T_2$};
\draw[-angle 90] (1) to (2);
\draw[out=150,in=210,looseness=4,-angle 90,overlay] (1) to (1);
\end{tikzpicture}\]
so $\cT=\add{T}$ has a loop at $T_1$.

Writing $a_i=a^{[T_i]}$, these being the cluster characters $\clucha[\cT]{\Aside}(T_i)$, we may calculate
\begin{align*}
\clucha[\cT]{\Aside}(U_1)&=a_1^{-1}(1+a_2+a_2^2),\\
\clucha[\cT]{\Aside}(U_2)&=a_2^{-1}(1+a_1^{-1}+a_1^{-1}a_2+a_1^{-1}a_2^2),\\
\clucha[\cT]{\Aside}(U_3)&=a_1a_2^{-2}(1+2a_1^{-1}+a_1^{-2}+a_1^{-1}a_2+a_1^{-2}a_2+a_1^{-2}a_2^{-2}),\\
\clucha[\cT]{\Aside}(U_4)&=a_1a_2^{-1}(1+a_1^{-1}).
\end{align*}
These are the generalised cluster variables of the generalised cluster algebra (in the sense of Chekhov--Shapiro \cite{CheSha}) with initial exchange matrix $B=\bigl(\begin{smallmatrix}0&-1\\2&0\end{smallmatrix}\bigr)$, different from $B_{\cT}=\bigl(\begin{smallmatrix}0&-1\\1&0\end{smallmatrix}\bigr)$ (cf.~Remark~\ref{r:gen-c-vec-mut}), and exchange polynomials $\theta_1=1+z+z^2$ and $\theta_2=1+z$.
These cluster characters are linearly independent, as predicted by Proposition~\ref{p:li-full-rank}, since $B_{\cT}$ has full rank.

The objects $Z_1$, $Z_2$ and $Z_3$ are not rigid. We have for example
\[\clucha[\cT]{\Aside}(Z_1)=1+a_2=\clucha[\cT]{\Aside}(0)+\clucha[\cT]{\Aside}(T_2),\]
demonstrating that the rigidity hypothesis in Proposition~\ref{p:li-full-rank} is necessary even when considering only indecomposable objects.

Many similar examples of cluster categories with only a weak cluster structure appear in work of Baur--Pasquali--Velasco \cite{BPV}, and we also expect these to decategorify to Chekhov--Shapiro's generalised cluster algebras, as suggested by work of Fraser \cite{Fraser}.
\end{example}

\subsection{\texorpdfstring{$\Xside$-cluster characters}{X-cluster characters}}

Our next goal is to write down an $\Xside$-cluster character, analogous to \eqref{eq:clucha-A} for the $\Aside$-side, which will produce $\Xside$-cluster variables in the presence of a cluster structure.
While this has been done implicitly by categorifying the individual ingredients ($\mathbf{c}$-vectors and $\Fpoly$-polynomials) of Fomin--Zelevinsky's separation formula for these variables \cite{FZ-CA4}, we will package things together to more closely resemble the $\Aside$-cluster character \eqref{eq:clucha-A}.
In particular, our proof that the $\Xside$-cluster character correctly computes the $\Xside$-cluster variables (when we have a cluster structure) is independent of the separation formula, and thus gives a new proof of this formula for any cluster algebra obtained from one of our cluster categories.

\begin{definition}
\label{d:U-pm}
Let $\cC$ be a Krull--Schmidt cluster category. For each $\cU\ctsubcat\cC$ and each $M\in\fpmod{\stab{\cU}}$, choose $\modlift{M}{\cU}{+},\modlift{M}{\cU}{-}\in\cU$ such that
\begin{equation}
\label{eq:TM-pm-def}
\beta_{\cU}[M]=[\modlift{M}{\cU}{+}]-[\modlift{M}{\cU}{-}]\in \Kgp{\cU}.
\end{equation}
\end{definition}

\begin{remark}
\label{r:UM-ambiguity}
While the objects $\modlift{M}{\cU}{\pm}$ are not defined uniquely up to isomorphism by \eqref{eq:TM-pm-def}, the fact that $\cU$ is Krull--Schmidt and has no non-split extensions means that the possible choices differ only by the addition or removal of common direct summands.
This ambiguity has no effect on what follows, in particular on Definition~\ref{d:clucha-X} below.
\end{remark}

Recall that $\Laurent{\Kgp{\fd{\stab{\cT}}}}$ denotes the algebra of Laurent pseudo-polynomials in $\Kgp{\fd{\stab{\cT}}}$, with respect to the basis of simple modules.

\begin{proposition}
\label{p:ppoly-id}
Let $\bV$ be a free abelian group with basis $\cB$.
Then the algebra $\Poly{\bV}$ is an integral domain.
\end{proposition}

\begin{proof}
This follows from \cite[Lem.~1]{Nishimura}, which shows that a larger ring of formal power series is an integral domain.
The argument also adapts directly to $\Laurent{\bV}$, as follows.
By choosing a total ordering of $\cB$, one can totally order the monomials $y^u$ for $u\in\bV$ in such a way that if $p,q\in\Laurent{\bV}$ have minimal non-zero terms $\lambda_vy^v$ and $\rho_wy^w$ respectively, the product $pq$ has the non-zero term $\lambda_v\rho_wy^{v+w}$.
\end{proof}

By Proposition~\ref{p:ppoly-id}, we may take the field of fractions of $\Laurent{\Kgp{\fd{\stab{\cT}}}}$, which we denote by $\Frac{\Kgp{\fd{\stab{\cT}}}}$.
This is naturally a module for the group algebra $\bK\Kgp{\fd{\stab{\cT}}}$, as is the group algebra $\bK\Kgpnum{\lfd{\stab{\cT}}}$, via the inclusion $\Kgp{\fd{\stab{\cT}}}\to\Kgpnum{\lfd{\stab{\cT}}}$.
Using the factorisation from Remark~\ref{r:pseudo-poly-factor}, one can show that any element of $\Frac{\Kgp{\fd{\stab{\cT}}}}$ has the form $p/q$, where $p$ and $q$ are pseudo-polynomials (rather than arbitrary Laurent pseudo-polynomials).

\begin{definition}
We write $\Fracbar{\Kgp{\fd{\stab{\cT}}}}=\bK\Kgpnum{\lfd{\stab{\cT}}}\otimes_{\bK\Kgp{\fd{\stab{\cT}}}}\Frac{\Kgp{\fd{\stab{\cT}}}}$.
\end{definition}
Elements of $\Fracbar{\Kgp{\fd{\stab{\cT}}}}$ can be thought of as (finite linear combinations of) products $x^v\cdot p/q$, where $p$ and $q$ are pseudo-polynomials in $\Kgp{\fd{\stab{\cT}}}$ and $x^v$ is a monomial with exponent in $\Kgpnum{\lfd{\stab{\cT}}}$.
Monomial factors of $x^v$ with exponent in the subgroup $\Kgp{\fd{\stab{\cT}}}$ may be absorbed into $p$ in the expected way.

\begin{remark}
If $\cC$ has finite rank, so $\Kgpnum{\lfd{\stab{\cT}}}=\Kgp{\fd\stab{\cT}}$ and $\Laurent{\fd\stab{\cT}}=\bK\Kgp{\fd{\stab{\cT}}}$, then $\Fracbar{\Kgp{\fd{\stab{\cT}}}}=\Frac{\Kgp{\fd{\stab{\cT}}}}$ is nothing but the field of fractions of the group algebra $\bK\Kgp{\fd{\stab{\cT}}}$.
\end{remark}

\begin{definition}
\label{d:clucha-X}
Let $\cC$ be a compact or skew-symmetric cluster category, and let $\cTU\ctsubcat\cC$.
Then the \emph{$\Xside$-cluster character} for $\cU$ with respect to $\cT$ is the function 
$ \clucha[\cTU]{\Xside}\colon \fpmod \stab{\cU} \to\Fracbar{\Kgp{\fd{\stab{\cT}}}}$ defined by
\begin{equation}
\label{eq:clucha-X-alt}
\clucha[\cTU]{\Xside}(M)=x^{\stabindbar{\cU}{\cT}{[M]}}\Fpoly(\Extfun{\cT}\modlift{M}{\cU}{+})\Fpoly(\Extfun{\cT}\modlift{M}{\cU}{-})^{-1}.
\end{equation}
\end{definition}

By \eqref{eq:F-poly-split}, modifying $\modlift{M}{\cU}{\pm}$ by adding or removing a common summand has no effect on \eqref{eq:clucha-X-alt}, as promised in Remark~\ref{r:UM-ambiguity}, so this expression is a well-defined function of $M$.
Expanding the sums gives
\begin{equation}
\label{eq:clucha-X}
\clucha[\cTU]{\Xside}(M)=x^{\stabindbar{\cU}{\cT}{[M]}}\frac{\sum_{[N]\in \Kgp{\fd \cU}}\chi(\QGra{[N]}{\Extfun{\cT}\modlift{M}{\cU}{+}})x^{[N]}}{\sum_{[N]\in \Kgp{\fd \cU}}\chi(\QGra{[N]}{\Extfun{\cT}\modlift{M}{\cU}{-}})x^{[N]}}.
\end{equation}

\begin{remark}\label{r:clucha-X} {\ }
\begin{enumerate}
\item Both sums in \eqref{eq:clucha-X} are non-zero, although it is certainly possible that at least one of them is equal to $1$.
They are finite if $\Extfun{\cT}{\modlift{M}{\cU}{\pm}}\in\fd{\stab{\cT}}$, for example if $\cU$ is reachable from $\cT$.
\item If $M\iso M'$ then $\clucha[\cTU]{\Xside}(M)=\clucha[\cTU]{\Xside}(M')$, since each part of the formula either explicitly involves Grothendieck group classes or an $\Fpoly$-polynomial.
\item The image of $\clucha[\cT]{\Xside}$ visibly lies in $\Fracbar{\Kgp{\fd{\stab{\cT}}}}$ and not $\powser{\bK}{\fd{\stab{\cT}}}$, except possibly in degenerate situations, i.e.\ even in finite rank cases, the values of the $\cX$-cluster character are (unavoidably) not Laurent polynomials, but more general rational functions.
\item\label{r:clucha-X-tropical} Since $\clucha[\cTU]{\Xside}(M)$ is not a Laurent polynomial, we cannot really discuss its leading terms in the same way as for $\clucha[\cT]{\Aside}(X)$. On the other hand, $\clucha[\cTU]{\Xside}(M)$ has natural tropicalisations which correspond to taking the minimal and maximal submodules in the two sums; the assumption that either $\cC$ has finite rank or $\cU$ is reachable from $\cT$ means that $\Extfun{\cT}{\modlift{M}{\cU}{\pm}}\in\fd{\stab{\cT}}$, so that there is a maximum. Under the minimal convention, we obtain $x^{\stabindbar{\cU}{\cT}{[M]}}$, and under the maximal convention we obtain
\[x^{\stabindbar{\cU}{\cT}{[M]}+[\Extfun{\cT}\modlift{M}{\cU}{+}]-[\Extfun{\cT}\modlift{M}{\cU}{-}]}=x^{\stabcoindbar{\cU}{\cT}{[M]}},\]
since $[\Extfun{\cT}\modlift{M}{\cU}{+}]-[\Extfun{\cT}\modlift{M}{\cU}{-}]=\stabcoindbar{\cU}{\cT}[M]-\stabindbar{\cU}{\cT}[M]$ by \eqref{eq:TM-pm-def} and Lemma~\ref{ind-coind-ext}.
\end{enumerate}
\end{remark}

If $\cU$ is maximally mutable, then the cluster character $\clucha[\cTU]{\Xside}$ may be restricted to $\fd{\stab{\cU}}\subseteq\fpmod{\stab{\cU}}$, and in particular evaluated on the simple $\stab{\cU}$-modules, as we will do below.
If we also assume that either $\cC$ has finite rank or $\cU$ is reachable from $\cT$, then we have $\clucha[\cTU]{\Xside}(M)\in\Frac{\Kgp{\fd{\stab{\cT}}}}$ for any $M\in\fd{\stab{\cT}}$, by Corollary~\ref{c:indbar-on-fd}.
In this case, $\clucha[\cTU]{\Xside}(M)$ even lies in the field of fractions of the Laurent polynomial algebra $\bK\Kgp{\fd{\stab{\cT}}}$, since the two $\Fpoly$-polynomials involved are both finite sums.

Together with Remark~\ref{r:clucha-X}, the following demonstrates that $\clucha[\cTU]{\Xside}$ induces a well-defined character on $\Kgp{\fpmod{\stab{\cU}}}$.

\begin{proposition}
\label{p:cluchaX-mult}
If $[M]=[K]+[L]\in\Kgp{\fpmod\stab{\cU}}$, then
\[\clucha[\cTU]{\Xside}(M)=\clucha[\cTU]{\Xside}(K)\clucha[\cTU]{\Xside}(L).\]
\end{proposition}
\begin{proof}
Since
\begin{align*}
\beta_{\cU}[M]=\beta_{\cU}([K]+[L])=[\modlift{K}{\cU}{+}]-[\modlift{K}{\cU}{-}]+[\modlift{L}{\cU}{+}]-[\modlift{L}{\cU}{-}]
=[\modlift{K}{\cU}{+}\oplus \modlift{L}{\cU}{+}]-[\modlift{K}{\cU}{-}\oplus\modlift{L}{\cU}{-}],
\end{align*}
we may take
$\modlift{M}{\cU}{\pm}=\modlift{K}{\cU}{\pm}\oplus\modlift{L}{\cU}{\pm}$.
Thus we obtain the statement using \eqref{eq:F-poly-split} and the fact that we also have $[M]=[K]+[L]\in\Kgpnum{\lfd{\stab{\cU}}}$.
\end{proof}

\begin{corollary}
\label{c:clucha-X}
Let $\cC$ be a compact or skew-symmetric cluster category and let $\cTU\ctsubcat\cC$.
Then $\clucha[\cTU]{\cX}$ induces a character
$ \clucha[\cTU]{\Xside}\colon \Kgp{\fpmod\stab{\cU}} \to \Fracbar{\Kgp{\fd{\stab{\cT}}}} $.
\end{corollary}

\begin{remark}
In contrast to the $\Aside$ case, we do not have an obvious `global' domain for the $\Xside$-cluster character: we cannot write `$\fpmod\stab{\cC}$' in place of $\fpmod\stab{\cU}$, for example.
Indeed, on the $\Aside$-side, we can take advantage of the fact that $\clucha[\cT]{\Aside}$ is agnostic about which cluster-tilting subcategories an object $X$ belongs to in its computation, whereas for $\clucha[\cTU]{\Xside}$ we start by expressing $\beta_{\cU}[M]$ as a difference of classes of objects in $\cU$, which is certainly not agnostic about $\cU$.
See Remark~\ref{r:XCC}\ref{r:XCC-bijection} and the discussion that follows for a partial resolution of this issue.
\end{remark}

Also as a consequence of Proposition~\ref{p:cluchaX-mult}, we have particular interest in the values of $\clucha[\cTU]{\Xside}(S)$ when $S$ is simple.
If $\cU$ is maximally mutable and has no loops, so that $\simpmod{\cU}{U}=\Extfun{\cU}{\mut{\cU}{U}}$ for each $U\in\indec{\stab{\cU}}$ (and in particular $\simpmod{\cU}{U}\in\fpmod{\stab{\cU}}$ is a valid input to the cluster character), these can be computed using the next result.

\begin{proposition}
\label{p:exch-seq-mid-terms} Let $\cC$ be a Krull--Schmidt cluster category, let $\cTU\ctsubcat\cC$, let $U\in \exch{\cU}$, and let $M=\Extfun{\cU}{(\mut{\cU}{U})}$. Then we may choose $\modlift{M}{\cU}{\pm}=\exchmon{\cU}{U}{\pm}$
to be the middle terms of the corresponding exchange conflations.
\end{proposition}
\begin{proof}
From the exchange conflations $\mut{\cU}{U} \infl \exchmon{\cU}{U}{+} \defl U \confl$ and $U \infl \exchmon{\cU}{U}{-} \defl \mut{\cU}{U}\confl$ we may calculate $\ind{\cC}{\cU}[\mut{\cU}{U}]=[\exchmon{\cU}{U}{-}]-[U]$ and $\coind{\cC}{\cU}[\mut{\cU}{U}]=[\exchmon{\cU}{U}{+}]-[U]$.
Thus $\beta_{\cU}[M]=\coind{\cC}{\cU}[\mut{\cU}{U}]-\ind{\cC}{\cU}[\mut{\cU}{U}]=[\exchmon{\cU}{U}{+}]-[\exchmon{\cU}{U}{-}]$,
and the result follows.
\end{proof}

\begin{corollary}
\label{c:X-cc-on-simples}
Let $\cC$ be a compact or skew-symmetric cluster category, and let $\cTU\ctsubcat\cC$. Then for each $U\in\exch{\cU}$ we have
\[\clucha[\cTU]{\Xside}(\Extfun{\cU}(\mut{\cU}{U}))=x^{\stabindbar{\cU}{\cT}{[\Extfun{\cU}(\mut{\cU}{U})]}}\Fpoly(\Extfun{\cT}\exchmon{\cU}{U}{+})\Fpoly(\Extfun{\cT}\exchmon{\cU}{U}{-})^{-1}.\]
\end{corollary}

When there is no loop at $U\in\exch{\cU}$, this gives us an expression for the value of $\clucha[\cTU]{\Xside}$ on $\simpmod{\cU}{U}=\Extfun{\cU}(\mut{\cU}{U})$, which we may make more explicit when $\cC$ is compact.

\begin{proposition}\label{p:prod-form-of-X-cc-on-simples}
Let $\cC$ be a compact cluster category and let $\cTU\ctsubcat\cC$. 
If there is no loop at $U\in\exch{\cU}$, then
\begin{equation}
\label{eq:prod-form-of-X-cc-on-simples}
\clucha[\cTU]{\Xside}(\simpmod{\cU}{U})=x^{\stabindbar{\cU}{\cT}[\simpmod{\cU}{U}]}\prod_{V\in \indec \cU} \Fpoly(\Extfun{\cT}{V})^{\exchmatentry{V,U}^{\cU}}.
\end{equation}
\end{proposition}

\begin{proof}
By Proposition~\ref{p:decomp-exch-terms}, we have $\exchmon{\cU}{U}{+}=\bigdsum_{V\in \indec \cU} V^{\Gabmatentry{V,U}^{\cU}}$ and $\exchmon{\cU}{U}{-}=\bigdsum_{V\in \indec \cU} V^{\frac{\dimdivalg{U}}{\dimdivalg{V}}\Gabmatentry{U,V}^{\cU}}$.
Then by \eqref{eq:F-poly-split},
\begin{align*} \Fpoly(\Extfun{\cT}\exchmon{\cU}{U}{+})\Fpoly(\Extfun{\cT}\exchmon{\cU}{U}{-})^{-1} & = \prod_{V\in \indec \cU} \Fpoly(\Extfun{\cT}V)^{\Gabmatentry{V,U}^{\cU}}\prod_{V\in \indec \cU} \Fpoly(\Extfun{\cT}V)^{-\frac{\dimdivalg{U}}{\dimdivalg{V}}\Gabmatentry{U,V}^{\cU}} \\ 
& = \prod_{V\in \indec \cU} \Fpoly(\Extfun{\cT}V)^{\Gabmatentry{V,U}^{\cU}-\frac{\dimdivalg{U}}{\dimdivalg{V}}\Gabmatentry{U,V}^{\cU}}.
\end{align*}
Finally, $\Gabmatentry{V,U}^{\cU}-\frac{\dimdivalg{U}}{\dimdivalg{V}}\Gabmatentry{U,V}^{\cU}=\exchmatentry{V,U}^\cU$ by \eqref{eq:exch-mat-vs-Gab-mat}, and so we obtain the result from Corollary~\ref{c:X-cc-on-simples}.
\end{proof}

Under slightly stronger local finiteness assumptions, we may extend $\clucha[\cTU]{\Xside}$ to a character $\Kgp{\fd{\cU}}\to\Frac{\Kgp{\fd{\cT}}}$, defined also on modules with support on the projective objects of $\cU$, although to do so requires using the basis of simple modules rather than giving a `basis-free' formula as in \eqref{eq:clucha-X}.

\begin{definition}
\label{d:X-CC-proj}
Assume $\cC$ is a compact cluster category and that $\cTU\ctsubcat\cC$.
Let $P\in\cU$ be an indecomposable projective such that $\cU\setminus P$ is functorially finite in $\cU$ and $\cU$ has no loop at $P$, and let $P\to\exchmon{\cU}{P}{-}$ and $\exchmon{\cU}{P}{+}\to P$ be, respectively, left and right $(\cU\setminus P)$-approximations of $P$.
Then define
\[\clucha[\cTU]{\Xside}(\simpmod{\cU}{P})=x^{\indbar{\cU}{\cT}[\simpmod{\cU}{P}]}\Fpoly(\Extfun{\cT}\exchmon{\cU}{P}{+})\Fpoly(\Extfun{\cT}\exchmon{\cU}{P}{-})^{-1}.\]
If $\cU$ is maximally mutable and has no loops, and $\cU\setminus P$ is functorially finite in $\cU$ for all indecomposable projectives $P$, then by combining this with \eqref{eq:clucha-X} we obtain a character $\clucha[\cTU]{\Xside}\colon\Kgp{\fd{\cU}}\to\Fracbar{\Kgp{\fd{\cT}}}$, using that the classes $[\simpmod{\cU}{P}]$ for $P$ indecomposable projective are the basis of a complement of $\Kgp{\fd{\stab{\cU}}}$ in $\Kgp{\fd{\cU}}$.
\end{definition}

\begin{remark}
The assumption that there is no loop at $P$ plays no logical role in Definition~\ref{d:X-CC-proj}, but instead makes the definition compatible with Corollary~\ref{c:X-cc-on-simples} for mutable indecomposables---the analogous expression for $U\in\exch\cU$ only computes the $\Xside$-cluster character of $\simpmod{\cU}{U}$ when there is no loop at $\cU$, suggesting that a different definition may be more natural if there is a loop at $P$.
Below, we will always be assuming that $\cU$ has no loop at $P$, and this assumption does have a logical role in the proof of Proposition~\ref{p:prod-form-of-X-cc-on-simples-at-projectives}.
\end{remark}

Recall that the expression $\exchmatentry{V,U}^{\cU}=\Gabmatentry{V,U}^{\cU}-\frac{\dimdivalg{U}}{\dimdivalg{V}}\Gabmatentry{U,V}^{\cU}$ from \eqref{eq:exch-mat-vs-Gab-mat}, used in the proof of Proposition~\ref{p:prod-form-of-X-cc-on-simples} for $U\in\exch{\cU}$, exploits the fact (Proposition~\ref{p:mut-v-lf}) that $\cU$ is locally finite at $U$ so that the quantities on the right-hand side of the expression are finite (Proposition~\ref{p:loc-finite-quiver}).
If $\cU$ is locally finite at an indecomposable projective $P\in\indec{\cU}$, we may simply extend \eqref{eq:exch-mat-vs-Gab-mat} by taking $\exchmatentry{V,P}^{\cU}\defeq\Gabmatentry{V,P}^{\cU}-\frac{\dimdivalg{P}}{\dimdivalg{V}}\Gabmatentry{P,V}^{\cU}$ as a definition.
We use this extension in the next result.

\begin{proposition}
\label{p:prod-form-of-X-cc-on-simples-at-projectives}
Assume $\cC$ is a compact cluster category, let $\cTU\ctsubcat\cC$, and let $P\in\cU$ be indecomposable projective.
If $\cU$ is locally finite at $P$ and has no loop at $P$, then $\cU\setminus P$ is functorially finite in $\cU$ and
\begin{equation}
\label{eq:X-CC-proj}
\clucha[\cTU]{\Xside}(\simpmod{\cU}{P})=x^{\indbar{\cU}{\cT}[\simpmod{\cU}{P}]}\prod_{V\in \indec \cU} \Fpoly(\Extfun{\cT}{V})^{\exchmatentry{V,P}^{\cU}}.
\end{equation}
\end{proposition}
\begin{proof}
The proof is essentially the same as for $U\in\exch{\cU}$: local finiteness at $P$ means that there is a source map $P\to \bigdsum_{V\in \indec \cU} V^{\frac{\dimdivalg{U}}{\dimdivalg{V}}\Gabmatentry{P,V}^{\cU}}$ and a sink map $\bigdsum_{V\in \indec \cU} V^{\Gabmatentry{V,P}^{\cU}}\to P$ by Lemma~\ref{l:sink-map-construction}, and these are left and right $(\cU\setminus P)$-approximations of $P$ since $\cU$ has no loop at $P$.
The proof of Proposition~\ref{p:prod-form-of-X-cc-on-simples} then applies without any further changes.
\end{proof}

\begin{theorem}\label{t:X-clucha-mutation}
Let $\cC$ be a compact cluster category, let $\cTU\ctsubcat\cC$, and assume $\cU$ is locally finite. Let $U\in\exch{\cU}$, with associated mutations $\cU'=\mut{U}{\cU}$ and $U'=\mut{\cU}{U}$, and assume that $\cU$ has no loop or $2$-cycle at $U$.
Let $V\in\exch{\cU'}$ and assume, if $V\ne U'$, that $\cU$ has no loop at $V$. Then we have
\[\clucha[\cTU']{\Xside}(\simpmod{\cU'}{V})=\begin{cases}
\clucha[\cTU]{\Xside}(\simpmod{\cU}{U})^{-1}&\text{if}\ V=U',\\
\clucha[\cTU]{\Xside}(\simpmod{\cU}{V})\clucha[\cT,\,\cU]{\Xside}(\simpmod{\cU}{U})^{[\exchmatentry{U,V}^{\cU}]_{+}}(1+\clucha[\cT,\,\cU]{\Xside}(\simpmod{\cU}{U}))^{-\exchmatentry{U,V}^{\cU}}&\text{otherwise.}
\end{cases}\]
\end{theorem}

\begin{proof}
Since there is no loop at $U\in\cU$ there is also no loop at $U'\in\cU'$ by Corollary~\ref{c:no-loops-mutates}, and so $\simpmod{\cU}{U}=\Extfun{\cU}{U'}$ and $\simpmod{\cU'}{U'}=\Extfun{\cU'}{U}$ by Lemma~\ref{l:props-of-K0-mod-T}\ref{l:props-mod-T-simple-eq-E}.
Now $\exchmon{\cU}{U}{\pm}=\exchmon{\cU'}{(U')}{\mp}$, and $\indbar{\cU'}{\cT}[\simpmod{\cU'}{U'}]=-\indbar{\cU}{\cT}[\simpmod{\cU}{U}]$ by Theorem~\ref{t:one-step-X-mut}\ref{t:one-step-X-mut-indbar-at-mut}, so Corollary~\ref{c:X-cc-on-simples} gives
\begin{align*}
\clucha[\cTU']{\Xside}(\simpmod{\cU'}{U'}) & = x^{\indbar{\cU'}{\cT}[\simpmod{\cU'}{U'}]}\Fpoly(\Extfun{\cT}\exchmon{\cU'}{(U')}{+})\Fpoly(\Extfun{\cT}\exchmon{\cU'}{(U')}{-})^{-1} \\
& = x^{-\indbar{\cU}{\cT}[\simpmod{\cU}{U}]}\Fpoly(\Extfun{\cT}\exchmon{\cU}{U}{-})\Fpoly(\Extfun{\cT}\exchmon{\cU}{U}{+})^{-1} \\
& = \clucha[\cTU]{\Xside}(\simpmod{\cU}{U})^{-1}.
\end{align*}
In the second case, we use Proposition~\ref{p:prod-form-of-X-cc-on-simples} to expand the expression, then prior results to simplify:
\begin{align*}
\clucha[\cTU']{\Xside}(\simpmod{\cU'}{V}) & =x^{\indbar{\cU'}{\cT}[\simpmod{\cU'}{V}]}\prod_{W\in \indec \cU'} \Fpoly(\Extfun{\cT}{W})^{\exchmatentry{W,V}^{U'}} \\ 
& \leftstackrel{\eqref{eq:exch-mat-mutation-clust-str}}{=} x^{\indbar{\cU'}{\cT}[\simpmod{\cU'}{V}]}\Fpoly(\Extfun{\cT}U')^{-\exchmatentry{U,V}^{\cU}}
\prod_{W\in \indec \cU'\setminus U'} \Fpoly(\Extfun{\cT}{W})^{\exchmatentry{W,V}^{\cU}+\exchmatentry{W,U}^{\cU}[\exchmatentry{U,V}^{\cU}]_{+}+[\exchmatentry{W,U}^{\cU}]_{-}\exchmatentry{U,V}^{\cU}}\\
& = \begin{multlined}[t]x^{\indbar{\cU'}{\cT}[\simpmod{\cU'}{V}]}\Fpoly(\Extfun{\cT}U')^{-\exchmatentry{U,V}^{\cU}} \prod_{W\in \indec \cU\setminus U} \Fpoly(\Extfun{\cT}{W})^{\exchmatentry{W,V}^{\cU}}\\
\cdot\prod_{W\in \indec \cU\setminus U} \Fpoly(\Extfun{\cT}{W})^{\exchmatentry{W,U}^{\cU}[\exchmatentry{U,V}^{\cU}]_{+}+[\exchmatentry{W,U}^{\cU}]_{-}\exchmatentry{U,V}^{\cU}}\end{multlined}\\
& \leftstackrel{\eqref{eq:prod-form-of-X-cc-on-simples}/\eqref{eq:X-CC-proj}}{=} \begin{multlined}[t]x^{\indbar{\cU'}{\cT}[\simpmod{\cU'}{V}]}\Fpoly(\Extfun{\cT}U')^{-\exchmatentry{U,V}^{\cU}}\bigl(x^{\indbar{\cU}{\cT}[\simpmod{\cU}{V}]}\Fpoly(\Extfun{\cT}U)^{\exchmatentry{U,V}^{\cU}}\bigr)^{-1} \clucha[\cTU]{\Xside}(\simpmod{\cU}{V})\\
\cdot \prod_{W\in \indec \cU\setminus U} \Fpoly(\Extfun{\cT}{W})^{\exchmatentry{W,U}^{\cU}[\exchmatentry{U,V}^{\cU}]_{+}+[\exchmatentry{W,U}^{\cU}]_{-}\exchmatentry{U,V}^{\cU}}\end{multlined}\\
& = \begin{multlined}[t]x^{\indbar{\cU'}{\cT}[\simpmod{\cU'}{V}]-\indbar{\cU}{\cT}[\simpmod{\cU}{V}]}(\Fpoly(\Extfun{\cT}U')\Fpoly(\Extfun{\cT}U))^{-\exchmatentry{U,V}^{\cU}}\clucha[\cTU]{\Xside}(\simpmod{\cU}{V})\\
\cdot \Bigl(\prod_{W\in \indec \cU\setminus U} \Fpoly(\Extfun{\cT}{W})^{\exchmatentry{W,U}^{\cU}}\Bigr)^{[\exchmatentry{U,V}^{\cU}]_{+}}\Bigl(\prod_{W\in \indec \cU\setminus U} \Fpoly(\Extfun{\cT}{W})^{[\exchmatentry{W,U}^{\cU}]_{-}}\Bigr)^{\exchmatentry{U,V}^{\cU}}\end{multlined}\\
& \leftstackrel{\eqref{eq:prod-form-of-X-cc-on-simples}}{=} \begin{multlined}[t]x^{\indbar{\cU'}{\cT}[\simpmod{\cU'}{V}]-\indbar{\cU}{\cT}[\simpmod{\cU}{V}]}(\Fpoly(\Extfun{\cT}U')\Fpoly(\Extfun{\cT}U))^{-\exchmatentry{U,V}^{\cU}}\clucha[\cTU]{\Xside}(\simpmod{\cU}{V})\\
\cdot \bigl(x^{-\indbar{\cU}{\cT}[\simpmod{\cU}{U}]}\clucha[\cTU]{\Xside}(\simpmod{\cU}{U})\bigr)^{[\exchmatentry{U,V}^{\cU}]_{+}}\Bigl(\prod_{W\in \indec \cU\setminus U} \Fpoly(\Extfun{\cT}{W})^{[\exchmatentry{W,U}^{\cU}]_{-}}\Bigr)^{\exchmatentry{U,V}^{\cU}}\end{multlined}\\
& \leftstackrel{\text{\ref{p:decomp-exch-terms}}}{=} \begin{multlined}[t]x^{\indbar{\cU'}{\cT}[\simpmod{\cU'}{V}]-\indbar{\cU}{\cT}[\simpmod{\cU}{V}]-[\exchmatentry{U,V}^{\cU}]_{+}\indbar{\cU}{\cT}[\simpmod{\cU}{U}]}(\Fpoly(\Extfun{\cT}U')\Fpoly(\Extfun{\cT}U))^{-\exchmatentry{U,V}^{\cU}}\clucha[\cTU]{\Xside}(\simpmod{\cU}{V})\\
\cdot\clucha[\cTU]{\Xside}(\simpmod{\cU}{U})^{[\exchmatentry{U,V}^{\cU}]_{+}}\Fpoly(\Extfun{\cT}\exchmon{\cU}{U}{-})^{\exchmatentry{U,V}^{\cU}}\end{multlined}
\end{align*}
Here we use the fact that $\cU$ has no loop or $2$-cycle at $U$ to apply \eqref{eq:exch-mat-mutation-clust-str} and Proposition~\ref{p:decomp-exch-terms}.
The second use of \eqref{eq:prod-form-of-X-cc-on-simples} requires only that $\cU$ has no loop at $U$, whereas for the first application of \eqref{eq:prod-form-of-X-cc-on-simples}, or the only use of \eqref{eq:X-CC-proj}, we use that $\cU$ has no loop at $V$.

To continue, we will use \eqref{eq:F-poly-non-split} to replace $\Fpoly(\Extfun{\cT}U')\Fpoly(\Extfun{\cT}U)$: since $\cU$ has no loop at $U$, we have in particular that $\rank_{\divalg{U}}\Ext{1}{\cC}{U}{U'}=1$ by Lemma~\ref{l:props-of-K0-mod-T}\ref{l:props-mod-T-simple-eq-E}.
Thus \eqref{eq:F-poly-non-split} applies to give
\[ \Fpoly(\Extfun{\cT}U')\Fpoly(\Extfun{\cT}U)= x^{[\Ker \Extfun{\cT}i_{+}]}\Fpoly(\Extfun{\cT}\exchmon{\cU}{U}{+})+x^{[\Ker \Extfun{\cT}i_{-}]}\Fpoly(\Extfun{\cT}\exchmon{\cU}{U}{-}),\]
where $i_-\colon U\infl\exchmon{U}{\cU}{-}$ and $i_+\colon U'\infl\exchmon{U}{\cU}{+}$ are taken from the (non-split) exchange conflations for $U\in\cU$.
We thus have
\begin{align*}
\Fpoly(\Extfun{\cT}U')\Fpoly(\Extfun{\cT}U)\Fpoly(\Extfun{\cT}\exchmon{\cU}{U}{-})^{-1}
&=x^{[\Ker \Extfun{\cT}i_{+}]}\Fpoly(\Extfun{\cT}\exchmon{\cU}{U}{+})\Fpoly(\Extfun{\cT}\exchmon{\cU}{U}{-})^{-1}+x^{[\Ker \Extfun{\cT}i_{-}]}\\
&\leftstackrel{\text{\ref{c:indbar-lift}}}{=}x^{[\Ker \Extfun{\cT}i_{-}]+\indbar{\cU}{\cT}[\simpmod{\cU}{U}]}\Fpoly(\Extfun{\cT}\exchmon{\cU}{U}{+})\Fpoly(\Extfun{\cT}\exchmon{\cU}{U}{-})^{-1}+x^{[\Ker \Extfun{\cT}i_{-}]}\\
&\leftstackrel{\eqref{eq:prod-form-of-X-cc-on-simples}}{=}x^{[\Ker\Extfun{\cT}i_{-}]}(1+\clucha[\cTU]{\Xside}(\simpmod{\cU}{U})).
\end{align*}
Substituting back to continue our previous calculation, we then have
\begin{align*}
\clucha[\cTU']{\Xside}(\simpmod{\cU'}{V})
& = \begin{multlined}[t]x^{\indbar{\cU'}{\cT}[\simpmod{\cU'}{V}]-\indbar{\cU}{\cT}[\simpmod{\cU}{V}]-[\exchmatentry{U,V}^{\cU}]_{+}\indbar{\cU}{\cT}[\simpmod{\cU}{U}]-\exchmatentry{U,V}^{\cU}[\Ker \Extfun{\cT}i_{-}]}\bigl(1+\clucha[\cTU]{\Xside}\bigl(\simpmod{\cU}{U})\bigr)^{-\exchmatentry{U,V}^{\cU}}\\
\cdot \clucha[\cTU]{\Xside}(\simpmod{\cU}{V})\clucha[\cTU]{\Xside}(\simpmod{\cU}{U})^{[\exchmatentry{U,V}^{\cU}]_{+}}\end{multlined}\\
& \leftstackrel{\text{\ref{c:one-step-X-mut}\ref{c:one-step-X-mut-indbar-away-from-mut}}}{=} \clucha[\cTU]{\Xside}(\simpmod{\cU}{V})\clucha[\cTU]{\Xside}(\simpmod{\cU}{U})^{[\exchmatentry{U,V}^{\cU}]_{+}}(1+\clucha[\cTU]{\Xside}(\simpmod{\cU}{U}))^{-\exchmatentry{U,V}^{\cU}}
\end{align*}
as required, using once more for the final equality that $\cU$ has no loop or $2$-cycle at $U$.
\end{proof}

\begin{theorem}
\label{t:X-clucha}
Let $\cC$ be a compact cluster category with $\cT\ctsubcat\cC$, and assume that $(\cC,\cT)$ has a cluster structure.
Let $\cU$ be the reachable cluster-tilting subcategory corresponding to a seed $s$ of the associated cluster algebra under the bijection of Theorem~\ref{t:Aside-bijection}.
Then the functions $\clucha[\cTU]{\Xside}(S)$, for $S$ a simple $\stab{\cU}$-module, are the $\Xside$-cluster variables of $s$ associated to mutable indices.

If we assume additionally that all cluster-tilting subcategories reachable from $\cT$ are locally finite and have no loops at any indecomposable objects (including the projectives), then the functions $\clucha[\cTU]{\Xside}(S)$, for $S$ any simple $\cU$-module, give all the $\Xside$-cluster variables of $s$.
\end{theorem}
\begin{proof}
Observe that $\clucha[\cT\!,\,\cT]{\Xside}(M)=x^{[M]}$ for $M\in \fd\stab{\cT}$, since in this case $\Extfun{\cT}\modlift{M}{\cU}{\pm}=0$ and $\indbar{\cT}{\cT}$ is the identity.
Similarly $\clucha[\cT\!,\,\cT]{\Xside}(\simpmod{\cU}{P})=x^{[\simpmod{\cU}{P}]}$ for $P\in\indec{\cU}$ projective, by \eqref{eq:X-CC-proj}.
Hence, the $\Xside$-cluster character computes the initial $\Xside$-cluster variables correctly.
Now the result follows by induction using Theorem~\ref{t:X-clucha-mutation}: the assumption that $(\cC,\cT)$ has a cluster structure means that this theorem applies to any $\cU\ctsubcat\cC$ reachable from $\cT$ to see that the $\Xside$-cluster characters of the simple $\stab{\cU}$-modules transform via $\Xside$-cluster mutations, whereas under the stronger assumptions this theorem gives the same conclusion for the $\Xside$-cluster characters of all simple $\cU$-modules.
\end{proof}

\begin{remark} {\ }
\label{r:XCC}
\begin{enumerate}
\item 
When $\cC$ has a (weak) cluster structure, any $\cU\ctsubcat\cC$ is locally finite at all $U\in\indec{\stab{\cU}}=\exch{\cU}$, so the local finiteness assumption in Theorem~\ref{t:X-clucha} reduces to local finiteness at indecomposable projective objects.
If $\cC$ has finite rank, then all cluster-tilting subcategories of $\cC$ are locally finite, so this assumption holds.

If $\cC$ is Hom-finite and finite rank, and $\bK$ is algebraically closed, then there is no loop at $U\in\indec{\cU}$ whenever $\simpmod{\cU}{U}\in\fd{\cU}$ has finite projective dimension, by the strong no loops theorem \cite{ILP}.
When $\cC$ is exact, it is not unusual \cite{GLS-PFV,BIRS1} (see also \cite[\S3]{Pressland}) that every $\cU\ctsubcat\cC$ has finite global dimension, so this result (or even the original no loops theorem \cite{Igusa}) applies.
In this case we therefore conclude that every $\cU\ctsubcat\cC$ is locally finite and has no loops at any of its indecomposable objects, as required by Theorem~\ref{t:X-clucha}.

\item Proposition~\ref{p:prod-form-of-X-cc-on-simples} tells us that the $\Xside$-cluster \emph{characters} have `separation' in the sense of Fomin--Zelevinsky \cite[Prop.~3.13]{FZ-CA4}, i.e.\ can be written as a monomial corresponding to a $\mathbf{c}$-vector (by Theorem~\ref{t:c-vec-mut-formula}) multiplied by a product of $\Fpoly$-polynomials (which agree with their cluster theoretic counterparts by Theorem~\ref{t:Fpoly-mutation}).
Since Theorem~\ref{t:X-clucha} says that $\Xside$-cluster characters compute $\Xside$-cluster \emph{variables} in the presence of a cluster structure, it reproves the separation formula for the latter in purely categorical terms, for those $\Xside$-cluster algebras admitting a cluster categorical realisation.
\item\label{r:XCC-bijection}
A straightforward corollary of Theorem~\ref{t:X-clucha} is that $\clucha[\cT\!,\,\blank]{\Xside}$ gives a surjection from the set of simple modules for all cluster-tilting subcategories $\cU$ reachable from $\cT$ to the set of $\Xside$-cluster variables for the cluster algebra associated to $\cC$, but this map is typically not injective.
\end{enumerate}
\end{remark}

To resolve Remark~\ref{r:XCC}\ref{r:XCC-bijection}, we make the following definition.

\begin{definition}
\label{d:mutation-pair}
A \emph{mutation pair} in a cluster category $\cC$ is an ordered pair of objects $(U,V)$ for which there exists $\cU\ctsubcat\cC$ with $U\in\cU$ and $V=\mut{\cU}{U}$.
In this case we say that $\cU$ \emph{realises} $(U,V)$.
A mutation pair for $(\cC,\cT)$, where $\cT\ctsubcat\cC$ is fixed, is a mutation pair for $\cC$ realised by some $\cU$ reachable from $\cT$.
\end{definition}

While the cluster-tilting subcategory $\cU$ in Definition~\ref{d:mutation-pair} is usually not unique when it exists, various quantities associated to $U\in\cU$ turn out to depend only on the mutation pair $(U,\mut{\cU}{U})$, and not on $\cU$ itself.

\begin{proposition}
\label{p:XCC-mutation-pair}
Let $\cC$ be a Krull--Schmidt cluster category and let $\cT\ctsubcat\cC$ have the property that any $\cT'\ctsubcat\cC$ reachable from $\cT$ has no loops.
Let $(U,V)$ be a mutation pair for $(\cC,\cT)$, realised by $\cU,\cU'\ctsubcat\cC$.
Then $\exchmon{\cU}{U}{\pm}\iso\exchmon{\cU'}{U}{\pm}$ and $\stabindbar{\cU}{\cT}[\simpmod{\cU}{U}]=\stabindbar{\cU'}{\cT}[\simpmod{\cU'}{U}]$.
\end{proposition}
\begin{proof}
Let $\cT'\ctsubcat\cC$ realise $(U,V)$ and be reachable from $\cT$, so that $\cT'$ has no loops.
Then $\rank_{\divalg{U}}\Ext{1}{\cC}{U}{V}=\rank_{\divalg{U}}\Ext{1}{\cC}{U}{\mut{\cT'}{U}}=1$ by Lemma~\ref{l:props-of-K0-mod-T}\ref{l:props-mod-T-simple-eq-E}.
Thus, by Lemma~\ref{l:rk1-unique-middle-term}, the exchange conflations $U\infl\exchmon{\cU}{U}{-}\defl V\confl$ and $U\infl\exchmon{\cU'}{U}{-}\defl V\confl$ are isomorphic, as are the exchange conflations in the opposite direction: in particular, $\exchmon{\cU}{U}{\pm}\iso\exchmon{\cU'}{U}{\pm}$.
Using Corollary~\ref{c:indbar-lift}, we may calculate $\stabindbar{\cU}{\cT}[\simpmod{\cU}{U}]=[\lefterror{1}{\cU}{U'}|_{\cT}]-[\righterror{1}{\cU}{U'}|_{\cT}]$ to see that this quantity also depends only on the exchange conflations, this being true of the modules $\lefterror{1}{\cU}{U'}$ and $\righterror{1}{\cU}{U'}$.
\end{proof}

Thanks to a result of Cao--Keller--Qin \cite[Thm.~7.8]{CKQ}, we may now deduce the following.

\begin{corollary}
Under the assumptions of Theorem~\ref{t:X-clucha}, let $\clucha[\cT]{\Xside}$ be the map sending an exchange pair $(U,V)$ for $(\cC,\cT)$ to $\clucha[\cTU]{\Xside}(\simpmod{\cU}{U})$, for any reachable $\cU\ctsubcat\cC$ realising $(U,V)$.
Then $\clucha[\cT]{\Xside}$ is a bijection between exchange pairs for $(\cC,\cT)$ and the $\Xside$-cluster variables of the associated cluster algebra associated to mutable indices.
\end{corollary}
\begin{proof}
The map $\clucha[\cT]{\Xside}$ is well-defined by Proposition~\ref{p:XCC-mutation-pair}, and surjective by Theorem~\ref{t:X-clucha}.
The $\Xside$-cluster variables are in bijection with exchange pairs in the $\Aside$-cluster algebra by \cite[Thm.~7.8]{CKQ}, which are in bijection with exchange pairs for $(\cC,\cT)$ by Theorem~\ref{t:Aside-bijection}.
Since all of these maps are compatible with mutations, $\clucha[\cT]{\Xside}$ is thus also a bijection.
\end{proof}

The following proposition demonstrates that the pushforward under $\beta_{\cT}$ of an $\Xside$-cluster character is a ratio of $\Aside$-cluster characters (from the same cluster).
Thus $(\beta_{\cT})_*$ is closely related to the change of variables from $y$ to $\hat{y}$ appearing in \cite[Eq.~3.7]{FZ-CA4}.

\begin{proposition}\label{p:cluchaX-ratio}
Let $\cC$ be a Krull--Schmidt cluster category, and let $\cTU\ctsubcat\cC$, with $\cU$ maximally mutable.
Assume either that $\cC$ has finite rank or that $\cU$ is reachable from $\cT$.
Then for $M\in\fd \stab{\cU}$, we have
\[(\beta_\cT)_*\clucha[\cTU]{\Xside}(M)=\frac{\clucha[\cT]{\Aside}(\modlift{M}{\cU}{+})}{\clucha[\cT]{\Aside}(\modlift{M}{\cU}{-})}.\]
\end{proposition}

\begin{proof}
The assumption that $\cC$ is finite rank or $\cU$ is reachable from $\cT$ means that $\Extfun{\cT}{\modlift{M}{\cU}{\pm}}\in\fd{\stab{\cT}}$, and so $\clucha[\cT]{\Aside}(\modlift{M}{\cU}{\pm})$ are well-defined.
Now given \eqref{eq:clucha-A-alt} and \eqref{eq:clucha-X-alt}, and using that $\modlift{M}{\cU}{\pm}\in\cU$, the statement reduces to the claim that
\[(\beta_{\cT})_*(x^{\stabindbar{\cU}{\cT}{[M]}})=a^{\ind{\cU}{\cT}{[\modlift{M}{\cU}{+}]}-\ind{\cU}{\cT}{[\modlift{M}{\cU}{-}]}},\]
which holds since
$\beta_{\cT}(\stabindbar{\cU}{\cT}{[M]})=\ind{\cU}{\cT}{[\modlift{M}{\cU}{+}]}-\ind{\cU}{\cT}{[\modlift{M}{\cU}{-}]}$ by Theorem~\ref{t:exch-isos} and \eqref{eq:TM-pm-def}.
\end{proof}

A consequence of Proposition~\ref{p:cluchaX-ratio} and Theorem~\ref{t:X-clucha-mutation} is that if $\cU\ctsubcat\cC$ has no loops or $2$-cycles, then the ratios $\clucha[\cT]{\Aside}(\exchmon{U}{\cU}{+})/\clucha[\cT]{\Aside}(\exchmon{U}{\cU}{-})$ obey the $\Xside$-cluster mutation rules under mutation of $\cU$; cf.~\cite[Prop.~3.9]{FZ-CA4}.

\sectionbreak
\section{Quantisation}

\subsection{Categorical quantum data}
\label{s:cat-quasi-comm}

We now carry out a similar programme as in Section~\ref{s:cat-ex-mat} for the quasi-commutation matrices in quantum cluster algebras.
To do so, we will continue to follow the philosophy of Fock and Goncharov's approach (see also Fan Qin \cite{Qin}) which sees the quasi-commutation matrix as encoding a form adjoint to the exchange matrix.
While the exchange and quasi-commutation matrices appear to play very different roles, these roles reverse when swapping $\Aside$ and $\Xside$, restoring the symmetry.

To start, we need a map $\lambda_{\cT}$ and a form $\pform{\blank}{\blank}{\cT}$ analogous to $\beta_{\cT}$ and $\sform{\blank}{\blank}{\cT}$.
To tie the maps $\beta_{\cT}$ and $\lambda_{\cT}$ together, in order to define quantum cluster algebras and categories, we will also need a compatibility condition.
In contrast to $\beta_{\cT}=-p_{\cT}$, which is defined in terms of projective resolutions, we do not necessarily have a single, natural choice for $\lambda_{\cT}$.
Rather, a choice must be made; the moduli of such choices is a linear algebra problem, considered in detail in \cite{GellertLampe}.An important consequence of this is that we do not necessarily have an intrinsic form $\pform{\blank}{\blank}{\cT}$ for each cluster-tilting subcategory, as we did for $\sform{\blank}{\blank}{\cT}$, but rather must define each form $\pform{\blank}{\blank}{\cU}$ in terms of some initial choice $\pform{\blank}{\blank}{\cT}$, which we will do using the index and coindex maps by analogy with Definition~\ref{d:mutated-s-form}.

Our definition uses adjunction, taken with respect to the canonical evaluation pairing $\evform{\blank}{\blank}\colon\Kgp{\cT}\times\dual{\Kgp{\cT}}\to\integ$, for which the induced map $\delta_{\dual{\Kgp{\cT}}}\colon\dual{\Kgp{\cT}}\to\dual{\Kgp{\cT}}$ is the identity (see Section~\ref{s:adj}).
It follows from \eqref{eq:evform-genform} that
\begin{equation}
\label{eq:canform-comparison}
\canform{\blank}{\blank}{\cT}=\evform{\blank}{\sdual{\cT}(\blank)},
\end{equation}
where $\canform{\blank}{\blank}{\cT}$ is as in Proposition~\ref{p:K-duality}.
We also wish to restrict $\beta_{\cT}$ to $\Kgp{\fd{\stab{\cT}}}$, so must assume that $\cT$ is maximally mutable for this to be a subgroup of $\Kgp{\fpmod{\stab{\cT}}}$.
Recall from Corollary~\ref{c:weak-clust-struct}\ref{c:weak-clust-struct-finite} that this assumption is always satisfied when $\cC$ is Krull--Schmidt and has finite rank.

\begin{definition}
\label{d:quantum-hom}
Let $\cC$ be a Krull--Schmidt cluster category and let $\cT\ctsubcat\cC$ be maximally mutable.
A \emph{quantum datum} for $\cT$ is a map
$\lambda_{\cT}\colon \Kgp{\cT} \to \dual{\Kgp{\cT}}$ satisfying
\begin{enumerate}
\item (skew-symmetry) $\adj{\lambda_{\cT}}=-\lambda_{\cT}$, and
\item (compatibility) $\adj{\lambda_{\cT}}\circ\beta_{\cT}|_{\Kgp{\fd{\stab{\cT}}}}=2(\sdual{\cT}\circ\sinc{\cT}|_{\Kgp{\fd{\stab{\cT}}}})$.
\end{enumerate}
\end{definition}

This is the same notion of skew-symmetry as in Corollary~\ref{c:beta-skew-symmetric}, and requires us to use $\dual{\Kgp{\cT}}$ as the codomain of $\lambda_{\cT}$, rather than $\Kgpnum{\lfd \cT}$ as might have been expected.  
Recall that $\sinc{\cT}\colon\Kgpnum{\lfd{\stab{\cT}}}\to\Kgpnum{\lfd{\cT}}$ is the map induced from the inclusion of categories.  Since $\sdual{\cT}$ and $\sinc{\cT}$ are injective, the compatibility condition implies that $\beta_{\cT}$ is injective on $\Kgp{\fd{\stab{\cT}}}$.
Below, we will mostly leave the restriction of $\beta_{\cT}$ and $\sinc{\cT}$ to this subgroup implicit.

\begin{remark}
\label{r:dual-compatibility}
By taking adjoints and using Proposition~\ref{p:K-duality}, the compatibility condition is equivalent to requiring that
$\adj{\beta_{\cT}}\circ\lambda_{\cT} = 2(\sproj{\cT}\circ\pdual{\cT})$, where $\sproj{\cT}=\adj{(\sinc{\cT})}$.
The coefficient $2$ appearing in this condition reflects the usual choices made in defining quantum cluster algebras: the map $\lambda$ corresponds to the quasi-commutation matrix $L$ of a seed, whose entries are usually the powers of $q^{1/2}$ (for $q$ the quantum parameter) appearing in the quasi-commutation relations for the cluster variables of that seed.
While it would be in some ways more natural (and reduce the proliferation of $\frac{1}{2}$s and $2$s) to absorb this coefficient into $\lambda_{\cT}$, by allowing it to take values in $\dual{\Kgp{\cT}}\otimes_{\integ}\halfinteg=\Hom{\integ}{\Kgp{\cT}}{\halfinteg}$, this makes it less convenient to discuss adjunction, as in the skew-symmetry condition.
\end{remark}

\begin{definition}[cf.~{\cite[Def.~2.4.1]{Qin}}]\label{d:p-form} Define $\pform{\blank}{\blank}{\cT}\colon \Kgp{\cT}\cross \Kgp{\cT}\to \integ$ by
\[ \pform{[T]}{[U]}{\cT}\defeq \evform{[T]}{\lambda_{\cT}[U]}. \]
\end{definition}

\begin{lemma}\label{l:p-form-skew-symmetric} The form $\pform{\blank}{\blank}{\cT}$ is skew-symmetric.
\end{lemma}

\begin{proof} This follows from the skew-symmetry of $\lambda_{\cT}$ and the properties of $\evform{\blank}{\blank}$:
\[
\pform{[T]}{[U]}{\cT}=\evform{[T]}{\lambda_{\cT}[U]}
=\evform{[U]}{\adj{\lambda_{\cT}}[T]}
=-\evform{[U]}{\lambda_{\cT}[T]}
=-\pform{[U]}{[T]}{\cT}.\qedhere\]
\end{proof}

Indeed, an examination of the proof shows that $\pform{\blank}{\blank}{\cT}$ being skew-symmetric implies, using non-degeneracy of $\evform{\blank}{\blank}$, that $\adj{\lambda_{\cT}}=-\lambda_{\cT}$, so the two notions of skew-symmetry coincide.
Analogous to \eqref{eq:beta-s-form}, it follows from the definitions and Lemma~\ref{l:p-form-skew-symmetric} that
\begin{equation}\label{eq:lambda-p-form} \lambda_{\cT}[T]=\evform{\blank}{\lambda_{\cT}[T]}=\pform{\blank}{[T]}{\cT}.
\end{equation}
We observe that while \eqref{eq:beta-s-form} involved the map $\pdual{\cT}$, the corresponding map in \eqref{eq:lambda-p-form} is $\delta_{\dual{\Kgp{\cT}}}=\id_{\dual{\Kgp{\cT}}}$, and so is not visible.
Due to the equivalence between specifying $\lambda_{\cT}$ or $\pform{\blank}{\blank}{\cT}$, we will also refer to the form as a quantum datum for $\cT$.

\begin{lemma}\label{l:form-compat} 
The following are equivalent:
\begin{enumerate}
\item\label{l:form-compat-maps} (compatibility of $\lambda_{\cT}$ and $\beta_{\cT}$ as maps) $\adj{\lambda_{\cT}}\circ\beta_{\cT}=2(\sdual{\cT}\circ\sinc{\cT})$;
\item\label{l:form-compat-adj} (pseudo-adjunction of $\lambda_{\cT}$ and $\beta_{\cT}$) for all $[M]\in \Kgp{\fd \stab{\cT}}$ and $[T]\in \Kgp{\cT}$ we have
\[ \evform{\beta_{\cT}[M]}{\lambda_{\cT}[T]}=2\canform{\sinc{\cT}[M]}{[T]}{\cT}; \]
\item\label{l:form-compat-beta-p-form} (compatibility of $\beta_{\cT}$ and $\pform{\blank}{\blank}{\cT}$) for all $[M]\in \Kgp{\fd \stab{\cT}}$ and $[T]\in \Kgp{\cT}$ we have
\[ \pform{\beta_{\cT}[M]}{[T]}{\cT}=2\canform{\sinc{\cT}[M]}{[T]}{\cT}. \]
\end{enumerate}
\end{lemma}

\begin{proof} We first show \ref{l:form-compat-maps} implies \ref{l:form-compat-adj}. By skew-symmetry, we have $\lambda_{\cT}=-\adj{\lambda_{\cT}}$. Thus 
\begin{align*} \evform{\beta_{\cT}[M]}{\lambda_{\cT}[T]} & = -\evform{\beta_{\cT}[M]}{\adj{\lambda_{\cT}}[T]} \\
& =-\evform{[T]}{\lambda_{\cT}(\beta_{\cT}[M])} \\
& =\evform{[T]}{\adj{\lambda_{\cT}}(\beta_{\cT}[M])} \\
& =2\evform{[T]}{\sdual{\cT}(\sinc{\cT}[M])}\\
& =2\canform{\sinc{\cT}[M]}{[T]}{\cT},
\end{align*}
noting the identity \eqref{eq:canform-comparison}.
The proof that \ref{l:form-compat-adj} implies \ref{l:form-compat-maps} is similar. Indeed, starting from \ref{l:form-compat-adj}, a completely analogous calculation to the above shows that $\evform{[T]}{2\sdual{\cT}(\sinc{\cT}[M])}=\evform{[T]}{\adj{\lambda_{\cT}}(\beta_{\cT}[M])}$, and so \ref{l:form-compat-maps} follows from the fact that $\evform{\blank}{\blank}$ is non-degenerate.
Finally, \ref{l:form-compat-adj} is equivalent to \ref{l:form-compat-beta-p-form} since 
$\pform{\blank}{\blank}{\cT}=\evform{\blank}{\lambda_{\cT}(\blank)}$
by definition.
\end{proof}

\begin{remark}
\label{r:form-compat-matrix}
Using \ref{l:form-compat-adj}, we may calculate for $T,U\in\indec{\cT}$ that
\[\evform{\beta_{\cT}[\simpmod{\cT}{U}]}{\lambda_{\cT}[T]}=\begin{cases}2d_T&\text{if}\ U=T,\\0&\text{otherwise,}\end{cases}\]
for $d_T=\dim\simpmod{\cT}{T}(T)$.
This expresses the compatibility condition (converting to matrices with respect to the natural bases) in its more usual form, saying that $\adj{B}L$ is a matrix with a diagonal block with positive integer entries on the mutable indices and a zero block on the frozen ones. 
\end{remark}

\begin{remark}
\label{r:form-compat-lambda-s-form}
Symmetry might reasonably lead one to expect a part (d) of Lemma~\ref{l:form-compat}, expressing the compatibility of $\lambda_{\cT}$ with $\sform{\blank}{\blank}{\cT}$.  This is possible but complicated by technicalities when $d_{T}>1$ for some $T\in\indec{\cT}$.

Let $D_{\cT}=\lcm\{d_T \mid T\in\indec{\cT}\}$. Then $D_{\cT}\lambda_{\cT}[T]$ is in the image of the injective map $\sdual{\cT}$, by Proposition~\ref{p:im-delta}, and a similar argument as in Lemma~\ref{l:form-compat} can be used to show that the compatibility condition is also equivalent to
\begin{equation}
\label{eq:compat-lambda-s-form}
\sform{[M]}{(\sdual{\cT})^{-1}(D_{\cT}\lambda_{\cT}[T])}{\cT}=2\canform{\sinc{\cT}[M]}{[T]}{\cT}.
\end{equation}
Recall that if $\cC$ is skew-symmetric then $D_{\cT}=1$ and $\sdual{\cT}$ is an isomorphism.
In this case, if we suppress $\sdual{\cT}$ from the notation by treating it as an identification of $\Kgp{\fd{\cT}}$ with $\dual{\Kgp{\cT}}$, further identifying the forms $\canform{\blank}{\blank}{\cT}$ and $\evform{\blank}{\blank}$, then \eqref{eq:compat-lambda-s-form} simplifies to
\[\sform{[M]}{\lambda_{\cT}[T]}{\cT}=2\canform{\sinc{\cT}[M]}{[T]}{\cT},\]
which is more recognisably analogous to the identity in Lemma~\ref{l:form-compat}\ref{l:form-compat-beta-p-form}.
\end{remark}

\subsection{Change of cluster-tilting subcategory}
\label{s:change-of-cts-lambda}

\begin{definition}
\label{d:mut-of-p-form}
Fix a quantum datum $\pform{\blank}{\blank}{\cT}$ for $\cT \ctsubcat \cC$.
For each $\cU\ctsubcat\cC$, define $\mu_{\cT}^{\cU}\pform{\blank}{\blank}{\cT}\colon\Kgp{\cU}\cross \Kgp{\cU}\to \integ$ by
\[\mu_{\cT}^{\cU}\pform{\blank}{\blank}{\cT}=\pform{\ind{\cU}{\cT}(\blank)}{\ind{\cU}{\cT}(\blank)}{\cT}. \]
\end{definition}

As in Section~\ref{s:change-of-cts-beta}, while we use the notation $\mu$, the definition does not use a sequence of mutations connecting $\cT$ to $\cU$, nor require the existence of one.

\begin{definition}
\label{d:quant-struct}
Let $\cC$ be a Krull--Schmidt cluster category with a weak cluster structure.
A \emph{quantum structure} for $\cC$ is a quantum datum $\pform{\blank}{\blank}{\cT}$ for every $\cT\ctsubcat\cC$ such that
\[\mu_{\cT}^{\cU}\pform{\blank}{\blank}{\cT}=\pform{\blank}{\blank}{\cU}\]
whenever $\cTU\ctsubcat\cC$.
\end{definition}

It is clear from this definition that a quantum structure on $\cC$ is completely determined by the quantum datum on a single cluster-tilting subcategory.
Our main goal in this section is to show that this determination is `free': if we choose some $\cT\ctsubcat\cC$ and a quantum datum $\pform{\blank}{\blank}{\cT}$ for $\cT$, then we always obtain a quantum structure on $\cC$ by setting $\pform{\blank}{\blank}{\cU}=\mu_{\cT}^{\cU}\pform{\blank}{\blank}{\cT}$ for any cluster-tilting subcategory $\cU\subset\cC$.
That is, we will show that each $\mu_{\cT}^{\cU}\pform{\blank}{\blank}{\cT}$ is a quantum datum for $\cU$, and these various quantum data satisfy the condition from Definition~\ref{d:quant-struct}, without any further assumptions on the initial choice of $\cT$ and $\pform{\blank}{\blank}{\cT}$.

\begin{lemma}
\label{l:mut-of-p-form-skewsymm}
In the setting of Definition~\ref{d:mut-of-p-form}, the form $\mu_{\cT}^{\cU}\pform{\blank}{\blank}{\cT}$ is skew-symmetric.
\end{lemma}

\begin{proof}
This follows immediately from skew-symmetry of $\pform{\blank}{\blank}{\cT}$ (Lemma~\ref{l:p-form-skew-symmetric}).
\end{proof}

We also claim that when $\cT$ and $\cU$ are related by a mutation, Definition~\ref{d:mut-of-lambda} recovers Berenstein--Zelevinsky mutation of compatible pairs \cite[(3.4)]{BZ-QCA} by combining the following with Theorem~\ref{t:exch-mat-mutation-clust-str}.
Write $\lambda^{\cT}_{U,V}=\pform{U}{V}{\cT}$, so that the $\lambda^{\cT}_{U,V}$ are the entries of the matrix of $\lambda$ with respect to the bases $[T]$ and $[T]^*$, for $T\in\indec{\cT}$, of $\Kgp{\cT}$ and $\dual{\Kgp{\cT}}$ respectively.

\begin{proposition}
\label{p:BZ-lambda-mut}
Let $\cC$ be a compact cluster category.
Let $\cT\ctsubcat \cC$ be maximally mutable, let $T\in \exch \cT$ and assume that $\cT$ has no loop or 2-cycle at $T$.
Then for $U,V\in \indec \mut{T}{\cT}$, we have
\begin{equation}
\lambda_{U,V}^{\mut{T}{\cT}} = \begin{cases}
\lambda_{U,V}^{\cT} & \text{if}\ U,V\neq T, \\
-\lambda_{U,T}^{\cT}+\sum_{W\in \indec \cT\setminus T} [\exchmatentry{W,T}]_{-}\lambda_{U,W}^{\cT} &  \text{if}\ U\neq T\ \text{and}\ V=T, \\
-\lambda_{T,V}^{\cT}+\sum_{W\in \indec \cT\setminus T} [\exchmatentry{W,T}]_{-}\lambda_{W,V}^{\cT} &  \text{if}\ U=T\ \text{and}\ V\neq T, \\
0 & \text{if}\ U=V=T.
\end{cases}
\end{equation}
\end{proposition}

\begin{proof}
This follows from Proposition~\ref{p:decomp-exch-terms} and \eqref{eq:ind-on-mut-T}:
\begin{align*}
\lambda_{U,V}^{\mut{T}{\cT}}=\mu_{\cT}^{\mut{T}{\cT}}\pform{[U]}{[V]}{\cT} & = \pform{\ind{\mut{T}{\cT}}{\cT}[U]}{\ind{\mut{T}{\cT}}{\cT}[V]}{\cT} \\
& = \begin{cases}
\pform{[U]}{[V]}{\cT} & \text{if}\ U,V\neq T, \\
\pform{[U]}{[\exchmon{\cT}{T}{-}]-[T]}{\cT} &  \text{if}\ U\neq T\ \text{and}\ V=T, \\
\pform{[\exchmon{\cT}{T}{-}]-[T]}{[V]}{\cT} &  \text{if}\ U=T\ \text{and}\ V\neq T, \\
0 & \text{if}\ U=V=T.
\end{cases} \\
& = \begin{cases}
\lambda_{U,V}^{\cT} & \text{if}\ U,V\neq T, \\
-\lambda_{U,T}^{\cT}+\sum_{W\in \indec \cT\setminus T} [\exchmatentry{W,T}]_{-}\lambda_{U,W}^{\cT} &  \text{if}\ U\neq T\ \text{and}\ V=T, \\
-\lambda_{T,V}^{\cT}+\sum_{W\in \indec \cT\setminus T} [\exchmatentry{W,T}]_{-}\lambda_{W,V}^{\cT} &  \text{if}\ U=T\ \text{and}\ V\neq T, \\
0 & \text{if}\ U=V=T.
\end{cases}
\end{align*}
as required.
\end{proof}

We now prove sign-invariance for the transfer of a quantum datum $\pform{\blank}{\blank}{\cT}$, analogous to Corollary~\ref{c:sign-invar-s-form} for $\sform{\blank}{\blank}{\cT}$.
Due to the lack of an intrinsically-defined form $\pform{\blank}{\blank}{\cU}$ on $\Kgp{\cU}$ to compare to, much more work is needed than in the case of $\sform{\blank}{\blank}{\cT}$.

\begin{proposition}
\label{p:sign-invar-for-p-form}
Let $\cC$ be a Krull--Schmidt cluster category, let $\cTU\ctsubcat\cC$, and let $\pform{\blank}{\blank}{\cT}$ be a quantum datum for $\cT$.
Then
\[\mu_{\cT}^{\cU}\pform{\blank}{\blank}{\cT}=\pform{\coind{\cU}{\cT}(\blank)}{\coind{\cU}{\cT}(\blank)}{\cT}.\]
\end{proposition}
\begin{proof}
Let $U,V\in\cU$. By Proposition~\ref{p:beta-proj-res}, we have
\begin{align*}
\mu_{\cT}^{\cU}\pform{[U]}{[V]}{\cT}&=\pform{\ind{\cU}{\cT}[U]}{\ind{\cU}{\cT}[V]}{\cT}\\
&=\pform{\coind{\cU}{\cT}[U]-\beta_{\cT}[\Extfun{\cT}{U}]}{\coind{\cU}{\cT}[V]-\beta_{\cT}[\Extfun{\cT}{V}]}{\cT}\\
&=\begin{multlined}[t]\pform{\coind{\cU}{\cT}[U]}{\coind{\cU}{\cT}[V]}{\cT}+\pform{\beta_{\cT}[\Extfun{\cT}{U}]}{\beta_{\cT}[\Extfun{\cT}{V}]}{\cT}\\
-\pform{\beta_{\cT}[\Extfun{\cT}{U}]}{\coind{\cU}{\cT}[V]}{\cT}-\pform{\coind{\cU}{\cT}[U]}{\beta_{\cT}[\Extfun{\cT}{V}]}{\cT}.\end{multlined}
\end{align*}
Thus it suffices to show that
\begin{equation}
\label{eq:pform-inv-error1}
\pform{\beta_{\cT}[\Extfun{\cT}{U}]}{\beta_{\cT}[\Extfun{\cT}{V}]}{\cT}-\pform{\beta_{\cT}[\Extfun{\cT}{U}]}{\coind{\cU}{\cT}[V]}{\cT}\\-\pform{\coind{\cU}{\cT}[U]}{\beta_{\cT}[\Extfun{\cT}{V}]}{\cT}=0.
\end{equation}
For any $T\in\cT$ and $M\in\fd{\underline{\cT}}$, we have
$\pform{[T]}{\beta_{\cT}[M]}{\cT}=-2\canform{\sinc{\cT}[M]}{[T]}{\cT}$
by Lemma~\ref{l:form-compat} and skew-symmetry of $\pform{\blank}{\blank}{\cT}$. Using this skew-symmetry again, we may therefore rewrite \eqref{eq:pform-inv-error1} as
\[-\canform{[\Extfun{\cT}{V}]}{\beta_{\cT}[\Extfun{\cT}{U}]}{\cT}-\canform{[\Extfun{\cT}{U}]}{\coind{\cU}{\cT}[V]}{\cT}+\canform{[\Extfun{\cT}{V}]}{\coind{\cU}{\cT}[U]}{\cT}=0.\]
Since $\beta_{\cT}[\Extfun{\cT}{U}]=\coind{\cU}{\cT}[U]-\ind{\cU}{\cT}[U]$, this simplifies to
\begin{equation}
\label{eq:pform-inv-error2}
\canform{[\Extfun{\cT}{U}]}{\coind{\cU}{\cT}[V]}{\cT}=\canform{[\Extfun{\cT}{V}]}{\ind{\cU}{\cT}[U]}{\cT}.
\end{equation}
Choose a $\cT$-index conflation $\rightker{\cT}{U}\infl\rightapp{\cT}{U}\defl U\confl$ for $U$ and a $\cT$-coindex conflation $V\infl\leftapp{\cT}{V}\defl\leftcok{\cT}{V}\confl$ for $V$, so that
$\ind{\cU}{\cT}[U]=[\rightapp{\cT}{U}]-[\rightker{\cT}{U}]$ and $\coind{\cU}{\cT}[V]=[\leftapp{\cT}{V}]-[\leftcok{\cT}{V}]$.
By the definitions of $\canform{\blank}{\blank}{\cT}$ and $\Extfun{\cT}$, equation \eqref{eq:pform-inv-error2} becomes
\[\ext{1}{\cC}{\leftapp{\cT}{V}}{U}-\ext{1}{\cC}{\leftcok{\cT}{V}}{U}=\ext{1}{\cC}{\rightapp{\cT}{U}}{V}-\ext{1}{\cC}{\rightker{\cT}{U}}{V}.\]
Since $\cC$ is stably $2$-Calabi--Yau, this follows from Lemma~\ref{l:ind-coind-adjointness} with $X=U$ and $Y=V$, noting that $U$ and $V$ both lie in the cluster-tilting subcategory $\cU$, so $\Ext{1}{\cC}{U}{V}=0$, and we may thus apply the stronger equalities from this lemma.
\end{proof}

With the same notation as in Remark~\ref{r:decategorification-of-square}, the previous two Propositions are equivalent to the claim that $\mu_{k}(L)=\adj{E_{\epsilon}(k)}LE_{\epsilon}(k)$, which is the usual expression \cite[Eq.~3.4]{BZ-QCA} of Berenstein--Zelevinsky mutation for the quasi-commutation matrix $L=(\lambda_{U,V}^{\cT})$ in the direction $k=T$.

The form $\mu_{\cT}^{\cU}\pform{\blank}{\blank}{\cT}$ is equivalent data to a map $\mu_{\cT}^{\cU}(\lambda_{\cT})\colon\Kgp{\cU}\to\dual{\Kgp{\cU}}$, the relationship of which to $\lambda_{\cT}$ is analogous to that between $\beta_{\cT}$ and $\beta_{\cU}$, as we now show.
  
\begin{definition}
\label{d:mut-of-lambda}
Let $\cC$ be a cluster category, $\cT\ctsubcat \cC$ and $\lambda_{\cT}$ a quantum datum.
Let $\cU\ctsubcat \cC$ and define $\mu_{\cT}^{\cU}(\lambda_{\cT})\colon \Kgp{\cU}\to \dual{\Kgp{\cU}}$ by 
\[\mu_{\cT}^{\cU}(\lambda_{\cT})[U]=\mu_{\cT}^{\cU}\pform{\blank}{[U]}{\cT}.\]
\end{definition}

Since $\mu_{\cT}^{\cT}\pform{\blank}{\blank}{\cT}=\pform{\blank}{\blank}{\cT}$, it follows from \eqref{eq:lambda-p-form} that $\mu_{\cT}^{\cT}(\lambda_{\cT})=\lambda_{\cT}$.
Since $\lambda_{\cT}$ and $\pform{\blank}{\blank}{\cT}$ are equivalent data (in a way compatible with the operation $\mu_{\cT}^{\cU}$), we will also refer to a choice of quantum datum $\lambda_{\cT}$ for each $\cT\ctsubcat\cC$ as a quantum structure for $\cC$ if $\mu_{\cT}^{\cU}(\lambda_{\cT})=\lambda_{\cU}$ for all $\cTU\ctsubcat\cC$.

\begin{proposition}
\label{p:lambda-square}
Let $\lambda_{\cT}$ be a quantum datum for $\cT\ctsubcat\cC$, and $\cU\ctsubcat\cC$ another cluster-tilting subcategory.
Then we have commutative diagrams
\[\begin{tikzcd}
\dual{\Kgp{\cT}} \arrow{d}[swap]{\dual{(\coind{\cU}{\cT})}}& \Kgp{\cT} \arrow{l}[swap]{\lambda_{\cT}} \arrow{d}{\ind{\cT}{\cU}} \\
\dual{\Kgp{\cU}}   & \Kgp{\cU}, \arrow{l}[swap]{\mu_{\cT}^{\cU}(\lambda_{\cT})}
\end{tikzcd}\qquad 
\begin{tikzcd}
\dual{\Kgp{\cT}} \arrow{d}[swap]{\dual{(\ind{\cU}{\cT})}}& \Kgp{\cT} \arrow{l}[swap]{\lambda_{\cT}} \arrow{d}{\coind{\cT}{\cU}} \\
\dual{\Kgp{\cU}}   & \Kgp{\cU}. \arrow{l}[swap]{\mu_{\cT}^{\cU}(\lambda_{\cT})}
\end{tikzcd}\]
\end{proposition}
\begin{proof}
Using \eqref{eq:lambda-p-form}, Proposition~\ref{p:ind-coind-inverse} and Proposition~\ref{p:sign-invar-for-p-form}, we have
\begin{align*}
\dual{(\coind{\cU}{\cT})}(\lambda_{\cT}[T])&=\lambda_{\cT}[T]\circ\coind{\cU}{\cT}\\
&=\pform{[T]}{\coind{\cU}{\cT}(\blank)}{\cT}\\
&=\pform{\coind{\cU}{\cT}\ind{\cT}{\cU}[T]}{\coind{\cU}{\cT}(\blank)}{\cT}\\
&=\mu_{\cT}^{\cU}\pform{\ind{\cT}{\cU}[T]}{\blank}{\cU}\\
&=\mu_{\cT}^{\cU}(\lambda_{\cT})(\ind{\cT}{\cU}[T])
\end{align*}
for any $T\in\cT$, and so the left-hand diagram commutes. Since $\lambda_{\cT}$ and (by Lemma~\ref{l:mut-of-p-form-skewsymm}) $\mu_{\cT}^{\cU}(\lambda_{\cT})$ are skew-symmetric, commutativity of the second diagram follows by taking the dual of the first.
Alternatively, this may be proved directly by a similar argument, in which Proposition~\ref{p:sign-invar-for-p-form} is not needed because of the choice made in defining $\mu_{\cT}^{\cU}\pform{\blank}{\blank}{\cT}$.
\end{proof}

\begin{remark}
The strategy here is reversed compared with the corresponding results for $\beta_{\cT}$ and $\sform{\blank}{\blank}{\cT}$, for which we proved the analogous commuting square first (Theorem~\ref{t:exch-isos}), then the equality of the transferred and intrinsic forms, before deducing the other properties of the $\sform{\blank}{\blank}{\cT}$ from these.
The lack of an intrinsic form $\pform{\blank}{\blank}{\cT}$ on each $\cT\ctsubcat\cC$
is at the heart of this: while $\beta_{\cT}$ is intrinsic, given by the formula in Proposition~\ref{p:beta-proj-res}, the map $\lambda_{\cT}$ is an additional choice, making a different approach necessary.
\end{remark}

\begin{proposition}
\label{p:mut-of-lambda-ss-compat} Let $\cC$ be a cluster category with $\cT\ctsubcat\cC$ maximally mutable, and let $\lambda_{\cT}$ be a quantum datum for $\cT$.
Then for each maximally mutable $\cU\ctsubcat\cC$, the map $\lambda_{\cU}\defeq\mu_{\cT}^{\cU}(\lambda_{\cT})$ is a quantum datum for $\cU$.
\end{proposition}

\begin{proof}
Skew-symmetry of $\lambda_{\cU}$ is equivalent to skew-symmetry of $\mu_{\cT}^{\cU}\pform{\blank}{\blank}{\cT}$, which is Lemma~\ref{l:mut-of-p-form-skewsymm}.
For compatibility, we may thus use $\adj{\lambda}_{\cU}=-\lambda_{\cU}$ and calculate using Corollary~\ref{c:exch-isos} and Proposition~\ref{p:lambda-square} that
\begin{align*}
-\lambda_{\cU}\circ\beta_{\cU}&=-\bigl(\dual{(\coind{\cU}{\cT})}\circ \lambda_{\cT}\circ \coind{\cU}{\cT}\bigr)\circ\bigl(\ind{\cT}{\cU}\circ \beta_{\cT} \circ \stabcoindbar{\cU}{\cT}\bigr)\\
&=\dual{(\coind{\cU}{\cT})}\circ(-\lambda_{\cT})\circ\beta_{\cT}\circ\stabcoindbar{\cU}{\cT}\\
&=\dual{(\coind{\cU}{\cT})}\circ\adj{\lambda}_{\cT}\circ\beta_{\cT}\circ\stabcoindbar{\cU}{\cT}\\
&=2\dual{(\coind{\cU}{\cT})}\circ\sdual{\cT}\circ\sinc{\cT}\circ\stabcoindbar{\cU}{\cT}.
\end{align*}
Now $\sinc{\cT}\circ\stabcoindbar{\cU}{\cT}=\coindbar{\cU}{\cT}\circ\sinc{\cU}$ by Proposition~\ref{p:iota-coind}, and the defining identity for $\indbar{\cT}{\cU}=\adj{(\coind{\cU}{\cT})}$ is
$\dual{(\coindbar{\cU}{\cT})}\circ\sdual{\cT}=\sdual{\cU}\circ\indbar{\cT}{\cU}$.
We thus have
\[\adj{\lambda}_{\cU}\circ\beta_{\cU}
=2\dual{(\coind{\cU}{\cT})}\circ\sdual{\cT}\circ\sinc{\cT}\circ\stabcoindbar{\cU}{\cT}
=2\sdual{\cU}\circ\indbar{\cT}{\cU}\circ\coindbar{\cU}{\cT}\circ\sinc{\cU}\circ\sinc{\cU}
=2\sdual{\cU}\circ\sinc{\cU},\]
as required, using the adjoint of Proposition~\ref{p:ind-coind-inverse}.
\end{proof}

Now we show that the transfer operation on $\pform{\blank}{\blank}{\cT}$ is transitive, analogous to Corollary~\ref{c:s-form-mut-transitive} for $\sform{\blank}{\blank}{\cT}$. Once again, the argument is different and more involved.

\begin{proposition}\label{p:mut-of-p-form-transitive}
If $\pform{\blank}{\blank}{\cT}$ is a quantum datum for a maximally mutable $\cT\ctsubcat\cC$, and $\cU,\cV\ctsubcat\cC$ with $\cV$ maximally mutable, then
\[\mu_{\cU}^{\cV}\mu_{\cT}^{\cU}\pform{\blank}{\blank}{\cT}=\mu_{\cT}^{\cV}\pform{\blank}{\blank}{\cT}.\]
\end{proposition}

\begin{proof}
We first rewrite the statement in a more tractable form. Let $X,Y\in\cV$. Unpacking the definition, the claim is that
\[\pform{\ind{\cU}{\cT}\ind{\cV}{\cU}[X]}{\ind{\cU}{\cT}\ind{\cV}{\cU}[Y]}{\cT}=
\pform{\ind{\cV}{\cT}[X]}{\ind{\cV}{\cT}[Y]}{\cT}.\]
By Corollary~\ref{c:ind-coind-mult-error-terms}, we have
$\ind{\cU}{\cT}\ind{\cV}{\cU}[X]=\ind{\cT}{\cV}[X]+\beta_{\cT}[\righterror{1}{\cU}{X}|_{\cT}]$
and similarly with $X$ replaced by $Y$. Substituting, expanding and simplifying, our claim is thus that
\[ \pform{\ind{\cV}{\cT}[X]}{\beta_{\cT}[\righterror{1}{\cU}{Y}|_{\cT}]}{\cT}+\pform{\beta_{\cT}[\righterror{1}{\cU}{X}|_{\cT}]}{\ind{\cV}{\cT}[Y]}{\cT}=-\pform{\beta_{\cT}[\righterror{1}{\cU}{X}|_{\cT}]}{\beta_{\cT}[\righterror{1}{\cU}{Y}|_{\cT}]}{\cT}.
\]
Since $\adj{\lambda_{\cT}}=-\lambda_{\cT}$ by skew-symmetry, the compatibility condition is that
$\lambda_{\cT}\circ\beta_{\cT}=-2(\sdual{\cT}\circ\sinc{\cT})$.
By the definition of $\pform{\blank}{\blank}{\cT}$, we thus have
\begin{align*}
\pform{\ind{\cV}{\cT}[X]}{\beta_{\cT}[\righterror{1}{\cU}{Y}|_{\cT}]}{\cT}&=\evform{\ind{\cV}{\cT}[X]}{\lambda_{\cT}(\beta_{\cT}[\righterrorold{1}{\cT'}{\blank}{Y}|_{\cT}])}\\
&=-2\evform{\ind{\cV}{\cT}[X]}{\sdual{\cT}(\sinc{\cT}[\righterror{1}{\cU}{Y}|_{\cT}])}\\
&=-2\canform{[\righterror{1}{\cU}{Y}|_{\cT}]}{\ind{\cV}{\cT}[X]}{\cT}.
\end{align*}
Continuing in this way, also using Lemma~\ref{l:p-form-skew-symmetric}, and dividing each side by $-2$, we see that our claim is equivalent to
\[
\canform{[\righterror{1}{\cU}{Y}|_{\cT}]}{\ind{\cV}{\cT}[X]}{\cT}-\canform{[\righterror{1}{\cU}{X}|_{\cT}]}{\ind{\cV'}{\cT}[Y]}{\cT}
=\canform{[\righterror{1}{\cU}{X}|_{\cT}]}{\beta_{\cT}[\righterror{1}{\cU}{Y}|_{\cT}]}{\cT}.
\]
Finally, applying Corollary~\ref{c:ind-coind-mult-error-terms} again, our claim becomes
\[\canform{[\righterror{1}{\cU}{Y}|_{\cT}]}{\ind{\cV}{\cT}[X]}{\cT}=\canform{[\righterror{1}{\cU}{X}|_{\cT}]}{\ind{\cU}{\cT}\ind{\cV}{\cU}[Y]}{\cT}.\]

To prove the claim, choose a $\cT$-index conflation $KX\infl RX\defl X\confl$ for $X$ and a $\cU$-index conflation $K'Y\infl R'Y\defl Y\confl$ for $Y$.
Further, pick $\cT$-index conflations $KR'Y\infl RR'Y\defl R'Y\confl$ and $KK'Y\infl RK'Y\defl K'Y\confl$ for $R'Y$ and $K'Y$. Then
\[
\ind{\cV}{\cT}[X]=[RX]-[KX],\quad
\ind{\cU}{\cT}\ind{\cV}{\cU}[Y]=[RR'Y]-[KR'Y]-[RK'Y]+[KK'Y],
\]
and so our claim is that
\begin{multline}
\label{eq:dimrighterrors}
\dim{\righterrorold{1}{\cU}{RX}{Y}}-\dim{\righterrorold{1}{\cU}{KX}{Y}}=
\dim{\righterrorold{1}{\cU}{RR'Y}{X}}-\dim{\righterrorold{1}{\cU}{KR'Y}{X}}\\
-\dim{\righterrorold{1}{\cU}{RK'Y}{X}}+\dim{\righterrorold{1}{\cU}{KK'Y}{X}}.
\end{multline}
By using Corollary~\ref{c:l-r-duality} to rewrite the right-hand side, this amounts to showing that
\begin{equation}
\label{eq:stabcoindbar-righterror1}
\stabcoindbar{\cT}{\cV}[\righterror{1}{\cU}{Y}|_{\cT}]=[\lefterror{2}{\cU}{(RR'Y)}|_{\cV}]-[\lefterror{2}{\cU}{(KR'Y)}|_{\cV}]
-[\lefterror{2}{\cU}{(RK'Y)}|_{\cV}]+[\lefterror{2}{\cU}{(KK'Y)}|_{\cV}],
\end{equation}
and then applying $\canform{\blank}{X}{\cV}$.
Applying $\beta_{\cV}$ to the left-hand side of \eqref{eq:stabcoindbar-righterror1}, we obtain
\begin{align*}
\beta_{\cV}\stabcoindbar{\cT}{\cV}[\righterror{1}{\cU}{Y}|_{\cT}]&=\coind{\cT}{\cV}\beta_{\cT}[\righterror{1}{\cU}{Y}|_{\cT}]\\
&=\coind{\cT}{\cV}(\ind{\cU}{\cT}\ind{\cV}{\cU}[Y]-\ind{\cV}{\cT}[Y])\\
&=\coind{\cT}{\cV}\ind{\cU}{\cT}\ind{\cV}{\cU}[Y]-[Y]
\end{align*}
by Corollary~\ref{c:exch-isos} (using that $\cV$ is maximally mutable), Corollary~\ref{c:ind-coind-mult-error-terms} and Proposition~\ref{p:ind-coind-inverse}. It follows from Corollary~\ref{c:ind-coind-mult-error-terms} that applying $\beta_{\cV}$ to the right-hand side of \eqref{eq:stabcoindbar-righterror1} gives
\begin{align*}
(\coind{\cT}{\cV}-\coind{\cU}{\cV}\coind{\cT}{\cU})([RR'Y]-[K&R'Y]-[RK'Y]+[KK'Y])\\
&=(\coind{\cT}{\cV}-\coind{\cU}{\cV}\coind{\cT}{\cU})(\ind{\cU}{\cT}\ind{\cV}{\cU}[Y])\\
&=\coind{\cT}{\cV}\ind{\cU}{\cT}\ind{\cV}{\cU}[Y]-[Y]
\end{align*}
by Proposition~\ref{p:ind-coind-inverse} again.
The two sides of \eqref{eq:stabcoindbar-righterror1} thus agree after applying the map $\beta_{\cV}$ to each, but since $\cT$, and hence $\cV$ by Proposition~\ref{p:mut-of-lambda-ss-compat}, admits a quantum datum, this map is injective.
The claim in \eqref{eq:stabcoindbar-righterror1} is therefore true, and we obtain \eqref{eq:dimrighterrors}, and hence the desired statement, by applying $\canform{\blank}{[X]}{\cV}$ to each side.
\end{proof}

\begin{remark}
We expect that the identity \eqref{eq:stabcoindbar-righterror1} holds for objects $X$ and $Y$ with $\Ext{1}{\cC}{X}{Y}=0$ in an arbitrary cluster category $\cC$, without requiring that $\beta_{\cV}$ is injective, but do not currently have a proof of this.
It is enough that there is a Frobenius category $\cE$ such that $\cC=\cE/\cP$ for a full and additively closed subcategory $\cP$ of projectives, and $\beta_{\cT}$ is injective for $\cT\ctsubcat\cE$, which can happen even when $\beta_{\cT/\cP}$ is not injective.
\end{remark}

\begin{remark}
Proposition~\ref{p:sign-invar-for-p-form} gives us many more identities in the style of \eqref{eq:dimrighterrors}. For example, by using $\coind{\cV}{\cU}$ to compute $\mu_{\cV}^{\cU}\pform{\blank}{\blank}{\cV}$, it follows from Proposition~\ref{p:mut-of-p-form-transitive} that
\[\pform{\ind{\cU}{\cT}\coind{\cV}{\cU}[X]}{\ind{\cU}{\cT}\coind{\cV}{\cU}[Y]}{\cT}=\pform{\ind{\cV}{\cT}[X]}{\ind{\cV}{\cT}[Y]}{\cT},\]
or equivalently, via a completely analogous argument to that in the preceding proof, that
\begin{multline*}
\dim{\lefterrorold{1}{\cU}{RX}{Y}}-\dim{\lefterrorold{1}{\cU}{KX}{Y}}=
\dim{\lefterrorold{1}{\cU}{RL'Y}{X}}-\dim{\lefterrorold{1}{\cU}{KL'Y}{X}}\\
-\dim{\lefterrorold{1}{\cU}{RC'Y}{X}}+\dim{\lefterrorold{1}{\cU}{KC'Y}{X}},
\end{multline*}
for $RL'Y$, $KL'Y$, $RC'Y$ and $KC'Y$ the objects involved in computing $\ind{\cU}{\cT}\coind{\cV}{\cU}[Y]$.
\end{remark}

\begin{corollary}\label{c:mut-of-lambda-are-q-structure}
Let $\cC$ be a cluster category with a weak cluster structure, fix $\cT\ctsubcat\cC$, and choose a quantum datum $\lambda_{\cT}$ for $\cT$. Then the maps $\lambda_{\cU}\defeq\mu_{\cT}^{\cU}(\lambda_{\cT})$ (equivalently, the forms $\pform{\blank}{\blank}{\cU}\defeq\mu_{\cT}^{\cU}\pform{\blank}{\blank}{\cT}$) are a quantum structure on $\cC$.
\end{corollary}
\begin{proof}
Each $\lambda_{\cU}$ is a quantum datum for $\cU$ by Proposition~\ref{p:mut-of-lambda-ss-compat}. Moreover, the statement of Proposition~\ref{p:mut-of-p-form-transitive}, for $\cT$ our initial choice of cluster-tilting subcategory, becomes
\[\mu_{\cU}^{\cV}\pform{\blank}{\blank}{\cU}=\pform{\blank}{\blank}{\cV}\]
for any $\cU,\cV\ctsubcat\cC$, and so we have a quantum structure on $\cC$.
\end{proof}

\begin{remark}
Given a cluster category with a quantum structure, it would be desirable to quantise the results of Section~\ref{s:clust-char} to produce quantum cluster characters, which compute quantum cluster variables under the usual extra assumptions.
For now, however, the geometric problems \cite[\S3.4]{Qin} concerning the appropriate replacement of the quantum Euler characteristic for a singular quiver Grassmannian continue to obstruct this.
\end{remark}

\subsection{A canonical quantum structure}

Above, we indicated that the form $\mu_{\cT}^{\cU}\pform{\blank}{\blank}{\cT}$ could not in general be matched up with a form $\pform{\blank}{\blank}{\cU}$ defined intrinsically for any cluster-tilting subcategory $\cU$.
However, in a particularly natural and important class of examples, there is a canonical quantum structure given by a global formula.

\begin{theorem}
\label{t:canonical-p-form}
Assume $\cE$ is a Hom-finite exact cluster category with a weak cluster structure. For each $\cT\ctsubcat\cE$ and each $T_1,T_2\in\cT$, define
\[\pform{[T_1]}{[T_2]}{\cT}=\dim_{\bK}\Hom{\cE}{T_1}{T_2}-\dim_{\bK}\Hom{\cE}{T_2}{T_1}.\]
Then the forms $\pform{\blank}{\blank}{\cT}$ defined in this way are a quantum structure on $\cE$.
\end{theorem}

\begin{proof}
To show that $\pform{\blank}{\blank}{\cT}$ is a quantum datum for $\cT$, we will establish condition~\ref{l:form-compat-beta-p-form} from Lemma~\ref{l:form-compat}, namely that
$\pform{\beta_{\cT}[M]}{[T]}{\cT}=2\canform{\sinc{\cT}[M]}{[T]}{\cT}$.
Recall that for $M\in\underline{\cT}$, we may write $M=\Extfun{\cT}X$ for some $X\in\cE$, and then calculate
$\beta_{\cT}[M]=\coind{\cE}{\cT}[X]-\ind{\cE}{\cT}[X]$.
To this end, we pick $\cT$-coindex and $\cT$-index sequences
\begin{equation}
\label{eq:quantum-eg-sequences}
\begin{tikzcd}[column sep=20pt]0\arrow{r}& X\arrow{r}& LX\arrow{r}& CX\arrow{r}& 0,\end{tikzcd}
\quad
\begin{tikzcd}[column sep=20pt] 0\arrow{r}& KX\arrow{r}& RX\arrow{r}& X\arrow{r}& 0,\end{tikzcd}
\end{equation}
so $\coind{\cE}{\cT}[X]=[LX]-[CX]$ and $\ind{\cE}{\cT}[X]=[RX]-[KX]$. 
Then for $T\in\cT$, we calculate
\begin{equation}
\label{eq:quantum-eg-calc1}
\begin{split}
\pform{\beta_{\cT}[M]}{[T]}{\cT}=\begin{multlined}[t]\dim\Hom{\cE}{LX}{T}-\dim\Hom{\cE}{CX}{T}\\
+\dim\Hom{\cE}{KX}{T}-\dim\Hom{\cE}{RX}{T}\\
+\dim\Hom{\cE}{T}{RX}-\dim\Hom{\cE}{T}{KX}\\
+\dim\Hom{\cE}{T}{CX}-\dim\Hom{\cE}{T}{LX}.\end{multlined}
\end{split}
\end{equation}
Applying the functors $\Hom{\cE}{\blank}{T}$ and $\Hom{\cE}{T}{\blank}$ to \eqref{eq:quantum-eg-sequences}, we obtain exact sequences
\begin{align*}
&\begin{tikzcd}[ampersand replacement=\&, column sep=15pt]0\arrow{r}\&\Hom{\cE}{CX}{T}\arrow{r}\&\Hom{\cE}{LX}{T}\arrow{r}\&\Hom{\cE}{X}{T}\arrow{r}\& 0,\end{tikzcd}\\
&\begin{tikzcd}[ampersand replacement=\&, column sep=15pt]0\arrow{r}\&\Hom{\cE}{X}{T}\arrow{r}\&\Hom{\cE}{RX}{T}\arrow{r}\&\Hom{\cE}{KX}{T}\arrow{r}\&\Ext{1}{\cE}{X}{T}\arrow{r}\& 0,\end{tikzcd}\\
&\begin{tikzcd}[ampersand replacement=\&, column sep=15pt]0\arrow{r}\&\Hom{\cE}{T}{KX}\arrow{r}\&\Hom{\cE}{T}{RX}\arrow{r}\&\Hom{\cE}{T}{X}\arrow{r}\& 0,\end{tikzcd}\\
&\begin{tikzcd}[ampersand replacement=\&, column sep=15pt]0\arrow{r}\&\Hom{\cE}{T}{X}\arrow{r}\&\Hom{\cE}{T}{LX}\arrow{r}\&\Hom{\cE}{T}{CX}\arrow{r}\&\Ext{1}{\cE}{T}{X}\arrow{r}\& 0.\end{tikzcd}
\end{align*}
Here we use that $\cE$ is a Frobenius category to get exactness at the left-hand end in each case. The various relations among dimensions arising from these exact sequences allow us to rewrite \eqref{eq:quantum-eg-calc1} as
\begin{align*}
\pform{\beta_{\cT}[M]}{[T]}{\cT}&=\begin{multlined}[t]\dim\Hom{\cE}{X}{T}-\dim\Hom{\cE}{X}{T}+\dim\Ext{1}{\cE}{X}{T}\\
+\dim\Hom{\cE}{T}{X}-\dim\Hom{\cE}{T}{X}+\dim\Ext{1}{\cE}{T}{X}\end{multlined}\\
&=\dim\Ext{1}{\cE}{X}{T}+\dim\Ext{1}{\cE}{T}{X}\\
&=2\dim\Ext{1}{\cE}{T}{X},
\end{align*}
since $\cE$ is stably $2$-Calabi--Yau.
Recalling that $M=\Extfun{\cT}X$, this calculation shows that $\pform{\beta_{\cT}[M]}{[T]}{\cT}=2\dim M(T)=2\canform{\sinc{\cT}[M]}{[T]}{\cT}$, as required.

It remains to show that $\mu_{\cT}^{\cU}\pform{\blank}{\blank}{\cT}=\pform{\blank}{\blank}{\cU}$ for $\cTU\ctsubcat\cE$. To do this, choose $U_1,U_2\in\cU$ and $\cT$-index sequences
\begin{equation}
\label{eq:index-seq-can-quant}
\begin{tikzcd}0\arrow{r}&KU_i\arrow{r}&RU_i\arrow{r}&U_i\arrow{r}&0\end{tikzcd}
\end{equation}
for each $U_i$, so that $\ind{\cU}{\cT}[U_i]=[RU_i]-[KU_i]$. We may thus calculate
\begin{align}
\begin{split}
\label{eq:mut-p-form-calc}
\mu_{\cT}^{\cU}\pform{[U_1]}{[U_2]}{\cT}=\begin{multlined}[t]\hom{\cE}{RU_1}{RU_2}-\hom{\cE}{RU_1}{KU_2}\\
-\hom{\cE}{RU_2}{RU_1}+\hom{\cE}{RU_2}{KU_1}\\
-\hom{\cE}{KU_1}{RU_2}+\hom{\cE}{KU_1}{KU_2}\\
+\hom{\cE}{KU_2}{RU_1}-\hom{\cE}{KU_2}{KU_1}.\end{multlined}
\end{split}
\end{align}
For any $T\in\cT$, applying $\Hom{\cE}{T}{\blank}$ to \eqref{eq:index-seq-can-quant} yields an exact sequence
\[\begin{tikzcd}[column sep=15pt]0\arrow{r}&\Hom{\cE}{T}{KU_i}\arrow{r}&\Hom{\cE}{T}{RU_i}\arrow{r}&\Hom{\cE}{T}{U_i}\arrow{r}&\Ext{1}{\cE}{T}{KU_i}=0,\end{tikzcd}\]
and hence
$\hom{\cE}{T}{RU_i}-\hom{\cE}{T}{KU_i}=\hom{\cE}{T}{U_i}$.
Applying this identity to \eqref{eq:mut-p-form-calc} (line by line), we see that
\begin{multline}
\label{eq:mut-p-form-calc2}
\mu_{\cT}^{\cU}\pform{[U_1]}{[U_2]}{\cT}=\hom{\cE}{RU_1}{U_2}-\hom{\cE}{RU_2}{U_1}\\
-\hom{\cE}{KU_1}{U_2}+\hom{\cE}{KU_2}{U_1}.
\end{multline}
Now applying $\Hom{\cE}{\blank}{U_j}$ to \eqref{eq:index-seq-can-quant} produces the exact sequence
\[\begin{tikzcd}[column sep=15pt]0\arrow{r}&\Hom{\cE}{U_i}{U_j}\arrow{r}&\Hom{\cE}{RU_i}{U_j}\arrow{r}&\Hom{\cE}{KU_i}{U_j}\arrow{r}&\Ext{1}{\cE}{U_i}{U_j}=0,\end{tikzcd}\]
and so
$\hom{\cE}{RU_i}{U_j}-\hom{\cE}{KU_i}{U_j}=\hom{\cE}{U_i}{U_j}$.
Applying this to \eqref{eq:mut-p-form-calc2} produces
\[\mu_{\cT}^{\cU}\pform{[U_1]}{[U_2]}{\cT}=\hom{\cE}{U_1}{U_2}-\hom{\cE}{U_2}{U_1}=\pform{[U_1]}{[U_2]}{\cU},\]
as required.
\end{proof}

In particular, this result covers the examples of Geiß--Leclerc--Schröer \cite{GLS-QuantumPFV}, who prove that certain subcategories of the module categories of preprojective algebras, with their canonical quantum structures as in Theorem~\ref{t:canonical-p-form}, categorify quantum cluster algebra structures on quantised coordinate rings of unipotent subgroups of Kac--Moody groups.
For example, the above proof directly generalises that of \cite[Prop.~10.1]{GLS-QuantumPFV}.
Our result extends this to any Hom-finite Frobenius cluster category, in a uniform way. 

Jensen--Su \cite{JS} have produced further examples, categorifying quantum partial flag varieties in type $\mathsf{A}$; their results show that their categories $\curly{F}_{\Delta}(\mathsf{J})$ are finite rank skew-symmetric Hom-finite exact cluster categories with a cluster structure.
The existence of a cluster character \cite[Lem.~9.5]{JS} and a quantum structure given by the difference of dimensions of Hom-spaces \cite[Thm.~10.13]{JS} then follow immediately from our results above.
They go on to identify the associated quantum cluster algebras as quantum partial flag varieties.

While the claim that the difference of Hom-dimensions determines valid initial quantum data is not so surprising, given \cite{GLS-QuantumPFV}, our theorem proves the stronger result that in the quantum structure induced from this initial data, the quantum datum on every cluster-tilting subcategory is obtained by computing an analogous Hom-difference; a priori there is no reason to expect this.
In particular, the quantum structure obtained this way is independent of the choice of $\cT$.

We also obtain the following corollary, which gives an alternative proof of an observation by Fu--Keller \cite[Rem.~4.5]{FuKeller}.

\begin{corollary}
\label{c:Hom-finite-exact-full-rank}
If $\cE$ is a Hom-finite Frobenius cluster category, then $\beta_{\cT}|_{\Kgp{\fd{\stab{\cT}}}}$ is injective for any $\cT\ctsubcat\cE$. \qed
\end{corollary}

There is an obvious obstruction to extending Theorem~\ref{t:canonical-p-form} to the Hom-infinite case.
Indeed, relatively few examples of quantum structures on Hom-infinite cluster categories are known, although one such is the categorification by Jensen--King--Su \cite{JKS-Quantum-Gr} of the quantum cluster algebra structure on the quantum Grassmannian due to the first author and Launois \cite{GradedQCAs}.

The examples of \cite{JKS-Quantum-Gr,JS} follow the pattern suggested by the work of Geiß--Leclerc--Schröer \cite{GLS-QuantumPFV}, namely that the quantum cluster algebra of interest is categorified by choosing a quantum structure on an existing categorification of the commutative cluster algebra, leaving the underlying category unchanged.
It is a somewhat remarkable phenomenon, currently unexplained, that the categorifications that have been discovered for these geometric examples quantise to exactly the noncommutative analogues that are best known, rather than some other quantisation.
It also suggests an important role for cluster categories in producing new quantum algebras, which provides further motivation for constructing quantum cluster characters in full generality.

\sectionbreak
\appendix
\section{Foundations}
\label{s:prelims}

\subsection{Adjunction}
\label{s:adj}

We denote by $(\blank)^*$ the functor $\Hom{\integ}{\blank}{\integ}$ on $\fpmod{\integ}$. Given a $\integ$-module $V$, there is a canonical \emph{evaluation pairing}
\[\evform[V]{\blank}{\blank}\colon V\times \dual{V}\to\integ,\ \evform[V]{v}{\varphi}=\varphi(v).\]
We usually omit $V$ in the notation for this pairing, since it will be clear from the context.
Let $V$ and $W$ be $\integ$-modules, and let $\ip{\blank}{\blank}\colon V\times W\to\integ$ be a ($\integ$-bilinear) form. This form determines maps
\[
\delta_V\colon V\to\dual{W},\ \delta_V(v)=\ip{v}{\blank},\quad
\delta_W\colon W\to\dual{V},\ \delta_W(w)=\ip{\blank}{w},
\]
and indeed either of these maps determines the form, via
\begin{equation}
\label{eq:evform-genform}
\evform[V]{v}{\delta_W(w)}=\ip{v}{w}=\evform[W]{w}{\delta_V(v)}.
\end{equation}

\begin{definition}
A form $\ip{\blank}{\blank}\colon V\times W\to\integ$, for $\integ$-modules $V$ and $W$, is \emph{non-degenerate} if both $\delta_V$ and $\delta_W$ are injective, and \emph{a perfect pairing} if both $\delta_V$ and $\delta_W$ are isomorphisms.
\end{definition}

In the case of the evaluation pairing $\evform{\blank}{\blank}\colon V\times \dual{V}\to\integ$, the map $\delta_{\dual{V}}\colon\dual{V}\to\dual{V}$ is the identity, whereas $\delta_{V}\colon V\to\ddual{V}$ is the evaluation map $v\mapsto(\varphi\mapsto\varphi(v))$. Thus while $\delta_{\dual{V}}$ is always an isomorphism, $\delta_{V}$ is injective (so $\evform{\blank}{\blank}$ is non-degenerate) if and only if $V$ is free, and $\delta_V$ is an isomorphism (so $\evform{\blank}{\blank}$ is a perfect pairing) if and only if $V$ is free and finitely generated.

Forms of this kind sometimes allow us to construct adjoints to $\integ$-linear maps. The most general form of adjunction which we will need is the following.

\begin{proposition}
\label{p:adjunction}
Let $V_1$, $V_2$, $W_1$ and $W_2$ be $\integ$-modules, and let $\ip{\blank}{\blank}_1\colon V_1\times W_1\to\integ$ and $\ip{\blank}{\blank}_2\colon V_2\times W_2\to\integ$ be bilinear forms with associated $\integ$-linear maps $\delta_{V_i}\colon V_i\to\dual{W}_i$ and $\delta_{W_i}\colon W_i\to\dual{V}_i$.
Then
\begin{enumerate}
\item\label{p:adjunction-left}
if $\delta_{W_1}$ is injective and $f\colon V_1\to V_2$ satisfies $\image(f^*\circ\delta_{W_2})\subset\image(\delta_{W_1})$, then there is a unique $\integ$-linear map $\adj{f}\colon W_2\to W_1$ such that
\begin{equation}
\label{eq:adjunction-left}
\ip{f(v)}{w}_2=\ip{v}{\adj{f}(w)}_1
\end{equation}
for all $v\in V_1$ and $w\in W_2$, and
\item\label{p:adjunction-right}
if $\delta_{V_1}$ is injective and $g\colon W_1\to W_2$ satisfies $\image(\dual{g}\circ\delta_{V_2})\subseteq\image(\delta_{V_1})$, then there is a unique $\integ$-linear map $\adj{g}\colon V_2\to V_1$ such that
\begin{equation}
\label{eq:adjunction-right}
\ip{v}{g(w)}_2=\ip{\adj{g}(v)}{w}_1
\end{equation}
for all $v\in V_2$ and $w\in W_1$.
\end{enumerate}
\end{proposition}
\begin{proof}
Note that
$f^*\circ\delta_{W_2}(w)=\ip{f(\blank)}{w}_2$
for all $w\in W_2$. Thus, by assumption, there exists some $w'\in W_1$ such that
\[\ip{f(\blank)}{w}_2=\delta_{W_1}(w')=\ip{\blank}{w'}_1.\]
Since $\ip{\blank}{\blank}_1$ is non-degenerate, $\delta_{W_1}$ is injective and hence $w'$ is unique.

It follows that $\adj{f}(w)=w'$ defines the unique map with the required properties for \ref{p:adjunction-left}. Its linearity also follows from the uniqueness of $w'$. Statement \ref{p:adjunction-right} can either be proved directly in a similar way, or deduced by applying \ref{p:adjunction-left} to the forms $\ip{\blank}{\blank}_i^{\mathrm{op}}\colon W_i\times V_i\to\integ$ with
$\ip{w}{v}_i^{\mathrm{op}}\defeq\ip{v}{w}_i$.
\end{proof}

\begin{corollary}
\label{c:identify-adj}
Under the assumptions of Proposition~\ref{p:adjunction}, let $f\colon V_1\to V_2$. If $\delta_{W_1}$ is injective and $h\colon W_2\to W_1$ satisfies
$\ip{f(v)}{w}_2=\ip{v}{h(w)}_1$
for all $v\in V_1$ and $w\in W_2$, then $\adj{f}$ exists and is equal to $h$.
\end{corollary}
\begin{proof}
The adjoint $\adj{f}$ exists since
\[\dual{f}\circ\delta_{W_2}(w)=\ip{f(\blank)}{w}_2=\ip{\blank}{h(w)}_1=\delta_{W_1}(h(w)),\]
so $\image(\dual{f}\circ\delta_{W_2})\subset\image(\delta_{W_1})$. Then $\adj{f}=h$ by the uniqueness result in Proposition~\ref{p:adjunction}.
\end{proof}

\begin{remark}
\label{r:delta-inverse}
The proof of Proposition~\ref{p:adjunction}, together with Corollary~\ref{c:identify-adj}, demonstrates that $\adj{f}$ is determined by the identity
$\delta_{W_1}\circ\adj{f}=\dual{f}\circ\delta_{W_2}$.
Thus when $\delta_{W_1}$ is an isomorphism (such as if $\ip{\blank}{\blank}_1$ is a perfect pairing), the adjoint $\adj{f}=\delta_{W_1}^{-1}\circ\dual{f}\circ\delta_{W_2}$ exists for any map $f\colon V_1\to W_1$.
As a special case, if $V$ is free and $f\colon V\to \dual{V}$ is any map, then with respect to the perfect pairings $\ip{\blank}{\blank}_1=\evform[V]{\blank}{\blank}$ and $\ip{\blank}{\blank}_2=(\evform[V]{\blank}{\blank})^{\mathrm{op}}$, we have $\adj{f}=\dual{f}\circ\delta_{V}$ since $\delta_{\dual{V}}=\id_{\dual{V}}$.
That is, $\adj{f}$ is obtained from $\dual{f}$ by restricting from $\ddual{V}$ to $V$ along the natural embedding $\delta_V$.
\end{remark}

\begin{remark}
\label{r:hidden-adjoints}
We sometimes find ourselves in the setting of Proposition~\ref{p:adjunction} but with $V_1=W_2$, $V_2=W_1$ and $\ip{\blank}{\blank}_1=\ip{\blank}{\blank}_2^{\mathrm{op}}$, so that the adjunction formula \eqref{eq:adjunction-left} becomes
\[\ip{f(v)}{w}_2=\ip{v}{\adj{f}(w)}_2^{\mathrm{op}}=\ip{\adj{f}(w)}{v}_2.\]
We usually avoid referring to the opposite form in this case, so the statement of the adjunction becomes just the equality of the outer two terms above; while this slightly disguises the fact that $\adj{f}$ is related to $f$ by adjunction, this can be seen by observing that $v$ and $w$ have swapped positions inside $\ip{\blank}{\blank}_2$.
As an example, it follows from \eqref{eq:evform-genform} and Corollary~\ref{c:identify-adj} that the maps $\delta_V$ and $\delta_W$ associated to a bilinear form $\ip{\blank}{\blank}\colon V\times W\to\integ$ are adjoint to each other with respect to the evaluation forms for $V$ and $W$.
\end{remark}

\begin{remark}
\label{r:adj-calc}
By Corollary~\ref{c:identify-adj}, the usual manipulations one does with adjoints in linear algebra may also be made here, giving that $\adj{(f+g)}=\adj{f}+\adj{g}$ and $\adj{(g\circ f)}=\adj{f}\circ \adj{g}$ when the adjoints $\adj{f}$ and $\adj{g}$ exist, and are defined with compatible choices of forms.
For the first identity one should take the same pair of forms for each of the maps $f$, $g$ and $f+g$.
For the second, the codomain of $f$ coincides with the domain of $g$, and the same form involving this space must be used in the definition of both $\adj{f}$ and $\adj{g}$.
Moreover, the other two forms needed to define these two adjoints should coincide with those used to define $\adj{(g\circ f)}$.
\end{remark}

\subsection{Modules over categories}
\label{s:modules}

A $\bK$-linear category is an additive category enriched in $\bK$-vector spaces. Such categories have modules (or representations), generalising the corresponding notion for $\bK$-algebras.
\begin{definition}
Let $\cA$ be a $\bK$-linear category. An \emph{$\cA$-module} is a contravariant $\bK$-linear functor
$\cA\to\Mod{\bK}$,
where $\Mod{\bK}$ denotes the category of (all) $\bK$-vector spaces.
The module is \emph{locally finite-dimensional} if it takes values in the full subcategory $\fd{\bK}$ of finite-dimensional vector spaces.
We write $\Mod{\cA}$ for the category of $\cA$-modules, and $\lfd{\cA}$ for the full subcategory of locally finite-dimensional modules.
These categories have a natural additive structure induced from that of $\Mod{\bK}$, and are even abelian categories, with kernels and cokernels computed pointwise.
\end{definition}

Given a $\bK$-linear category $\cA$, we write $\Yonfun{\cA}\colon\cA\to\Mod\cA$ for the (covariant) Yoneda functor, i.e.\ $\Yonfun{\cA}{X}=\Hom{\cA}{\blank}{X}$, and $\opYonfun{\cA}=\Yonfun{\op{\cA}}\colon\cA\to\Mod{\op{\cA}}$ for the contravariant Yoneda functor, i.e.\ $\opYonfun{\cA}{X}=\Hom{\cA}{X}{\blank}$.
If $\cB\subseteq\cA$ is a full subcategory, the restricted Yoneda functor $\cA\to\Mod{\cB}$ given by $X\mapsto\Hom{\cA}{\blank}{X}|_{\cB}$ extends $\Yonfun{\cB}$, and so we reuse the notation $\Yonfun{\cB}$ for this functor on $\cA$.
Similarly, $\opYonfun{\cB}$ will denote both the contravariant Yoneda functor on $\cB$, and the restricted contravariant Yoneda functor on $\cA$.

\begin{remark}
The convention that $\cA$-modules are contravariant functors on $\cA$ is common but can be surprising at first sight: it means that the covariant Yoneda functor $\Yonfun{\cA}$, rather than the contravariant Yoneda functor $\opYonfun{\cA}$, takes values in $\cA$-modules.
This fact is not sensitive to any convention concerning the direction of function composition, or of left versus right modules over rings.
\end{remark}

\begin{definition}
\label{d:fp}
An $\cA$-module $M\colon\cA\to\Mod{\bK}$ is \emph{finitely generated} if there is an epimorphism $\Yonfun{\cA}{X}\defl M$ for some $X\in\cA$.
Moreover, $M$ is \emph{finitely presented} if there is an exact sequence
\[\begin{tikzcd}
\Yonfun{\cA} Y\arrow{r}&\Yonfun{\cA} X\arrow{r}& M\arrow{r}&0
\end{tikzcd}\]
for $X,Y\in\cA$.
The full subcategory of $\Mod{\cA}$ consisting of finitely presented modules is denoted by $\fpmod{\cA}$.
This category is not necessarily abelian, but it is always exact, since it is full and extension-closed in $\Mod{\cA}$.
We write $\gldim{\cA}$ for the supremum of projective dimensions of $\cA$-modules.
\end{definition}

The following well-known result justifies some of the preceding terminology.

\begin{proposition}
\label{p:Yoneda-image}
Assume that $\cA$ is idempotent complete.
Then the essential image of the Yoneda functor $\Yonfun{\cA}$ is the category $\proj{\cA}\subseteq\fpmod{\cA}$ of projective objects in $\fpmod{\cA}$.\qed
\end{proposition}

More generally, one can identify $\proj{\cA}$ with the idempotent completion $\idcomp{\cA}$ of $\cA$, in such a way that $\Yonfun{\cA}$ is identified with the universal fully faithful functor $\cA\to\idcomp{\cA}$.
Under extra assumptions on the additive category $\cA$, we obtain more familiar descriptions of the various categories of $\cA$-modules above.

\begin{definition}
An additive category $\cA$ is \emph{Krull--Schmidt}\footnote{This terminology, excluding the contribution of Remak, is unfortunately by now well-established.
Some further comments on its history may be found in \cite[Rem.~1.1]{ShahKRS}.} if each of its objects is isomorphic to a direct sum of indecomposable objects, and such objects have local endomorphism rings.
We say $\cA$ is \emph{additively finite} if it has finitely many isomorphism classes of indecomposable objects.
\end{definition}

This definition of Krull--Schmidt is that of Krause \cite[\S4]{KrauseKS}.
Since $\cA$ is additive, it is Krull--Schmidt in this sense if and only if it is a Krull--Schmidt prevariety in the sense of Bautista \cite{Bautista}.
By \cite[Cor.~4.4]{KrauseKS} (see also \cite[Prop.~2.1]{LNP}), the category $\cA$ is Krull--Schmidt if and only if it is idempotent complete (also called Karoubian) and the endomorphism ring of each of its objects is semi-perfect.
In particular, this means that any Hom-finite idempotent complete additive category is Krull--Schmidt.
In a Krull--Schmidt additive category, decompositions of objects into indecomposables are essentially unique \cite[Cor.~4.3]{KrauseKS}, as in the classical Krull--Remak--Schmidt theorem, as a consequence of the condition on endomorphism rings.

\begin{definition}
\label{d:fd}
Let $\cA$ be a $\bK$-linear category and $M\in \Mod{\cA}$.
The support of $M$ is
\[ \supp{M}=\{ X\in \cA \mid M(X')\neq 0\ \text{for all non-zero summands}\ X'\ \text{of}\ X\}.\]
We say $M$ is \emph{finite-dimensional} if $M\in\lfd{\cA}$ and there exists $V\in \cA$ such that $\supp{M}=\add{V}$. 
The (abelian) category of finite-dimensional $\cA$-modules is denoted by $\fd{\cA}$.
\end{definition}

The first of the next two propositions is thus immediate from the definition, and the second is again well-known.

\begin{proposition}
If $\cA$ is additively finite, then $\fd{\cA}=\lfd{\cA}$.\qed
\end{proposition}

\begin{proposition}
\label{p:cat-vs-alg}
Assume that $\cA$ is Krull--Schmidt and additively finite. Then there exists an object $X\in\cA$ such that $\add{X}=\cA$, and for any such $X$ there is an equivalence of categories
\[\Mod{\cA}\isoto\Mod{\op{\End{\cA}{X}}},\]
restricting to an equivalence $\fpmod{\cA}\isoto\fpmod{\op{\End{\cA}{X}}}$ between the categories of finitely presented modules, an equivalence $\fd{\cA}\isoto\fd{\op{\End{\cA}{X}}}$ of the categories of (locally) finite-dimensional modules and, precomposing with the Yoneda functor, an equivalence $\cA\isoto\proj{\op{\End{\cA}{X}}}$.
Moreover, $\gldim{\cA}=\gldim{\op{\End{\cA}{X}}}$.\qed
\end{proposition}

\begin{remark}
The appearance of the opposite algebra in Proposition~\ref{p:cat-vs-alg} is also not a result of any convention concerning left or right modules.
Rather, it comes from our convention of reading algebra multiplication in the same direction as function composition (for us, right-to-left).
The alternative, in which algebra multiplication and function composition are read in opposite directions, can be sensible in some contexts (because it puts the commuting actions of $A$ and $\End{A}{M}$ on an $A$-module $M$ on opposite sides), but this would be confusing here since most of our algebras will have functions as elements.
\end{remark}

\begin{definition}[{\cite{Kelly}, see also \cite[Def.~A.3.3]{ASS-Book}, \cite[\S2]{KrauseKS}}]
\label{d:radcat}
Let $\cA$ be an additive category.
The ideal $\rad{\cA}$, called the \emph{radical} of $\cA$, is that for which $\radHom{\cA}{X}{Y}$ consists of the morphisms $f\colon X\to Y$ such that $\id_X-fg$ is invertible for all $g\colon Y\to X$.
\end{definition}

The ideal $\rad{\cA}$ is analogous to the Jacobson radical of an algebra; indeed, for any $X\in\cA$, the space $\radHom{\cA}{X}{X}$ coincides with the Jacobson radical of the endomorphism algebra of $X$.
It is also immediate from the definition that if $\cB\subset\cA$ is a full subcategory, then $\radHom{\cB}{X}{Y}=\radHom{\cA}{X}{Y}$ for all $X,Y\in\cB$.

When $\cA$ is Krull--Schmidt, the ideal $\rad{\cA}$ has a simpler description, with $\radHom{\cA}{X}{Y}$ consisting of the non-isomorphisms from $X$ to $Y$ when these objects are indecomposable \cite[Prop.~2.1(b)]{Bautista}.
Since $\radHom{\cA}{\blank}{\blank}$ is an additive bifunctor on the Krull--Schmidt category $\cA$, it can be computed on an arbitrary pair of objects using this description \cite[Lem.~3.4(b)]{ASS-Book}.

Let $M\in\Mod{\cA}$ and let $X,Z\in\cA$.
We write
\[M(Z)\radHom{\cA}{X}{Z}=\{M(\varphi)(m):m\in M(Z),\ \varphi\in\radHom{\cA}{X}{Z}\}.\]
It may seem more natural to write $\radHom{\cA}{X}{Z}M(Z)$ for this subspace, but our notation reflects the contravariance of $M$, as well as the fact that we compose functions right-to-left. Its advantage is made clearer when $M(Z)$ is itself a set of functions, as in \eqref{eq:radHom} below.

\begin{definition}
\label{d:radmod}
Let $\cA$ be a Krull--Schmidt category and $M$ an $\cA$-module.
Then the $\cA$-module $\rad{\cA}M$ is the subfunctor of $M$ defined by 
\begin{equation}
\label{eq:rad-module}
\rad{\cA} M(X)=\bigcup_{Z\in\cA}M(Z)\radHom{\cA}{X}{Z}
\end{equation}
on an object $X$, and on a morphism $\varphi\colon X\to Y$ by the restriction of $M(\varphi)\colon M(Y)\to M(X)$ to the appropriate subspaces.
\end{definition}

The fact that both $M$ and $\radHom{\cA}{\blank}{X}$ commute with finite direct sums (being $\bK$-linear functors) implies that $\rad{\cA}M(X)$ really is a subspace of $M(X)$:
\[M(Z)\radHom{\cA}{X}{Z}+M(Z')\radHom{\cA}{X}{Z'}\subset M(Z\dsum Z')\radHom{\cA}{X}{Z\dsum Z'}.\]
Moreover, $\rad{\cA}M(\varphi)$ is well-defined for $\varphi\colon X\to Y$, i.e.\ it takes values in $\rad{\cA}M(X)$, because $\radHom{\cA}{\blank}{\blank}$ is an ideal of $\cA$.

For objects $X$ and $Y$ in an additive category $\cA$, we define $\radHom[n]{\cA}{X}{Y}$ for $n\geq 2$ inductively by $\radHom[1]{\cA}{X}{Y}=\radHom{\cA}{X}{Y}$ and
\begin{equation}
\label{eq:radHom}
\radHom[n]{\cA}{X}{Y}=\bigcup_{Z\in\cA}\radHom[n-1]{\cA}{Z}{Y}\radHom{\cA}{X}{Z}.
\end{equation}
That is, an element of $\radHom[n]{\cA}{X}{Y}$ is a morphism which may be realised as a composition of $n$ morphisms from $\rad{\cA}$.
With this description, we may check that we also have
\[\radHom[n]{\cA}{X}{Y}=\bigcup_{Z\in\cA}\radHom{\cA}{Z}{Y}\radHom[n-1]{\cA}{X}{Z}.\]
By convention, we set $\radHom[0]{\cA}{X}{Y}=\Hom{\cA}{X}{Y}$.

We will frequently be interested in the case that $\cA$ is Krull--Schmidt and $\cB\subseteq\cA$ is a full and additively closed subcategory.
We then have $\radHom{\cB}{X}{Y}=\radHom{\cA}{X}{Y}$ for all $X,Y\in\cB$, but the inclusion $\radHom[n]{\cB}{X}{Y}\subseteq\radHom[n]{\cA}{X}{Y}$ can be strict for $n\geq2$.
For example, for a morphism to lie in $\radHom[2]{\cA}{X}{Y}$ it must be expressible as a composition $g\circ f$ for $f\in\radHom{\cA}{X}{Z}$ and $g\in\radHom{\cA}{Z}{Y}$, for some $Z\in\cA$, whereas to lie in $\radHom[2]{\cB}{X}{Y}$ we additionally require $Z\in\cB$.

For $M\in\Mod{\cA}$, one can check using \eqref{eq:rad-module} and \eqref{eq:radHom} that defining $\rad[n]{\cA}{M}$ recursively by $\rad[n]{\cA}{M}=\rad{\cA}\rad[n-1]{\cA}M$ (with $\rad[0]{\cA}{M}=M$) is equivalent to defining it directly by replacing $\radHom{\cA}{\blank}{\blank}$ by $\radHom[n]{\cA}{\blank}{\blank}$ in Definition~\ref{d:radmod}.
As a special case,
\begin{equation}
\label{eq:radfun}
\radfun[n]{\cA}X=\radHom[n]{\cA}{\blank}{X}
\end{equation}
for any $X\in\cA$ and $n\in\nat$.
As in the previous paragraph, if $M\in\Mod{\cA}$ and $\cB\subseteq\cA$ is a full and additively closed subcategory then we have an inclusion $\rad[n]{\cB}(M|_{\cB})\subseteq(\rad[n]{\cA}M)|_{\cB}$, but typically this is strict (even for $n=1$).

As a final extension of this notation, we write
\[\radHom[\infty]{\cA}{X}{Y}=\bigcap_{n\in\nat}\radHom[n]{\cA}{X}{Y},\quad \rad[\infty]{\cA}M(X)=\bigcap_{n\in\nat}\rad[n]{\cA}M(X)\]
for any $X,Y\in\cA$ and any $M\in\Mod{\cA}$.

\begin{example}
\label{eg:Kronecker}
The infinite radical $\radHom[\infty]{\cA}{X}{Y}$ can be non-trivial even when $\cA$ is a fairly benign category.
Indeed, when $\cA=\fpmod{A}$ is the category of finite-dimensional modules over a finite-dimensional algebra, the vanishing of $\radHom[\infty]{\cA}{\blank}{\blank}$ is equivalent to $A$ being representation-finite \cite[Thm.~3.1]{Auslander-RTAA2} (see also \cite[Cor.~1.8]{KS}), that is, to $\fpmod{A}$ being additively finite.
In particular, $\radHom[\infty]{\cA}{\blank}{\blank}\ne0$ whenever $\cA=\fpmod\bK Q$ for $Q$ a non-Dynkin acyclic quiver.
\end{example}

\begin{definition}
\label{d:simple-functor}
Assume $\cA$ is a Krull--Schmidt category. For each $X\in\indec{\cA}$, we write $\simpmod{\cA}{X}=\Yonfun{\cA}{X}/\radfun{\cA}{X}$.
\end{definition}

If $X,Y\in\indec{\cA}$ with $X\not\iso Y$, then $\radHom{\cA}{Y}{X}=\Hom{\cA}{Y}{X}$, so $\simpmod{\cA}{X}(Y)=0$. On the other hand, $\simpmod{\cA}{X}(X)=\Hom{\cA}{X}{X}/\radHom{\cA}{X}{X}\ne0$, because $\id_X$ represents a non-zero element.

\begin{proposition}[{\cite[Prop.~2.1(f)]{Bautista}}]
\label{p:KS-simples}
For a Krull--Schmidt category $\cA$, the representation $\simpmod{\cA}{X}$ is a simple object of $\Mod{\cA}$ for any $X\in\indec{\cA}$---that is, it has no proper subobjects---and all simple objects of $\Mod{\cA}$ are of this form.\qed
\end{proposition}

\subsection{Pseudocompactness}
\label{ss:pseudocompact}

Since there are notable examples (e.g.~\cite{Plamondon-ClustCat,JKS,PresslandPostnikov,KellerWu}) of Hom-infinite categorifications of cluster algebras, we do not restrict to Hom-finite categories in this paper.
This requires some additional technicalities, which can be safely ignored in the Hom-finite case.

\begin{definition}
\label{d:pseudocompact}
A \emph{topological $\bK$-vector space} is a $\bK$-vector space $V$ equipped with a topology for which addition and scalar multiplication are continuous.
We say that $V$ is \emph{pseudocompact} if it is Hausdorff and has a system $(U_i)_{i \in I}$ of open subspaces such that $\dim_{\bK}V/U_i<\infty$ for all $i\in I$ and the natural map $V\to\colim_I V/U_i$ is an isomorphism.
In this case we say that the system $(U_i)_{i\in I}$ \emph{exhibits} the pseudocompactness of $V$.

The category $\pcMod{\bK}$ has pseudocompact vector spaces as objects, with morphisms given by continuous linear maps.
We say that a $\bK$-linear category $\cA$ is pseudocompact if it is enriched in $\pcMod{\bK}$, that is, its Hom-spaces are pseudocompact topological vector spaces, and composition of morphisms is continuous.
A pseudocompact $\cA$-module is a contravariant functor $\cA\to\pcMod{\bK}$, and we denote the category of such by $\pcMod{\cA}$.
\end{definition}

In a pseudocompact $\bK$-linear category, the endomorphism algebra of any object is a pseudocompact (unital) $\bK$-algebra, as surveyed in \cite{IMacQ-survey}.
Further information concerning pseudocompact algebras and categories can be found in \cite{Gabriel-AbCats,Brumer,Simson,VdB-CY,KellerYang}.

\begin{example}
Let $Q$ be a (possibly infinite) quiver.
We define the $\bK$-linear \emph{complete path category} $\cpa{\bK}{Q}$ to be the Krull--Schmidt category whose indecomposable objects are the vertices of $Q$, with morphisms between them defined as follows.
First, for a pair of vertices $v,w\in Q_0$, let $P(v,w)$ be the $\bK$-vector space consisting of finite $\bK$-linear combinations of paths in $Q$ from $v$ to $w$, and let $J^n(v,w)\leq P(v,w)$ be the subspace spanned by those paths consisting of at least $n$ arrows. We may then define
\[\Hom{\cpa{\bK}{Q}}{v}{w}={\colim}_nP(v,w)/J^n(v,w).\]
Composition is induced from concatenation of paths by continuous extension.

If we assume additionally that $Q$ is locally finite, meaning that each vertex is incident with only finitely many arrows, then $P(v,w)/J^n(v,w)$ is finite-dimensional for all $v$, $w$ and $n$, its dimension being given by the number of paths from $v$ to $w$ of length at most $n$.
In this case $\cpa{\bK}{Q}$ is thus, by construction, a pseudocompact $\bK$-linear category in which the pseudocompactness of each $\Hom{\cpa{\bK}{Q}}{v}{w}$ is exhibited by the system $J^n(v,w)=\radHom[n]{\cpa{\bK}{Q}}{v}{w}$, for $n\geq0$.

If $Q$ is a finite quiver, then $\cpa{\bK}{Q}$ is nothing but the idempotent completion of the complete path algebra $A$ of $Q$, considering $A$ as a $\bK$-linear category with one object.
In other words, $\cpa{\bK}{Q}$ is equivalent to the category $\proj{\op{A}}$.

The definition of $\cpa{\bK}{Q}$ above may give unexpected (and arguably undesirable) results when applied to quivers which are not locally finite.
For example, if $Q$ has one vertex and a countably-infinite number of loops, then $\cpa{\bK}{Q}$ is not the power series ring on a countably-infinite number of variables (cf.\ \cite[Eg.~2.14]{IMacQ-survey}), but rather the subalgebra consisting of power series with only finitely many terms in each degree.
\end{example}

The Hausdorff condition in the definition of a pseudocompact vector space is sometimes omitted (and, on the other hand, sometimes already made part of the definition of a topological vector space).
For us it will be useful because of the next proposition.
It is equivalent to requiring that $\{0\}\subseteq V$ is closed, since any topological vector space $V$ has the property that $V/\overline{\{0\}}$ is Hausdorff in the quotient topology.

\begin{proposition}
\label{p:pc-abelian}
If $\cA$ is a $\bK$-linear category, then $\pcMod{\cA}$ is abelian (with kernels and cokernels computed pointwise).
\end{proposition}
\begin{proof}
This reduces to the fact that $\pcMod{\bK}$ is an abelian category.
By \cite[\S II.27]{Lefschetz} (see also \cite[Lem.~2.2]{IMacQ}), the image of a continuous linear map between pseudocompact vector spaces is closed, and hence the quotient by this image admits a natural topology in which the projection is continuous---we use here that pseudocompact vector spaces are Hausdorff.
One may then check that the cokernel is pseudocompact (see Proposition~\ref{p:fp-radpc} below for the style of argument), and so $\pcMod{\bK}$ admits cokernels (unlike the full category of topological vector spaces).
Verifying the remaining conditions is more straightforward.
\end{proof}

\begin{remark}
\label{r:open-int}
The kernel of the natural map $V\to\colim_IV/U_i$ is the intersection $\bigcap_{i\in I}U_i$, so that a necessary condition for a system $(U_i)_{i\in I}$ to exhibit pseudocompactness of $V$ is that $\bigcap_{i\in I}U_i=\{0\}$. For any $i,j\in I$ we have $\dim_\bK(U_i\cap U_j)\leq\dim_{\bK}U_i$.
Thus, if $V$ is finite-dimensional, there is always a finite subset $J\subseteq I$, of cardinality at most $\dim_{\bK}(V)$, such that $\bigcap_{i\in I}U_i=\bigcap_{j\in J}U_j$ is open.
In particular, if $V$ is finite-dimensional and pseudocompact, then $\{0\}\subseteq V$ is open (as well as closed).

As a corollary, if $V$ is pseudocompact, $W\leq V$ is closed and $\dim_{\bK}V/W<\infty$, then $W$ is also open, because it is the preimage under the quotient map of the open set $\{0\}$ in the pseudocompact vector space $V/W$.
\end{remark}

The category of ordinary (non-topological) vector spaces admits a fully-faithful embedding into the category of topological vector spaces by equipping each object with the discrete topology (for which linear maps, like all functions, are continuous).
This restricts to a fully-faithful embedding $\fd{\bK}\to\pcMod{\bK}$---that is, it makes every finite-dimensional vector space pseudocompact---since for a finite-dimensional vector space with the discrete topology the single open set $\{0\}$ exhibits pseudocompactness.
As a result, any Hom-finite $\bK$-linear category becomes pseudocompact on equipping all of its Hom-spaces with the discrete topology.
In the same way, we have a fully-faithful embedding $\lfd{\cA}\to\pcMod{\cA}$ for any $\bK$-linear category $\cA$.

The definition of a pseudocompact $\cA$-module does not strictly require $\cA$ itself to be pseudocompact.
However, under this assumption on $\cA$, the Yoneda functor $\Yonfun{\cA}$ takes values in $\pcMod{\cA}$ (and $\opYonfun{\cA}$ takes values in $\pcMod{\op{\cA}}$).
By Proposition~\ref{p:pc-abelian}, it then follows that every finitely presented $\cA$-module is pseudocompact.

Instead of using the discrete topology, we can also enrich any $\bK$-linear category $\cA$ in topological vector spaces by equipping each Hom-space $\Hom{\cA}{X}{Y}$ with the topology whose basis of open neighbourhoods of $0$ is given by the powers $\radHom[n]{\cA}{X}{Y}$ of the radical.
The fact that this leads to continuous composition may be proved analogously to the fact that multiplication is continuous in the $p$-adic topology on $\integ$ (or in the $I$-adic topology on any ring with ideal $I$), using that $\rad{\cA}$ is an ideal of $\cA$.
We call this the \emph{radical topology}\footnote{despite the tempting possibilities `rad-adic topology' or `r-adic-al topology'} on $\cA$.
In exactly the same way, we define the radical topology on an $\cA$-module $M$ to be that with basis of open neighbourhoods $\rad[n]{\cA}{M}$.

\begin{example}
\label{eg:add-fin-pc}
Continuing Proposition~\ref{p:cat-vs-alg}, if $\cA=\add{X}$ is an additively finite and Krull--Schmidt $\bK$-linear category, then a choice of topology on the Hom-spaces of $\cA$ making composition continuous is equivalent to the choice of a topology on the algebra $A=\op{\End{\cA}{X}}$ making multiplication continuous.
As a special case, the radical topology on $\cA$ corresponds to the $J$-adic topology on $A$, for $J=\rad{}{A}$ the Jacobson radical; we call this the radical topology on $A$.
The category $\cA$ is pseudocompact for a given topology if and only if $A$ is a pseudocompact algebra in the corresponding topology.
\end{example}

\begin{definition}\label{d:rad-pc}
We say that a category, algebra or module is \emph{radically pseudocompact} if it is pseudocompact in the radical topology.
\end{definition}

Unfortunately, even if $\cA$ is pseudocompact in some topology, it may fail to be radically pseudocompact.
Indeed, pseudocompactness in the radical topology requires the infinite radical to vanish as in Remark~\ref{r:open-int}, because $\radHom[\infty]{\cA}{X}{Y}$ is the intersection of all open neighbourhoods of $0\in\Hom{\cA}{X}{Y}$ in this topology.
We saw in Example~\ref{eg:Kronecker} that the infinite radical can be non-zero even for Hom-finite categories, which are always pseudocompact in the discrete topology. However, under some reasonable conditions we may deduce that a pseudocompact category is also radically pseudocompact.

\begin{definition}
\label{d:loc-finite}
A $\bK$-linear category $\cA$ is \emph{locally finite at $X\in\cA$} if $\Yonfun{\cA}{X}/\rad[2]{\cA}\Yonfun{\cA}{X}\in\fd{\cA}$ and $\opYonfun{\cA}{X}/\rad[2]{\cA}\opYonfun{\cA}{X}\in\fd{\op{\cA}}$.
We call $\cA$ \emph{locally finite} if it is locally finite at all of its objects.
\end{definition}

This is compatible with the corresponding terminology for quivers; cf.\ Proposition~\ref{p:loc-finite-quiver}.
If $\cA$ is Krull--Schmidt, then it is locally finite if and only if it is locally finite at each of its indecomposable objects.

Because the inclusion $\radfun[2]{\cB}{X}\subseteq\radfun[2]{\cA}{X}$ can be strict, for $\cB\subset\cA$ full and additively closed, it is possible that $\cA$ is locally finite but $\cB$ is not; indeed, this can happen even when $\cA$ is additively finite.
For example, let $\cA$ be the complete path category of the locally finite quiver
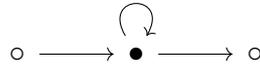
\[\begin{tikzcd}
\circ\arrow{r}&\bullet\ar[looseness=6, out=120, in=60]\arrow{r}&\circ
\end{tikzcd}\]
and let $\cB$ be the full and additively closed subcategory generated by the two white vertices.
In this case, there is an infinite-dimensional space of morphisms between the two indecomposable objects of $\cB$, and while every morphism between these two objects lies in $\rad[2]{\cA}$, we have $\rad[2]{\cB}=0$.

\begin{proposition}
\label{p:rad-compact-finite}
Assume that $\cA$ is pseudocompact, locally finite, additively finite and Krull--Schmidt.
Then $\cA$ is radically pseudocompact.
\end{proposition}
\begin{proof}
In this case $\cA=\add(X)$ is equivalent to the category $\proj{A}$ for the pseudocompact unital algebra $A=\op{\End{\cA}{X}}$.
Thus $\cA$ is radically pseudocompact if and only if $A$ is, and this follows from the local finiteness assumption as in \cite[Prop.~2.7]{IMacQ}.
\end{proof}

\subsection{Approximations}
\label{s:approximations}

\begin{definition}
Let $\cA$ be an additive category, let $\cB\subset\cA$ be additively closed, and let $X\in\cA$.
A \emph{right $\cB$-approximation} of $X$ is a map $\varphi\colon B\to X$ such that $B\in\cB$ and $\Yonfun{\cB}{\varphi}\colon\Yonfun{\cB}{B}\to\Yonfun{\cB}{X}$ is surjective.
Dually, a \emph{left $\cB$-approximation} of $X$ is a map $\psi\colon X\to B$ such that $B\in\cB$ and $\opYonfun{\cB}{\varphi}\colon\opYonfun{\cB}{X}\to\opYonfun{\cB}{B}$ is surjective. We say $\cB$ is \emph{covariantly finite} (in $\cA$) if every $X\in\cA$ has a right $\cB$-approximation, \emph{contravariantly finite} if every $X\in\cA$ has a left $\cB$-approximation, and \emph{functorially finite} if both properties hold.
\end{definition}

\begin{definition}
Let $\cA$ be an additive category.
A \emph{sink map} for $X\in\cA$ is a map $\varphi\colon Y\to X$ in $\cA$ such that the sequence
\begin{equation}
\label{eq:sink-map}
\begin{tikzcd}
\Yonfun{\cA}Y\arrow{r}{\Yonfun{\cA}\varphi}&\radHom{\cA}{\blank}{X}\arrow{r}&0
\end{tikzcd}
\end{equation}
of $\cA$-modules is exact, i.e.\ such that the image of the natural transformation $\Yonfun{\cA}{\varphi}$ is the subfunctor $\radHom{\cA}{\blank}{X}\leq\Yonfun{\cA}{X}$.
Dually, a \emph{source map} for $X\in\cA$ is a map $\psi\colon X\to Y$ such that the sequence
\begin{equation}
\label{eq:source-map}
\begin{tikzcd}
\opYonfun{\cA}{Y}\arrow{r}{\opYonfun{\cA}\psi}&\radHom{\cA}{X}{\blank}\arrow{r}&0
\end{tikzcd}
\end{equation}
of $\op{\cA}$-modules is exact.
We emphasise that when applying these definitions to a full subcategory $\cB\subseteq\cA$, we use the Yoneda functors $\Yonfun{\cB}$ and $\opYonfun{\cB}$ only on $\cB$---that is, the object $Y$ should be chosen in $\cB$, and sequences \eqref{eq:sink-map} and \eqref{eq:source-map} are only required to be exact as sequences of $\cB$-modules.
\end{definition}

\begin{remark}
The terminology of sink and source maps is taken from \cite{IyamaYoshino}, although here we do not require such maps to be minimal.
A right $\cB$-approximation is sometimes called a $\cB$-precover, and a left $\cB$-approximation a pre-envelope, with the prefix `pre' being dropped if the map is additionally minimal.
\end{remark}

The object $X\in\cA$ has a right $\cB$-approximation if and only if its image $\Yonfun{\cB}X$ under the contravariant Yoneda functor is finitely generated, because any surjection $\Yonfun{\cB}{B}\to\Yonfun{\cB}{X}$ with $B\in\cB$ must be of the form $\Yonfun{\cB}\varphi$ for some $\varphi\colon B\to X$, which is then a right $\cB$-approximation.
Dually, $X$ has a left $\cB$-approximation if and only if $\opYonfun{\cB}{X}$ is finitely generated.
This is the origin of the terminology of covariantly and contravariantly finite subcategories, due to Auslander and Smalø \cite{AS80,AS81}; note that this terminology refers not to the variance of the Yoneda functors, but to the (opposite) variance of their values as functors on $\cB$.
The existence of sink and source maps may be phrased similarly, as follows.

\begin{proposition}
\label{p:sink-v-fp-simp}
Let $\cA$ be a Krull--Schmidt category, and let $X\in\cA$.
Then the simple $\cA$-module $\simpmod{\cA}{X}$ is finitely presented if and only if $X$ admits a sink map, and the simple $\op{\cA}$-module $\simpmod{\op{\cA}}{X}$ is finitely presented if and only if $X$ admits a source map.
\end{proposition}
\begin{proof}
We give the proof for sink maps, that for source maps being completely analogous.
If $\varphi\colon R\to X$ is a sink map, then the sequence
\[\begin{tikzcd}
\Yonfun{\cA}R\arrow{r}{\Yonfun{\cA}\varphi}&\Yonfun{\cA}{X}\arrow{r}&\simpmod{\cA}{X}\arrow{r}&0
\end{tikzcd}\]
is exact, since $\image({\Yonfun{\cA}{\varphi}})=\radfun{\cA}{X}$ is the kernel of the projection $\Yonfun{\cA}{X}\to\simpmod{\cA}{X}$.
Thus, it is a projective presentation of $\simpmod{\cA}{X}$, which is therefore finitely presented.

Conversely, since the minimal projective cover of $\simpmod{\cA}{X}$ is $\Yonfun{\cA}{X}$, if $\simpmod{\cA}{X}$ is finitely presented then by Proposition~\ref{p:Yoneda-image} it has a projective presentation of the form
\[\begin{tikzcd}
\Yonfun{\cA}R\arrow{r}{\Yonfun{\cA}\varphi}&\Yonfun{\cA}{X}\arrow{r}&\simpmod{\cA}{X}\arrow{r}&0
\end{tikzcd}\]
for some $\varphi\colon R\to X$.
It follows that the image of $\Yonfun{\cA}\varphi$ is precisely $\radfun{\cA}{X}$, and hence $\varphi\colon R\to X$ is a sink map.
\end{proof}

For $\cA$ an additive category and $X\in\cA$, write
\[\divalg{X}\defeq\op{\End{\cA}{X}}/\rad{}\op{\End{\cA}{X}}, \quad
\dimdivalg{X}=\dim\divalg{X}.\]
While we may have $\dimdivalg{X}=\infty$ in general, mild additional assumptions on $\cA$ will imply that $\dimdivalg{X}<\infty$; for example, this follows immediately if $\cA$ is radically pseudocompact, or locally finite at $X$.
If $X\in\cB\subseteq\cA$ for some full subcategory $\cB$, then $\End{\cA}{X}=\End{\cB}{X}$ and so the algebra $\divalg{X}$, and its dimension $\dimdivalg{X}$, are insensitive to whether we consider $X$ to be an object of $\cA$ or of $\cB$.
If $\cA$ is Krull--Schmidt and $X$ is indecomposable, then $\divalg{X}=\simpmod{\cA}{X}(X)$ is an associative division algebra, or skew-field, over $\bK$, because $\op{\End{\cA}{X}}$ is local.

\begin{remark}\label{r:field-dX}
The choice of field $\bK$ gives rise to constraints on the possible values of $\dimdivalg{X}$ for indecomposable objects $X$ in a $\bK$-linear Krull--Schmidt category $\cA$.
For example, if $\bK$ is algebraically closed and $\dimdivalg{X}<\infty$ then $\divalg{X}=\bK$, and so $\dimdivalg{X}=1$, for any indecomposable object $X$, since $\bK$ is the only finite-dimensional division algebra over $\bK$ in this case. However, if $\bK=\real$, then for each indecomposable object $X$ with $\dimdivalg{X}<\infty$ we have $\dimdivalg{X}\in\{1,2,4\}$, since $\divalg{X}\in\{\real,\complex,\mathbb{H}\}$ by the Frobenius theorem \cite{Frobenius}.
\end{remark}

We now describe minimal approximations in pseudocompact categories, in the case that the ground field $\bK$ is perfect: this means that if $\bK$ has positive characteristic $p$, then every element of $\bK$ is a $p$-th power.
In particular, any field with characteristic $0$ is perfect, as is any algebraically closed field.
Under this assumption, we may use the following proposition to see that the algebra $\divalg{X}$ can be made to act on any $\op{\End{\cA}{X}}$-module. While the action is not canonical, this does not affect our results.

\begin{proposition}
\label{p:Wedderburn-splitting}
Let $\cA$ be a Krull--Schmidt $\bK$-linear category, for $\bK$ a perfect field, and let $X\in\cA$. If $\op{\End{\cA}{X}}$ is pseudocompact and $\dimdivalg{X}<\infty$, then $\op{\End{\cA}{X}}\to\divalg{X}$ splits as an algebra homomorphism.
\end{proposition}

\begin{proof}
By a result of Iusenko and MacQuarrie \cite[Thm.~4.6]{IMacQ-survey} (based on \cite[Thm.~2.3.11]{Abe}), it suffices to show that $D_X=\op{\End{\cA}{X}}/\rad{}\op{\End{\cA}{X}}$ is separable.
But this holds since it is a finite-dimensional algebra over the perfect field $\bK$.
\end{proof}
\begin{remark}
When $\dim_{\bK}\op{\End{\cA}{X}}$ is finite-dimensional, Proposition~\ref{p:Wedderburn-splitting} is just the Wedderburn principal theorem, and when $\add{X}$ is locally finite it is due to Curtis \cite{Curtis}.

While the splitting from Proposition~\ref{p:Wedderburn-splitting} is not generally unique, any two splittings are conjugate, by an element of the form $1-x$ with $x\in\rad{}\op{\End{\cA}{X}}$; this is \cite[Thm.~17]{Eckstein} (see also \cite[Thm.~4.7]{IMacQ-survey}) in this generality, and due to Malcev \cite{Malcev} for finite-dimensional algebras.
In particular, if $\op{\End{\cA}{X}}$ is commutative, so $\divalg{X}$ is a product of field extensions of $\bK$, then there is a unique splitting of the form required by Proposition~\ref{p:Wedderburn-splitting}.
\end{remark}

If $\cA$ is $\bK$-linear for $\bK$ a perfect field, $\cB\subset\cA$ is additively closed and radically pseudocompact, and $X\in\cA$ has the property that the $\cB$-module $\Yonfun{\cB}{X}$ is radically pseudocompact, then in particular $\Yonfun{\cB}{X}/\rad{\cB}\Yonfun{\cB}{X}$ is finite-dimensional.
Moreover, $\dimdivalg{B}<\infty$ for all $B\in\cB$, since this follows from the radical pseudocompactness of $\cB$, so we may choose $\divalg{B}$-linear splittings of the quotient maps $\op{\End{\cA}{B}}\to\divalg{B}$ and $\Yonfun{\cB}{X}(B)\to(\Yonfun{\cB}{X}/\radfun{\cB}{X})(B)$ by Proposition~\ref{p:Wedderburn-splitting}.
The first of these splittings puts a right $\divalg{B}$-module structure on each object $B\in\cB$.
This allows us to define a map
\begin{equation}
\label{eq:right-app}
r\colon\bigdsum_{B\in\indec{\cB}}B\otimes_{\divalg{B}}\bigl((\Yonfun{\cB}{X}/\radfun{\cB}{X})(B)\bigr)\to X,
\end{equation}
where, on the summand of the domain indexed by $B\in\indec{\cB}$, we have $r(b\otimes\bar{\varphi})=\varphi(b)$, where $\bar{\varphi}\mapsto\varphi$ splits the quotient map $\Yonfun{\cB}{X}(B)\to(\Yonfun{\cB}{X}/\radfun{\cB}{X})(B)$.

Dually, if the $\op{\cB}$-module $\opYonfun{\cB}{X}$ is radically pseudocompact, we may similarly define
\begin{equation}
\label{eq:left-app}
\ell\colon X\to\bigdsum_{B\in\indec{\cB}}\bigl((\opYonfun{\cB}{X}/\rad{\op{\cB}}\opYonfun{\cB}{X})(B)\bigr)\otimes_{\op{\divalg{B}}}B
\end{equation}
by $\ell(x)=\sum\bar{\varphi}\otimes\varphi(x)$, where the sum is over a union of bases $\{\bar{\varphi}\}$ of the  various spaces $(\opYonfun{\cB}{X}/\rad{\op{\cB}}\opYonfun{\cB}{X})(B)$ appearing in the codomain, and $\bar{\varphi}\mapsto\varphi$ is a choice of splitting.
\begin{lemma}
\label{l:approx-construction}
Let $\cA$ be a pseudocompact Krull--Schmidt $\bK$-category for $\bK$ a perfect field, let $\cB\subseteq\cA$ be full, additively closed and radically pseudocompact, and let $X\in\cA$.
If the $\cB$-module $\Yonfun{\cB}{X}$ is radically pseudocompact, then the map $r$ from \eqref{eq:right-app} is a minimal right $\cB$-approximation of $X$.
Dually, if the $\op{\cB}$-module $\opYonfun{\cB}{X}$ is radically pseudocompact, then the map $\ell$ from \eqref{eq:left-app} is a minimal left $\cB$-approximation of $X$.
\end{lemma}

\begin{proof}
We give the proof for $r$, that for $\ell$ being similar.
Because $\Yonfun{\cB}X$ is radically pseudocompact, we have in particular that $\Yonfun{\cB}{X}/\rad{\cB}\Yonfun{\cB}{X}\in\fd{\cB}$.
The domain of $r$ is thus a finite direct sum, and hence a well-defined object of $\cB$.

So let $B\in\cB$, and let $f\in\Hom{\cA}{B}{X}$.
To see that $r$ is a right $\cB$-approximation, we need to construct a map $\hat{f}\colon B\to R$ such that $f=r\hat{f}$.
To do this, we first claim that there are maps $\hat{f}_n\colon B\to R$, for each $n\geq 1$, such that $f-r\hat{f}_n\in\radfun[n]{\cB}{X}$ and, if $n\geq2$, we have $\hat{f}_n-\hat{f}_{n-1}\in\radHom[n-1]{\cB}{B}{R}$.
\medskip

To prove the claim, let $\bar{f}\in(\Yonfun{\cB}{X}/\radfun{\cB}{X})(B)$ be the projection of $f$.
In the case that $B$ is indecomposable, we may then define $\hat{f}_1\colon B\to R$, with image contained in the summand indexed by $B$, by $\hat{f}_1(b)=b\otimes\bar{f}$.
In general, we may use the Krull--Schmidt property to decompose $B$ into indecomposable summands and define $\hat{f}_1$ componentwise in the same way.
Then, by construction, $r\hat{f}_1$ also projects to $\bar{f}\in(\Yonfun{\cB}{X}/\radfun{\cB}{X})(B)$, and so $f-r\hat{f}_1\in\radfun{\cB}{X}(B)$, as required.

Now assume we have defined $\hat{f}_N$ with the required properties for some $N\geq1$.
Since $f-r\hat{f}_N\in\radfun[N]{\cB}{X}(B)$, there is an object $B'\in\cB$ such that $f-r\hat{f}_N=h g$ for some $g\in\radHom[N]{\cB}{B}{B'}$ and $h\in\Hom{\cA}{B'}{X}$. Construct $\hat{h}\colon B'\to R$ from $h$ in the same way as $\hat{f}_1$ was constructed from $f$, so that $h-r\hat{h}\in\radfun{\cB}{X}(B')$, and let $\hat{f}_{N+1}=\hat{f}_N+\hat{h} g$.

Now $f-r\hat{f}_{N+1}=f-r\hat{f}_N-r\hat{h} g=(h-r\hat{h}) g$, and this element lies in $\radfun[N+1]{\cB}{X}(B)$ since $g\in\radHom[N]{\cB}{B}{B'}$ and $h-r\hat{h}\in\radfun{\cB}{X}(B')$.
Moreover, $\hat{f}_{N+1}-\hat{f}_N=\hat{h} g\in\radHom[N]{\cB}{B}{R}$, and so $\hat{f}_{N+1}$ satisfies the necessary properties, proving the claim.
\medskip

Since $\cB$ is radically pseudocompact, the fact that $\hat{f}_{n}-\hat{f}_{n-1}\in\radHom[n-1]{\cB}{B}{R}$ for all $n\in\nat$ implies that there is a unique map $\hat{f}\colon B\to R$ such that $\hat{f}-\hat{f}_n\in\radHom[n]{\cB}{B}{R}$ for all $n$.
Postcomposing with $r$, we see that $r\hat{f}-r\circ\hat{f}_n\in\radfun[n]{\cB}{X}(B)$.
Since we also have $f-r\hat{f}_n\in\radfun[n]{\cB}{X}(B)$, it follows that $r\hat{f}-f\in\radfun[n+1]{\cB}{X}(B)$ for all $n\in\nat$.
Since $\Yonfun{\cB}{X}$ is radically pseudocompact, we have $\bigcap_{n\in\nat}\radfun[n]{\cB}{X}(B)=0$, and it follows that $r\hat{f}=f$.
Hence, $r$ is a right $\cB$-approximation.

To see that $r$ is minimal, we first show that no indecomposable summand of $R$ lies in $K=\ker{r}$.
Observe first that any such summand has the form $R'=\{b\otimes\bar{\varphi}:b\in B\}$ for a fixed $B\in\indec{\cB}$ and non-zero $\bar{\varphi}\in(\Yonfun{\cB}{X}/\radfun{\cB}{X})(B)$.
Let $\varphi\in\Yonfun{\cB}{X}(B)$ be the image of $\bar{\varphi}$ under the chosen splitting.
If $r(R')=0$ then $\varphi(b)=0$ for all $b\in B$, that is $\varphi=0$. But then $\bar{\varphi}=0$, and so $R'=0$.

The kernel $k\colon K\to R$ of $r$ thus lies in $\radHom{\cA}{K}{R}$.
If $r\alpha=r$ for some $\alpha\colon R\to R$, it follows that $r(1-\alpha)=0$, so $1-\alpha=ks$ for some $s\colon R\to K$.
Now $ks\in\radHom{\cA}{R}{R}=\rad{}\op{\End{\cA}{R}}$, and so $\alpha=1-ks$ is invertible by definition of the Jacobson radical.
\end{proof}

\begin{corollary}
\label{c:approx-count}
Under the assumptions of Lemma~\ref{l:approx-construction}, assume that $R\to X$ is a minimal right $\cB$-approximation of an indecomposable $X$, and let $[R:B]$ denotes the multiplicity of $B\in\indec{\cB}$ as a summand of $R$.
Then if $\dimdivalg{X}<\infty$, we have
\begin{equation}
\label{eq:d-divisibility}
\dimdivalg{X}\divides\dimdivalg{B}[R:B].
\end{equation}
\end{corollary}

\begin{proof}
The vector space $(\Yonfun{\cB}{X}/\radfun{\cB}{X})(B)$ is finite-dimensional because $\Yonfun{\cB}{X}$ is radically pseudocompact, and it is both a left $\divalg{B}$-module and right $\divalg{X}$-module.
The $\divalg{X}$-module structure is via a choice of splitting as in Proposition~\ref{p:Wedderburn-splitting}; here we use that $\dimdivalg{X}<\infty$.
On the other hand, the $\divalg{B}$-module structure is choice-free, since $\rad{}\op{\End{\cA}{B}}$ is in the kernel of the left action of $\op{\End{\cA}{B}}$ on $(\Yonfun{\cB}{X}/\radfun{\cB}{X})(B)$.
Because $\divalg{B}$ and $\divalg{X}$ are division algebras, since both $B$ and $X$ are indecomposable, both of these module structures are free.

Since any two minimal right $\cB$-approximations have isomorphic domains, we have
\[[R:B]=\rank_{\divalg{B}}\bigl((\Yonfun{\cB}{X}/\radfun{\cB}{X})(B)\bigr)\]
by Lemma~\ref{l:approx-construction}.
Since $(\Yonfun{\cB}{X}/\radfun{\cB}{X})(B)$ is free over $\divalg{B}$ it follows that
\[\dimdivalg{B}[R:B]=\dim_{\bK}\bigl((\Yonfun{\cB}{X}/\radfun{\cB}{X})(B)\bigr).\]
As $(\Yonfun{\cB}{X}/\radfun{\cB}{X})(B)$ is also free over $\divalg{X}$, this dimension is divisible by $\dimdivalg{X}$.
\end{proof}

We also state a version of Lemma~\ref{l:approx-construction} for sink and source maps, but omit the proof since it is almost identical.

\begin{lemma}
\label{l:sink-map-construction}
Let $\cA$ be a Krull--Schmidt $\bK$-category for $\bK$ a perfect field.
Assume $\cA$ is radically pseudocompact and locally finite at $X$.
Then $X$ admits a minimal sink map
\[r\colon\bigdsum_{A\in\indec{\cA}}A\otimes_{\divalg{A}}\bigl(\radHom{\cA}{A}{X}/\radHom[2]{\cA}{A}{X}\bigr)\to X,\]
defined analogously to \eqref{eq:right-app}, and a minimal source map
\[\ell\colon X\to\bigdsum_{A\in\indec{\cA}}\bigl(\radHom{\cA}{X}{A}/\radHom[2]{\cA}{X}{A}\bigr)\otimes_{\op{\divalg{A}}}A\]
analogous to \eqref{eq:left-app}.\qed
\end{lemma}

When applying Lemma~\ref{l:approx-construction} in practice, the following result is useful.

\begin{proposition}
\label{p:fp-radpc}
If $\cA$ is radically pseudocompact, then so is any $M\in\fpmod{\cA}$.
\end{proposition}
\begin{proof}
Let $M\in\fpmod{\cA}$, with projective presentation $P_1\to P_0\to M\to 0$. Then for any $n\in\nat$, there is an exact sequence
\[\begin{tikzcd}
P_1/\rad[n]{\cA}P_1\arrow{r}&P_0/\rad[n]{\cA}P_0\arrow{r}&M/\rad[n]{\cA}M\arrow{r}&0.
\end{tikzcd}\]
Since $\cA$ is radically pseudocompact, so is the projective $\cA$-module $P_0$ (which lies in the image of $\Yonfun{\cA}$), and so it follows that $\dim_{\bK}M/\rad[n]{\cA}M\leq\dim_{\bK}P_0/\rad[n]{\cA}P_0<\infty$.
Moreover, the terms in this exact sequence form inverse systems which have surjective morphisms, and so in particular satisfy the Mittag-Leffler condition.
The sequence thus remains exact under taking colimits, and so we see from the commutative diagram
\[\begin{tikzcd}[row sep=15pt]
P_1\arrow{r}\arrow{d}{\vsim}&P_0\arrow{r}\arrow{d}{\vsim}&M\arrow{r}\arrow{d}&0\\
\colim_nP_1/\rad[n]{\cA}P_1\arrow{r}&\colim_nP_0/\rad[n]{\cA}P_0\arrow{r}&\colim_nM/\rad[n]{\cA}M\arrow{r}&0.
\end{tikzcd}\]
that $M\isoto\colim_nM/\rad[n]{\cA}M$ is radically pseudocompact.
\end{proof}

\sectionbreak

\footnotesize

\bibliographystyle{habbrv}
\bibliography{biblio}\label{references}

\end{document}